%% file: main.tex
\documentclass[11pt,reqno]{amsart}
\input{library}

\begin{document}

\title[\PaperShortTitle]{\PaperTitle}
\author{\PaperAuthor}
\address{\PaperAddress}
\email{\PaperEmail}
\input{sections/00-abstract}
\maketitle

\setcounter{tocdepth}{1}
\tableofcontents

\input{sections/01-introduction}

\part{Existence of solutions with a prescribed number of zeros}
\input{sections/02-setting}
\input{sections/03-well-posedness}
\input{sections/04-intermediate-properties}
\input{sections/05-scaled-limit}
\input{sections/06-no-gaps}
\input{sections/07-proof-existence}
\input{sections/08-manifolds}
\input{sections/09-singular-tips}

\part{Uniqueness for the model nonlinearity}
\input{sections/part2-overview}
\input{sections/10-uniqueness-setting}
\input{sections/11-multiplier-field}
\input{sections/12-basic-properties}
\input{sections/13-pohozaev}
\input{sections/14-variation-function}
\input{sections/15-later-phases}
\input{sections/16-critical-point-transition}
\input{sections/17-renewal}
\input{sections/18-tail}
\input{sections/19-assembly}
\input{sections/20-restriction-on-f}
\input{sections/21-singular-tips-uniqueness}

\FloatBarrier
\appendix
\input{sections/A-singular-tips}
\input{sections/B-identity-table}

\clearpage
\input{sections/C-notation}

\bibliographystyle{amsalpha}
\bibliography{references}
\end{document}

%% file: library.tex
\usepackage[T1]{fontenc}
\usepackage[utf8]{inputenc}
\IfFileExists{lmodern.sty}{\usepackage{lmodern}}{}
\usepackage[margin=2.6cm]{geometry}
\usepackage{amsmath,amssymb,amsthm,mathtools}
\usepackage{enumitem}
\usepackage{booktabs}
\usepackage{placeins}
\usepackage{array}
\usepackage{xcolor}
\usepackage{tikz}
\usetikzlibrary{arrows.meta,positioning,calc,fit,backgrounds,decorations.pathreplacing}
\usepackage[colorlinks=true,linkcolor=blue!60!black,citecolor=green!40!black,urlcolor=blue!60!black]{hyperref}

\numberwithin{equation}{section}

\setlist[enumerate]{font=\upshape}
\theoremstyle{plain}
\newtheorem{theorem}{Theorem}[section]
\newtheorem{proposition}[theorem]{Proposition}
\newtheorem{lemma}[theorem]{Lemma}
\newtheorem{corollary}[theorem]{Corollary}
\newtheorem*{introtheorem}{Theorem}
\theoremstyle{definition}
\newtheorem{definition}[theorem]{Definition}
\newtheorem{assumption}[theorem]{Assumption}
\newtheorem{example}[theorem]{Example}
\theoremstyle{remark}
\newtheorem{remark}[theorem]{Remark}

\newcommand{\PaperTitle}{Existence and uniqueness of radial bound states with prescribed
    number of nodes for nonlinear dispersive equations on spherically
    symmetric manifolds}
\newcommand{\PaperShortTitle}{Nodal bound states for dispersive equations on
    spherically symmetric manifolds}
\newcommand{\PaperAuthor}{Valentino Badalucco}
\newcommand{\PaperAddress}{D\'epartement de Math\'ematiques et Applications,
    \'Ecole Normale Sup\'erieure, Universit\'e PSL, 45, rue d'Ulm 75005 Paris, France}
\newcommand{\PaperEmail}{valentino.badalucco@ens.fr}

\newcommand{\R}{\mathbb{R}}

\newcommand{\defeq}{\vcentcolon =}
\newcommand{\eqdef}{{= \vcentcolon}}
\newcommand{\abs}[1]{\lvert #1 \rvert}
\newcommand{\norm}[1]{\lVert #1 \rVert}

\DeclareMathOperator{\sgn}{sgn}
\DeclareMathOperator{\sech}{sech}
\DeclareMathOperator{\arcsinh}{arcsinh}
\DeclareMathOperator{\vol}{vol}
\DeclareMathOperator{\area}{area}
\let\div\relax
\DeclareMathOperator{\div}{div}

\newcommand{\Hyp}{\mathbb{H}}

\newcommand{\sB}{\mathcal{B}}
\newcommand{\cK}{\mathcal{K}}
\newcommand{\cJ}{\mathcal{J}}
\newcommand{\cN}{\mathcal{N}}
\newcommand{\fc}{\mathfrak{c}}
\newcommand{\lmu}{\underline{\mu}}
\newcommand{\hrho}{\hat\rho}

%% file: sections/00-abstract.tex
\begin{abstract}
    Standing waves $e^{i\omega t}u(x)$ of nonlinear dispersive equations, such as the nonlinear Schr\"odinger and Klein--Gordon equations, have a profile $u$ that solves the semilinear elliptic equation $\Delta u+f(u)=0$. We prove the existence of radial bound states with a prescribed number of nodal surfaces for this equation on a large class of spherically symmetric manifolds and for a broad class of nonlinearities $f$. Under further assumptions on the metric, we prove uniqueness of such bound states for the model nonlinearity $f(u)=|u|^{p-1}u-u$. Both existence and uniqueness results apply to the hyperbolic spaces $\mathbb{H}^n_\lambda$ of any curvature $-\lambda^2$ and dimension $n\geq 3$, and to some manifolds with a singularity at the origin.
\end{abstract}

%% file: sections/01-introduction.tex
\section{Introduction}\label{sec:intro}

Nonlinear dispersive equations, such as the nonlinear Schr\"odinger and Klein--Gordon equations
\[
  i\partial_t\Psi+\Delta\Psi+\tilde f(|\Psi|)\Psi=0,\qquad
  \partial_{tt}\Psi-\Delta\Psi-\tilde f(|\Psi|)\Psi=0,
\]
are fundamental models for the propagation of nonlinear waves. Their standing waves $\Psi(t,x)=e^{i\omega t}u(x)$ have a profile $u$ that solves the elliptic equation
\[
  \Delta u+f(u)=0,
\]
where $f$ incorporates the frequency $\omega$. Bound states of this equation have classically been studied by variational and by ODE shooting methods. The existence of ground states (solutions with a strictly positive profile, which are radial~\cite{GidasNiNirenberg1981}) is well established, by variational methods~\cite{Strauss1977,BerestyckiLions1983I,BerestyckiGallouetKavian1983} and by shooting~\cite{BerestyckiLionsPeletier1981}, while excited states are much harder to treat.

In this work we study radial solutions on spherically symmetric manifolds. Let $M$ be a manifold of dimension $n$ with the warped product metric
\[
  g_M=dr^2+\psi(r)^2g_{\mathbb{S}^{n-1}},
\]
where $r$ is the distance from a fixed point and $\psi(r)>0$ for $r>0$. For a function $u=u(r)$ that depends only on $r$, the Laplace--Beltrami operator is $\Delta_{g_M}u=u''+(n-1)(\psi'/\psi)u'$, so the equation $\Delta_{g_M}u+f(u)=0$ becomes
\begin{equation}\label{eq:radial}
  u''+g(r)u'+f(u)=0
\end{equation}
with $g=g_\psi:=(n-1)\psi'/\psi$. For $\psi(r)=r$ this is $\R^n$, with $g(r)=\frac{n-1}{r}$, and for $\psi(r)=\lambda^{-1}\sinh(\lambda r)$ it is the hyperbolic space $\Hyp^n_\lambda$ of curvature $-\lambda^2$, with $g(r)=(n-1)\lambda\coth(\lambda r)$. Our results are proved for \eqref{eq:radial} with a general function $g$, and then applied to manifolds.

For the model nonlinearity $f(u)=|u|^{p-1}u-u$, our results take the following form on a spherically symmetric manifold. A \emph{$k$-node bound state} is a radial solution that tends to $0$ at infinity and vanishes on exactly $k$ spheres centered at the origin.

\begin{introtheorem}
  Let $M$ be a $C^2$ manifold of dimension $n\ge3$ with metric $dr^2+\psi(r)^2g_{\mathbb{S}^{n-1}}$, where $\psi,\psi'>0$ for $r>0$, and let $1<p<\frac{n+2}{n-2}$. Consider the problem
  \[
    \Delta_{g_M}u-u+|u|^{p-1}u=0,\qquad u\to0\ \text{ as }r\to\infty.
  \]
  \begin{enumerate}[label=(\roman*)]
    \item \textup{(Existence.)} If $r^2\psi'(r)/\psi(r)\to\infty$ as $r\to\infty$, then for every $k\ge0$ there is a $k$-node bound state.
    \item \textup{(Uniqueness.)} If moreover $\psi$ satisfies the regularity, integrability and structural conditions of Section~\ref{sec:uniqueness}, then for every $k\ge0$ the $k$-node bound state is unique up to the sign $u\mapsto-u$.
  \end{enumerate}
  Both parts apply to $\R^n$ and to the hyperbolic space $\Hyp^n_\lambda$ of any curvature $-\lambda^2$.
\end{introtheorem}

Part (i) is a special case of Theorem~\ref{thm:manifold-existence}, which holds for a large class of nonlinearities, and also in dimension two with $1<p<\infty$. Part (ii) is Corollary~\ref{cor:warped}. The main structural condition it needs, $(\Phi)$, holds for instance when $\psi/\psi'$ is concave and not eventually constant (Lemma~\ref{lem:concave-reciprocal}), and for the ground state ($k=0$) the second one, (H), is not needed (Corollary~\ref{cor:ground-state}). By scaling, both parts remain true when the linear term $-u$ is replaced by $-\omega u$ for any $\omega>0$, so they apply to radial standing waves at every frequency. Both parts also extend to some manifolds with a singular tip at the origin, where uniqueness holds in dimension two as well (Sections~\ref{sec:tips} and~\ref{sec:tip-uniqueness}).

\subsection*{Part I: existence for a class of nonlinearities}
The first part of the paper solves the existence problem under general assumptions on the nonlinearity $f$: sign conditions near $0$, which match the linear term coming from the frequency $\omega$, and growth $|u|^p$ at infinity. The relevant subcritical range for the exponent depends only on the parameter
\[
  \gamma:=\lim_{s\to0^+}s\,g(s),
\]
which depends only on the metric near the origin: $\gamma=n-1$ on a $C^2$ warped product of dimension $n$. We require
\[
  1<p<\frac{\gamma+3}{\gamma-1}\quad\text{when }\gamma>1,\qquad\text{and any }1<p<\infty\text{ when }\gamma=1.
\]
This includes the Euclidean case: for $g(r)=\frac{n-1}{r}$, so that $\gamma=n-1$, we get for $n\ge3$ the classical subcritical bound $1<p<\frac{n+2}{n-2}$, while for $n=2$ it yields the unrestricted range $1<p<\infty$.

Radial solutions with a prescribed number of zeros were first obtained in dimension three by Nehari~\cite{Nehari1963} and Ryder~\cite{Ryder1967}. For general nonlinearities, Berestycki and Lions~\cite{BerestyckiLions1983I,BerestyckiLions1983II} proved the existence of infinitely many radial solutions and conjectured that, at least for some classes of nonlinearities, exactly one exists for each number of nodes. Radial nodal solutions were later also obtained by variational methods, by Bartsch and Willem~\cite{BartschWillem1993} and by Conti, Merizzi and Terracini~\cite{ContiMerizziTerracini2000}. On hyperbolic space, infinitely many radial sign-changing solutions in $H^1(\Hyp^n)$ were obtained variationally by Bhakta and Sandeep~\cite[Theorem~5.1]{BhaktaSandeep2012}; Theorem~\ref{thm:manifold-existence} gives, for each $k$, one with exactly $k$ nodes. We use instead the shooting method for the radial ODE: by treating the initial amplitude $u(0)=\alpha$ as a shooting parameter, and tracking the evolution of the nodes of the solution as $\alpha$ varies, we inductively prove the existence of radial excited states with exactly $k$ zeros, for any integer $k\ge0$. We follow the proof structure of~\cite{McLeodTroyWeissler1990} for the Euclidean case, which avoids the dynamical-systems methods used by Jones and K\"upper~\cite{JonesKupper1986} for general subcritical nonlinearities. The two analytic components are an oscillation lemma for the scaled limit equation, which we prove through a Pohozaev identity~\cite{Pohozaev1965} (Lemma~\ref{lem:limit-zeros}), and a ``no gaps'' lemma showing that a small perturbation of the shooting parameter around a bound state creates at most one new zero (Lemma~\ref{lem:no-gaps}). What is new in Part~I is the class of damping terms. Besides the behavior of $g$ at the origin, it is described by a single condition at infinity, $r^2g(r)\to\infty$, which is (G3) of Section~\ref{sec:setting}: neither monotonicity of $g$ nor a bound on the growth of the volume is needed. In particular, over-damped equations (considered in~\cite{McLeodTroyWeissler1990}), in which solutions may creep toward a level that is not a zero of $f$, are included (Remark~\ref{rem:overdamped}). In Section~\ref{sec:manifolds} we apply this to radial bound states on a class of $C^2$ spherically symmetric manifolds, and in Section~\ref{sec:tips} to manifolds with singular tips, including tips without a unique tangent cone.

\subsection*{Part II: uniqueness for the model nonlinearity}
The conjecture of Berestycki and Lions~\cite{BerestyckiLions1983I,BerestyckiLions1983II}, that, at least for some classes of nonlinearities, the radial $k$-node bound state of $\Delta u+f(u)=0$ in $H^1(\R^n)$ is unique for each $k$, has a long history.
The case $k=0$ was treated by Coffman~\cite{Coffman1972}, McLeod and Serrin~\cite{McLeodSerrin1987}, and Kwong~\cite{Kwong1989}, who proved uniqueness for the model nonlinearity in every dimension (see also~\cite{PeletierSerrin1983,McLeod1993}, and \cite{ChenLin1991,PucciSerrin1998,SerrinTang2000} for general nonlinearities). For $k=1$, uniqueness was first proved by Troy~\cite{Troy2005} for a piecewise linear nonlinearity in dimension three, and by Cort\'azar, Garc\'ia-Huidobro and Yarur~\cite{CortazarGarciaHuidobroYarur2009} for a class of superlinear nonlinearities. A computer-assisted proof of uniqueness of the first twenty excited states for the cubic nonlinearity in dimension three was given in~\cite{CohenLiSchlag2024}, and a complete proof in the Euclidean setting, for the model nonlinearity $f(u)=|u|^{p-1}u-u$ and $n\ge3$, was obtained recently by Tang~\cite{Tang2026}, also by a shooting method. In dimension two, uniqueness of the $k$-node bound state for every $k$ follows from an earlier theorem of Cort\'azar, Garc\'ia-Huidobro and Yarur~\cite[Theorem~1.1]{CortazarGarciaHuidobroYarur2011}, whose hypothesis $(F/f)'\ge\frac{n-2}2$ on $(\alpha_*,\infty)$ holds for the model nonlinearity when $n=2$ and fails when $n=3,4$ (Remark~\ref{rem:literature}(b)). Uniqueness of the $k$-node bound state on $\R^n$ has since been proved for a sublinear model nonlinearity~\cite{ZhangZhang2026} and, with Tang's auxiliary functions, for the logarithmic Schr\"odinger equation~\cite{LiuSunZou2026} and for a double-power nonlinearity with a small defocusing term~\cite{GuoJiLuo2026}.

Uniqueness for equations with a radial weight, a class that contains \eqref{eq:radial}, has also been studied. For ground states, generalized Pohozaev identities (see~\cite{KawanoNiYotsutani1990} for an early form) give uniqueness for $\Delta u+a(r)u+b(r)u^p=0$~\cite{Yanagida1991,ShiojiWatanabe2013} and for $\div(\rho\nabla u)+\rho(-au+bu^p)=0$~\cite{ShiojiWatanabe2016}. Cort\'azar, Garc\'ia-Huidobro and Herreros~\cite{CortazarGarciaHuidobroHerreros2019} treat the divergence form $(e^Gu')'+e^Gf(u)=0$ of \eqref{eq:radial}, where $G'=g$, and prove uniqueness of the ground state, of the bound state with one sign change and, under stronger conditions, of higher bound states, when $g$ is strictly decreasing and $f$ satisfies convexity and growth conditions. See~\cite{CortazarGarciaHuidobroYarur2011} for the Euclidean case. On hyperbolic space, the positive solution is unique up to isometries by a theorem of Mancini and Sandeep~\cite{ManciniSandeep2008}. Remark~\ref{rem:literature} compares these results with ours.

The second part of this paper takes up the uniqueness question on the class of manifolds treated in Part~I, by an argument close to the one of~\cite{Tang2026} in the Euclidean case. It assumes regularity and integrability conditions and one structural hypothesis $(\Phi)$ on $g$, described below. Some restriction on the geometry is to be expected: Christianson, Marzuola, Metcalfe and Taylor~\cite{ChristiansonMarzuolaMetcalfeTaylor2014}, who construct ground states on a class of Riemannian manifolds, also identify manifolds on which positive solutions of the associated elliptic equation are not essentially unique. Under these assumptions, Part~II proves uniqueness of the ground state (on hyperbolic spaces this is the radial case of~\cite{ManciniSandeep2008}; ODE proofs of the radial statement were given in~\cite{Wang2014,Selvitella2020}). It also proves uniqueness of the $k$-node bound state for every $k$ under two further assumptions: a second structural hypothesis (H) on $g$, which holds on Euclidean and hyperbolic spaces, and the elementary restriction (R): $\gamma\ge\frac85$ or $p\ge\frac2{\gamma-1}$.\footnote{The value $\frac85$ is not sharp. It is chosen so that one numerical inequality in the proof can be checked by hand, and the same argument works for every $\gamma\ge\gamma_*\approx1.560$ (see Remark~\ref{rem:river-integral}(b)).} This restriction is superfluous when $\gamma=n-1$ with $n\ge3$: on every $C^2$ warped product of dimension $n\ge3$ that satisfies the hypotheses of Part~II, hyperbolic spaces of all curvatures included, the $k$-node bound state is unique for every $k$ and every $1<p<\frac{n+2}{n-2}$. The same holds on the manifolds with singular tips of Section~\ref{sec:tips}, over arbitrary links and also in dimension two, whenever the tip index gives $\gamma>1$ (Section~\ref{sec:tip-uniqueness}). On these manifolds $\gamma$ need not be an integer, and (R) becomes a condition on the tip. For $k\ge1$ and the model nonlinearity we are not aware of earlier uniqueness results on hyperbolic spaces. The comparison with the results above is in Remark~\ref{rem:literature}. We comment on three points.

First, uniqueness is proved for the model nonlinearity $f(u)=|u|^{p-1}u-u$, not for the whole class (f1)--(f3) under which existence holds. This restriction is not a limitation of the method: in Section~\ref{sec:restriction-f} we observe that the class (f1)--(f3) contains nonlinearities for which the ground state is already not unique in $\R^3$, by the construction of D\'avila, del Pino and Guerra~\cite{DavilaDelPinoGuerra2013}; for nodal bound states, non-uniqueness is known for nonlinearities outside this class~\cite{CortazarGarciaHuidobroHerreros2015}. Some structural restriction on $f$ is therefore necessary, and the model nonlinearity is the natural first case.

Second, every Euclidean uniqueness proof uses the scaling vector field $r\partial_r$, which is not available on a general warped product. The natural substitute is a multiplier field $\varphi(r)\partial_r$, where $\varphi$ is the solution of
\[
  \varphi'=g\varphi-(\gamma-1),\qquad
  \varphi(r)=(\gamma-1)e^{G(r)}\int_r^\infty e^{-G(s)}\,ds,\qquad
  G(r)=\int_1^r g.
\]
On Euclidean space $\varphi(r)\equiv r$, and in general $\varphi$ keeps the two properties that the Euclidean argument uses: it kills the first-order term of the linearized operator applied to $\varphi u'$, and it produces a variable quantity which we name the \emph{effective dimension}
\[
  \nu(r):=2\varphi'(r)+\gamma-1
\]
in the Pohozaev identity, equal to $n$ in the Euclidean case. The auxiliary functions and differential identities on which the argument is based (Pohozaev functions, two Wronskians $M$ and $Q$ of the variation $v=\partial_\alpha u$ against $u$ and against $\varphi u'$, the comparison functions $T_1,T_2,W_a$, and the bridge functions $B_a$) then carry over almost verbatim, up to correction terms with a factor $\varphi''$ or $\varphi'-1$. The new structural hypothesis, $(\Phi)$, is that
\[
  D(r):=g(r)e^{G(r)}\int_r^\infty e^{-G}\ \text{ is nonincreasing and }D>1,\qquad
  \text{equivalently}\qquad\varphi'>0\ge\varphi'',
\]
which fixes the sign of every correction term. Geometrically, $(\Phi)$ says that $g$ is strictly decreasing and $\varphi$ is concave (Lemma~\ref{lem:phi-consequences}(iv) and its converse). For a warped product, the first condition is the strict log-concavity of $\psi$. We verify $(\Phi)$ on Euclidean space and on hyperbolic spaces of all curvatures, and give a differential criterion for it, the concavity of $1/g$ (the reciprocal of the mean curvature of the geodesic spheres), which also covers warped products of variable negative curvature such as $\psi=r\cosh(\lambda r)$ (Lemmas~\ref{lem:concave-reciprocal} and~\ref{lem:variable-curvature}).

Third, the counting of zeros relies on two dissipation estimates: every solution has finitely many zeros, and just above a bound state the node count jumps by exactly one and the solution is trapped. In the Euclidean case the finiteness was proved in~\cite[Proposition~2.4]{Tang2026} by a dissipation estimate with the explicit weight $r^{2(n-1)}$, while the jump was obtained there by a comparison argument (proof of Lemma~3.1), the count of exactly one new zero following from the staircase itself. Here both are proved for a general damping $g$ that is bounded at infinity, hypothesis (G5), by a single argument (Proposition~\ref{prop:classification}, Lemma~\ref{lem:jump}). The dissipation across one oscillation is at least of order $1/r$, because $(\Phi)$ forces $g>(\gamma-1)/r$. Moreover, the zeros cannot spread out faster than $j\log j$, because a Pohozaev lower bound on the energy at the zeros (Lemma~\ref{lem:pohozaev-lower}) bounds the time the solution can spend near the equilibrium. The two facts together make the total dissipation diverge.
The same argument shows that the energy left at the extra zero just above a bound state is exponentially small while the energy needed for one more oscillation is of order $1/r$.

The proof scheme is described informally in the overview at the start of Part~II. Under $(\Phi)$ and the regularity and integrability conditions (G1)--(G6), without (H) or (R), the ground state is unique (Theorem~\ref{thm:unconditional}(ii), Corollary~\ref{cor:ground-state}), hyperbolic spaces of all curvatures included. For nodal bound states, uniqueness reduces to a single sign per bound state: the variation $v=\partial_\alpha u$ grows at infinity with the sign opposite to the solution (Theorem~\ref{thm:tail-criterion}). This is a parity statement on the number of zeros of $v$, whose lower bound $k+1$ is a Sturm comparison (Lemma~\ref{lem:v-zeros}). On Euclidean space all hypotheses hold and our results recover the theorem of~\cite{Tang2026} (Theorem~\ref{thm:euclidean}).

The sign is obtained phase by phase in Sections~\ref{sec:transition} and~\ref{sec:renewal-below}, which carry over to warped products the organization of~\cite{Tang2026} around the critical points of $u$. This needs one extra hypothesis, (H): $e^G\varphi'$ nondecreasing, which holds trivially on $\R^n$ and on all hyperbolic spaces. The argument is an induction over the critical points of $u$, and its most delicate step, as in~\cite{Tang2026}, is the renewal at each critical point of a bound on $v'/v$, the slope bound.
When the zero of the variation between a critical point and the next zero of $u$ (the descent) lies above the level of the next extremum of $u$, a bridge identity gives the renewal (Theorem~\ref{thm:renewal-above}). Otherwise we pair the descent with the following ascent at equal levels, as in~\cite{Tang2026}, but integrate $T_2'$ level by level (Section~\ref{sec:renewal-below}). Three Euclidean facts are used in this pairing. The first, the comparison of $Q/(e^G|u'|)$ across the zero of $u$, is proved by combining a weighted energy inequality with the monotonicity of $e^G\varphi'v/u'$, which is where (H) is used. The second, the monotonicity of $P/r^n$, becomes the positivity of a second Pohozaev function $P_2$. The third, $n\ge3$, is not needed.
In the residual case of~\cite{Tang2026}, in which the zero of the variation lies below the level $\alpha_*$, the Euclidean proof shows that the relevant Wronskian has no negative part when $n\ge4$ or $p\ge2$, and treats $n=3$, $1<p<2$ by a separate integral estimate~\cite[Lemma~5.4]{Tang2026}. We prove a warped analogue of that estimate for every $\gamma\ge\frac85$ and every $p$ in the range \eqref{eq:p-range} by a short argument: the integral is affine in a parameter, one endpoint is trivially positive, and the other is controlled by a single monotone quantity (Lemmas~\ref{lem:X-growth}--\ref{lem:river-integral}). The region around the zero of $u$ where $|u|<\alpha_*$ (the river) is then crossed with a bridge function whose weight is a constant. For $1<\gamma<\frac85$ we use instead the warped form of the first Euclidean argument, which requires $p\ge\frac2{\gamma-1}$. This is the origin of the restriction (R), and it gives the main theorem in critical-point form (Theorem~\ref{thm:critical-form}), from which the theorems of Section~\ref{sec:uniqueness} follow through the criterion on the tail. On $\R^n$ the argument reorganizes the renewal step of~\cite{Tang2026}: what is new there is the level-by-level integration of $T_2'$ and the weighted energy inequality at paired levels (Lemma~\ref{lem:paired-energy}). Like~\cite{Tang2026}, which crosses the river with a bridge function, it does not use the sign of $Q$ inside the river, where $Q$ can change sign (Remark~\ref{rem:R-complement}).

\subsection*{Organization}
Part~I is self-contained. Section~\ref{sec:setting} sets up the problem and states the hypotheses and the existence theorem. Section~\ref{sec:well-posedness} treats well-posedness, and Section~\ref{sec:intermediate} the energy and the elementary properties of solutions. Section~\ref{sec:scaled-limit} studies the scaled limit and shows that solutions with arbitrarily many zeros exist. Section~\ref{sec:no-gaps} shows that a small perturbation of the shooting parameter creates at most one new zero. Section~\ref{sec:existence-proof} contains the proof of the existence theorem, and Sections~\ref{sec:manifolds} and~\ref{sec:tips} its consequences on manifolds, for smooth and for singular tips.

Part~II opens with an overview of the proof and a map of its dependencies (Figure~\ref{fig:dependency}). Section~\ref{sec:uniqueness} fixes the model nonlinearity, adds the hypotheses on $g$ needed for uniqueness and states the results. Section~\ref{sec:multiplier} studies the multiplier field, and Sections~\ref{sec:basic-properties} and~\ref{sec:pohozaev} develop the elementary and Pohozaev properties used. Section~\ref{sec:variation} contains the identities and the analysis of the first phase, and Section~\ref{sec:later-phases} the conventions for the later phases and the scaling field. Section~\ref{sec:transition} contains the critical-point form of the transition and the renewal of the slope bound above the level $|u(c_i)|$, and Section~\ref{sec:renewal-below} the renewal in the remaining case and the main theorem in critical-point form. Section~\ref{sec:tail} contains the tail analysis and nondegeneracy, and Section~\ref{sec:staircase} the tail criterion, the staircase and the proofs of the main theorems. Section~\ref{sec:restriction-f} shows the failure of uniqueness for general nonlinearities, and Section~\ref{sec:tip-uniqueness} treats manifolds with singular tips. Appendix~\ref{sec:tip-appendix} contains the examples of Sections~\ref{sec:tips} and~\ref{sec:tip-uniqueness} and the proof of Theorem~\ref{thm:blowups}. Appendix~\ref{sec:identity-table} lists the differential identities used in Part~II, and Appendix~\ref{sec:notation} collects the notation of Part~II.

\subsection*{Statement of AI use}
LLMs were used to assist with the computations in the table of identities of Part~II and, during the research phase, to check the plausibility of partial results numerically on explicit examples. The paper was written by the author. LLMs were used for grammar and formatting suggestions and to generate the figures. The author takes full responsibility for the content of this paper.
\subsection*{Acknowledgments}
The author would like to thank Thomas Duyckaerts for many helpful discussions on the topic of existence of bound states, and for his comments and overall guidance on an earlier draft of this paper. The author would also like to thank Lorenzo Catani for spotting a gap in the proof of an earlier version of Part~I.

%% file: sections/02-setting.tex
\section{Setting, hypotheses and the existence theorem}\label{sec:setting}

The existence proof for excited states consists of two main parts. First, we show that the equation admits solutions with arbitrarily many zeros by a limiting argument. Second, we prove that a small perturbation of the initial value around a bound state generates a solution with at most one additional zero. Together, these two facts imply the existence of solutions with any prescribed finite number of zeros. Before turning to these two steps, we establish global existence and uniqueness for the initial value problem and prove several preliminary results.

\subsection{The boundary value problem}\label{sec:bvp}
We study the boundary value problem
\begin{alignat}{2}
    u'' + g(r)u' + f(u) & = 0  & \qquad & \text{for } r > 0, \label{eq:ode}                    \\
    u'(0)               & = 0, &        & u(r) \to 0 \ \text{ as } r \to \infty. \label{eq:bc}
\end{alignat}
We assume the following on $f$, where $F(u) = \int_0^u f(s)\,ds$:
\begin{enumerate}[label=(f\arabic*), leftmargin=4em]
    \item $f : \R \to \R$ is locally Lipschitz continuous.
    \item $uf(u) < 0$ for $|u|$ small, $u \neq 0$.
    \item there are values $\beta > 0$ and $\beta' < 0$ such that
          \[
              F(u) < 0 \ \text{ on }\ (0,\beta), \quad f(u) > 0 \ \text{ on }\ [\beta,\infty), \quad F(u) < 0 \ \text{ on }\ (\beta',0), \quad f(u) < 0 \ \text{ on }\ (-\infty,\beta'].
          \]
    \item $f(u) = k_\pm |u|^{p-1}u + o(|u|^p)$ as $u \to \pm\infty$, with $k_\pm > 0$.
\end{enumerate}
By (f3), $F$ is strictly increasing on $[\beta,\infty)$ and strictly decreasing on $(-\infty,\beta']$, and by (f4), $F(u) \to +\infty$ as $|u| \to \infty$. We denote by $\beta_F \ge \beta$ and $\beta_F' \le \beta'$ the positive and the negative zero of $F$, so that $F \le 0$ exactly on $[\beta_F',\beta_F]$, with $F < 0$ on $(\beta_F',\beta_F) \setminus \{0\}$ (on $[\beta,\beta_F)$ this follows from the strict monotonicity of $F$). Note that by (f3) every zero of $f$ in $(0,\infty)$ lies in $(0,\beta)$, where $F < 0$.

We also assume the following hypotheses on the function $g$, where $G(r) = \int_1^r g(s)\,ds$ denotes the primitive vanishing at $r = 1$:
\begin{enumerate}[label=(G\arabic*), wide=0pt]
    \item $g : (0,\infty) \to (0,\infty)$ is locally Lipschitz continuous.
    \item $\lim_{s \to 0^+} s\,g(s) \eqdef \gamma \ge 1$ exists.
    \item $\lim_{r \to \infty} r^2 g(r) = \infty$.
\end{enumerate}
The allowed range of $p$ is then determined by the value of $\gamma$:
\begin{equation}\label{eq:p-range}
    1 < p < \frac{\gamma+3}{\gamma-1} \quad \text{if } \gamma > 1, \qquad \text{and} \qquad 1 < p < \infty \quad \text{if } \gamma = 1.
\end{equation}

\begin{remark}[On the hypotheses on $g$]\label{rem:hyp-g}
    (G2) fixes the behavior at the origin. Since $sg(s) \to \gamma \ge 1$ we have $g(s) \ge 1/(2s)$ near $0$, so $G(r) \to -\infty$ and $e^{G(r)} \to 0$ as $r \downarrow 0$. (G3) is the only condition at infinity, and it is used only in Lemma~\ref{lem:no-gaps}: it asks the damping to decay strictly more slowly than $r^{-2}$. Some condition of this kind is needed (Remark~\ref{rem:end}). It holds for $g=(n-1)/r$, for every $g$ bounded below by a positive constant at infinity, and for $g=(n-1)/r+r^b$ with $b\ge0$. No upper bound and no monotonicity is required, and in particular no restriction is placed on the growth of the volume of geodesic balls (Section~\ref{sec:manifolds}).

    The proof of Theorem~\ref{thm:existence} distinguishes two regimes at infinity through the function
    \[
        V(t) := e^{-G(t)} \int_1^t e^{G(s)}\,ds, \qquad V' = 1 - gV.
    \]
    For a warped product with $g = g_\psi$ one has $e^G = (\psi/\psi(1))^{n-1}$, so $V(t)$ is the ratio $(\vol B_t - \vol B_1)/\area \partial B_t$ of geodesic balls. If $g$ is bounded at infinity then $\int_1^\infty V=\infty$, since $V(t) \ge \int_{t-1}^t e^{-\int_s^t g}\,ds \ge e^{-\sup_{[t-1,t]}g}$ for $t\ge2$. The two cases $\int_1^\infty V=\infty$ and $\int_1^\infty V<\infty$ are compared in Remark~\ref{rem:overdamped}.
\end{remark}

The main theorem of Part~I is the following.

\begin{theorem}[Existence]\label{thm:existence}
    Let $k \ge 0$ be an integer, let $f$ satisfy \textup{(f1)--(f4)} with exponent $p$ as in \eqref{eq:p-range}, and let $g$ satisfy \textup{(G1)--(G3)}. Then there exists a solution of \eqref{eq:ode}--\eqref{eq:bc} with exactly $k$ zeros.
\end{theorem}

We now define $u(\cdot,\alpha)$, for $\alpha > 0$, as the solution of \eqref{eq:ode} with
\begin{equation}\label{eq:ivp}
    u(0) = \alpha > 0, \qquad u'(0) = 0.
\end{equation}
We will show that, with an appropriate scaling, $u(r,\alpha)$ converges to the solution of
\begin{align}
    w'' + \frac{\gamma}{r} w' + k_\pm |w|^{p-1} w & = 0, \label{eq:limit-ode} \\
    w(0) = 1, \qquad w'(0)                        & = 0, \label{eq:limit-ivp}
\end{align}
which we prove has infinitely many zeros when $p$ is in the range \eqref{eq:p-range}. Here and below $k_\pm$ stands for $k_{\sgn w}$, i.e.\ the nonlinearity in \eqref{eq:limit-ode} is $k_+ w^p$ where $w > 0$ and $-k_- |w|^p$ where $w < 0$.

\begin{remark}\label{rem:model-f}
    The model nonlinearity of Part II,
    \[
        f(u) = |u|^{p-1}u - u, \qquad F(u) = \frac{|u|^{p+1}}{p+1} - \frac{u^2}{2},
    \]
    belongs to the class (f1)--(f4): (f1) is clear, and (f2) holds since $uf(u) = |u|^{p+1} - u^2 < 0$ for $0 < |u| < 1$. Moreover, (f3) holds with $\beta = -\beta' = \alpha_* := \bigl(\frac{p+1}{2}\bigr)^{1/(p-1)}$, the positive zero of $F$, since $f > 0$ on $(1,\infty) \supset [\alpha_*,\infty)$, and (f4) holds with $k_\pm = 1$. Consequently Theorem~\ref{thm:existence} applies to it, and Part II shows that for this nonlinearity, under additional hypotheses on $g$, the solutions it produces are unique.
\end{remark}

%% file: sections/03-well-posedness.tex
\section{Local and global well-posedness}\label{sec:well-posedness}

For the sake of completeness, we make a small note regarding local and global well-posedness of the Cauchy problem defined by \eqref{eq:ode} and \eqref{eq:ivp}, and regarding the dependence of its solution on $\alpha$. On a manifold with an isometric group action of cohomogeneity one, existence and uniqueness for the analogous Cauchy problem, with the data at a point that is not fixed by the action, are proved in~\cite{MartinezAlbaRiano2025}. In \eqref{eq:ivp} the data are given at $r=0$, where $g$ is singular.

\subsection{Local uniqueness and existence near the origin}\label{sec:local}
The condition $g > 0$ is sufficient for the equation to have local well-posedness at $0$. Indeed, rewriting the equation as $\bigl(e^{G(r)}u'\bigr)' = -e^{G(r)} f(u)$ and integrating twice (the boundary term at $0$ vanishes since $e^{G(r)} \to 0$ as $r \downarrow 0$, see Remark~\ref{rem:hyp-g}), we define the operator
\[
    T(u)(r) \defeq \alpha - \int_0^r \cK(r,\tau) f(u(\tau))\,d\tau, \qquad
    \cK(r,\tau) \defeq \int_\tau^r e^{G(\tau)-G(s)}\,ds \le r - \tau \le r,
\]
the inequality holding since $G$ is increasing. Let $\delta>0$, let $M := \sup_{\abs{t-\alpha}\le\delta} \abs{f(t)}$, and let $L$ be the Lipschitz constant of $f$ on $[\alpha-\delta,\alpha+\delta]$. A standard calculation yields, for $r\in[0,R]$,
\[
    \abs{T(u)(r) - \alpha} \le \int_0^r \cK(r,\tau) \abs{f(u(\tau))}\,d\tau \le M \int_0^R (R-\tau)\,d\tau = M\frac{R^2}{2},
\]
so that $T$ maps the closed ball $\bar B_\delta^\alpha := \{u \in C([0,R]) : \ \norm{u - \alpha}_\infty \le \delta\}$, a complete metric space, into itself, provided $MR^2/2 \le \delta$. Similarly,
\[
    \abs{T(u)(r) - T(v)(r)} \le \int_0^r \cK(r,\tau) \abs{f(u(\tau)) - f(v(\tau))}\,d\tau \le L\frac{R^2}{2} \norm{u - v}_\infty .
\]
Taking the supremum with respect to $r \in [0,R]$ and choosing $R$ with $LR^2\le1$ shows that $T$ is a contraction with constant $\frac12 < 1$. Banach's fixed point theorem then ensures the existence and uniqueness of a local solution at the origin. This solution belongs to $C^1([0,R]) \cap C^2((0,R])$ and satisfies $u'(0) = 0$, since
\[
    u'(r)=-\int_0^r e^{G(\tau)-G(r)}f(u(\tau))\,d\tau.
\]

Conversely, let $u\in C^1([0,R_1])\cap C^2((0,R_1])$ solve \eqref{eq:ode} with $u(0)=\alpha$ and $u'(0)=0$. Integrating $(e^Gu')'=-e^Gf(u)$ from $\varepsilon$ to $r$ and letting $\varepsilon\downarrow0$ kills the boundary term $e^{G(\varepsilon)}u'(\varepsilon)$, because $e^G$ is bounded near $0$ ($G$ is increasing) and $u'(\varepsilon)\to0$. Integrating once more gives $u=T(u)$ on $[0,R_1]$. By continuity $\|u-\alpha\|_{C([0,R'])}\le\delta$ for some $R'\in(0,\min\{R,R_1\}]$. The estimates above hold with $R'$ in place of $R$, so $T$ has a unique fixed point in the ball of $C([0,R'])$, and $u$ coincides on $[0,R']$ with the solution constructed above. Beyond $R'$ the two agree by the Cauchy--Lipschitz theorem.

Away from the origin, local uniqueness is guaranteed by the Cauchy--Lipschitz theorem, which cannot be used at the origin as $g$ is not Lipschitz continuous near $0$. The same argument applies to \eqref{eq:limit-ode}--\eqref{eq:limit-ivp}, with $g = \gamma/r$.

\begin{remark}[Dependence on $\alpha$]
    The constants can be chosen uniformly for $\alpha$ in some compact interval $I$: take $\delta=1$ and let $M$, $L$ be the bound and the Lipschitz constant of $f$ on the $1$-neighborhood of $I$. For $\alpha,\tilde\alpha\in I$ the fixed points then satisfy
    \[
        \|u_\alpha-u_{\tilde\alpha}\|_\infty\le|\alpha-\tilde\alpha|+\tfrac12\|u_\alpha-u_{\tilde\alpha}\|_\infty.
    \]
    Rearranging gives the first of the following two bounds, and the formula for $u'$ gives the second:
    \[
        \|u_\alpha-u_{\tilde\alpha}\|_{C([0,R])}\le2|\alpha-\tilde\alpha|, \qquad
        \|u_\alpha'-u_{\tilde\alpha}'\|_{C([0,R])}\le2LR|\alpha-\tilde\alpha|.
    \]
    Away from the origin the equation is regular, and the solutions are bounded uniformly for $\alpha\in I$ (Section~\ref{sec:global}). Gronwall's lemma then shows that $\alpha\mapsto(u,u')(\cdot,\alpha)$ is continuous, and indeed locally Lipschitz, into $C([0,T])^2$ for every $T>0$. This continuous dependence in $C^1$ on compact sets is used repeatedly and sometimes tacitly in later sections (Proposition~\ref{prop:trapping}, Lemma~\ref{lem:no-gaps}, Section~\ref{sec:existence-proof}).

    If moreover $f\in C^1$, as in Part~II, the same fixed point argument (the uniform contraction principle) and $C^1$ dependence on initial data beyond $R$ show that $\alpha\mapsto u(\cdot,\alpha)$ is $C^1$ into $C^1([0,T])$. The derivative $v=\partial_\alpha u$ then solves
    \[
        v''+gv'+f'(u)v=0, \qquad v(0)=1, \qquad v'(0)=0.
    \]
\end{remark}

\subsection{Global existence}\label{sec:global}
Let $E(r) := \tfrac{1}{2} u'(r)^2 + F(u(r))$ be the energy of Definition~\ref{def:energy} below. By \eqref{eq:energy-decay}, $E$ is nonincreasing, hence $F(u(r)) \le E(r) \le E(0) = F(\alpha)$ on the maximal interval of existence. Since $F(s) \to +\infty$ as $|s| \to \infty$ by (f4), $u$ is bounded, and thus so is $u'$. Therefore the solution is global. Hence for each $\alpha$ there is a unique global solution satisfying \eqref{eq:ode} and \eqref{eq:ivp}. The bounds depend on $\alpha$ only through $F(\alpha)$, and they are uniform for $\alpha$ in a compact set.

%% file: sections/04-intermediate-properties.tex
\section{Intermediate properties}\label{sec:intermediate}

We begin by defining an energy-like quantity that will be very useful throughout this work.

\begin{definition}\label{def:energy}
  For a solution $u(\cdot,\alpha)$ of \eqref{eq:ode}, \eqref{eq:ivp} we define
  \[
    E(u(r,\alpha)) := \tfrac{1}{2}\bigl(u'(r,\alpha)\bigr)^2 + F(u(r,\alpha)).
  \]
\end{definition}

By \eqref{eq:ode},
\begin{equation}\label{eq:energy-decay}
  \frac{d}{dr}\bigl[E(u(r,\alpha))\bigr] = -g(r)\bigl(u'(r,\alpha)\bigr)^2 \le 0,
\end{equation}
which, since the critical points of a nonconstant solution are isolated, implies that $E$ is in fact strictly decreasing along such nonconstant solutions. As the only constant solutions are those of the form $u \equiv c$ with $f(c) = 0$, we shall only consider nonconstant solutions. We will often abbreviate $E(u(r,\alpha))$ as $E(r,\alpha)$ or $E(r)$ if $\alpha$ is clear from context, and $u(r,\alpha)$ as $u(r)$. We also see the following:

\begin{proposition}\label{prop:strict-extrema}
  Let $r_0 \ge 0$ be a critical point of a solution $u(\cdot,\alpha)$. If $f(u(r_0,\alpha)) > 0$ (resp.\ $f(u(r_0,\alpha)) < 0$) then $u(r_0,\alpha)$ is a strict local maximum (resp.\ minimum), and $u(r,\alpha) < u(r_0,\alpha)$ (resp.\ $u(r,\alpha) > u(r_0,\alpha)$) for all $r > r_0$.
\end{proposition}

\begin{proof}
  For $r_0 > 0$, by \eqref{eq:ode} we have $u''(r_0) = -f(u(r_0)) < 0$, hence the strict local maximum. For $r_0 = 0$ one has $u''(0) = -f(\alpha)/(\gamma+1)$, with the same conclusion. Indeed $u'(r)=-e^{-G(r)}\int_0^re^Gf(u)$ (Section~\ref{sec:local}), and
  \[
    e^{-G(r)}\int_0^re^G=r\int_0^1e^{G(rt)-G(r)}\,dt=\frac{r}{\gamma+1}(1+o(1)) \qquad \text{as } r\downarrow0.
  \]
  The last equality holds by dominated convergence, since $e^{G(rt)-G(r)}=\exp\big(-\int_{rt}^rg\big)\in(0,1]$ tends to $t^\gamma$ by (G2). Hence $u'(r)/r\to-f(\alpha)/(\gamma+1)$.

  Now suppose $u(r_1) = u(r_0)$ for some $r_1 > r_0$. Then
  \[
    E(r_1) = \tfrac{1}{2} u'(r_1)^2 + F(u(r_1)) \ge F(u(r_1)) = F(u(r_0)) = E(r_0),
  \]
  which contradicts the strictly decreasing property of $E$. The case of a minimum is analogous.
\end{proof}

A simple but important consequence is the following.

\begin{proposition}[Trapping]\label{prop:trapping}
  If $E(u(r_0,\alpha)) < 0$ for some $r_0 \ge 0$, then $u(r_0,\alpha) \neq 0$, $u(r,\alpha)$ has the sign of $u(r_0,\alpha)$ for all $r \ge r_0$, and $u(r,\alpha) \not\to 0$ as $r \to \infty$. Moreover, this is stable with respect to the shooting parameter: for $\tilde\alpha$ in a small enough neighborhood of $\alpha$ one has $E(u(r_0,\tilde\alpha)) < 0$ as well.
\end{proposition}

\begin{proof}
  If $u(r) = 0$ for some $r \ge r_0$, then $0 > E(r_0) \ge E(r) = \tfrac{1}{2}(u'(r))^2 \ge 0$, which is a contradiction. Similarly, were $u(r)$ to tend to $0$, then $F(u(r)) \to 0$, so that $E(r) \ge F(u(r)) > E(r_0) \ge E(r)$ for $r$ large, again a contradiction. Stability is a consequence of the continuity of $(u,u')(r_0,\cdot)$ in $\alpha$.
\end{proof}

We use this property to prove the following lemma:

\begin{lemma}\label{lem:small-alpha}
  For $0 < \alpha \le \beta$, $u(r,\alpha) > 0$ for all $r \ge 0$ and $u(r,\alpha) \not\to 0$ as $r \to \infty$.
\end{lemma}

\begin{proof}
  If $f(\alpha) = 0$ then $u \equiv \alpha > 0$ and there is nothing to prove. Otherwise $u$ is nonconstant, so $E$ is strictly decreasing, and since $E(0) = F(\alpha) \le 0$ by (f2) and the continuity of $F$ (recall $F(\beta) \le 0$), $E(r_0) < 0$ for any $r_0 > 0$, and since $u > 0$ on $[0,r_0]$ for $r_0$ small, Proposition~\ref{prop:trapping} gives the claim.
\end{proof}

In the notation of Section~\ref{sec:existence-proof}, Lemma~\ref{lem:small-alpha} says that $(0,\beta]\subset A_0$: the set of shooting parameters producing no zero is nonempty, which is the base case of the inductive proof of Theorem~\ref{thm:existence} (for every $k$, since $A_0\subset A_k$). The following lemma on critical points of positive solutions is used repeatedly.

\begin{lemma}\label{lem:positive-critical}
  Let $u$ be a solution, and let $c \geq 0$ be a critical point with $u(c) > 0$ and $E(c) \ge 0$. Then $u(c) \ge \beta_F \ge \beta$, $f(u(c)) > 0$, and $c$ is a strict local maximum of $u$. In particular, on an interval on which $u > 0$ and $E \ge 0$, $u$ has no local minimum, hence at most one critical point, and it is monotone on each side of it. The same holds with the reversed signs, that is with $\beta_F',\beta'$ in place of $\beta_F,\beta$.
\end{lemma}

\begin{proof}
  $F(u(c)) = E(c) \ge 0$ and $u(c) > 0$ force $u(c) \ge \beta_F \ge \beta$, hence $f(u(c)) > 0$ by (f3), and Proposition~\ref{prop:strict-extrema} applies. If $u$ had two critical points $c < \tilde c$ on such an interval, both would be strict local maxima, and the minimum of $u$ on $[c,\tilde c]$ would be attained at an interior critical point that is a local minimum, contradicting what we just proved.
\end{proof}

%% file: sections/05-scaled-limit.tex
\section{The scaled limit}\label{sec:scaled-limit}

We show that, after rescaling, $u(\cdot,\alpha)$ converges as $\alpha\to\infty$ to the solution of \eqref{eq:limit-ode}--\eqref{eq:limit-ivp}, which has infinitely many zeros. Hence the number of zeros of $u(\cdot,\alpha)$ is unbounded as $\alpha\to\infty$.

\begin{lemma}\label{lem:scaled-convergence}
    For $\lambda > 0$, let
    \[
        v(r,\lambda) = \lambda^{-2/(p-1)}\, u\Bigl(\frac{r}{\lambda},\ \lambda^{2/(p-1)}\Bigr).
    \]
    Then as $\lambda \to \infty$, $v(r,\lambda) \to w(r)$ uniformly on compact sets, where $w$ is the solution to \eqref{eq:limit-ode} and \eqref{eq:limit-ivp}.
\end{lemma}

\begin{proof}
    A direct substitution gives
    \begin{align}
        v'' + \frac{1}{\lambda} g\Bigl(\frac{r}{\lambda}\Bigr) v' + \lambda^{-2p/(p-1)} f\bigl(\lambda^{2/(p-1)} v\bigr) & = 0, \label{eq:scaled-ode} \\
        v(0) = 1, \qquad v'(0)                                                                                           & = 0. \label{eq:scaled-ivp}
    \end{align}
    We define an energy-like quantity for $v$, similarly to $E$, as
    \begin{equation}\label{eq:scaled-energy}
        \mathcal{E}(r,\lambda) \defeq \tfrac{1}{2} v'(r,\lambda)^2 + \lambda^{-\frac{2p}{p-1}} \int_0^{v(r,\lambda)} f\bigl(\lambda^{\frac{2}{p-1}} s\bigr)\,ds = \tfrac{1}{2} v'(r,\lambda)^2 + \lambda^{-\frac{2p+2}{p-1}} F\bigl(\lambda^{\frac{2}{p-1}} v(r,\lambda)\bigr).
    \end{equation}
    Since $\mathcal{E}$ is also nonincreasing, thanks to (f4) we have, for $\lambda \ge 1$,
    \[
        \mathcal{E}(r,\lambda) \le \mathcal{E}(0,\lambda) = \lambda^{-(2p+2)/(p-1)} F\bigl(\lambda^{2/(p-1)}\bigr) \lesssim \lambda^{-(2p+2)/(p-1)} \lambda^{\frac{2(p+1)}{p-1}} = 1.
    \]
    By (f4) there are $c, C, T > 0$ with $F(t) \ge c|t|^{p+1}$ for $|t| \ge T$ and $F \ge -C$ on $[-T,T]$. Hence the potential term in \eqref{eq:scaled-energy} is bounded below by $c|v|^{p+1} - C\lambda^{-(2p+2)/(p-1)} \ge c|v|^{p+1} - C$ for $\lambda \ge 1$, and we conclude that $|v(r,\lambda)| \le M$ for all $r \ge 0$, $\lambda \ge 1$ and some $M > 0$. Taking \eqref{eq:scaled-ode} and \eqref{eq:scaled-ivp}, multiplying by the integrating factor $e^{G(\cdot/\lambda)}$ and integrating, we see they are equivalent to
    \begin{equation}\label{eq:scaled-integral}
        e^{G(r/\lambda)} v'(r,\lambda) = -\lambda^{-2p/(p-1)} \int_0^r e^{G(s/\lambda)} f\bigl(\lambda^{2/(p-1)} v(s,\lambda)\bigr)\,ds.
    \end{equation}
    Since $|f(t)| \le C(1 + |t|^p)$ and $|v| \le M$, we have $\lambda^{-2p/(p-1)} |f(\lambda^{2/(p-1)} v)| \le C$ for $\lambda \ge 1$, so
    \[
        |v'(r,\lambda)| \le C\, e^{-G(r/\lambda)} \int_0^r e^{G(s/\lambda)}\,ds \le Cr,
    \]
    using again that $G$ is increasing. Thus $v'(r,\lambda)$ is bounded on compact subsets of $\R_{\ge 0}$, independently of $\lambda$. By the Ascoli--Arzel\`a theorem we can extract a subsequence $\lambda_j \to \infty$ and a continuous function $w : \R_{\ge 0} \to \R$ such that $v(r,\lambda_j) \to w(r)$ uniformly on compact sets.

    To identify the limit equation we normalize \eqref{eq:scaled-integral} by the constant $e^{G(1/\lambda)}$:
    \begin{equation}\label{eq:scaled-normalized}
        e^{G(r/\lambda)-G(1/\lambda)} v'(r,\lambda) = -\lambda^{-2p/(p-1)} \int_0^r e^{G(s/\lambda)-G(1/\lambda)} f\bigl(\lambda^{2/(p-1)} v(s,\lambda)\bigr)\,ds.
    \end{equation}
    By (G2), writing $tg(t) = \gamma + h(t)$ with $h(t) \to 0$ as $t \downarrow 0$, for $0 < s \le T$ we have
    \[
        G(s/\lambda) - G(1/\lambda) = \int_{1/\lambda}^{s/\lambda} \frac{\gamma + h(t)}{t}\,dt = \gamma \log s + \rho_\lambda(s), \qquad |\rho_\lambda(s)| \le \sup_{(0,\max(1,T)/\lambda]} |h| \cdot |\log s|,
    \]
    so $e^{G(s/\lambda)-G(1/\lambda)} \to s^\gamma$ uniformly on compact subsets of $(0,T]$, with $e^{G(s/\lambda)-G(1/\lambda)} \le Cs^{\gamma/2}$ for $s \le 1$ and $\lambda$ large. Moreover, by (f4), $\lambda^{-2p/(p-1)} f(\lambda^{2/(p-1)} v(s,\lambda_j)) \to k_{\sgn w(s)} |w(s)|^{p-1} w(s)$ wherever $w(s) \neq 0$, and both sides are bounded by $C$. Wherever $w(s) = 0$ the left-hand side tends to $0$ as well. By dominated convergence the right-hand side of \eqref{eq:scaled-normalized} converges to
    \[
        -\int_0^r s^\gamma\, k_{\sgn w(s)} |w(s)|^{p-1} w(s)\,ds
    \]
    for every $r > 0$. Hence $v'(r,\lambda_j)$ converges pointwise on $(0,\infty)$ to the continuous function
    \[
        \tilde w(r) \defeq -r^{-\gamma} \int_0^r s^\gamma k_{\sgn w(s)} |w(s)|^{p-1} w(s)\,ds,
    \]
    and applying dominated convergence to $v(r,\lambda_j) = 1 + \int_0^r v'(s,\lambda_j)\,ds$ (which we can do since $v'(\cdot,\lambda)$ is bounded on compact sets) we obtain $w(r) = 1 + \int_0^r \tilde w$. Thus $w \in C^1$, $w' = \tilde w$, $w(0) = 1$, $w'(0) = 0$, and $r^\gamma w' = -\int_0^r s^\gamma k_\pm |w|^{p-1} w$, which is the integrated form of \eqref{eq:limit-ode}--\eqref{eq:limit-ivp}. Since this problem has a unique solution (Section~\ref{sec:well-posedness}), the limit does not depend on the subsequence, and the whole family $v(\cdot,\lambda)$ converges to $w$ as $\lambda \to \infty$.
\end{proof}

The main result of this section is the following lemma. Its proof uses a Pohozaev identity, the special case $g = \gamma/r$, $\varphi(r) = r$ of the identity \eqref{eq:pohozaev-identity}, on which Part II is built.

\begin{lemma}[Zeros of the limit equation]\label{lem:limit-zeros}
    Let $\gamma \ge 1$, $k_\pm > 0$ and let $p$ be in the range \eqref{eq:p-range}. The solution $w$ of \eqref{eq:limit-ode}--\eqref{eq:limit-ivp} has infinitely many zeros, all of them simple.
\end{lemma}

\begin{proof}
    Zeros are simple by uniqueness ($w \equiv 0$ is the only solution with $(w,w') = (0,0)$ somewhere), and they cannot accumulate at a finite point.

    \emph{Step 1: the Pohozaev function.} Set $f_k(w) := k_{\sgn w} |w|^{p-1} w$, $F_k(w) \defeq k_{\sgn w} |w|^{p+1}/(p+1)$, so that $F_k' = f_k$, and
    \[
        P(r) \defeq r^{\gamma+1}\bigl(w'^2 + 2F_k(w)\bigr) + (\gamma-1) r^\gamma w w'.
    \]
    Using \eqref{eq:limit-ode}, a direct computation gives
    \[
        P'(r) = r^\gamma \bigl[2(\gamma+1) F_k(w) - (\gamma-1) w f_k(w)\bigr] = c_p\, r^\gamma k_{\sgn w} |w|^{p+1}, \qquad
        c_p \defeq \frac{(\gamma+3) - p(\gamma-1)}{p+1} > 0,
    \]
    the sign of $c_p$ being exactly the condition \eqref{eq:p-range}. Indeed the $w'^2$, $f_k w'$ and $ww'$ terms cancel: their coefficients are $(\gamma+1) - 2\gamma + (\gamma-1)$, $2 - 2$ and $\gamma(\gamma-1) - \gamma(\gamma-1)$. Since $P(0^+) = 0$ and $P' > 0$ wherever $w \neq 0$, $P$ is strictly increasing and $P(r) > 0$ for every $r > 0$.

    \emph{Step 2: the last zero.} Suppose, for contradiction, that $w$ has only finitely many zeros, and let $R \ge 0$ be the last one ($R = 0$ if there is none). Changing $w$ to $-w$ and exchanging $k_+$ and $k_-$ if necessary we may assume $w > 0$ on $(R,\infty)$ and write $k = k_+$. On $(R,\infty)$ the energy $E_w := \tfrac{1}{2} w'^2 + k w^{p+1}/(p+1)$ satisfies $E_w' = -(\gamma/r) w'^2 \le 0$, so $w$ and $w'$ are bounded.

    \emph{Step 3: eventually monotone.} We claim that $w$ has a critical point $R' \ge R$ with $w(R') > 0$. If $R = 0$ take $R' = 0$. If $R > 0$ then $w'(R) > 0$. If $w' > 0$ on all of $(R,\infty)$, then $w$ increases to a limit $0<L<+\infty$, and for $r > R_1 > R$ we would have $w'' = -(\gamma/r) w' - k w^p \le -k w(R_1)^p < 0$, so $w' \to -\infty$, a contradiction. Hence a first critical point $R' > R$ exists. Now $(r^\gamma w')' = -k r^\gamma w^p < 0$ on $(R',\infty)$, so $r^\gamma w'(r) < R'^\gamma w'(R') = 0$ in that range: $w' < 0$ and $w$ decreases on $(R',\infty)$ to a limit $w_\infty \ge 0$.

    \emph{Step 4: decay.} If $w_\infty > 0$ then $r^\gamma w'(r) \le -k w_\infty^p (r^{\gamma+1} - R'^{\gamma+1})/(\gamma+1)$, so $w' \le -cr$ for large $r$ and $w \to -\infty$, again a contradiction. Hence $w_\infty = 0$. Moreover, since $w$ is decreasing, for $r \ge R'' := \max(1,2R')$,
    \[
        r^\gamma w'(r) = -k \int_{R'}^r s^\gamma w(s)^p\,ds \le -k w(r)^p \frac{r^{\gamma+1} - R'^{\gamma+1}}{\gamma+1} \le -\frac{k}{2(\gamma+1)}\,r^{\gamma+1} w(r)^p,
    \]
    that is $w'/w^p \le -c\,r$ with $c = k/(2(\gamma+1))$. Integrating from $R''$ we get $w(r)^{1-p} \ge w(R'')^{1-p} + (p-1)c(r^2 - R''^2)/2$, hence
    \begin{equation}\label{eq:limit-decay}
        w(r) \le Cr^{-\frac{2}{p-1}}, \qquad \text{for }r \ge 2R''.
    \end{equation}

    \emph{Step 5: the case $\gamma = 1$.} Here $rw'(r) = -k \int_{R'}^r s w^p\,ds$ is decreasing, and by \eqref{eq:limit-decay} the integral converges as $r \to \infty$ (the integrand is $O(s^{1-2p/(p-1)})$ with $1 - \frac{2p}{p-1} = -\frac{p+1}{p-1} < -1$) to a limit $I_\infty > 0$. Thus $rw'(r) \to -kI_\infty$, and $rw'(r) \le -\tfrac12 kI_\infty$ for $r \ge R'''$, so
    \[
        w(r) \le w(R''') - \tfrac12 kI_\infty \log(r/R''') \to -\infty,
    \]
    contradicting $w > 0$.

    \emph{Step 6: the case $\gamma > 1$.} On $(R',\infty)$ we have $ww' < 0$, so $0 < P(r) \le r^{\gamma+1}\bigl(w'^2 + 2F_k(w)\bigr)$, and it's enough to show that $r^{\gamma+1} w^{p+1} \to 0$ and $r^{\gamma+1} w'^2 \to 0$: this yields $P(r) \to 0$, contradicting $P(r) \ge P(R'') > 0$. The first limit follows from \eqref{eq:limit-decay}, since $\gamma + 1 - \frac{2(p+1)}{p-1} < 0$ is equivalent to $(\gamma+1)(p-1) < 2(p+1)$, i.e.\ to $p < \frac{\gamma+3}{\gamma-1}$. For the second, \eqref{eq:limit-decay} gives, with $\kappa \defeq \gamma - \frac{2p}{p-1}$,
    \[
        r^\gamma |w'(r)| = k \int_{R'}^r s^\gamma w^p\,ds \le C\Bigl(1 + \int_{2R''}^r s^\kappa\,ds\Bigr) \le C\bigl(1 + \log r + r^{\kappa+1}\bigr),
    \]
    so that
    \[
        r^{\gamma+1} w'^2 \le C\, r^{1-\gamma} (1 + \log r)^2 + C\, r^{\gamma+3-\frac{4p}{p-1}}.
    \]
    The first term tends to $0$ because $\gamma > 1$, and the exponent of the second is $\gamma + 3 - \frac{4p}{p-1}$, which is negative if and only if $(\gamma+3)(p-1) < 4p$, i.e.\ $p(\gamma-1) < \gamma + 3$. Hence $P(r) \to 0$, the desired contradiction.
\end{proof}

\begin{remark}\label{rem:lane-emden}
    For $\gamma = n - 1$ with $n \ge 3$ an integer, Lemma~\ref{lem:limit-zeros} is the classical statement that every radial solution of the Lane--Emden equation $\Delta w + |w|^{p-1} w = 0$ oscillates in the subcritical range~\cite{JonesKupper1986,McLeodTroyWeissler1990}. The argument above is the Pohozaev proof used there, written for real $\gamma \ge 1$ and for a nonlinearity whose coefficient may depend on the sign of $w$. The decay rate \eqref{eq:limit-decay} is exactly the scaling rate of the equation, and the two exponents that must be negative in Step~6 are both equivalent to \eqref{eq:p-range}. A version for a general damping $h(r)$ in place of $\gamma/r$ is Lemma~\ref{lem:oscillation}.
\end{remark}

%% file: sections/06-no-gaps.tex
\section{No gaps}\label{sec:no-gaps}

In this section we prove that a small perturbation of a shooting parameter that produces a bound state with $k$ zeros gives a solution with at most $k+1$ zeros.

We write $N(\alpha)\in\{0,1,2,\dots\}\cup\{\infty\}$ for the number of zeros of $u(\cdot,\alpha)$ in $(0,\infty)$.

\begin{lemma}[No gaps]\label{lem:no-gaps}
    Let $\tilde\alpha>0$ be a value for which $u(\cdot,\tilde\alpha)$ has exactly $k$ zeros ($k\ge 0$) and $u(r,\tilde\alpha)\to 0$ as $r\to\infty$ (so that $\tilde\alpha>\beta$ by Lemma~\ref{lem:small-alpha}). If $|\alpha-\tilde\alpha|$ is sufficiently small, then $u(\cdot,\alpha)$ has at most $k+1$ zeros on $[0,\infty)$.
\end{lemma}

\begin{proof}
    By \eqref{eq:energy-decay}, $E(r)=\frac12u'(r)^2+F(u(r))$ satisfies
    \begin{equation}\label{eq:weighted-energy}
        E'+2gE=2gF(u),\qquad\text{i.e.}\qquad \bigl(e^{2G}E\bigr)'=2g\,e^{2G}F(u).
    \end{equation}

    \emph{Preliminaries on $\tilde u:=u(\cdot,\tilde\alpha)$.} Since $\tilde u\to 0$, we have $E(\cdot,\tilde\alpha)\to 0$, as $\tilde u'^2=2(E-F(\tilde u))$ has a limit which must vanish. Hence $E(\cdot,\tilde\alpha)>0$ everywhere. Let $z_1<\cdots<z_k$ be the zeros of $\tilde u$ ($k$ may be $0$). To fix ideas we assume $\tilde u>0$ after its last zero (the other case is symmetric, with $\beta'$ in place of $\beta$ and vice versa). By Lemma~\ref{lem:positive-critical}, $\tilde u$ has exactly one critical point $c_k$ after $z_k$ (with $c_0=0$ if $k=0$), $\tilde u$ is decreasing on $(c_k,\infty)$, and $\tilde u(c_k)\ge\beta_F>\beta/2$. (For $k=0$ the lemma is applied at $c=0$, where $E(0)=F(\tilde\alpha)>0$: it gives $\tilde\alpha\ge\beta_F$ and makes $0$ a strict local maximum, so a further critical point in $(0,\infty)$ would produce an interior local minimum, which the lemma excludes.) Let $\sigma\in(c_k,\infty)$ be the unique point such that $\tilde u(\sigma)=\beta/2$. It is important to know that $\sigma$ is fixed throughout the argument, and does not depend on $\alpha$.

    \emph{Setup.} Let $\alpha$ be close to $\tilde\alpha$. Since the zeros of $\tilde u$ are simple and $\tilde u>0$ on $[z_k+\delta,T]$ for every $T$, uniform convergence on compact sets (in $C^1$, Section~\ref{sec:local}) shows that for $\alpha$ close to $\tilde\alpha$ the solution $u=u(\cdot,\alpha)$ has exactly $k$ zeros $r_1,\dots,r_k$ in $[0,T]$, close to those of $\tilde u$, and $u>0$ on $(r_k,T]$. Moreover, $u'(\sigma,\alpha)<0$ and $|u(\sigma,\alpha)-\beta/2|<\beta/4$.
    Suppose, for contradiction, that for $\alpha$ arbitrarily close to $\tilde\alpha$ the solution $u$ has at least two further zeros $r_{k+1}<r_{k+2}$. As $\alpha\to\tilde\alpha$ we have $r_{k+1}\to\infty$. Since $E(r_{k+2})=\frac12u'(r_{k+2})^2\ge 0$, we have $E>0$ on $[0,r_{k+2})$.

    \emph{Shape of $u$ between $\sigma$ and $r_{k+1}$.} On the interval $(\sigma,r_{k+1})$ the solution is positive with $E>0$. By Lemma~\ref{lem:positive-critical} it has no local minimum in $(\sigma,r_{k+1})$, and since $u'(\sigma)<0$ it is strictly decreasing on $[\sigma,r_{k+1}]$, from $u(\sigma)<3\beta/4$ to $0$. Let $c_{k+1}\in(r_{k+1},r_{k+2})$ be the first critical point after $r_{k+1}$. On $(r_{k+1},r_{k+2})$ we have $u<0$, and $u$ is decreasing on $(r_{k+1},c_{k+1})$. Since $E(c_{k+1})>0$ we have $F(u(c_{k+1}))>0$, so $u(c_{k+1})<\beta'_F\le\beta'$ by (f3). By the intermediate value theorem there exist unique $r',r''\in(r_{k+1},c_{k+1})$ with $u(r')=\beta'/4$ and $u(r'')=\beta'/2$, and $r''>r'$. Summing up, on $(\sigma,r')$ we have $u\in(\beta'/4,3\beta/4)$, hence $F(u)\le 0$. On $(r',r'')$ we have instead $u\in(\beta'/2,\beta'/4)$.

    \emph{Step 1: a lower bound on $E(r')$.} Since $E(r'')\ge E(r_{k+2})\ge 0$, integrating \eqref{eq:weighted-energy} from $r'$ to $r''$ and using that $G$ is increasing yields
    \begin{equation}\label{eq:no-gaps-integral}
        e^{2G(r')}E(r')\ge-2\int_{r'}^{r''}g(s)e^{2G(s)}F(u(s))\,ds
        \ge 2m_0\,e^{2G(r')}\int_{r'}^{r''}g(s)\,ds,
    \end{equation}
    where $m_0\defeq\inf_{w\in[\beta'/2,\beta'/4]}\big[-F(w)\big]>0$, since $F<0$ on $[\beta'/2,\beta'/4]$. As $E$ is decreasing,
    \[
        \tfrac12u'^2\le E(0)-F(u)\le F(\alpha)-\inf F\eqdef F_\alpha<\infty,
    \]
    so $|u'|\le\sqrt{2F_\alpha}$. Hence, for $\alpha$ close enough to $\tilde\alpha$ such that $F_\alpha\le2F_{\tilde\alpha}$,
    \[
        r''-r'\ge\frac{|\beta'|}{4\sqrt{2F_\alpha}}\ge \frac{|\beta'|}{8\sqrt{F_{\tilde\alpha}}} \eqdef d_0>0 .
    \]
    Write $\Gamma(r)\defeq\inf_{s\ge r}s^2g(s)$, which is positive, nondecreasing, and tends to $\infty$ by (G3). Since $g>0$ and $s^2g(s)\ge\Gamma(r')$ for $s\ge r'$, \eqref{eq:no-gaps-integral} gives
    \begin{equation}\label{eq:no-gaps-energy}
        E(r')\ge 2m_0\int_{r'}^{r'+d_0}\frac{\Gamma(r')}{s^2}\,ds
        =\frac{2m_0\,d_0\,\Gamma(r')}{r'(r'+d_0)}\ge\frac{2m_0\,d_0\,\Gamma(r')}{(r'+d_0)^2},
    \end{equation}
    with $m_0,d_0$ independent of $\alpha$.

    \emph{Step 2: a lower bound on the slope.} On $(\sigma,r')$ we have $F(u)\le 0$ and $E>E(r')$, hence $\frac12u'^2=E-F(u)\ge E>E(r')$, that is, by \eqref{eq:no-gaps-energy}, $|u'|>2\sqrt{m_0d_0\Gamma(r')}/(r'+d_0)$ on $(\sigma,r')$. By the mean value theorem,
    \[
        \frac{u(r')-u(\sigma)}{r'-\sigma}=u'(\zeta)\quad\text{for some }\zeta\in(\sigma,r'),\qquad
        |u(r')-u(\sigma)|\le\frac{|\beta'|}{4}+\frac{3\beta}{4},
    \]
    hence
    \begin{equation}\label{eq:no-gaps-distance}
        r'-\sigma<C'\,\frac{r'+d_0}{\sqrt{\Gamma(r')}},\qquad
        \text{where }C'\defeq\frac{|\beta'|/4+3\beta/4}{2\sqrt{m_0d_0}}.
    \end{equation}

    \emph{Step 3: conclusion.} By (G3) there is $R_1\ge d_0$ such that $\Gamma(r)\ge16C'^2$ for all $r\ge R_1$. Then the right-hand side of \eqref{eq:no-gaps-distance} is at most $(r'+d_0)/4\le r'/2$ whenever $r'\ge R_1$. Since $r'>r_{k+1}\to\infty$ as $\alpha\to\tilde\alpha$, for $\alpha$ close enough to $\tilde\alpha$ we have $r'\ge\max(R_1,4\sigma)$, and then \eqref{eq:no-gaps-distance} gives $r'-\sigma<r'/2$, i.e.\ $r'<2\sigma$, contradicting $r'\ge 4\sigma$. Hence for $|\alpha-\tilde\alpha|$ small $u(\cdot,\alpha)$ has at most $k+1$ zeros.
\end{proof}

\begin{remark}\label{rem:no-gaps-trapped}
    The proof actually gives slightly more: for $|\alpha-\tilde\alpha|$ small, if $u(\cdot,\alpha)$ has a $(k+1)$-th zero $r_{k+1}$, then $E(r_0,\alpha)<0$ for some $r_0>r_{k+1}$, so the solution is trapped in the sense of Proposition~\ref{prop:trapping} beyond $r_0$, and in particular it is not a bound state. Indeed, suppose $E\ge0$ on $[r_{k+1},\infty)$. Just after $r_{k+1}$, $u$ is negative and decreasing. If $u$ had a critical point $c>r_{k+1}$ with $u<0$ on $(r_{k+1},c]$, then $u(c)\le\beta'_F\le\beta'$ by Lemma~\ref{lem:positive-critical}, so $u$ reaches $\beta'/2$ before $c$. If it had none, $u$ would decrease to a limit $\bar u<0$ with $u'\to0$, and $E\to F(\bar u)<0$ unless $\bar u<\beta'/2$. In either case the points $r'<r''$ of the proof exist, and Steps~1--3, which use only $E(r'')\ge0$, the positivity of $E$ on $[0,r'']$ and the shape of the function on $(\sigma,r'')$, give the same contradiction.
\end{remark}

%% file: sections/07-proof-existence.tex
\section{Proof of the existence theorem}\label{sec:existence-proof}

First, we introduce the sets
\[
    A_k=\{\alpha>0:\ u(\cdot,\alpha)\text{ has at most $k$ zeros}\},\qquad \alpha_k:=\sup A_k.
\]
By Lemma~\ref{lem:small-alpha}, $(0,\beta]\subset A_0\subset A_k$, so every $A_k$ is nonempty. Every $A_k$ is bounded above: by Lemma~\ref{lem:limit-zeros} the function $w$ has infinitely many simple zeros, so given $K$ it has at least $K$ zeros in some $[0,T]$. By Lemma~\ref{lem:scaled-convergence}, $v(\cdot,\lambda)$ has at least $K$ zeros in $[0,T]$ for $\lambda\ge\Lambda_K$, and since the zeros of $v(\cdot,\lambda)$ correspond to those of $u(\cdot,\lambda^{2/(p-1)})$, we get $N(\alpha)\ge K$ for $\alpha\ge\Lambda_K^{2/(p-1)}$. Hence $\alpha_k\in[\beta,\infty)$ is well defined. We shall prove that $u(\cdot,\alpha_k)$ is a bound state with exactly $k$ zeros, which is the statement of Theorem~\ref{thm:existence}.

We first record two facts. Zeros of a nonconstant solution are simple ($(u,u')\ne(0,0)$ by uniqueness), so they persist under small perturbations of $\alpha$ (Section~\ref{sec:local}). In particular, $N$ is lower semicontinuous, $N(\alpha)\le\liminf_{\alpha'\to\alpha}N(\alpha')$, and if $u(\cdot,\alpha)$ does not vanish at $r^*$ then the number of zeros of $u(\cdot,\alpha')$ in $[0,r^*]$ equals that of $u(\cdot,\alpha)$ for $\alpha'$ near $\alpha$. Second, $u(\cdot,\alpha_k)$ is not a constant solution: a constant solution $u\equiv c$ has $f(c)=0$, hence $c<\beta$ by (f3), while $\alpha_k\ge\beta$.

\emph{Step 1: $N(\alpha_k)\le k$.} By definition of supremum there are $\alpha\in A_k$ arbitrarily close to $\alpha_k$ (possibly $\alpha_k$ itself). If $N(\alpha_k)\ge k+1$, lower semicontinuity gives $N(\alpha)\ge k+1$ for all $\alpha$ near $\alpha_k$, a contradiction. Let $j:=N(\alpha_k)\le k$ and let $z_j$ be the last zero of $u(\cdot,\alpha_k)$ ($z_0:=0$).

\emph{Step 2: $u(r,\alpha_k)\to 0$ as $r\to\infty$.} Write $u=u(\cdot,\alpha_k)$ and assume, without loss of generality, that $u>0$ on $(z_j,\infty)$ (the other sign is identical, with $\beta'$ in place of $\beta$). Recall the function $V$ of Remark~\ref{rem:hyp-g}.

(a) \emph{$u$ has no strict local minimum on $(z_j,\infty)$.} Suppose $r_0>z_j$ is a strict local minimum, so $u(r_0)>0$ and, by \eqref{eq:ode}, $f(u(r_0))\le 0$. As $u$ is not constant, $f(u(r_0))<0$ (if $f(u(r_0))=0$ then $u\equiv u(r_0)$). By Proposition~\ref{prop:strict-extrema}, $u(r)>u(r_0)>0$ for all $r>r_0$. By continuous dependence, for $\alpha$ close to $\alpha_k$ the solution $u(\cdot,\alpha)$ has a strict local minimum $\tilde r_0$ near $r_0$ with $u(\tilde r_0,\alpha)>0$ and $f(u(\tilde r_0,\alpha))<0$. Then $u(\cdot,\alpha)>0$ on $[\tilde r_0,\infty)$ by Proposition~\ref{prop:strict-extrema}. Hence all the zeros of $u(\cdot,\alpha)$ lie in $[0,\tilde r_0]$, where there are exactly $j$ of them. Thus $N(\alpha)=j\le k$ for all $\alpha$ in a neighborhood of $\alpha_k$, contradicting the definition of $\alpha_k$ as a supremum.

(b) \emph{$u$ is eventually monotone.} Every critical point $c>z_j$ of $u$ is a strict local maximum. Indeed, $f(u(c))\ne0$, since otherwise $u$ would be constant, and $f(u(c))<0$ would make $c$ a strict local minimum (Proposition~\ref{prop:strict-extrema}), which is excluded by (a). Between two strict local maxima the minimum of $u$ is attained at an interior critical point, that is a local minimum, which is impossible. Hence $u$ has at most one critical point in $(z_j,\infty)$, and there is $r_1$ such that $u'$ has a constant sign on $(r_1,\infty)$. Since $u$ is bounded (Section~\ref{sec:well-posedness}), $u(r)\to\bar u\ge 0$. Moreover, $E$ has a limit, so $u'^2=2(E-F(u))$ has a limit, which must be $0$ because $u$ is bounded and monotone.

(c) \emph{If $\int_1^\infty V=\infty$ and $\bar u>0$, then $f(\bar u)=0$.}\\ Suppose $f(\bar u)>0$.
Then $f(u)\ge f(\bar u)/2$ for $r\ge r_2$, and integrating $\bigl(e^Gu'\bigr)'=-e^Gf(u)$,
\[
    e^{G(r)}u'(r)\le e^{G(r_2)}u'(r_2)-\frac{f(\bar u)}{2}\int_{r_2}^re^{G(s)}\,ds
    \le-\frac{f(\bar u)}{4}\int_{r_2}^re^{G(s)}\,ds\qquad(r\ge r_3),
\]
for some $r_3>r_2$, because $\int_{r_2}^re^G\ge e^{G(r_2)}(r-r_2)\to\infty$. Hence, with $V_{r_2}(t):=e^{-G(t)}\int_{r_2}^te^G$,
\[
    u(r)\le u(r_3)-\frac{f(\bar u)}{4}\int_{r_3}^rV_{r_2}(t)\,dt\longrightarrow-\infty
    \qquad(r\to\infty),
\]
since $V_{r_2}\ge V/2$ for large $t$, contradicting $u>0$. If $f(\bar u)<0$ the same computation gives $e^Gu'\ge\frac{|f(\bar u)|}{4}\int_{r_2}^re^G$ for large $r$, hence $u\to+\infty$, contradicting that $u$ is bounded. Thus $f(\bar u)=0$.

(c$'$) \emph{If $\int_1^\infty V<\infty$, then $\bar u=0$.}\\ Put $\cJ(s):=\int_s^\infty e^{-G}$. By Tonelli's theorem,
\begin{equation}\label{eq:tonelli}
    \int_1^\infty V(t)\,dt=\int_1^\infty e^{G(s)}\cJ(s)\,ds ,
\end{equation}
so $\cJ$ is finite and $e^{G}\cJ$ is integrable on $(1,\infty)$. In particular, $e^{G(R_m)}\cJ(R_m)\to0$ along some sequence $R_m\to\infty$. Let $I$ be a compact neighborhood of $\alpha_k$. By Section~\ref{sec:global} there are constants $A,B$ with $|u'(r,\alpha)|\le A$ and $|f(u(r,\alpha))|\le B$ for all $r\ge0$ and $\alpha\in I$. For $R\ge1$, $r\ge R$ and $\alpha\in I$, integrating $(e^Gu')'=-e^Gf(u)$ from $R$ to $r$ gives $|u'(r,\alpha)|\le Ae^{G(R)-G(r)}+B\,V_R(r)$, with $V_R\le V$, and hence
\begin{equation}\label{eq:total-variation}
    \int_R^\infty|u'(r,\alpha)|\,dr\le A\,e^{G(R)}\cJ(R)+B\int_R^\infty V(t)\,dt .
\end{equation}
Suppose $\bar u>0$, and fix $m$ with $R_m>\max(z_j,1)$, $u(R_m)>\bar u/2$ and $Ae^{G(R_m)}\cJ(R_m)+B\int_{R_m}^\infty V<\bar u/4$. By continuous dependence, for $\alpha$ close to $\alpha_k$ the solution $u(\cdot,\alpha)$ has exactly $j$ zeros in $[0,R_m]$ and $u(R_m,\alpha)>\bar u/2$, and then, by \eqref{eq:total-variation}, $u(r,\alpha)>\bar u/2-\bar u/4>0$ for all $r\ge R_m$. Thus $N(\alpha)=j\le k$ on a neighborhood of $\alpha_k$, contradicting the definition of $\alpha_k$. Hence $\bar u=0$.

(d) \emph{Conclusion.} Suppose $\bar u>0$. By (c$'$), $\int_1^\infty V=\infty$, so by (c) and (f3), $\bar u$ is a positive zero of $f$, hence $\bar u<\beta$ and $F(\bar u)<0$. Since $u'\to 0$, $E(r)\to F(\bar u)<0$, and there is $r^*$ with $E(r^*,\alpha_k)<0$. Note that $u(r^*,\alpha_k)>0$. By continuity, for $\alpha$ near $\alpha_k$ we have $E(r^*,\alpha)<0$ and $u(r^*,\alpha)>0$, so by Proposition~\ref{prop:trapping} $u(\cdot,\alpha)$ has no zeros in $[r^*,\infty)$, while it has exactly $j$ zeros in $[0,r^*]$. Hence $N(\alpha)=j\le k$ on a neighborhood of $\alpha_k$, contradicting the definition of $\alpha_k$. Therefore $\bar u=0$: $u(\cdot,\alpha_k)$ is a bound state with $j\le k$ zeros.

\emph{Step 3: $j=k$.} By Lemma~\ref{lem:no-gaps} applied at $\tilde\alpha=\alpha_k$ (note $\alpha_k\ge\beta$, and $\alpha_k>\beta$ since $\alpha_k=\beta$ would give the trapped or constant solution of Lemma~\ref{lem:small-alpha}, not a bound state), $N(\alpha)\le j+1$ for all $\alpha$ near $\alpha_k$. On the other hand $N(\alpha)\ge k+1$ for every $\alpha>\alpha_k$, since $\alpha\notin A_k$. Hence $k+1\le j+1$, i.e.\ $j=k$. This completes the proof of Theorem~\ref{thm:existence}.\qed

\begin{remark}\label{rem:overdamped}
    The two cases of Step~2 show that there are two distinct regimes for solutions as $r\to\infty$. If $\int^\infty V<\infty$ (for instance $g=(n-1)/r+r^b$ with $b>1$, for which $V(t)\sim t^{-b}$), the equation is over-damped: by \eqref{eq:total-variation} every solution has $\int^\infty|u'|<\infty$, and a solution can ``creep'' monotonically to a level $\bar u$ with $f(\bar u)\ne0$. Step~2(c$'$) shows that this cannot happen at the endpoint $\alpha_k$ of $A_k$: creeping to a positive level is stable under perturbations of $\alpha$, so it would make $N$ constant near $\alpha_k$.

    The function $e^G\cJ$ in \eqref{eq:tonelli} is dual to $V$. Both solve first-order linear equations,
    \[
        V'+gV=1,\qquad\big(e^G\cJ\big)'-g\,e^G\cJ=-1,
    \]
    whose operators $\frac{d}{dr}+g$ and $-\frac{d}{dr}+g$ are formal adjoints of each other. The solution $V$ is fixed by $V(1)=0$, while, when $\cJ$ is finite, $e^G\cJ$ is the only solution $h$ with $e^{-G}h\to0$ at infinity. Identity \eqref{eq:tonelli} is the corresponding duality, $\int_1^\infty V\cdot1=\int_1^\infty1\cdot e^G\cJ$. When $\gamma>1$, $e^G\cJ=\varphi/(\gamma-1)$, where $\varphi$ is the multiplier field of Part~II (Definition~\ref{def:multiplier}), so $V$ and $\varphi$ are two substitutes for $r$: for $g=(n-1)/r$ one has $V(t)=(t-t^{1-n})/n$ and $\varphi(r)=r$. On a warped product, $V$ is a volume/area ratio, while $\cJ$ is, up to a constant factor, the radial Green's function of the Laplacian, and $e^G\cJ=\cJ/|\cJ'|$. By \eqref{eq:tonelli}, the over-damped regime $\int^\infty V<\infty$ is the regime $\int^\infty\varphi<\infty$. Part~II never meets it: there $\varphi$ is increasing (Lemma~\ref{lem:phi-basic}), so $\int^\infty\varphi=\infty$. Accordingly, $(\Phi)$ fails for the over-damped examples above (Remark~\ref{rem:phi-fails}).
\end{remark}

%% file: sections/08-manifolds.tex
\section{Bound states on spherically symmetric manifolds}\label{sec:manifolds}

We now apply Theorem~\ref{thm:existence} to a spherically symmetric manifold $M$ with warped product metric $g_M=dr^2+\psi(r)^2g_{\mathbb{S}^{n-1}}$. For a radial function $u=u(r)$,
\[
    \Delta_{g_M}u=\frac{1}{\psi^{n-1}}\bigl(\psi^{n-1}u'\bigr)'
    =u''+(n-1)\frac{\psi'}{\psi}u',
\]
so, as in the introduction, the profiles of radial standing waves on $M$ solve \eqref{eq:ode} with $g=g_\psi:=(n-1)\psi'/\psi$. Conversely, a damping $g$ with primitive $G$ comes from the warping function $\psi=\exp\big(G/(n-1)\big)$, up to a constant factor. Hence the question of which manifolds can be treated reduces to translating the assumptions on $g$ into assumptions on $\psi$.

Warped products are not the only manifolds on which \eqref{eq:ode} arises. If a compact Lie group acts isometrically on $M$ with cohomogeneity one, the functions that are constant on its orbits satisfy $(Au')'+Af(u)=0$, where $r$ is the arclength along a geodesic orthogonal to the orbits and $A(r)$ is the volume of the orbit through the point at arclength $r$~\cite{BecerraGalvisMartinez2019,MartinezAlbaRiano2025}. This is \eqref{eq:ode} with $g=A'/A$. For a warped product, with the orthogonal group acting on the spheres $\{r\}\times\mathbb{S}^{n-1}$, $A$ is a constant multiple of $\psi^{n-1}$.

\subsection{An existence theorem on manifolds}\label{sec:manifold-existence}
Consider the following boundary value problem:
\begin{align}
    \Delta_{g_M}u+f(u)      & =0,\label{eq:manifold-eq} \\
    \lim_{|x|\to\infty}u(x) & =0,\label{eq:manifold-decay}
\end{align}
where $g_M$ has warping function $\psi$ and $r(x)$ is the distance from the pole. Theorem~\ref{thm:existence} gives the following.

\begin{theorem}\label{thm:manifold-existence}
    Let $M$ be a $C^2$ spherically symmetric manifold of dimension $n\ge 2$ with warping function $\psi$ such that $\psi,\psi'>0$ for $r>0$. Suppose that
    \[
        \lim_{r\to\infty}r^2g_\psi(r)=\infty .
    \]
    Then for each $k\ge 0$ there exists a radial solution of \eqref{eq:manifold-eq} and \eqref{eq:manifold-decay} with exactly $k$ nodes, provided $f$ satisfies \textup{(f1)--(f4)} with $1<p<\frac{n+2}{n-2}$, or $1<p<\infty$ if $n=2$.
\end{theorem}

\begin{proof}
    Since $\psi\in C^2$ and $\psi,\psi'>0$, the function $g_\psi$ is positive and locally Lipschitz on $(0,\infty)$, which is (G1). The regularity of the manifold at the origin ($\psi(0)=0$, $\psi'(0)=1$, $\psi''(0)=0$) ensures $s\,g_\psi(s)\to n-1$, which is (G2) with $\gamma=n-1$, and the assumption is (G3). Then Theorem~\ref{thm:existence} can be directly applied to the ODE obtained by writing out the equation in $r$. The range of $p$ is obtained from $\gamma=n-1$.
\end{proof}

This result extends to manifolds with a controlled singularity at the origin (Section~\ref{sec:tips}). We now give some examples of manifolds to which it applies:
\begin{itemize}
    \item $\psi(r)=r$ (flat Euclidean space), which recovers the standard radial Laplacian and the classical subcritical bounds. Here $g_\psi=(n-1)/r$, $r^2g_\psi=(n-1)r$, and the volume/area ratio is $r/n$.
    \item $\psi(r)=\lambda^{-1}\sinh(\lambda r)$ (hyperbolic space of curvature $-\lambda^2$), for which $g_\psi=(n-1)\lambda\coth(\lambda r)\to(n-1)\lambda$, and the volume/area ratio tends to $1/((n-1)\lambda)$.
    \item $\psi(r)=e^r-1$ (exponentially expanding space), for which $g_\psi=(n-1)e^r/(e^r-1)\to n-1$. Here $\psi''(0)=1$, so the metric is only Lipschitz at the pole, and this example is covered by Theorem~\ref{thm:tip-existence} (tip index $1$) rather than by Theorem~\ref{thm:manifold-existence}.
    \item $\psi(r)=r\exp\Big(\dfrac{r^{b+1}}{(b+1)(n-1)}\Big)$ with $b\ge0$, which corresponds to the generalized damping term $g_\psi(r)=\dfrac{n-1}{r}+r^b$, so that $r^2g_\psi\to\infty$ for every $b\ge0$. This case was foreseen in~\cite{McLeodTroyWeissler1990}. The metric is $C^2$ at the pole for $b\ge1$, only $C^{1,b}$ for $0<b<1$ and only Lipschitz for $b=0$. In the last two cases Theorem~\ref{thm:tip-existence} applies (tip index $1$). The volume/area ratio behaves like $r^{-b}$, so it is integrable at infinity for $b>1$: these over-damped manifolds (Remark~\ref{rem:overdamped}) are covered as well.
    \item $\psi(r)=r+\dfrac{r^3\sin r}{1+r^4}$ (a mildly oscillating space), which behaves like flat space at large $r$: $\psi'\ge 0.6$, $r^2g_\psi\sim(n-1)r$, and the volume/area ratio behaves like $r/n$.
\end{itemize}
Since only (G3) is required at infinity, Theorem~\ref{thm:manifold-existence} has no condition on the oscillations of $\psi'$ beyond $\psi'>0$, and none on the growth of the volume. For instance, a warping function that agrees near infinity with $\psi(r)=r+\frac12\sin r$ is admissible, although $(g_\psi e^{2G})'=(g'_\psi+2g_\psi^2)e^{2G}$ with $g'_\psi+2g_\psi^2=-(n-1)\frac{\sin r}{2r}+O(r^{-2})$, so that $g_\psi e^{2G}$ is not eventually monotone. So is one that agrees near infinity with the super-exponentially growing $\psi(r)=e^{r^2}$. Strongly oscillating metrics such as $\psi(r)=r+\sin(e^r)$ are excluded only because $\psi'$ changes sign.

%% file: sections/09-singular-tips.tex
\section{Singular tips}\label{sec:tips}

In Theorem~\ref{thm:manifold-existence} the manifold is assumed to be $C^2$ up to the origin, so that $\psi(r)=r+o(r^2)$ and $s\,g_\psi(s)\to n-1$. At a second glance, one can notice that the proof of Theorem~\ref{thm:existence} uses considerably less than this at the origin: hypothesis (G2) is a statement about one limit, and neither $\psi(0)$, nor $\psi'(0)$, nor the curvature near the tip are ever used. In this section we determine what is used at the origin, translate it into a condition on the warping function that admits a large variety of singular tips, and give some notable examples in Appendix~\ref{sec:tip-examples}.

Throughout, $\cN$ is a closed Riemannian manifold of dimension $n-1\ge1$, not necessarily a sphere, and
\begin{equation}\label{eq:tip-metric}
  M=(0,\infty)\times \cN,\qquad g_M=dr^2+\psi(r)^2g_{\cN}.
\end{equation}
For a function $u=u(r)$ one has $\Delta_{g_M}u=\psi^{1-n}\big(\psi^{n-1}u'\big)'=u''+g_\psi u'$ with $g_\psi:=(n-1)\psi'/\psi$, so the link $\cN$ enters only through its dimension. We write
\begin{equation}\label{eq:pc}
  p_c(\gamma):=
  \begin{cases}
    \dfrac{\gamma+3}{\gamma-1}, & \gamma>1,       \\
    +\infty,                    & 0\le\gamma\le1,
  \end{cases}
\end{equation}
so that the admissible range in Theorem~\ref{thm:existence} is $1<p<p_c(\gamma)$.

\subsection{What the proof uses at the origin}\label{sec:origin}

\begin{lemma}[Weaker condition at the origin]\label{lem:origin}
  Theorem~\ref{thm:existence} remains true if hypothesis \textup{(G2)} is replaced by
  \begin{equation}\label{eq:origin-hyp}
    \textup{(G2$'$)}\qquad \gamma:=\lim_{s\to0^+}s\,g(s)\quad\text{exists in }[0,\infty),
  \end{equation}
  the admissible range being $1<p<p_c(\gamma)$. In particular the requirement $\gamma\ge1$ in \textup{(G2)} is superfluous, and for $\gamma\le1$ every $p>1$ is admissible.
\end{lemma}

\begin{proof}
  We go through the proof of Theorem~\ref{thm:existence} and record where the origin enters.

  \emph{Local well-posedness.} The contraction argument of \S\ref{sec:local}, including the converse statement and the continuous dependence on $\alpha$, uses only two facts. The first is that $G$ is nondecreasing on $(0,R)$, which enters through $\cK(r,\tau)=\int_\tau^r e^{G(\tau)-G(s)}\,ds\le r-\tau$ and through the boundedness of $e^G$ near $0$. The second is the local Lipschitz continuity of $g$ on $(0,R)$, used for the Cauchy--Lipschitz step away from~$0$. Neither the integrability of $g$ at $0$ nor a rate is needed. Indeed, the fixed point satisfies
  \[
    u'(r)=-e^{-G(r)}\int_0^r e^{G(\tau)}f(u(\tau))\,d\tau,
  \]
  hence $|u'(r)|\le Mr$. Therefore $u\in C^1([0,R))$ with $u'(0)=0$, and $u\in C^2((0,R))$ since $g$ is continuous on $(0,R)$. So existence, uniqueness and continuous dependence of $u(\cdot,\alpha)$ hold under (G1) alone.

  \emph{Energy, a priori bounds, no gaps, induction.} The monotonicity of the energy $E$, Propositions~\ref{prop:strict-extrema}--\ref{prop:trapping}, Lemma~\ref{lem:small-alpha}, Lemma~\ref{lem:no-gaps} and the argument of \S\ref{sec:existence-proof} use $g>0$ and the behavior of $g$ at infinity only. (In Proposition~\ref{prop:strict-extrema}, $u''(0)=-f(\alpha)/(\gamma+1)$ for every $\gamma\ge0$, since $e^{-G(r)}\int_0^re^G=r\int_0^1e^{G(rt)-G(r)}\,dt=\frac{r}{\gamma+1}(1+o(1))$ as $r\downarrow0$ by (G2$'$) and dominated convergence.)

  \emph{Scaled limit.} In Lemma~\ref{lem:scaled-convergence} the bounds on $v$ and $v'$ use only that $G$ is nondecreasing, and the boundary term at $0$ in \eqref{eq:scaled-integral} vanishes because $v'(0)=0$ and $e^G$ is bounded near $0$. Identifying the limit in \eqref{eq:scaled-integral} requires
  \begin{equation}\label{eq:blowup-weight}
    e^{G(s/\lambda)-G(r/\lambda)}=\exp\Big(-\int_s^r g_\lambda(t)\,dt\Big)\longrightarrow
    \Big(\frac{s}{r}\Big)^{\gamma},\qquad 0<s\le r,
  \end{equation}
  where $g_\lambda(t):=\lambda^{-1}g(t/\lambda)$. By (G2$'$), $g_\lambda(t)=t^{-1}\,(t/\lambda)g(t/\lambda)\to\gamma/t$ pointwise on $(0,\infty)$, and $g_\lambda(t)\le C/t$ for $t\le\lambda s_0$, so \eqref{eq:blowup-weight} follows by dominated convergence. The left side of \eqref{eq:blowup-weight} is bounded by $1$, so the passage to the limit in \eqref{eq:scaled-integral}, multiplied by $e^{-G(r/\lambda)}$, is as before, and the limit $w$ satisfies $r^\gamma w'(r)=-k_\pm\int_0^r s^\gamma|w|^{p-1}w\,ds$, i.e.\ the integral form of \eqref{eq:limit-ode}. This holds for every $\gamma\ge0$. For $\gamma=0$ the limit problem is the autonomous equation $w''+k_\pm|w|^{p-1}w=0$.

  \emph{Infinitely many zeros of the limit.} For $\gamma\ge1$ this is Lemma~\ref{lem:limit-zeros}. Let $0\le\gamma<1$. Steps~2 and~3 of the proof of Lemma~\ref{lem:limit-zeros} hold verbatim for every $\gamma\ge0$: if $w$ had a last zero $R$, then, normalizing $w>0$ on $(R,\infty)$, there would be a critical point $R'\ge R$ with $(r^\gamma w')'=-kr^\gamma w^p<0$ on $(R',\infty)$. Fix $r_1>R'$. Then $r^\gamma w'(r)\le r_1^\gamma w'(r_1)\eqdef-c<0$ for $r\ge r_1$, and
  \[
    w(r)\le w(r_1)-c\int_{r_1}^r s^{-\gamma}\,ds\longrightarrow-\infty\qquad(r\to\infty),
  \]
  because $\gamma<1$, contradicting $w>0$. Hence $w$ has infinitely many zeros, all simple by uniqueness, and the boundedness of the sets $A_k$ follows as before. (This is also the case $h=\gamma/r$ of Lemma~\ref{lem:oscillation}, for which $\hrho_h<0$.)
\end{proof}

\subsection{The tip condition}\label{sec:tip-condition}
For $\psi$ as in \eqref{eq:tip-metric} set
\begin{equation}\label{eq:phi-psi}
  \varphi_\psi(r):=\frac{r\psi'(r)}{\psi(r)},\qquad
  \psi_\lambda(r):=\frac{\psi(r/\lambda)}{\psi(1/\lambda)}\quad(\lambda>0),
\end{equation}
so that $g_\psi(r)=(n-1)\varphi_\psi(r)/r$ and $\psi_\lambda'/\psi_\lambda=g_\lambda/(n-1)$ with $g_\lambda$ as in the proof of Lemma~\ref{lem:origin}. We consider the following hypotheses on the warping function.
\begin{enumerate}
  \item[($\Psi$1)] $\psi\in C^{1,1}_{\mathrm{loc}}((0,\infty))$ and $\psi>0$, $\psi'>0$ on $(0,\infty)$.
  \item[($\Psi$2)] The \emph{tip index} $a:=\lim_{r\to0^+}\varphi_\psi(r)$ exists in $[0,\infty)$.
  \item[($\Psi$3)] $r^2g_\psi(r)\to\infty$ as $r\to\infty$.
\end{enumerate}
Nothing is assumed about $\psi(0)$, $\psi'(0)$ or $\psi''$ near $0$: $\psi$ may vanish to any order or not vanish at all, and $\psi'$ may tend to $0$ or to $+\infty$.

\begin{proposition}[Equivalent forms of the tip condition]\label{prop:tip-condition}
  Assume \textup{($\Psi$1)}. The following are equivalent.
  \begin{enumerate}[label=(\roman*)]
    \item \textup{($\Psi$2)} holds with tip index $a$.
    \item $\psi_\lambda\to r^a$ in $C^1_{\mathrm{loc}}((0,\infty))$ as $\lambda\to\infty$.
    \item $\psi(r)=r^aL(r)$ with $rL'(r)/L(r)\to0$ as $r\to0^+$. In particular $L$ is slowly varying at $0$ and $\psi$ is regularly varying at $0$ of index $a$ in the sense of Karamata~\cite{BinghamGoldieTeugels1987}.
  \end{enumerate}
  Moreover, if $\psi_\lambda\to\Psi$ in $C^1_{\mathrm{loc}}((0,\infty))$ as $\lambda\to\infty$ for some positive $\Psi$, then $\Psi(r)=r^a$ for some $a\ge0$, so \textup{($\Psi$2)} holds.
\end{proposition}

\begin{proof}
  Since $\log\psi_\lambda(r)=\int_1^r\varphi_\psi(t/\lambda)\,t^{-1}\,dt$ and $\psi_\lambda'(r)/\psi_\lambda(r)=\varphi_\psi(r/\lambda)/r$, (i)$\Rightarrow$(ii) follows because $\varphi_\psi(r/\lambda)\to a$ uniformly for $r$ in compact subsets of $(0,\infty)$.

  (ii)$\Rightarrow$(i) follows by evaluating $\psi_\lambda'(1)/\psi_\lambda(1)=\varphi_\psi(1/\lambda)$.

  (i)$\Leftrightarrow$(iii) holds because $rL'/L=\varphi_\psi-a$ for $L=\psi/r^a$.

  For the last claim, the identity $\psi_{\lambda\mu}(r)=\psi_\lambda(r/\mu)/\psi_\lambda(1/\mu)$ gives in the limit $\lambda\to\infty$ that $\Psi(r/\mu)=\Psi(r)\Psi(1/\mu)$ for all $r,\mu>0$. With $\Psi(1)=1$ this is Cauchy's multiplicative equation, whose positive continuous solutions are the powers $r^a$, and $a\ge0$ since $\Psi'\ge0$.
\end{proof}

\begin{remark}[A geometric interpretation]\label{rem:tangent-cone}
  That a limit in Proposition~\ref{prop:tip-condition}(ii) must be a power is the analytic counterpart of the fact that a unique tangent cone is a metric cone. If $a=1$ and moreover $\psi(r)/r\to c\in(0,\infty)$, the pointed rescalings $(M,\lambda^2g_M,o)$ converge, in $C^1$ on compact sets away from the tip, to the metric cone $C(\cN_c)$ over $\cN_c:=(\cN,c^2g_{\cN})$, so ($\Psi$2) then says that the tip has a unique tangent cone in the usual sense of pointed limits of rescalings~\cite{Gromov1999}. In general ($\Psi$2) is a statement about the blow-ups after the link metric has been renormalized by the factor $(\lambda\psi(1/\lambda))^{-2}$: for $a>1$ the unnormalized blow-ups collapse the link and converge to a half-line, for $a<1$ they expand it and do not converge at all, and only the normalized limit $r^a$ contains the information that determines the critical exponent. The model spaces $C_a(\cN):=(\R_+\times \cN,\ dr^2+r^{2a}g_{\cN})$ are metric cones for $a=1$ (see~\cite{Cheeger1983,BruningSeeley1987,Lesch1997} for analysis on cones) and metric horns in the sense of Cheeger~\cite{Cheeger1980} for $a>1$.
\end{remark}

\begin{theorem}[Existence with a singular tip]\label{thm:tip-existence}
    Let $n\ge2$, let $\cN$ be a closed Riemannian manifold of dimension $n-1$, and let $\psi$ satisfy \textup{($\Psi$1)--($\Psi$3)} with tip index $a$. Put $\gamma:=(n-1)a$ and let $f$ satisfy \textup{(f1)--(f4)} with $1<p<p_c(\gamma)$. Then for every integer $k\ge0$ there exists $u\in C^1([0,\infty))\cap C^2((0,\infty))$ with $u'(0)=0$, $u(r)\to0$ as $r\to\infty$, exactly $k$ zeros in $(0,\infty)$, solving
  \begin{equation}\label{eq:tip-ode}
    u''+(n-1)\frac{\psi'}{\psi}u'+f(u)=0\qquad\text{on }(0,\infty).
  \end{equation}
  The function $u(r,y):=u(r)$ is a classical solution of $\Delta_{g_M}u+f(u)=0$ on $M$, it satisfies $\psi(r)^{n-1}u'(r)\to0$ as $r\to0^+$, and it is a weak solution across the tip in the sense that
  \begin{equation}\label{eq:weak-tip}
    \int_M\big(\langle\nabla u,\nabla\phi\rangle-f(u)\phi\big)\,d\vol_{g_M}=0
  \end{equation}
  for every $\phi\in C^1_c(\overline{M})$, where $\overline{M}$ is the metric completion of $M$ and $C^1_c(\overline{M})$ denotes functions with compact support in $\overline{M}$ that are $C^1$ in $(r,y)$ up to $r=0$.
\end{theorem}

\begin{proof}
  By ($\Psi$1), $g_\psi=(n-1)\psi'/\psi$ is positive and locally Lipschitz on $(0,\infty)$, i.e.\ (G1) holds. Hypothesis ($\Psi$2) is (G2$'$) with $\gamma=(n-1)a$, and ($\Psi$3) is (G3). Lemma~\ref{lem:origin} gives $u$. Since $u'(0)=0$ and $\psi$ is bounded near $0$ (it is increasing), $\psi^{n-1}u'\to0$.

  For \eqref{eq:weak-tip}, let $\bar\phi(r)$ denote the average of $\phi(r,\cdot)$ over $\cN$. As $u$ is radial, $\langle\nabla u,\nabla\phi\rangle=u'\,\partial_r\phi$. Integrating first over $\cN$ turns the left side of \eqref{eq:weak-tip} into
  \[
    \vol(\cN)\int_0^\infty\psi^{n-1}\big(u'\bar\phi'-f(u)\bar\phi\big)\,dr.
  \]
  Integrating by parts on $[\varepsilon,\infty)$ and using \eqref{eq:tip-ode} leaves the boundary term $-\psi(\varepsilon)^{n-1}u'(\varepsilon)\bar\phi(\varepsilon)$, which tends to $0$.
\end{proof}

\begin{remark}[What happens at the tip]\label{rem:tip-capacity}
  The two solutions of $(\psi^{n-1}u')'=0$ near $0$ are $u_1\equiv1$ and $u_2(r)=\int_r^{r_0}\psi^{1-n}$.

  Suppose first that $\int_0\psi^{1-n}\,dr=\infty$. This is the case whenever $\gamma>1$, and for $\gamma=1$ it depends on the slowly varying factor $L$ of Proposition~\ref{prop:tip-condition}(iii). Then $u_2$ is unbounded, bounded solutions near the tip are exactly those with $\psi^{n-1}u'\to0$, and the tip has zero capacity. No boundary condition is needed, and the solutions of Theorem~\ref{thm:tip-existence} are the natural ones.

  Suppose now that $\int_0\psi^{1-n}\,dr<\infty$. This is the case whenever $\gamma<1$, and also for slowly varying $\psi$ ($a=0$). Then $u_2$ is bounded and the tip has positive capacity: the cut-off $u_2/u_2(0)$, equal to $1$ at the tip and to $0$ at $r_0$, has energy $1/\int_0^{r_0}\psi^{1-n}>0$, and no cut-off supported in $\{r\le r_0\}$ does better. In this case the Laplacian defined by the Dirichlet form on $H^1(M)$ (the Neumann realization, whose domain is characterized by $\psi^{n-1}u'\to0$) differs from the Friedrichs realization obtained from $C_c^\infty(M\setminus\{o\})$, which imposes $u(0)=0$. In the terminology of Sturm--Liouville theory~\cite{Zettl2005} this is the distinction between a singular and a regular endpoint of $-(\psi^{n-1}u')'$ at $0$ (the weight $\psi^{n-1}$ being always integrable), since $\int_0\psi^{1-n}<\infty$ is exactly the condition that $1/\psi^{n-1}$ be integrable. It is not Weyl's limit-point/limit-circle alternative, which asks whether $u_2$ lies in $L^2(\psi^{n-1}dr)$ near $0$ and has, for $\psi\sim r^a$, the threshold $\gamma=3$ instead of $\gamma=1$.

  Theorem~\ref{thm:tip-existence} concerns the Neumann realization, which is the natural one when $\overline{M}$ is regarded as a metric measure space with the full Sobolev space $H^1$. The shooting parameter $u(0)=\alpha$ is then a boundary value in the usual sense. The Dirichlet problem at a positive-capacity tip is not covered. It would require shooting in the flux $\lim_{r\to0}\psi^{n-1}u'$ with $u(0)=0$, and we do not know how to adapt the scaled limit of Lemma~\ref{lem:scaled-convergence} or the oscillation Lemma~\ref{lem:oscillation} to this setting.
\end{remark}

\begin{remark}[Effective dimension]\label{rem:effective-dim}
  In the model $C_a(\cN)$ one has
  \[
    \vol B_r(o)=\vol(\cN)\,\frac{r^{1+(n-1)a}}{1+(n-1)a}.
  \]
  Writing $N_{\mathrm{eff}}:=\gamma+1=1+(n-1)a$ for this volume-growth exponent,
  \begin{equation}\label{eq:pc-neff}
    p_c(\gamma)=\frac{N_{\mathrm{eff}}+2}{N_{\mathrm{eff}}-2}\qquad(N_{\mathrm{eff}}>2),
  \end{equation}
  and $p_c=\infty$ when $N_{\mathrm{eff}}\le2$. So the critical exponent of Theorem~\ref{thm:tip-existence} is the Sobolev exponent of the volume-growth dimension of the tip, which need not be an integer and need not coincide with the topological dimension $n$. The tip and the end play separate roles: the tip index alone determines the range of $p$, while ($\Psi$3) only requires the end to open up.
\end{remark}

\subsection{Tips without a unique tangent cone}\label{sec:no-tangent-cone}
The existence of the limit in ($\Psi$2) can be dispensed with, at the price of a smaller range of $p$. We replace ($\Psi$2) by
\begin{enumerate}
  \item[($\Psi$2$'$)] $\psi\in C^2((0,r_0])$, and $\varphi_\psi$ and $r^2\psi''/\psi$ are bounded on $(0,r_0]$,
\end{enumerate}
i.e.\ $r\psi'/\psi$ is bounded and the radial curvature is $O(r^{-2})$, the conical rate. Put $a_+:=\limsup_{r\to0}\varphi_\psi$ and $\gamma_+:=(n-1)a_+$. With $G$ a primitive of $g_\psi$, set
\begin{equation}\label{eq:rho-star}
  V_{g_\psi}(r):=e^{-G(r)}\int_0^re^{G(s)}\,ds,\qquad
  \hrho_*:=\limsup_{s\to0^+}\Big(g_\psi(s)V_{g_\psi}(s)\Big)-\frac12 .
\end{equation}
Under ($\Psi$2) one has $\gamma_+=\gamma$ and $\hrho_*=(\gamma-1)/(2(\gamma+1))$, so that the range \eqref{eq:blowup-p-range} below is the range $1<p<p_c(\gamma)$ of Theorem~\ref{thm:tip-existence}. In general $\hrho_*\ge(\gamma_+-1)/(2(\gamma_++1))$ (Appendix~\ref{sec:blowups}), so the first condition in \eqref{eq:blowup-p-range} is implied by the second when $\gamma_+>1$ and is void when $\gamma_+\le1$; it is kept for readability.

\begin{theorem}[Non-unique blow-ups]\label{thm:blowups}
    Let $\psi$ satisfy \textup{($\Psi$1)}, \textup{($\Psi$2$'$)} and \textup{($\Psi$3)}, and let $f$ satisfy \textup{(f1)--(f4)} with
  \begin{equation}\label{eq:blowup-p-range}
    1<p<p_c(\gamma_+)\qquad\text{and, if }\hrho_*>0,\qquad p<\frac{1-\hrho_*}{\hrho_*}.
  \end{equation}
  Then the conclusions of Theorem~\ref{thm:tip-existence} hold.
\end{theorem}

The proof is given in Appendix~\ref{sec:blowups}, together with an example of a tip whose tangent cone is not unique (Example~\ref{ex:log-periodic}).

\begin{remark}[Uniqueness]\label{rem:tip-uniqueness}
  For the model nonlinearity $f(u)=|u|^{p-1}u-u$, the solutions of Theorem~\ref{thm:tip-existence} are unique when $\gamma>1$ and $g_\psi$ satisfies the hypotheses of Part~II (Theorem~\ref{thm:tip-uniqueness}). Theorem~\ref{thm:blowups} has no such counterpart: the structural hypothesis $(\Phi)$ of Part~II forces the limit $\gamma$ to exist whenever $\limsup_{r\to0}rg_\psi(r)>1$ (Proposition~\ref{prop:phi-tip}(i)), which is the only case in which the hypothesis $\gamma>1$ of Part~II could hold.
\end{remark}

%% file: sections/part2-overview.tex
\section*{Overview of Part~II}

Part~II proves that the bound states found in Part~I are unique when $f$ is the model nonlinearity $|u|^{p-1}u-u$. The proof is long and dense, but its main idea is simple. We describe it here informally, at the cost of some precision, and with some harmless white lies. The exact statements start in Section~\ref{sec:uniqueness}.

A large part of Part~II builds on the impressive work of Tang~\cite{Tang2026}, which proves uniqueness of bound states for the Euclidean case $g=(n-1)/r$. To make the paper self-contained, we include the main arguments of \cite{Tang2026} in the present setting, with some simplifications and improvements. The main new ideas are in Section~\ref{sec:multiplier}, which introduces the multiplier field, in Section~\ref{sec:pohozaev}, which introduces the Pohozaev functions and the moving Pohozaev level, and in Sections~\ref{sec:transition}--\ref{sec:renewal-below}, which carry out the renewal described below on warped products. To make the paper easier to read and reference, we try to keep the notation and terminology of \cite{Tang2026} as much as possible, and we refer to it for some technical details.

\emph{The goal: a staircase.} As in Part~I, we follow the solution $u(\cdot,\alpha)$ that starts at height $\alpha$, and we count its zeros $N(\alpha)$. Uniqueness follows if the graph of $N$ is a staircase: $N$ is constant except at the values of $\alpha$ that give bound states, and at each of these it increases by exactly one. Then the bound state with $k$ zeros is the unique value of $\alpha$ at which $N$ jumps from $k$ to $k+1$.

\emph{Away from bound states nothing happens.} Sections~\ref{sec:basic-properties} and~\ref{sec:pohozaev} show that every solution is of one of two kinds. Either it is a bound state, or it is trapped: its energy eventually becomes negative, it stays forever in one of the two wells around $\pm1$, and it has no further zeros. Being trapped is not affected by small changes of $\alpha$, so $N$ is constant near every trapped solution and can change only at bound states. The proof uses two quantities. The energy decreases because of the damping. A Pohozaev function bounds the energy from below at the zeros. Together they show that a solution cannot have infinitely many zeros.

\emph{At a bound state, everything depends on one sign.} To see what happens at a bound state we differentiate in $\alpha$: the variation $v=\partial_\alpha u$ describes how the solution moves when $\alpha$ increases. Two facts about $v$ explain the staircase. First, if $v$ and $u'$ have the same sign at every zero of $u$, the zeros move toward the origin as $\alpha$ grows, so no zero is lost and $N$ cannot decrease. Second, if at a bound state $v$ grows at infinity with the sign opposite to $u$, then for $\alpha$ slightly larger the solution overshoots, gains exactly one new zero, and is then trapped. Section~\ref{sec:staircase} shows that the second fact alone, checked at every bound state, already gives the staircase. The whole proof therefore reduces to one sign per bound state: the behavior of $v$ at infinity. Section~\ref{sec:tail} studies this behavior: $v$ either decays or grows, and a simple condition on the last part of the solution rules out decay.

\emph{Following $v$ phase by phase.} To find the sign of $v$ at infinity, we follow $v$ along the oscillations of $u$. A phase is the stretch between two consecutive critical points of $u$: on the descent the solution falls from an extremum to a zero, and on the ascent it climbs to the next extremum, of opposite sign. Around the zero lies the river, where $|u|$ is below the level $\alpha_*$ at which the potential $F$ vanishes. The river is the hardest stretch to control, because the nonlinearity changes sign inside it. We do not follow $v$ directly. Instead we use various auxiliary functions built from $u$, $v$, their derivatives and the multiplier $\varphi$: two Wronskians, Pohozaev functions, comparison functions and bridge functions. Each has an explicit formula for its derivative, and on each stretch of a phase the signs of $u$, $u'$, $v$ and $f(u)$ fix the sign of that derivative. Most arguments consist of reading off such signs.

\emph{What is calculation.} Several parts are pure computation and can be skipped on a first reading: the properties of the multiplier $\varphi$ and the checks of the hypotheses on Euclidean and hyperbolic spaces (Section~\ref{sec:multiplier} and Lemma~\ref{lem:H-criterion}), the formulas for the derivatives of the auxiliary functions (Section~\ref{sec:identities}, Lemmas~\ref{lem:weighted-identities} and~\ref{lem:renewal-identities}, all that is collected in Appendix~\ref{sec:identity-table}), and the numerical estimates of Lemmas~\ref{lem:X-growth} and~\ref{lem:river-integral}. The multiplier $\varphi$ plays the role of $r$ in the Euclidean proof. Every formula is the Euclidean one plus correction terms that involve the derivatives of $\varphi$, and hypothesis $(\Phi)$ is what we use to make sure that these corrections have the right sign for the argument to go through.

\emph{The first phase.} In the first phase (Section~\ref{sec:variation}), $v$ starts at $1$, changes sign once before $u$ comes down to the level $1$, and then keeps that sign at least up to the first zero of $u$. For a ground state there is only one phase: $v$ never changes sign again, which is the required sign at infinity. Hence ground states are unique without the hypotheses (H) and (R) introduced below.

\emph{The slope condition.} The later phases are handled by induction over the critical points of $u$. The induction hypothesis, the slope condition, consists of three inequalities at a critical point. The first two say that $v$ has the same sign as $u$ at the critical point while $v'$ has the opposite sign, so that $v$ is already moving toward zero. The third, the slope bound, says that $v'$ is not too large compared with $v$. If the slope condition holds at the critical point where a phase starts, then under hypothesis (H) the whole phase is under control (Section~\ref{sec:transition}): $v$ has exactly one zero, early in the descent, it has the right sign at the zero of $u$, it crosses the river with the help of a bridge function, and at the next critical point the first two inequalities hold again.

\emph{The renewal.} The slope bound does not come back as easily, and proving it again at the next critical point, which completes the induction step, is the renewal. As in the Euclidean case, this is by far the hardest step. When the starting height is large, $v$ is close to the scaling field of the equation, for which the third quantity vanishes identically (Section~\ref{sec:later-phases}), so its sign depends on lower-order effects. The renewal splits into cases according to where the zero of $v$ lies on the descent. If it lies higher than the next extremum of $u$, a single bridge function gives the renewal (Section~\ref{sec:renewal-above}). Otherwise, we compare the descent and the ascent at equal heights, level by level (Section~\ref{sec:renewal-below}). In the hardest case, where the zero of $v$ lies inside the river, we need an integral estimate, and this is the only place where the restriction (R) is used.

\emph{Conclusion.} By induction, the slope condition holds at every critical point. For a bound state the last phase never ends: after the last zero of $u$, $v$ has one more zero and then grows at infinity with the sign opposite to $u$. This is the required sign at every bound state, and Section~\ref{sec:staircase} turns it into the staircase.

\emph{Dependencies.} Figure~\ref{fig:dependency} shows the logical dependencies. The main line consists of the following steps, in this order, with the corresponding results:
\begin{enumerate}
  \item Pohozaev estimates: Proposition~\ref{prop:pohozaev-positive} and Lemma~\ref{lem:pohozaev-lower},
  \item classification: Proposition~\ref{prop:classification},
  \item first phase: Theorem~\ref{thm:first-phase}(i)--(iii),
  \item later phases and renewal: Theorems~\ref{thm:phase}, \ref{thm:renewal-above} and~\ref{thm:renewal-below},
  \item tail sign at every bound state: Theorem~\ref{thm:critical-form}(i) with Lemma~\ref{lem:tail-sign},
  \item tail criterion: Theorem~\ref{thm:tail-criterion},
  \item staircase: Theorem~\ref{thm:staircase}.
\end{enumerate}
Only Sections~\ref{sec:transition}--\ref{sec:renewal-below} use (H) and (R). The ground state (Theorem~\ref{thm:unconditional}) branches off after the first phase. There are a few forward references:
\begin{itemize}
  \item Proposition~\ref{prop:pohozaev-positive} uses the decay Lemma~\ref{lem:decay},
  \item Theorem~\ref{thm:phase}(v) uses Lemmas~\ref{lem:decay} and~\ref{lem:dichotomy}, Proposition~\ref{prop:nondegeneracy}, Corollary~\ref{cor:divergence} and Lemma~\ref{lem:tail-sign},
  \item Theorem~\ref{thm:critical-form} uses Lemmas~\ref{lem:zero-motion} and~\ref{lem:tail-sign}.
\end{itemize}
These forward references create no cycle. Indeed, Section~\ref{sec:tail}, Lemmas~\ref{lem:zero-motion}--\ref{lem:tail-jump} and Theorem~\ref{thm:tail-criterion} use only Part~I, Sections~\ref{sec:multiplier}--\ref{sec:variation} and one another, and Lemmas~\ref{lem:decay} and~\ref{lem:dichotomy} use only Sections~\ref{sec:multiplier}--\ref{sec:basic-properties} and the definition \eqref{eq:pohozaev} of $P$. Section~\ref{sec:tip-uniqueness} applies the theorems of Section~\ref{sec:uniqueness}, together with Part~I, to manifolds with singular tips.

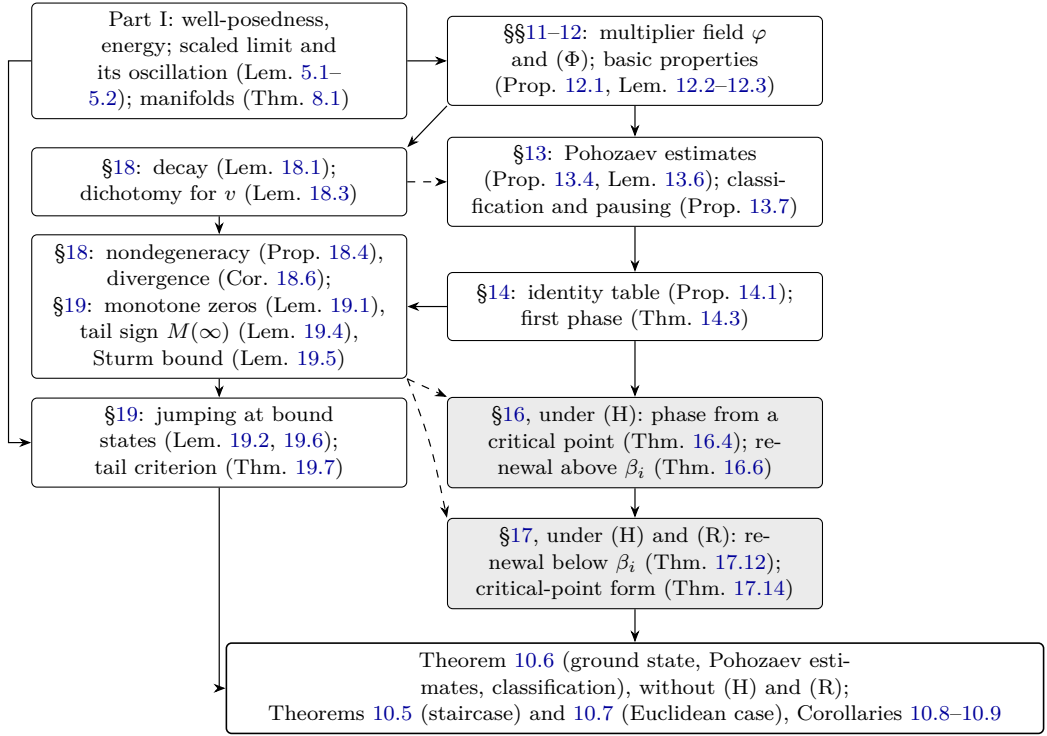
\begin{figure}[t]
  \centering
  \input{pictures/dependency-map}
  \caption{Dependencies in Part~II. An arrow $A\to B$ means that $B$ uses $A$, possibly through the results on a path from $A$ to $B$. Dashed arrows are forward references, from a later section to an earlier one. Shaded boxes use (H), and Section~\ref{sec:renewal-below} also uses (R). In the bottom box, Theorems~\ref{thm:staircase}, \ref{thm:euclidean} and Corollary~\ref{cor:warped} use (H) and (R) through them. Apart from the clause on (H) and (R) in its part (iv), Theorem~\ref{thm:unconditional} uses only the unshaded boxes.}
  \label{fig:dependency}
\end{figure}

%% file: pictures/dependency-map.tex
\begin{tikzpicture}[
    x=1cm, y=1cm,
    box/.style={draw, rounded corners=2pt, align=center, font=\scriptsize,
                text width=4.75cm, inner sep=3pt, minimum height=0.9cm},
    hyp/.style={box, fill=black!8},
    main/.style={box, text width=10.6cm, line width=0.6pt},
    use/.style={-{Stealth[length=1.6mm]}, line width=0.45pt},
    fwd/.style={-{Stealth[length=1.6mm]}, line width=0.45pt, dashed},
  ]
  \node[box] (A) at (5.5,0)
    {\S\S\ref{sec:multiplier}--\ref{sec:basic-properties}: multiplier field $\varphi$ and $(\Phi)$;
     basic properties (Prop.~\ref{prop:critical-points}, Lem.~\ref{lem:passage-times}--\ref{lem:band-dissipation})};
  \node[box] (B) at (5.5,-1.6)
    {\S\ref{sec:pohozaev}: Pohozaev estimates (Prop.~\ref{prop:pohozaev-positive}, Lem.~\ref{lem:pohozaev-lower});
     classification and pausing (Prop.~\ref{prop:classification})};
  \node[box] (C) at (5.5,-3.25)
    {\S\ref{sec:variation}: identity table (Prop.~\ref{prop:identity-table});\\
     first phase (Thm.~\ref{thm:first-phase})};
  \node[hyp] (D) at (5.5,-5.05)
    {\S\ref{sec:transition}, under (H): phase from a critical point
     (Thm.~\ref{thm:phase}); renewal above $\beta_i$ (Thm.~\ref{thm:renewal-above})};
  \node[hyp] (E) at (5.5,-6.65)
    {\S\ref{sec:renewal-below}, under (H) and (R): renewal below $\beta_i$ (Thm.~\ref{thm:renewal-below});
     critical-point form (Thm.~\ref{thm:critical-form})};
  \node[box] (P) at (0,0)
    {Part~I: well-posedness, energy; scaled limit and its oscillation
     (Lem.~\ref{lem:scaled-convergence}--\ref{lem:limit-zeros}); manifolds (Thm.~\ref{thm:manifold-existence})};
  \node[box] (T1) at (0,-1.6)
    {\S\ref{sec:tail}: decay (Lem.~\ref{lem:decay});\\ dichotomy for $v$ (Lem.~\ref{lem:dichotomy})};
  \node[box] (T2) at (0,-3.25)
    {\S\ref{sec:tail}: nondegeneracy (Prop.~\ref{prop:nondegeneracy}), divergence (Cor.~\ref{cor:divergence});\\
     \S\ref{sec:staircase}: monotone zeros (Lem.~\ref{lem:zero-motion}), tail sign $M(\infty)$
     (Lem.~\ref{lem:tail-sign}), Sturm bound (Lem.~\ref{lem:v-zeros})};
  \node[box] (F) at (0,-5.05)
    {\S\ref{sec:staircase}: jumping at bound states (Lem.~\ref{lem:jump}, \ref{lem:tail-jump});
     tail criterion (Thm.~\ref{thm:tail-criterion})};
  \node[main] (G) at (5.5,-8.3)
    {Theorem~\ref{thm:unconditional} (ground state, Pohozaev estimates, classification), without (H) and (R);\\
     Theorems~\ref{thm:staircase} (staircase) and~\ref{thm:euclidean} (Euclidean case),
     Corollaries~\ref{cor:warped}--\ref{cor:ground-state}};
  \draw[use] (P) -- (A);
  \draw[use] (A) -- (B);
  \draw[use] (B) -- (C);
  \draw[use] (D) -- (E);
  \draw[use] (E) -- (E.south |- G.north);
  \draw[use] (A.south west) -- (T1.north east);
  \draw[fwd] (T1) -- (B);
  \draw[use] (T1) -- (T2);
  \draw[use] (C) -- (T2);
  \draw[use] (T2) -- (F);
  \draw[use] (P.west) -- ++(-0.3,0) |- (F.west);
  \draw[use] (F.south) |- (G.west);
  \draw[fwd] (T2.south east) -- (D.north west);
  \draw[fwd] (T2.south east) .. controls (2.8,-5.05) .. (E.north west);
  \draw[use] (C) -- (D);
\end{tikzpicture}

%% file: sections/10-uniqueness-setting.tex
\section{The uniqueness problem: setting and statements}\label{sec:uniqueness}

From now on we fix the model nonlinearity
\begin{equation}\label{eq:model-f}
    f(u)=|u|^{p-1}u-u,\qquad F(u)=\frac{|u|^{p+1}}{p+1}-\frac{u^2}{2},
\end{equation}
which belongs to the class (f1)--(f4) by Remark~\ref{rem:model-f}, so that Theorem~\ref{thm:existence} provides, for every $k\ge 0$, a bound state with exactly $k$ nodes. The purpose of Part II is to prove that this solution is \emph{unique}. As Section~\ref{sec:restriction-f} shows, uniqueness fails for some members of the class (f1)--(f4), so a restriction of this kind is unavoidable, although the largest class of nonlinearities for which uniqueness holds is not known.

Throughout Part II we use the two structural constants of \eqref{eq:model-f},
\begin{equation}\label{eq:levels}
    \alpha_*=\Big(\frac{p+1}{2}\Big)^{\frac{1}{p-1}}\quad\text{(the positive zero of $F$)},\qquad
    \alpha^*(0)=\Big(\frac{2(p+1)}{(\gamma+3)-p(\gamma-1)}\Big)^{\frac{1}{p-1}},
\end{equation}
the second one being the Pohozaev level introduced in Section~\ref{sec:pohozaev}. Note that \eqref{eq:p-range} makes $\alpha^*(0)$ finite and, as checked in Lemma~\ref{lem:level-structure}, $\alpha^*(0)>\alpha_*>1$. (The reader should distinguish the three symbols $\alpha_*$, $\alpha^*(r)$ and $\alpha_k$: the first is the zero of $F$, the second the moving Pohozaev level of Definition~\ref{def:pohozaev-level}, the third the shooting values of the staircase.)

\subsection{Terminology}\label{sec:terminology}
A solution of \eqref{eq:ode}, \eqref{eq:ivp} is called
\begin{itemize}
    \item a \emph{bound state with $k$ nodes} ($k\ge 0$) if it has exactly $k$ zeros in $(0,\infty)$ and $u(r)\to 0$ as $r\to\infty$ (a $0$-node bound state is a \emph{ground state}).
    \item \emph{nodal} if it has at least one zero.
    \item \emph{trapped} if beyond its last zero (or from $r=0$ if there is none) there is $r_0$ with $E(r_0)<0$. By Proposition~\ref{prop:trapping}, a trapped solution has no further zeros and does not tend to $0$.
\end{itemize}
Recall that $N(\alpha)$ is the number of zeros of $u(\cdot,\alpha)$ in $(0,\infty)$ (Section~\ref{sec:no-gaps}). We write
\begin{equation}\label{eq:variation}
    v(r,\alpha):=\partial_\alpha u(r,\alpha),\qquad v''+g\,v'+f'(u)v=0,\qquad v(0)=1,\ v'(0)=0
\end{equation}
for the variation, which solves the linearization of \eqref{eq:ode} along $u$.

\begin{remark}\label{rem:trapped}
    Trapped solutions stay in the potential well $\{F<0\}$ around $\pm1$. If in addition $\liminf_{r\to\infty}g>0$ and $g$ is bounded at infinity (the latter as under (G5) below), they converge to $\pm1$: indeed $E\downarrow E_\infty\in[F(1),0)$ and $\int^\infty gu'^2<\infty$ give $u'\in L^2$, $u''=-gu'-f(u)$ is bounded so $u'\to0$, then $F(u)\to E_\infty$ and $|u|\to s$ with $F(s)=E_\infty$. If $s\ne1$, then $u''=-gu'-f(u)\to-f(\pm s)\ne0$, which is incompatible with $u'\to0$. Whether the approach to $\pm1$ is oscillatory or eventually monotone depends on the linearization $z''+gz'+(p-1)z=0$ at the equilibrium: for asymptotically constant $g\to g_\infty$ it is oscillatory when $p-1>g_\infty^2/4$ (for instance $p>2$ on $\Hyp^3$) and eventually monotone when $p-1<g_\infty^2/4$. For this reason we use the neutral word ``trapped'' rather than ``oscillatory''.
\end{remark}

We use a few names for the parts of each phase of a solution. The points of a phase and the river are defined in Section~\ref{sec:later-phases}, the other objects where they are first used, and all of them are shown in Figure~\ref{fig:phase-scheme}, which is only a sketch: the actual solution may have a different shape, but its parts and their names are the same.

\begin{figure}[t]
    \centering
    \input{pictures/phase-scheme}
    \caption{Sketch of $|u|$ along phase $i$, with $u$ normalized to be positive on the descent. Solid horizontal segments join the two points at which a bridge function is compared, and the dashed one joins two paired levels. The bars below the axis show where the first zero $\tau_i$ of $v$ after $c_{i-1}$ can lie, and where each case is treated.}
    \label{fig:phase-scheme}
\end{figure}
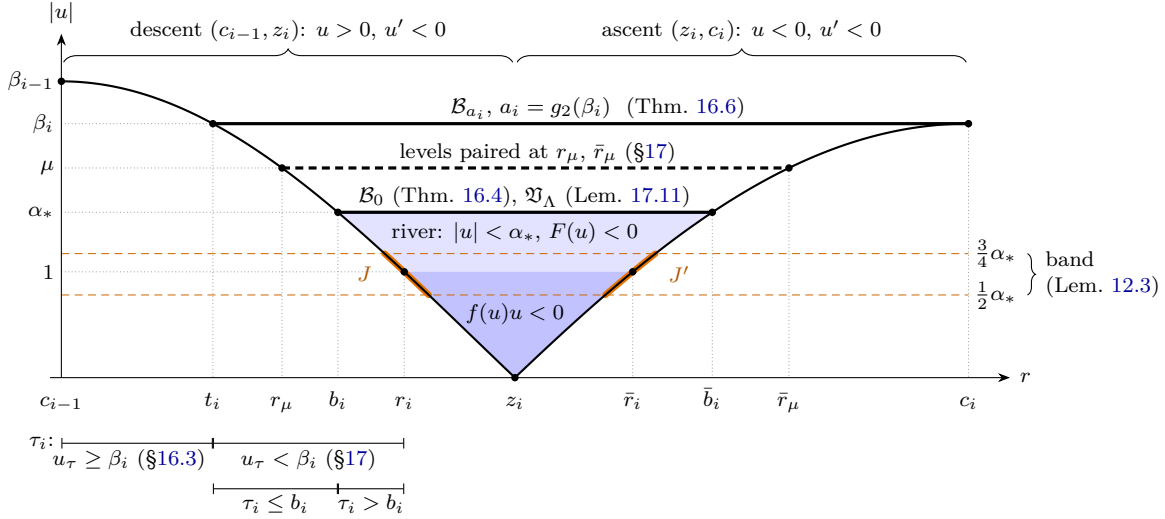

\subsection{Additional hypotheses on \texorpdfstring{$g$}{g}}\label{sec:extra-hyp}

\begin{assumption}\label{ass:extra-hyp}
    In addition to (G1)--(G3) of Section~\ref{sec:setting} we assume:
    \begin{enumerate}[leftmargin=*]
        \item[(G4)] $g\in C^2((0,\infty))$ and $\gamma>1$.
        \item[(G5)] $\displaystyle\int_1^\infty e^{-G(s)}\,ds<\infty$ and $\displaystyle\limsup_{r\to\infty}g(r)<\infty$ (integrable and bounded damping).
        \item[(G6)] $g\in C^3([R,\infty))$ for some $R>0$, and, with $q:=1+\frac{g^2}{4}+\frac{g'}{2}$: $g'(r)\to0$ as $r\to\infty$ and $\displaystyle\int_R^\infty\big(|q''|+q'^2\big)\,dr<\infty$.
        \item[$(\Phi)$] $D(r)>1$ for all $r>0$ and $D$ is nonincreasing. Equivalently, by \eqref{eq:phi-ode},
              \[
                  \varphi'(r)>0\qquad\text{and}\qquad\varphi''(r)\le0\qquad(r>0).
              \]
    \end{enumerate}
    Here $D$, $\varphi$, $\nu$ are the functions of Definition~\ref{def:multiplier}.
\end{assumption}

\begin{remark}\label{rem:extra-hyp}
    (G4) says that the manifold has ``dimension $\gamma+1>2$'' at the origin: for a $C^2$ warped product of dimension $n$ one has $g_\psi(s)=\frac{n-1}{s}+o(1)$, so $\gamma=n-1$, and the requirement $\gamma>1$ corresponds to $n\ge3$. At a singular tip of index $a$ one has $\gamma=(n-1)a$ (Section~\ref{sec:tip-uniqueness}). The regularity $g_\psi\in C^2$ requires $\psi\in C^3$ on $(0,\infty)$; it is used only in the sufficient criteria for $(\Phi)$ and (H) (Lemmas~\ref{lem:phi-criterion}, \ref{lem:concave-reciprocal} and~\ref{lem:H-criterion}), while the proofs of the main theorems differentiate $g$ at most once on $(0,\infty)$. As in Part~I (Lemma~\ref{lem:origin}), no rate is needed in the limit $sg(s)\to\gamma$.

    The first half of (G5) is the integral formulation of ``$n\ge3$ at infinity'': for $g=(n-1)/r$ it reads $\int^\infty s^{1-n}\,ds<\infty$. It makes the multiplier field of Section~\ref{sec:multiplier} well defined. The second half, boundedness of $g$, holds for $g=(n-1)/r$, for $g=(n-1)\lambda\coth(\lambda r)$ and for $g_\psi$ with $\psi=e^r-1$, but not for the expanding examples $g=(n-1)/r+r^b$ with $b>0$. Its role is to bound the time a solution spends near its turning points (Lemma~\ref{lem:passage-times}) and, in Lemma~\ref{lem:dichotomy}, to make the dominant solution of the linearized equation at $u=0$ grow. No lower bound on $g$ at infinity is assumed: the lower bound $g>(\gamma-1)/r$ that the dissipation arguments of Proposition~\ref{prop:classification} and Lemma~\ref{lem:jump} need is a consequence of $(\Phi)$ (Lemma~\ref{lem:phi-consequences}).

    (G6) is Liouville--Green admissibility \cite[Ch.~6]{Olver1974} of the linearization at $u=0$ after the substitution $u=e^{-G/2}z$. It is immediate for $g=(n-1)/r$ and for $g=(n-1)\lambda\coth(\lambda r)$.

    Finally, $(\Phi)$ is the one new structural hypothesis: it is studied in Section~\ref{sec:multiplier}, and it fails for the expanding examples $g=(n-1)/r+r^b$ with $b>0$ (Remark~\ref{rem:phi-fails}), for which uniqueness remains open. For $b=0$ and $n\ge3$, instead, $g=(n-1)/r+1$ satisfies (G4)--(G6), $(\Phi)$, (H) and (R). With $m=n-1$ one has $\ell=1/g=r/(r+m)$ and $\ell''<0$, so $(\Phi)$ follows from Lemma~\ref{lem:concave-reciprocal}(ii). Hypothesis (H) follows from Lemma~\ref{lem:H-criterion}(i), since
    \begin{gather*}
        \ell'=m/(r+m)^2<2,\\
        \big(2\ell\ell''-\ell'(\ell'-2)\big)(r+m)^4=m\big(2r^2+(4m-4)r+2m^2-m\big)>0.
    \end{gather*}
    Thus Theorem~\ref{thm:staircase} applies to this damping. Corollary~\ref{cor:warped} does not, the metric being only Lipschitz at the pole, but Theorem~\ref{thm:tip-uniqueness} does (Example~\ref{ex:lipschitz-poles}).

    Not all of these hypotheses are independent: (G3) is a consequence of $(\Phi)$, since Lemma~\ref{lem:phi-consequences}(i) gives $g>(\gamma-1)/r$, and so is the boundedness of $g$ in (G5), since Lemma~\ref{lem:phi-consequences}(iv) gives $g'<0$, hence $g\le g(1)$ on $[1,\infty)$. We keep it in (G5) because Lemmas~\ref{lem:passage-times}, \ref{lem:decay} and~\ref{lem:dichotomy} are stated without $(\Phi)$. By Lemma~\ref{lem:phi-consequences}, $(\Phi)$ means that $g$ is strictly decreasing and $\varphi$ is concave. For instance, $(\Phi)$ fails for the mildly oscillating metric of Section~\ref{sec:manifolds}, for which $g_\psi$ is not monotone ($g_\psi'>0$ on $(2.85,4.37)$).

    Thus Part~II assumes (G1), (G2), (G4)--(G6) and $(\Phi)$, together with (H) and (R) for Theorem~\ref{thm:staircase}. The hypotheses $(\Phi)$, (H) and (R) are limitations of our method, and we do not expect any of them to be necessary for uniqueness. A further condition, $\nu\ge1$ on an interval, appears only in Theorem~\ref{thm:first-phase}(iv) and the second half of Theorem~\ref{thm:unconditional}(i). This condition holds when $\gamma\ge2$ and plays no role in the other results (Remark~\ref{rem:nu-condition}).
\end{remark}

\subsection{The restriction (R) and the main results}\label{sec:main-results}
The argument of Part~II is an induction over the critical points of $u$, from one to the next (Sections~\ref{sec:transition} and~\ref{sec:renewal-below}): its induction hypothesis is a condition on the variation $v$, the slope condition of Section~\ref{sec:phase}. It comes with one more structural hypothesis on $g$, (H) (Assumption~\ref{ass:H}: $e^G\varphi'$ is nondecreasing), which holds on Euclidean and hyperbolic spaces (Lemma~\ref{lem:H-criterion}). One step of this argument, the renewal of the slope bound at the next critical point, in the case in which the zero of $v$ on the descent lies below the level $\alpha_*$, uses in addition the following elementary restriction on the pair $(\gamma,p)$.

\begin{assumption}[R]\label{ass:R}
    $\gamma\ge\frac85$, or $p\ge\frac{2}{\gamma-1}$.
\end{assumption}

(R) involves neither the solution nor the variation. It holds for every $\gamma\ge\frac85$, in particular on every $C^2$ warped product of dimension $n\ge3$, where $\gamma=n-1\ge2$ (Remark~\ref{rem:extra-hyp}), and hence on $\R^n$ and on all hyperbolic spaces. For $1<\gamma<\frac85$ it restricts $p$ to $\big[\frac2{\gamma-1},\frac{\gamma+3}{\gamma-1}\big)$, a nonempty subinterval of the range \eqref{eq:p-range}. Such values of $\gamma$ occur at singular tips, for instance at the horns of Example~\ref{ex:real-horns}. It enters only through hypotheses (a) and (c) of Theorem~\ref{thm:renewal-below}, which are used only at phases in which the zero of $v$ on the descent lies below the level $\alpha_*$ (Remark~\ref{rem:R-complement}(b)). Nondegeneracy of bound states, that is, $|v|\to\infty$, is not an assumption either: it holds at every ground state (Proposition~\ref{prop:nondegeneracy}) and, under (H) and (R), at every bound state (Theorems~\ref{thm:phase}(v) and~\ref{thm:critical-form}).

\begin{theorem}[Uniqueness staircase]\label{thm:staircase}
    Let $f$ be as in \eqref{eq:model-f} with $p$ as in \eqref{eq:p-range}, let $g$ satisfy (G1)--(G6), $(\Phi)$ and (H), and assume (R). Then there is a strictly increasing sequence $\alpha_0<\alpha_1<\alpha_2<\cdots\to\infty$, with $\alpha_0>\alpha^*(0)$, such that:
    \begin{enumerate}[label=(\roman*)]
        \item for each $k\ge0$, $u(\cdot,\alpha_k)$ is the unique bound state with exactly $k$ nodes.
        \item $N(\alpha)=k$ for $\alpha\in(\alpha_{k-1},\alpha_k]$ (with $\alpha_{-1}:=0$), and for $\alpha\in(\alpha_{k-1},\alpha_k)$ the solution is trapped.
        \item $N$ is nondecreasing, constant on each $(\alpha_{k-1},\alpha_k]$, and jumps by exactly one at each $\alpha_k$.
        \item at each $\alpha_k$ the variation $v$ has exactly $k+1$ zeros $\tau_1<\cdots<\tau_{k+1}$, interlaced with the zeros $z_1<\cdots<z_k$ of $u$ as $\tau_1<z_1<\tau_2<\cdots<z_k<\tau_{k+1}$, it satisfies $u'v'>0$ at every zero of $v$, and $|v(r)|\to\infty$ as $r\to\infty$.
    \end{enumerate}
\end{theorem}

\begin{theorem}[Unconditional part]\label{thm:unconditional}
    Let $f$ and $p$ be as in Theorem~\ref{thm:staircase} and let $g$ satisfy (G1)--(G6) and $(\Phi)$. Neither (H) nor (R) is assumed.
    \begin{enumerate}[label=(\roman*)]
        \item \textup{(Phase one.)} Let $u$ be a nodal solution or a bound state with first zero $z_1$, and let $r_1\in(0,z_1)$ be the point with $u(r_1)=1$. Then the variation $v$ has exactly one zero $\tau_1$ in $(0,z_1]$, located in $(0,r_1)$. Moreover, $u'v'(\tau_1)>0$, $v<0$ on $(\tau_1,z_1]$, and, with the functions of Section~\ref{sec:variation}, $M,Q>0$ on $(0,z_1]$, $T_1>0$ on $(0,r_1]$ with $Q(r_1)=T_1(r_1)$, and $T_2>0$ on $(0,z_1)$. If, in addition, there is a first point $\bar r_1>z_1$ with $|u(\bar r_1)|=1$, $E>0$ on $(z_1,\bar r_1]$ (both hold when $u$ has a second zero or is a bound state), and $\nu\ge1$ on $[z_1,\bar r_1]$ (which holds when $\gamma\ge2$), then $v<0$ on $(\tau_1,\bar r_1]$ and $M,Q_1>0$ on $[z_1,\bar r_1]$. For a ground state, $v$ has exactly one zero $\tau_1\in(0,r_1)$, and $v<0$, $M>0$, $Q>0$ on $(\tau_1,\infty)$.
        \item \textup{(Ground state uniqueness.)} The ground state is unique. At every ground state the variation is dominant at infinity, with the sign opposite to that of $u$.
        \item \textup{(Pohozaev estimates.)} Every ground state, nodal solution or bound state has $\alpha>\alpha^*(0)$, and the Pohozaev functions $P,P_1$ of Section~\ref{sec:pohozaev} are positive on $(0,z_k]$, $z_k$ being the last zero (on $(0,\infty)$ for bound states), with $\omega=-\varphi u'/u$ strictly increasing between zeros of $u$.
        \item \textup{(Classification.)} Every solution has finitely many zeros, and is either a bound state or trapped. The node count $N$ is locally constant on the set of trapped solutions, and jumps by exactly one at every bound state at which the variation is dominant at infinity with the sign opposite to $u$, which is the case at every ground state and, under (H) and (R), at every bound state (Theorem~\ref{thm:critical-form}).
    \end{enumerate}
\end{theorem}

\begin{theorem}[The Euclidean case]\label{thm:euclidean}
    For $g=(n-1)/r$, $n\ge3$, the hypotheses (G1)--(G6) and $(\Phi)$ all hold, with $q=1+\frac{(n-1)(n-3)}{4r^2}$, $D\equiv\frac{n-1}{n-2}$ and $\nu\equiv n$, and \eqref{eq:p-range} reads $1<p<\frac{n+2}{n-2}$. Consequently all the conclusions of Theorem~\ref{thm:unconditional} hold on $\R^n$, and the conclusions of Theorem~\ref{thm:staircase} hold on $\R^n$, since its hypotheses (H) and (R) hold (Lemma~\ref{lem:H-criterion}, and $\gamma=n-1\ge2$): this recovers the uniqueness of the $k$-node bound state of $\Delta u-u+|u|^{p-1}u=0$ in $\R^n$ for every $k\ge0$ and $1<p<\frac{n+2}{n-2}$ \cite{Tang2026}.
\end{theorem}

Combining Theorem~\ref{thm:manifold-existence}, the geometric form of Theorem~\ref{thm:existence}, with Theorem~\ref{thm:staircase} we obtain the following.

\begin{corollary}\label{cor:warped}
    Let $M$ be a $C^2$ spherically symmetric manifold of dimension $n\ge3$ with warping function $\psi$ such that $\psi,\psi'>0$ for $r>0$, and suppose that $g_\psi$ satisfies (G4)--(G6), $(\Phi)$ and (H). Then for every $k\ge0$ the problem
    \[
        \Delta_{g_M}u-u+|u|^{p-1}u=0,\qquad u\to0\ (r\to\infty),\qquad 1<p<\frac{n+2}{n-2},
    \]
    has exactly one nontrivial radial solution with $k$ nodes, up to the reflection $u\mapsto-u$. In particular this holds on the hyperbolic space $\Hyp^n_\lambda$ of any curvature $-\lambda^2$, $n\ge3$.
\end{corollary}

The case $k=0$ needs neither (H) nor (R), by Theorem~\ref{thm:unconditional}(ii).

\begin{corollary}[Ground states on warped products]\label{cor:ground-state}
    Let $M$ be as in Corollary~\ref{cor:warped} and suppose that $g_\psi$ satisfies (G4)--(G6) and $(\Phi)$. Then the problem above has exactly one positive radial solution. In particular, on the hyperbolic space $\Hyp^n_\lambda$ of any curvature $-\lambda^2$, $n\ge3$, the positive radial solution of $\Delta u-u+|u|^{p-1}u=0$ decaying at infinity is unique for every $1<p<\frac{n+2}{n-2}$. On $\Hyp^n_\lambda$ this particular case is known \cite{ManciniSandeep2008}, and the kernel of the linearization at this solution consists of the derivatives along hyperbolic translations \cite{GangulySandeep2015}.
\end{corollary}

All the hypotheses on $g$ are invariant under the scaling $g_\lambda(r)\defeq\lambda^{-1}g(r/\lambda)$, $\lambda>0$ ($\gamma$ is unchanged, $D_\lambda(r)=D(r/\lambda)$, and (G5), (G6) and (H) are preserved), and $\omega^{-1/(p-1)}u(r/\sqrt\omega)$ takes the solutions for $g$ and $f(u)=|u|^{p-1}u-\omega u$ to those for $g_{\sqrt\omega}$ and \eqref{eq:model-f}. Hence the uniqueness results above hold for every $\omega>0$, and so apply to the standing waves $e^{i\omega t}u$ of the introduction with a power nonlinearity, at every frequency $\omega>0$ for the nonlinear Schr\"odinger equation, and at every frequency $|\omega|<m$ for the nonlinear Klein--Gordon equation with mass $m$, $\partial_{tt}\Psi-\Delta\Psi+m^2\Psi-|\Psi|^{p-1}\Psi=0$.

Section~\ref{sec:tip-uniqueness} extends Corollaries~\ref{cor:warped} and~\ref{cor:ground-state} to manifolds with singular tips over arbitrary closed links, in every dimension $n\ge2$ (Theorem~\ref{thm:tip-uniqueness}).

\begin{remark}[Relation to previous results]\label{rem:literature}
    (a) \emph{Hyperbolic space.} For $n\ge3$ and $1<p<\frac{n+2}{n-2}$, Mancini and Sandeep~\cite{ManciniSandeep2008} proved that the equation $\Delta_{\Hyp^n}u+cu+|u|^{p-1}u=0$ has a positive solution in $H^1(\Hyp^n)$ if and only if $c<(n-1)^2/4$, and that this solution is unique up to hyperbolic isometries. On $\Hyp^n_\lambda$, the metric is $\lambda^{-2}$ times that of $\Hyp^n=\Hyp^n_1$, and $u=\lambda^{2/(p-1)}w$ turns $\Delta_{\Hyp^n_\lambda}u-u+|u|^{p-1}u=0$ into $\Delta_{\Hyp^n}w-\lambda^{-2}w+|w|^{p-1}w=0$, which is their equation with $c=-\lambda^{-2}<(n-1)^2/4$. Radial solutions decaying at infinity decay exponentially (Lemma~\ref{lem:decay}) and lie in $H^1(\Hyp^n)$. Hence the hyperbolic case of Corollary~\ref{cor:ground-state} is also contained in \cite{ManciniSandeep2008}, which moreover does not assume radial symmetry; ODE proofs of the radial statement are given in~\cite{Wang2014} and, by a shooting argument in the polar model, in~\cite{Selvitella2020}. Infinitely many radial sign-changing solutions on $\Hyp^n$ are obtained variationally in~\cite[Theorem~5.1]{BhaktaSandeep2012}, without a count of the nodes. Ganguly and Sandeep~\cite[Theorem~1.1]{GangulySandeep2015} proved that the kernel of the linearization at this solution in $H^1(\Hyp^n)$ is spanned by $n$ functions obtained by differentiating along hyperbolic translations, none of which is radial. If the variation $v$ decayed, it would be a multiple of $V_-$ by Lemma~\ref{lem:dichotomy} and would lie in this kernel. So this result also gives the hyperbolic case of the nondegeneracy $|v|\to\infty$ in Theorem~\ref{thm:unconditional}(ii). On other warped products, the uniqueness in Corollary~\ref{cor:ground-state} also follows from \cite[Theorem~2]{CortazarGarciaHuidobroHerreros2019} when $\lim_{r\to\infty}\varphi'(r)>0$, by (b). Compared with \cite{CortazarGarciaHuidobroHerreros2019}, the new content of Corollary~\ref{cor:ground-state} is therefore the case $\lim_{r\to\infty}\varphi'(r)=0$, which holds whenever $g_\psi$ has a positive limit at infinity, for instance for $\psi=r\cosh(\lambda r)$ (Lemma~\ref{lem:variable-curvature}). We are not aware of previous uniqueness results for nodal bound states on $\Hyp^n$, that is, for the case $k\ge1$ of Corollary~\ref{cor:warped}. Radial solutions of the Lane--Emden--Fowler equation $-\Delta u=|u|^{p-1}u$, without the linear term, are classified on $\Hyp^n$ in~\cite{BonforteGazzolaGrilloVazquez2013}, and their existence, uniqueness and stability on Riemannian models $dr^2+\psi(r)^2g_{\mathbb S^{n-1}}$ are studied in~\cite{BerchioFerreroGrillo2014}.

    (b) \emph{Weighted equations.} Radial solutions of \eqref{eq:ode} solve the divergence-form equation $(e^Gu')'+e^Gf(u)=0$. This is the equation with weight $e^G$ studied by Cort\'azar, Garc\'ia-Huidobro and Herreros~\cite{CortazarGarciaHuidobroHerreros2019}, who prove uniqueness of the ground state, of the bound state with one sign change, and of higher bound states under stronger conditions. Their assumption that the logarithmic derivative $g$ of the weight is strictly decreasing is the first half of the equivalent form of $(\Phi)$ in Lemma~\ref{lem:phi-consequences}(iv) and its converse. Their function $e^G\int_r^\infty e^{-G}$ is $\varphi/(\gamma-1)$, and their condition that its derivative be positive and nonincreasing with a positive limit at infinity is $(\Phi)$ together with $\lim_{r\to\infty}\varphi'(r)>0$. Their alternative condition requires the derivative of the volume/area ratio $e^{-G}\int_0^re^G$ to have a positive limit at infinity. If $g$ has a positive limit $g_\infty$ at infinity, as on $\Hyp^n$ and for $\psi=r\cosh(\lambda r)$, then $\varphi'=(\gamma-1)(D-1)\to0$, since $g$ is decreasing, so that $\int_r^\infty e^{-G}\le e^{-G(r)}/g_\infty$ and $1<D\le g/g_\infty\to1$. The volume/area ratio then converges, so neither condition holds. Their Theorem~2, which uses the condition on $\varphi$, also assumes that $sf'(s)/f(s)$ is nonincreasing for $s>s_0$ (the condition of Serrin and Tang~\cite{SerrinTang2000}), where $s_0$ is the positive zero of $f$. The model nonlinearity satisfies this, since $sf'(s)/f(s)=p+\frac{p-1}{s^{p-1}-1}$ for $s>1$. Hence \cite[Theorem~2]{CortazarGarciaHuidobroHerreros2019} gives the uniqueness of the ground state for the model nonlinearity whenever $(\Phi)$ holds and $\lim_{r\to\infty}\varphi'(r)>0$, while Theorem~\ref{thm:unconditional}(ii) also covers the case $\lim_{r\to\infty}\varphi'(r)=0$. For solutions with one sign change, \cite[Theorem~3]{CortazarGarciaHuidobroHerreros2019} assumes one of two conditions. The first is $f(s)\ge f'(s)(s-s_0)$ for $s\ge s_0$, and it fails for the model nonlinearity ($s_0=1$). The second is
    \[
        s_1f'(s_1)/f(s_1)\le1+2\lim_{r\to\infty}\varphi'(r)/(\gamma-1),
    \]
    where $s_1$ is the positive zero of $F$. For the model nonlinearity, $s_1=\alpha_*$ and $\alpha_*f'(\alpha_*)/f(\alpha_*)=p+2$, while $\lim\varphi'\le1$ by \eqref{eq:phi-bounds}, so the second condition requires $p\le\frac{3-\gamma}{\gamma-1}$ and fails whenever $\gamma\ge2$. For solutions with any number of sign changes, \cite[Theorem~4]{CortazarGarciaHuidobroHerreros2019} requires $(F/f)'\ge C$ on $(s_1,\infty)$ for a constant $C\le1$ with $C\ge\sup_{r>0}\big(\frac{g(r)}{\varphi(r)}\int_0^r\varphi-\frac12\big)$, together with further conditions on $g$. For the model nonlinearity, with $t=s^{p-1}$, a direct computation gives
    \begin{equation}\label{eq:F-over-f}
        \Big(\frac Ff\Big)'(s)=\frac1{p+1}+\frac{(p-1)\big((p-2)t+1\big)}{2(p+1)(t-1)^2}\qquad(s>\alpha_*).
    \end{equation}
    Hence $(F/f)'(s)\to\frac1{p+1}$ as $s\to\infty$. The infimum of $(F/f)'$ on $(\alpha_*,\infty)$ is $\frac1{p+1}$ if $p\ge2$, and if $1<p<2$ it is attained at $t=\frac p{2-p}$ and equals $\frac{8-(2-p)^2}{8(p+1)}$. In particular $(F/f)'>0$ on $(\alpha_*,\infty)$. The expression in the supremum tends to $\frac{\gamma-1}2$ as $r\to0$, so the condition of Theorem~4 requires $\frac{\gamma-1}2\le\inf_{(\alpha_*,\infty)}(F/f)'\le\frac1{p+1}$, in particular again $p\le\frac{3-\gamma}{\gamma-1}$ (for $1<p<2$ the exact condition is the stricter $8-(2-p)^2\ge4(\gamma-1)(p+1)$, as in Example~\ref{ex:real-horns}), and it fails whenever $\gamma\ge2$. On $\Hyp^n$ the supremum is moreover infinite. Hence \cite{CortazarGarciaHuidobroHerreros2019} does not give uniqueness of nodal bound states for the model nonlinearity on any $C^2$ warped product of dimension $n\ge3$. For $1<\gamma<2$, which occurs at singular tips, the two theorems can apply only for $p\le\frac{3-\gamma}{\gamma-1}=\frac2{\gamma-1}-1$, below the range $p\ge\frac2{\gamma-1}$ of (R), and on the horns of Example~\ref{ex:real-horns} they do apply in part of this range. The Euclidean predecessor of \cite{CortazarGarciaHuidobroHerreros2019} for sign-changing solutions is \cite{CortazarGarciaHuidobroYarur2011}.

    (c) \emph{Generalized Pohozaev identities.} Ground states of $\Delta u+a(r)u+b(r)u^p=0$ and of $\div(\rho\nabla u)+\rho(-au+bu^p)=0$ were shown to be unique by Yanagida~\cite{Yanagida1991} and by Shioji and Watanabe~\cite{ShiojiWatanabe2013,ShiojiWatanabe2016}, by means of Pohozaev-type identities (an early generalized Pohozaev identity is in~\cite{KawanoNiYotsutani1990}) whose coefficients solve an auxiliary linear system and whose derivative is a multiple of $u^2$. Radial solutions of $\Delta_{g_M}u-u+|u|^{p-1}u=0$ on a warped product are radial solutions of the second equation in $\R^n$ with $\rho=(\psi/r)^{n-1}$ and $a=b\equiv1$, so these results apply to the ground states of Corollary~\ref{cor:ground-state} whenever their conditions hold. The identity \eqref{eq:pohozaev-identity} is of a different kind: it keeps both a $|u|^{p+1}$-term and a $u^2$-term, whose ratio defines the moving Pohozaev level of Definition~\ref{def:pohozaev-level}, and the same multiplier field controls the Wronskian identities used for nodal bound states.

    (d) \emph{Other nonlinearities on $\R^n$.} Uniqueness of the $k$-node bound state on $\R^n$ has recently been proved for a sublinear model nonlinearity~\cite{ZhangZhang2026}, by a shooting argument with a detailed analysis of the variation, and for the logarithmic Schr\"odinger equation~\cite{LiuSunZou2026}, with the auxiliary functions of~\cite{Tang2026}, and for a double-power nonlinearity with a small defocusing term~\cite{GuoJiLuo2026}. Part~II extends the method of~\cite{Tang2026} to warped products.
\end{remark}

%% file: pictures/phase-scheme.tex
\begin{tikzpicture}[
    x=1cm, y=1.4cm,
    declare function={prof(\t)=abs((2.8-0.2*(1-cos(15*\t)))*cos(15*\t));},
    lab/.style={font=\scriptsize},
    guide/.style={densely dotted, black!45, line width=0.3pt},
    deck/.style={line width=1.2pt},
    bandline/.style={orange!80!black, densely dashed, line width=0.35pt},
    crossing/.style={orange!90!black, line width=2.4pt, line cap=round},
    pt/.style={circle, fill, inner sep=0pt, minimum size=2.8pt},
    interval/.style={{Bar[width=4pt]}-{Bar[width=4pt]}, line width=0.4pt},
    ]
    \def\bA{2.8}       
    \def\bB{2.4}       
    \def\mU{1.98125}   
    \def\aS{1.5625}    
    \def\bH{1.171875}  
    \def\bL{0.78125}   
    \def\xt{2.004}  \def\xrm{2.919} \def\xb{3.658}  \def\xr{4.536} \def\xz{6}
    \def\xrb{7.558} \def\xbb{8.611} \def\xrmb{9.625}  \def\xc{12}
    \def\xHd{4.276} \def\xLd{4.861} \def\xLa{7.195}   \def\xHa{7.858} 
    \fill[blue!11] (\xb,\aS) -- plot[domain=\xb:\xz, samples=60] (\x,{prof(\x)})
    -- plot[domain=\xz:\xbb, samples=60] (\x,{prof(\x)}) -- cycle;
    \fill[blue!24] (\xr,1) -- plot[domain=\xr:\xz, samples=40] (\x,{prof(\x)})
    -- plot[domain=\xz:\xrb, samples=40] (\x,{prof(\x)}) -- cycle;
    \node[lab, align=center] at (\xz,1.37) {river: $|u|<\alpha_*$, $F(u)<0$};
    \node[lab] at (\xz,0.6) {$f(u)u<0$};
    \draw[bandline] (0,\bH) -- (\xc,\bH);
    \draw[bandline] (0,\bL) -- (\xc,\bL);
    \draw[crossing] plot[domain=\xHd:\xLd, samples=20] (\x,{prof(\x)});
    \draw[crossing] plot[domain=\xLa:\xHa, samples=20] (\x,{prof(\x)});
    \node[lab, orange!70!black, anchor=east] at (\xHd-0.04,0.98) {$J$};
    \node[lab, orange!70!black, anchor=west] at (\xHa+0.04,0.98) {$J'$};
    \foreach \yy/\xx in {\bB/\xt, \mU/\xrm, \aS/\xb, 1/\xr} \draw[guide] (0,\yy) -- (\xx,\yy);
    \foreach \xx/\yy in {\xt/\bB, \xrm/\mU, \xb/\aS, \xr/1, \xrb/1, \xbb/\aS, \xrmb/\mU, \xc/\bB}
    \draw[guide] (\xx,\yy) -- (\xx,0);
    \draw[-{Stealth[length=1.6mm]}, line width=0.4pt] (-0.15,0) -- (12.55,0) node[right, lab] {$r$};
    \draw[-{Stealth[length=1.6mm]}, line width=0.4pt] (0,0) -- (0,3.25) node[above, lab] {$|u|$};
    \draw[line width=0.8pt] plot[domain=0:\xz, samples=121] (\x,{prof(\x)})
    -- plot[domain=\xz:\xc, samples=121] (\x,{prof(\x)});
    \draw[deck] (\xt,\bB) -- (\xc,\bB);
    \draw[deck, densely dashed] (\xrm,\mU) -- (\xrmb,\mU);
    \draw[deck] (\xb,\aS) -- (\xbb,\aS);
    \node[lab, above, inner sep=1.5pt] at (7.1,\bB) {$\sB_{a_i}$, $a_i=g_2(\beta_i)$ \ (Thm.~\ref{thm:renewal-above})};
    \node[lab, above, inner sep=1.5pt] at (6.3,\mU) {levels paired at $r_\mu$, $\bar r_\mu$ (\S\ref{sec:renewal-below})};
    \node[lab, above, inner sep=1.5pt] at (6.13,\aS) {$\sB_0$ (Thm.~\ref{thm:phase}), $\mathfrak V_\Lambda$ (Lem.~\ref{lem:frozen-bridge})};
    \foreach \xx/\yy in {0/\bA, \xt/\bB, \xrm/\mU, \xb/\aS, \xr/1, \xz/0, \xrb/1, \xbb/\aS, \xrmb/\mU, \xc/\bB}
    \node[pt] at (\xx,\yy) {};
    \foreach \xx/\ll in {0/c_{i-1}, \xt/t_i, \xrm/r_\mu, \xb/b_i, \xr/r_i, \xz/z_i, \xrb/\bar r_i,
    \xbb/\bar b_i, \xrmb/\bar r_\mu, \xc/c_i}
    \node[lab, anchor=base] at (\xx,-0.27) {$\ll$};
    \foreach \yy/\ll in {\bA/\beta_{i-1}, \bB/\beta_i, \mU/\mu, \aS/\alpha_*, 1/1}
    \node[lab, left, inner sep=2pt] at (-0.02,\yy) {$\ll$};
    \node[lab, right, inner sep=2pt] at (\xc,\bH) {$\tfrac34\alpha_*$};
    \node[lab, right, inner sep=2pt] at (\xc,\bL) {$\tfrac12\alpha_*$};
    \draw[decorate, decoration={brace, amplitude=3pt}, line width=0.4pt]
    (12.78,\bH) -- (12.78,\bL) node[midway, right=3pt, lab, align=left] {band\\(Lem.~\ref{lem:band-dissipation})};
    \draw[decorate, decoration={brace, amplitude=5pt}, line width=0.4pt]
    (0.03,2.98) -- (5.97,2.98) node[midway, above=4pt, lab] {descent $(c_{i-1},z_i)$: $u>0$, $u'<0$};
    \draw[decorate, decoration={brace, amplitude=5pt}, line width=0.4pt]
    (6.03,2.98) -- (11.97,2.98) node[midway, above=4pt, lab] {ascent $(z_i,c_i)$: $u<0$, $u'<0$};
    \node[lab, left, inner sep=2pt] at (-0.02,-0.62) {$\tau_i$:};
    \draw[interval] (0,-0.62) -- (\xt,-0.62) node[midway, below, lab, inner sep=2pt, xshift=-3pt] {$u_\tau\ge\beta_i$ (\S\ref{sec:renewal-above})};
    \draw[interval] (\xt,-0.62) -- (\xr,-0.62) node[midway, below, lab, inner sep=2pt] {$u_\tau<\beta_i$ (\S\ref{sec:renewal-below})};
    \draw[interval] (\xt,-1.05) -- (\xb,-1.05) node[midway, below, lab, inner sep=2pt] {$\tau_i\le b_i$};
    \draw[interval] (\xb,-1.05) -- (\xr,-1.05) node[midway, below, lab, inner sep=2pt] {$\tau_i>b_i$};
\end{tikzpicture}

%% file: sections/11-multiplier-field.tex
\section{The multiplier field}\label{sec:multiplier}

Every uniqueness proof in the Euclidean case uses the scaling vector field $r\partial_r$. A valid and intuitive substitute in the present setting is the field $\varphi(r)\partial_r$ defined below.

\begin{definition}\label{def:multiplier}
    Assume (G1), (G4) and the first half of (G5), $\int^\infty e^{-G}<\infty$. Define
    \begin{equation}\label{eq:multiplier}
        \cJ(r)\defeq\int_r^\infty e^{-G(s)}\,ds,\qquad
        D(r)\defeq g(r)\,e^{G(r)}\cJ(r),\qquad
        \varphi(r)\defeq(\gamma-1)\,e^{G(r)}\cJ(r),
    \end{equation}
    so that $D=\frac{g\varphi}{\gamma-1}$ and
    \begin{equation}\label{eq:phi-ode}
        \varphi'=g\varphi-(\gamma-1),\qquad\varphi''=(g\varphi)'=(\gamma-1)D'.
    \end{equation}
    The \emph{effective dimension} is
    \begin{equation}\label{eq:nu}
        \nu(r)\defeq2\varphi'(r)+\gamma-1.
    \end{equation}
\end{definition}

\begin{lemma}[Basic properties]\label{lem:phi-basic}
    Assume (G1), (G2), (G4) and $\int^\infty e^{-G}<\infty$ (the first half of (G5); the boundedness of $g$ is not used in this section). Then $\varphi$ is well defined, positive and $C^3$ on $(0,\infty)$, and satisfies \eqref{eq:phi-ode}. Moreover
    \[
        \lim_{r\downarrow0}\varphi(r)=0,\qquad
        \lim_{r\downarrow0}\frac{\varphi(r)}{r}=\lim_{r\downarrow0}\varphi'(r)=1,\qquad
        \lim_{r\downarrow0}D(r)=\frac{\gamma}{\gamma-1},\qquad
        \lim_{r\downarrow0}\nu(r)=\gamma+1.
    \]
    If in addition $(\Phi)$ holds, then on $(0,\infty)$
    \begin{equation}\label{eq:phi-bounds}
        0<\varphi'\le1,\qquad\gamma-1<g\varphi\le\gamma,\qquad\gamma-1<\nu\le\gamma+1,
    \end{equation}
    $\varphi$ is strictly increasing, concave and sublinear ($\varphi(r)\le r$), and $\varphi'$, $g\varphi$, $\nu$, $D$ are nonincreasing.
\end{lemma}

\begin{proof}
    Finiteness of the integral requires $\int^\infty e^{-G}<\infty$, which is the first half of (G5). Differentiating \eqref{eq:multiplier} gives $\varphi'=g\varphi-(\gamma-1)$ and then $\varphi''=(g\varphi)'$, so $\varphi\in C^3$ by (G4). For the limit at $0$, fix $\gamma'\in(1,\gamma)$. By (G2), $sg(s)\ge\gamma'$ on some interval $(0,s_0]$, so $e^{-G(s)}\ge e^{-G(s_0)}(s_0/s)^{\gamma'}$ for $s\in(0,s_0]$. Hence $\cJ(r)$ and $re^{-G(r)}$ tend to $+\infty$ as $r\downarrow0$. Since $\big(re^{-G}\big)'=(1-rg)e^{-G}$, l'H\^opital's rule gives
    \[
        \lim_{r\downarrow0}\frac{\varphi(r)}{r}
        =\lim_{r\downarrow0}\frac{(\gamma-1)\cJ(r)}{re^{-G(r)}}
        =\lim_{r\downarrow0}\frac{\gamma-1}{rg(r)-1}=1,
    \]
    and in particular $\varphi(r)\to0$. Consequently $sg(s)\cdot\frac{\varphi(s)}{s}\to\gamma$, so $\varphi'(s)=g\varphi-(\gamma-1)\to1$ and $D=\frac{g\varphi}{\gamma-1}\to\frac{\gamma}{\gamma-1}$, $\nu=2\varphi'+\gamma-1\to\gamma+1$. Under $(\Phi)$, $\varphi''\le0$ makes $\varphi'$ nonincreasing from its limit $1$ at $0$, whence $0<\varphi'\le1$ and $\varphi(r)=\int_0^r\varphi'\le r$. The remaining claims follow from $g\varphi=\varphi'+\gamma-1$ and $\nu=2\varphi'+\gamma-1$.
\end{proof}

\begin{lemma}[Consequences of (G5) and $(\Phi)$]\label{lem:phi-consequences}
    Assume (G1), (G2), (G4), (G5), $(\Phi)$, and let $g\le M_0$ on $[R_1,\infty)$. Then:
    \begin{enumerate}[label=(\roman*)]
        \item $g(r)>\dfrac{\gamma-1}{r}$ for every $r>0$, hence $G(r)\ge(\gamma-1)\log r$ for $r\ge1$ and $e^{G(r)}\to\infty$.
        \item $\displaystyle\int_{R_1}^r e^G\ge\frac{e^{G(r)}-e^{G(R_1)}}{M_0}$ for $r\ge R_1$.
        \item $\displaystyle e^{-G(t)}\int_a^t e^{G(s)}\,ds\le t-a$ for $0<a\le t$.
        \item $g'<0$ on $(0,\infty)$.
    \end{enumerate}
    Conversely, if (G1), (G4), (G5) hold and $g'<0$ on $(0,\infty)$, then $D>1$ on $(0,\infty)$; if moreover $\varphi''\le0$ on $(0,\infty)$, then $(\Phi)$ holds. Thus, under (G1), (G4), (G5), $(\Phi)$ is equivalent to: \emph{$g$ is strictly decreasing ($g'<0$) and $\varphi$ is concave.} For $g=g_\psi$, $g'<0$ is the strict log-concavity of the warping function, $\psi\psi''<\psi'^2$.
\end{lemma}

\begin{proof}
    (i) By $(\Phi)$, $D>1$, i.e.\ $g\varphi>\gamma-1$, and $\varphi\le r$ by Lemma~\ref{lem:phi-basic}.

    (ii) $\int_{R_1}^r e^G=\int_{R_1}^r\frac{(e^G)'}{g}\ge\frac{1}{M_0}\big(e^{G(r)}-e^{G(R_1)}\big)$.

    (iii) $G$ is increasing.

    (iv) By \eqref{eq:phi-ode}, $\varphi''=(g\varphi)'=g'\varphi+g\varphi'$, so $g'\varphi=\varphi''-g\varphi'<0$ under $(\Phi)$. Conversely, let $g'<0$ and $\varphi''\le0$. For $t>s$ we have $G(t)-G(s)<g(s)(t-s)$, hence
    \[
        \cJ(s)>e^{-G(s)}\int_s^\infty e^{-g(s)(t-s)}\,dt=\frac{e^{-G(s)}}{g(s)},
    \]
    that is $D(s)=g(s)e^{G(s)}\cJ(s)>1$; this uses only $g'<0$. If moreover $\varphi''\le0$, then $D'\le0$ by \eqref{eq:phi-ode}, and $(\Phi)$ holds. For $g=g_\psi$, $g_\psi'=(n-1)(\psi\psi''-\psi'^2)/\psi^2$.
\end{proof}

\begin{remark}\label{rem:phi-euclidean}
    In the Euclidean case $g=\frac{n-1}{r}$, so $\gamma=n-1$: one has $\varphi(r)=r$, $\varphi'\equiv1$, $\varphi''\equiv0$, $D\equiv\frac{n-1}{n-2}$ and $\nu\equiv n$. All the correction terms appearing below, which carry a factor $\varphi''$ or $\varphi'-1$, vanish identically, so the Euclidean case is a special case of the present one. (Here $\int^\infty e^{-G}=\int^\infty s^{1-n}\,ds<\infty$ for $n\ge3$ and $g$ is bounded, so (G5) holds.)
\end{remark}

\subsection{A differential criterion for \texorpdfstring{$(\Phi)$}{(Phi)}}\label{sec:phi-criterion}

\begin{lemma}[Criterion]\label{lem:phi-criterion}
    Assume (G1), (G4), $\int^\infty e^{-G}<\infty$, let $g'+g^2>0$ on $(0,\infty)$, set $\mathcal{H}\defeq\dfrac{g\,e^{-G}}{g'+g^2}$, and suppose
    \begin{equation}\label{eq:phi-criterion}
        \Big(\frac{g}{g'+g^2}\Big)'\le-\frac{g'}{g'+g^2}\quad\text{on }(0,\infty),\qquad\text{and}\qquad
        \mathcal{H}(r)\to0\ \ (r\to\infty).
    \end{equation}
    Then $D'\le0$ on $(0,\infty)$. If moreover the inequality in \eqref{eq:phi-criterion} is strict and $D(r)\to1$ as $r\to\infty$, then $D'<0$ and $D>1$ on $(0,\infty)$, so $(\Phi)$ holds.
\end{lemma}

\begin{proof}
    A direct computation from \eqref{eq:multiplier} gives
    \[
        D'=(g'+g^2)e^G\cJ-g=(g'+g^2)e^G\big(\cJ-\mathcal{H}\big),
    \]
    so $D'\le0$ if and only if $\cJ\le\mathcal{H}$. Now
    \[
        (\mathcal{H}-\cJ)'=e^{-G}\Big[\Big(\frac{g}{g'+g^2}\Big)'-\frac{g^2}{g'+g^2}\Big]+e^{-G}
        =e^{-G}\Big[\Big(\frac{g}{g'+g^2}\Big)'+\frac{g'}{g'+g^2}\Big]\le0
    \]
    by \eqref{eq:phi-criterion}. Since $\mathcal{H}-\cJ\to0$ at infinity, $\mathcal{H}-\cJ\ge0$ everywhere, proving $D'\le0$. If \eqref{eq:phi-criterion} is strict, then $(\mathcal{H}-\cJ)'<0$, so $\mathcal{H}-\cJ$ is strictly decreasing to $0$, hence $\mathcal{H}-\cJ>0$ and $D'<0$ everywhere. Since $D$ is strictly decreasing with limit $1$, we get $D>1$.
\end{proof}

The criterion can be written in a much simpler form.

\begin{lemma}[The criterion is the concavity of $1/g$]\label{lem:concave-reciprocal}
    Let $g\in C^2((0,\infty))$, $g>0$, and put $\ell\defeq1/g$. Wherever $\ell'\ne1$ (that is, $g'+g^2\ne0$),
    \begin{equation}\label{eq:concave-reciprocal}
        \Big(\frac{g}{g'+g^2}\Big)'+\frac{g'}{g'+g^2}=\frac{\ell\,\ell''}{(1-\ell')^2}.
    \end{equation}
    Hence the inequality in \eqref{eq:phi-criterion} holds if and only if $\ell''\le0$, i.e.\ $g\,g''\ge2g'^2$: \emph{the reciprocal damping $1/g$ is concave}. For $g=g_\psi=(n-1)\psi'/\psi$ one has $\ell=\psi/((n-1)\psi')$ and $(n-1)\ell'=1-\psi\psi''/\psi'^2$, so the condition is that $\psi\psi''/\psi'^2=-K_{\mathrm{rad}}\,(\psi/\psi')^2$ be nondecreasing, $K_{\mathrm{rad}}=-\psi''/\psi$ being the radial curvature.

    Assume moreover (G1), (G2), (G4), (G5), and that $\ell$ is concave on $(0,\infty)$. Then:
    \begin{enumerate}[label=(\roman*)]
        \item $\ell'$ is nonincreasing with $\ell'(0^+)=\frac{1}{\gamma}$. Hence $0\le\ell'\le\frac{1}{\gamma}$, $g'\le0$, $g\ge\frac{\gamma}{r}$, and, for $g=g_\psi$ with $\gamma\ge n-1$, $\psi''\ge0$.
        \item $D'\le0$ and $D\ge1$ on $(0,\infty)$, and $D(r)>1$ unless $g$ is constant on $(r,\infty)$. In particular, if $g$ is not eventually constant, then $(\Phi)$ holds.
    \end{enumerate}
\end{lemma}

\begin{proof}
    With $g=1/\ell$: $g'+g^2=(1-\ell')/\ell^2$, $\frac{g}{g'+g^2}=\frac{\ell}{1-\ell'}$, $\frac{g'}{g'+g^2}=-\frac{\ell'}{1-\ell'}$, and $\big(\frac{\ell}{1-\ell'}\big)'=\frac{\ell'(1-\ell')+\ell\ell''}{(1-\ell')^2}$, which gives \eqref{eq:concave-reciprocal}.

    (i) A concave function has a nonincreasing derivative. By the mean value theorem, $\ell(r)/r=\ell'(\xi_r)$ for some $\xi_r\in(0,r)$, and $\ell(r)/r=1/(rg(r))\to\frac{1}{\gamma}$ by (G2), so $\ell'(0^+)=\frac{1}{\gamma}$ and $\ell'\le\frac{1}{\gamma}$. If $\ell'(r_1)<0$ for some $r_1$ then $\ell'\le\ell'(r_1)<0$ on $(r_1,\infty)$ and $\ell\to-\infty$, contradicting $\ell>0$. Thus $\ell'\ge0$, i.e.\ $g'=-\ell'/\ell^2\le0$, and $\ell(r)=\int_0^r\ell'\le\frac{r}{\gamma}$, i.e.\ $g\ge\frac{\gamma}{r}$. Finally, if $\gamma\ge n-1$, then $(n-1)\ell'=1-\psi\psi''/\psi'^2\le(n-1)/\gamma\le1$ gives $\psi''\ge0$.

    (ii) Since $e^{-G}$ is decreasing and integrable by (G5), $r\,e^{-G(r)}\le2\int_{r/2}^r e^{-G}\to0$. Hence $\mathcal{H}=\frac{ge^{-G}}{g'+g^2}=\frac{e^{-G}\ell}{1-\ell'}\le\frac{r\,e^{-G}}{\gamma-1}\to0$, and by \eqref{eq:concave-reciprocal} and the proof of Lemma~\ref{lem:phi-criterion}, $(\mathcal{H}-\cJ)'=e^{-G}\frac{\ell\ell''}{(1-\ell')^2}\le0$. Therefore $\mathcal{H}\ge\cJ$, that is $D'\le0$. Next, $e^{-G}=-\ell\,(e^{-G})'$, and an integration by parts, in which the boundary term $\ell e^{-G}\le re^{-G}/\gamma$ vanishes at infinity, gives
    \[
        \cJ(r)=\int_r^\infty e^{-G}=\ell(r)e^{-G(r)}+\int_r^\infty\ell'e^{-G}\,ds,\qquad\text{i.e.}\qquad
        D(r)=1+g(r)e^{G(r)}\int_r^\infty\ell'e^{-G}\,ds\ \ge1,
    \]
    with equality iff $\ell'=0$ a.e.\ on $(r,\infty)$, i.e.\ (as $\ell'$ is nonincreasing and nonnegative) iff $g$ is constant on $(r,\infty)$. Since $(\Phi)$ is exactly $D>1$ and $D'\le0$, this proves the last statement.
\end{proof}

\begin{lemma}[Euclidean and hyperbolic spaces satisfy $(\Phi)$]\label{lem:space-forms}
    Let $n\ge3$.
    \begin{enumerate}[label=(\roman*)]
        \item \textup{(Euclidean.)} $g=\frac{n-1}{r}$ satisfies \eqref{eq:phi-criterion} with equality and $D\equiv\frac{n-1}{n-2}>1$, so $(\Phi)$ holds.
        \item \textup{(Hyperbolic, any curvature.)} For $\lambda>0$, the damping $g=(n-1)\lambda\coth(\lambda r)$, corresponding to the warping $\psi(r)=\lambda^{-1}\sinh(\lambda r)$, satisfies $(\Phi)$: $D>1$ and $D'<0$ on $(0,\infty)$, with $D(0^+)=\frac{n-1}{n-2}$ and $D(\infty)=1$. For $n=3$, $\lambda=1$ one has explicitly $\varphi(r)=\frac{1}{2}(1-e^{-2r})$ and $D(r)=1+e^{-2r}$.
    \end{enumerate}
\end{lemma}

\begin{proof}
    (i) With $g=\frac{n-1}{r}$ we have $g'+g^2=\frac{(n-1)(n-2)}{r^2}>0$ and $\frac{g}{g'+g^2}=\frac{r}{n-2}$, so the left-hand side of \eqref{eq:phi-criterion} is $\frac{1}{n-2}$, while $-\frac{g'}{g'+g^2}=\frac{1}{n-2}$: equality. Directly, $D\equiv\frac{n-1}{n-2}$.

    (ii) By the scaling $r\mapsto\lambda r$ it suffices to take $\lambda=1$ (one checks that $D_\lambda(r)=D_1(\lambda r)$ from \eqref{eq:multiplier}). Put $x=\coth r\in(1,\infty)$, so that $\frac{dx}{dr}=1-x^2<0$. Then $g=(n-1)x$, $g'=(n-1)(1-x^2)$, and
    \[
        g'+g^2=(n-1)\big[1+(n-2)x^2\big]>0,\qquad\frac{g}{g'+g^2}=\frac{x}{1+(n-2)x^2}.
    \]
    Hence
    \[
        \Big(\frac{g}{g'+g^2}\Big)'=(1-x^2)\frac{1-(n-2)x^2}{\big[1+(n-2)x^2\big]^2},\qquad
        -\frac{g'}{g'+g^2}=-\frac{1-x^2}{1+(n-2)x^2},
    \]
    and since $1-x^2<0$, the inequality \eqref{eq:phi-criterion} is equivalent, after dividing by $1-x^2$ and reversing the sign, to
    \[
        \frac{1-(n-2)x^2}{1+(n-2)x^2}\ge-1\qquad\Longleftrightarrow\qquad1\ge-1,
    \]
    which holds strictly (equivalently, $\ell=\frac{\tanh r}{n-1}$ has $\ell''=-\frac{2\tanh r}{(n-1)\cosh^2r}<0$, Lemma~\ref{lem:concave-reciprocal}). Moreover $e^{-G}\sim c\,e^{-(n-1)r}$, so $\mathcal{H}=\frac{ge^{-G}}{g'+g^2}\to0$, and $\cJ(r)=\int_r^\infty e^{-G}=\frac{e^{-G(r)}}{g(r)}(1+o(1))$ (by l'H\^opital, since $g'/g^2\to0$), whence $D\to1$. Lemma~\ref{lem:phi-criterion} applies and gives $(\Phi)$, with strict inequality $D'<0$ because \eqref{eq:phi-criterion} is strict. The limit at $0$ is Lemma~\ref{lem:phi-basic}. For $n=3$, $\int_r^\infty\sinh^{-2}=\coth r-1$ gives $\varphi=\sinh^2r\,(\coth r-1)=\frac{1}{2}(1-e^{-2r})$ and $D=\coth r\,(1-e^{-2r})=1+e^{-2r}$.
\end{proof}

\begin{lemma}[A family of variable negative curvature]\label{lem:variable-curvature}
    Let $n\ge3$, $\lambda>0$ and $\psi(r)\defeq r\cosh(\lambda r)$. This is a $C^\infty$ warping function ($\psi(0)=0$, $\psi'(0)=1$, $\psi''(0)=0$), with radial curvature $-\psi''/\psi=-\lambda^2-\frac{2\lambda\tanh(\lambda r)}{r} \in(-3\lambda^2,-\lambda^2)$, asymptotic to the hyperbolic space of curvature $-\lambda^2$ but not of constant curvature. The damping $g=(n-1)\big(\frac{1}{r}+\lambda\tanh(\lambda r)\big)$ satisfies (G1)--(G6), $1/g$ is strictly concave on $(0,\infty)$, and $(\Phi)$ holds.
\end{lemma}

\begin{proof}
    (G1)--(G4) and the boundedness in (G5) are immediate ($g\to(n-1)\lambda$, and $g-\frac{n-1}{r}=(n-1)\lambda\tanh(\lambda r)$ is bounded). The function
    \[
        e^{-G}=(\psi(1)/\psi)^{n-1}\le\psi(1)^{n-1}(2/(re^{\lambda r}))^{n-1}
    \]
    is integrable, which is the first half of (G5). Finally, (G6) holds since $g$ is smooth and
    \[
        g'=(n-1)\big(\lambda^2\cosh^{-2}(\lambda r)-r^{-2}\big)\to0,\qquad q'=\tfrac{1}{2}(gg'+g'')=O(r^{-2}),\qquad q''=O(r^{-3}).
    \]

    For the concavity put $x=\lambda r$, $t=\tanh x$: then $\ell=\frac{1}{g}=\frac{1}{(n-1)\lambda}\,\frac{x}{1+xt}$, and a direct computation (with $\frac{dt}{dx}=1-t^2$) gives
    \[
        \ell''(r)=-\frac{2\lambda}{n-1}\,\frac{\Theta(x)}{(1+xt)^3},\qquad
        \Theta(x)\defeq t+x(1-t^2)\big(2-xt-x^2\big).
    \]
    If $x^2+xt\le2$ then $\Theta\ge t>0$. If $x^2+xt>2$, then, since $t\le1$ and $1-t^2=\cosh^{-2}x$, we have $x(1-t^2)(x^2+xt-2)\le x(x^2+x-2)\cosh^{-2}x$. In this case $\Theta>0$ follows from $x^3+x^2-2x<\tanh x\cosh^2x=\frac{1}{2}\sinh2x$. Indeed, $\frac{1}{2}\sinh2x\ge x+\frac{2}{3}x^3+\frac{2}{15}x^5$, and
    \begin{align*}
        x+\frac{2}{3}x^3+\frac{2}{15}x^5-(x^3+x^2-2x)
        &=x\Big(3-x-\frac{x^2}{3}+\frac{2}{15}x^4\Big)\\
        &\ge\frac{2x}{15}\Big(x^4-\frac{25}{4}x^2+\frac{75}{4}\Big)>0,
    \end{align*}
    using $x\le\frac{x^2+1}{2}$ and the negativity of the discriminant $\frac{625}{16}-75$. Hence $\ell''<0$ on $(0,\infty)$, and $(\Phi)$ follows from Lemma~\ref{lem:concave-reciprocal}(ii), as $g$ is not eventually constant.
\end{proof}

\begin{remark}[When $(\Phi)$ fails]\label{rem:phi-fails}
    Suppose that $g\in C^1$ is nondecreasing on $(r_0,\infty)$ and not constant on any interval $(r,\infty)$ with $r\ge r_0$. Then the integration by parts in the proof of Lemma~\ref{lem:concave-reciprocal}(ii) gives
    \[
        D(r)-1=g(r)\,e^{G(r)}\int_r^\infty\Big(\frac1g\Big)'e^{-G}\,ds<0\qquad(r\ge r_0),
    \]
    where the boundary term vanishes because $1/g$ is bounded on $(r_0,\infty)$ and $e^{-G}\to0$. Hence $(\Phi)$ fails. In particular, the expanding examples $g=\frac{n-1}{r}+r^b$ ($b>0$) of Section~\ref{sec:manifolds} do not satisfy $(\Phi)$ (nor (G5)). For $n=3$ and $b=1$ the exact primitive is $G(s)=2\log s+\frac{s^2-1}{2}$, and one computes that $D$ decreases from $D(0^+)=2$ and crosses $1$ at $r\approx1.16$. Then $D$ attains its minimum $D\approx0.9423$ at $r\approx2.24$ and returns to $1$ from below. Correspondingly, $\varphi'$ changes sign. Uniqueness for such $g$ is not covered by the present method.
\end{remark}

\begin{remark}[Geometric interpretation]\label{rem:phi-geometry}
    Let $g=g_\psi$. By the converse part of Lemma~\ref{lem:phi-consequences}, strict log-concavity of $\psi$ already gives $D>1$, the first half of $(\Phi)$. Both space forms have $g'<0$. The second half of $(\Phi)$, $D$ nonincreasing, holds under \eqref{eq:phi-criterion} (Lemma~\ref{lem:phi-criterion}), which by Lemma~\ref{lem:concave-reciprocal} is the concavity of $1/g_\psi=\psi/((n-1)\psi')$, i.e.\ the monotonicity of $\psi\psi''/\psi'^2=-K_{\mathrm{rad}}(\psi/\psi')^2$. This is not a monotonicity condition on $K_{\mathrm{rad}}$ itself: in Lemma~\ref{lem:variable-curvature} the radial curvature increases from $-3\lambda^2$ to $-\lambda^2$, and $1/g_\psi$ is nevertheless concave. For the models, $1/g_\psi$ is $\frac{r}{n-1}$ and $\frac{\tanh(\lambda r)}{(n-1)\lambda}$.
\end{remark}

%% file: sections/12-basic-properties.tex
\section{Basic properties of solutions for the model nonlinearity}\label{sec:basic-properties}

We collect the elementary properties of solutions, now for the model nonlinearity \eqref{eq:model-f}. Throughout this section we use only (G1), (G2), (G4), (G5) and, in Lemma~\ref{lem:band-dissipation}, $(\Phi)$, together with the following properties of $f$:
\begin{itemize}
    \item $f$ is odd and locally Lipschitz,
    \item $f<0$ on $(0,1)$ and $f>0$ on $(1,\infty)$, with $f$ increasing on $[1,\infty)$,
    \item $F<0$ on $(0,\alpha_*)$ and $F>0$ on $(\alpha_*,\infty)$,
    \item $f'(0)=-1<0$,
    \item $F(u)\ge F(1)>-\infty$.
\end{itemize}
Recall the energy of Definition~\ref{def:energy},
\begin{equation}\label{eq:energy-model}
    E(r)=\tfrac12 u'(r)^2+F(u(r)), \qquad E'=-gu'^2\le 0,
\end{equation}
which is strictly decreasing along nonconstant solutions, as shown after \eqref{eq:energy-decay}. We also recall that zeros and critical points of a nonconstant solution are simple, in the sense that $(u,u')\neq(0,0)$ everywhere (by uniqueness for the ODE away from $r=0$), and that $u\equiv 1$ if and only if $\alpha=1$. The trapping property of Proposition~\ref{prop:trapping} also holds when $E(r_0)=0$ and $u(r_0)\ne0$: then $E(r_1)<0$ for every $r_1>r_0$, and $u$ cannot vanish on $(r_0,r_1]$, since $E\ge0$ at a zero. Hence a solution with $E(r_0)\le0$ at some $r_0$ is not a bound state and has no zeros after $r_0$.

\begin{proposition}[Critical points]\label{prop:critical-points}
    Let $u$ be nonconstant with $u(0)=\alpha>0$.
    \begin{enumerate}[label=(\roman*)]
        \item At every critical point $c$ with $E(c)>0$ one has $|u(c)|>\alpha_*$, $f(u(c))u(c)>0$, and $c$ is a strict local extremum of $|u|$ (a local maximum of $u$ if $u(c)>0$, a local minimum if $u(c)<0$).
        \item If $c<\tilde c$ are critical points with $E(\tilde c)>0$, then $|u(\tilde c)|<|u(c)|$. Moreover, if $\alpha\ge 1$ then $\|u\|_\infty=\alpha$.
        \item Let $z_k$ be the $k$-th zero (so that $E>0$ on $(0,z_k]$, since $E(z_k)=\frac12u'(z_k)^2>0$ and $E$ is nonincreasing). Then between consecutive zeros there is exactly one critical point, $u$ has exactly the critical points $0=c_0<c_1<\cdots<c_{k-1}$ below $z_k$ with $c_i\in(z_i,z_{i+1})$ for $1\le i\le k-1$, and $\alpha=|u(c_0)|>|u(c_1)|>\cdots>|u(c_{k-1})|>\alpha_*$.
        \item If $u$ is a bound state, then $\alpha>\alpha_*$ and $E>0$ on $[0,\infty)$ with $E(\infty)=0$. A ground state has $u'<0$ on $(0,\infty)$. A bound state with $k\ge 1$ zeros has exactly one critical point $c_k>z_k$ beyond its last zero, with $|u(c_k)|>\alpha_*$, and $|u|$ is strictly decreasing on $(c_k,\infty)$.
    \end{enumerate}
\end{proposition}

\begin{proof}
    (i) For $c>0$ this is Lemma~\ref{lem:positive-critical}, with $\beta_F=\alpha_*$: the inequality is strict since $F(u(c))=E(c)>0$. For $c=0$ use $u''(0)=-f(\alpha)/(\gamma+1)$.

    (ii) $F(u(\tilde c))=E(\tilde c)<E(c)=F(u(c))$, and $E(\tilde c)>0$ forces $E(c)>0$, so both critical values exceed $\alpha_*$ in absolute value by (i). The strict monotonicity of $F$ on $(\alpha_*,\infty)$ and the evenness of $F$ then give $|u(\tilde c)|<|u(c)|$. For the supremum: for $r>0$ one has $F(u(r))\le E(r)<E(0)=F(\alpha)$. If $|u(r)|\ge1$, then, $F$ being even and strictly increasing on $[1,\infty)$, this forces $|u(r)|<\alpha$, while if $|u(r)|<1$ then $|u(r)|<1\le\alpha$. Hence $\|u\|_\infty=\alpha$.

    (iii) Between consecutive zeros $z_i<z_{i+1}$ there is at least one critical point by Rolle's theorem, and at most one by Lemma~\ref{lem:positive-critical}, applied to $\pm u$ on $(z_i,z_{i+1})$, where $E>0$. There is no critical point in $(0,z_1)$: since $E(0)=F(\alpha)>0$, $0$ is a strict local maximum of $u$ by (i), and a critical point $c\in(0,z_1)$ would be one as well, so $u$ would have a local minimum in $(0,c)$, which is a critical point and contradicts (i). Then $|u(c_i)|>\alpha_*$ by (i), and the chain of inequalities is (ii).

    (iv) If $E(r_0)\le 0$ at some $r_0$ then $u(r_0)\neq 0$ (otherwise $E(r_0)=\frac12 u'(r_0)^2$ forces $(u,u')(r_0)=0$, so $u\equiv 0$), and trapping (Proposition~\ref{prop:trapping} and the remark after \eqref{eq:energy-model}) contradicts $u\to 0$. Hence $E>0$ everywhere. In particular $E(0)=F(\alpha)>0$, that is $\alpha>\alpha_*$. Now $E$ decreases to a limit $E(\infty)\ge 0$. Along the tail $F(u)\to 0$, so $u'^2\to 2E(\infty)$, and if $E(\infty)>0$ then $|u'|\ge c>0$ eventually, which is incompatible with $u\to 0$, thus $E(\infty)=0$. For a ground state, $u'<0$ just after $0$ since $f(\alpha)>0$. An interior critical point would then, by (i), be a strict local maximum of $u$, so before it there would be a local minimum, also a critical point, contradicting (i). Hence $u'<0$ throughout. For $k\ge 1$: $u'(z_k)\neq 0$ drives $|u|$ up after $z_k$ while $u\to 0$, so there is a critical point beyond $z_k$, and (i), applied as in the proof of (iii) on the sign-definite interval $(z_k,\infty)$, shows it is unique. Hence $|u|$ decreases strictly on $(c_k,\infty)$.
\end{proof}

\subsection{Passage times and dissipation across the band}\label{sec:passage-times}
The boundedness of $g$ in (G5) is used in the dissipation arguments through Lemma~\ref{lem:passage-times}, through Lemma~\ref{lem:phi-consequences}(ii) in Proposition~\ref{prop:classification}, and through the bound $\sup m<\infty$ (Section~\ref{sec:tail}) in Lemma~\ref{lem:jump}. Lemma~\ref{lem:band-dissipation} uses $(\Phi)$ instead. The two lemmas of this subsection replace, for a general $g$, the explicit bookkeeping of the power of $r$ in the Euclidean case. Define:
\begin{gather*}
    u_1\defeq \frac{1+\alpha_*}{2}\in(1,\alpha_*), \qquad f_{\min}\defeq f(u_1)>0,\\
    \delta\in(0,1)\text{ with }\delta^{p-1}\le\frac{p+1}{4}, \qquad c_\delta\defeq \min_{\delta\le|u|\le u_1}\big(-F(u)\big)>0,
\end{gather*}
so that $-F(u)\ge u^2/4$ for $|u|\le\delta$, and for a solution $u(\cdot,\alpha)$ let $S\defeq \sqrt{2(E(0)-F(1))}$, so that $|u'|\le S$ everywhere.

\begin{lemma}[Passage times]\label{lem:passage-times}
    Assume \textup{(G1)}, \textup{(G2)}, and let $g\le M_0$ on $[R_1,\infty)$ (the boundedness half of \textup{(G5)}; neither \textup{(G4)} nor the integrability of $e^{-G}$ is used). Let $u$ be a nonconstant solution and $[a,b]\subset[R_1,\infty)$ an interval on which $E\ge 0$ and $E(b)>0$. Then:
                \begin{enumerate}[label=(\roman*)]
                    \item on any subinterval $I\subset[a,b]$ on which $|u|\le\delta$, one has $u'^2\ge 2E(b)+u^2/2$, hence $u$ is monotone on $I$ and $|I|\le C_0\big(1+\log^+(1/E(b))\big)$, with $C_0$ depending only on $p$.
                    \item on any subinterval on which $\delta\le|u|\le u_1$, $u$ is monotone and the subinterval has length at most $(u_1-\delta)/\sqrt{2c_\delta}$.
                    \item on any subinterval on which $|u|\ge u_1$ and which contains at most one critical point, the length is at most $S/f_{\min}+1/M_0+2M_0\alpha/f_{\min}$.
                \end{enumerate}
                Consequently, if $z<z'$ are consecutive zeros of $u$ with $z\ge R_1$, then
                \begin{equation}\label{eq:zero-spacing}
                    z'-z\;\le\;T_0\big(1+\log^+(1/E(z'))\big),
                \end{equation}
                and if $\zeta\ge R_1$ is a zero followed by a critical point $c'$ with no zero in $(\zeta,c']$ and $E\ge 0$ on $[\zeta,c']$, then
    \begin{equation}\label{eq:ascent-time}
        c'-\zeta\;\le\;T_0\big(1+\log^+(1/E(c'))\big),
    \end{equation}
    where $T_0$ depends only on $p$, $M_0$ and $\alpha$.
\end{lemma}

\begin{proof}
    (i) On $I$, $u'^2=2(E-F(u))\ge 2E(b)+u^2/2$, so $u'\neq 0$ and $u$ is monotone. Hence
    \begin{align*}
        |I|=\int_{u(I)}\frac{du}{|u'|}
        \le\int_{-\delta}^{\delta}\frac{du}{\sqrt{2E(b)+u^2/2}}
         & =2\sqrt{2}\arcsinh\frac{\delta}{2\sqrt{E(b)}} \\
         & \le C_0\big(1+\log^+(1/E(b))\big),
    \end{align*}
    using $\arcsinh x\le\log(1+2x)$ and $\delta\le 1$.

    (ii) On a subinterval on which $\delta\le|u|\le u_1$ we have $u'^2\ge-2F(u)\ge 2c_\delta$, so $u$ is monotone and the length is at most $(u_1-\delta)/\sqrt{2c_\delta}$.

    (iii) Without loss of generality, we may assume $u\ge u_1>1$, so $f(u)\ge f_{\min}$. Split the interval at the critical point (if any) into an ascent ($u'>0$) and a descent ($u'<0$). On the ascent $u''=-gu'-f(u)\le-f_{\min}$, so $u'$ decreases at rate at least $f_{\min}$ from a value $\le S$: this makes the length of the ascent at most $S/f_{\min}$. On the descent $u''=g|u'|-f(u)\le M_0|u'|-f_{\min}$. While $|u'|\le f_{\min}/(2M_0)$ we have $u''\le-f_{\min}/2$, so $|u'|$ reaches $f_{\min}/(2M_0)$ within time $1/M_0$, and it cannot fall below this threshold afterward (at the threshold $u''\le-f_{\min}/2<0$). Since $F(\alpha)=E(0)>E(b)>0$ gives $\alpha>\alpha_*>1$, Proposition~\ref{prop:critical-points}(ii) yields $u\le\alpha$, and the descent lasts at most $1/M_0+2M_0\alpha/f_{\min}$.

    For \eqref{eq:zero-spacing}: on $(z,z')$ the solution has one sign, $E>0$ (as $E(z')=\frac12 u'(z')^2>0$), and, by Proposition~\ref{prop:critical-points}(iii), exactly one critical point $c$, with $|u(c)|>\alpha_*>u_1$, so that $|u|$ increases on $(z,c)$ and decreases on $(c,z')$. Hence $(z,z')$ is covered by two intervals of type (i), two of type (ii) and one of type (iii), and $E\ge E(z')$ on all of them. Summing the three bounds gives \eqref{eq:zero-spacing}. For \eqref{eq:ascent-time} the same decomposition of the ascent $(\zeta,c']$ has one interval of each type, and $E\ge E(c')$ on $(\zeta,c']$.
\end{proof}

\begin{lemma}[Dissipation across the band]\label{lem:band-dissipation}
    Assume \textup{(G1)}, \textup{(G2)}, \textup{(G4)}, $\int^\infty e^{-G}<\infty$ and $(\Phi)$. Set $\fc_1\defeq \min\{-F(u): \frac{\alpha_*}{2}\le|u|\le\frac{3\alpha_*}{4}\}>0$ and $\fc_2\defeq \alpha_*/(4S)$. Let $J\subset(0,\infty)$ be a bounded interval on which $E\ge 0$ and on which $|u|$ runs monotonically from $\frac{\alpha_*}{2}$ to $\frac{3\alpha_*}{4}$ or back \emph{(a band crossing)}. Then $|J|\ge \fc_2$ and
    \begin{equation}\label{eq:band-dissipation}
        E(\inf J)-E(\sup J)\;\ge\;\int_J g\,u'^2\;\ge\;2\fc_1\fc_2\,\frac{\gamma-1}{\sup J}.
    \end{equation}
\end{lemma}

\begin{proof}
    $|u|$ changes by $\alpha_*/4$ on $J$ and $|u'|\le S$, so $|J|\ge \fc_2$. On $J$, $u'^2=2(E-F(u))\ge-2F(u)\ge 2\fc_1$, and $g>(\gamma-1)/r\ge(\gamma-1)/\sup J$ by Lemma~\ref{lem:phi-consequences}(i). Integrating $E'=-gu'^2$ over $J$ gives \eqref{eq:band-dissipation}.
\end{proof}

Every solution that leaves a zero $z$ and reaches a critical point $c$ with $|u(c)|>\alpha_*$ performs a band crossing in $(z,c)$, and every solution that leaves such a critical point and reaches a zero (or tends to $0$) performs one on the way down: the two lemmas will be applied repeatedly to these crossings.

%% file: sections/13-pohozaev.tex
\section{Pohozaev identities with variable effective dimension}\label{sec:pohozaev}

Throughout Sections~\ref{sec:pohozaev} to~\ref{sec:staircase} we assume (G1)--(G6), $(\Phi)$ and \eqref{eq:p-range}, and only additional hypotheses, or weaker ones where this matters, are stated in the results. Here $u$ is a nonconstant solution of \eqref{eq:ode}. Recall $\nu=2\varphi'+\gamma-1$ and set
\begin{gather}
    P:=e^G\Big[\varphi\big(u'^2+2F(u)\big)+(\gamma-1)uu'\Big], \qquad P_1:=P+\varphi\,e^Gh_1(u), \qquad h_1(u):=uf(u)-2F(u),
    \label{eq:pohozaev}\\
    \omega:=-\frac{\varphi\,u'}{u} \qquad (\text{where } u\neq 0). \label{eq:omega}
\end{gather}

\begin{lemma}[Identities]\label{lem:pohozaev-identities}
    With the notation above, under \textup{(G1)}, \textup{(G4)} and $\int^\infty e^{-G}<\infty$ alone,
    \begin{gather}
        P'=e^Gh_r(u), \qquad h_r(u):=2\nu(r)F(u)-(\gamma-1)uf(u)=\frac{4\varphi'-(p-1)(\gamma-1)}{p+1}|u|^{p+1}-2\varphi'u^2,
        \label{eq:pohozaev-identity}\\
        h_1(u)=\frac{p-1}{p+1}|u|^{p+1}\;\ge\;0, \qquad \omega'=\frac{P_1}{e^Gu^2} \quad (u\neq 0). \label{eq:omega-prime}
    \end{gather}
\end{lemma}

\begin{proof}
    Direct computation using \eqref{eq:ode} and $\varphi'=g\varphi-(\gamma-1)$ (see Appendix~\ref{sec:identity-table}).

    For \eqref{eq:pohozaev-identity}, write $P=Au'^2+2AF(u)+Buu'$ with $A=\varphi e^G$ and $B=(\gamma-1)e^G$. The choice of $B$ kills the $uu'$ term because $B'=gB$. The choice of $A$ kills the $u'^2$ term because $A'-2Ag+B=e^G(\varphi'-g\varphi+\gamma-1)=0$. What remains is
    \[
        2A'F-Buf=e^G\big[2(2g\varphi+1-\gamma)F-(\gamma-1)uf\big],
    \]
    and $2g\varphi+1-\gamma=\nu$.

    For \eqref{eq:omega-prime},
    \[
        \omega'=(g\varphi-\varphi')\frac{u'}{u}+\frac{\varphi f(u)}{u}+\frac{\varphi u'^2}{u^2}
        =(\gamma-1)\frac{u'}{u}+\frac{\varphi f(u)}{u}+\frac{\varphi u'^2}{u^2},
    \]
    which is $P_1/(e^Gu^2)$.
\end{proof}

\begin{definition}[Moving Pohozaev level]\label{def:pohozaev-level}
    For $r>0$, define:
    \[
        \alpha^*(r):=
        \begin{cases}
            \Big[\dfrac{2(p+1)\varphi'(r)}{4\varphi'(r)-(p-1)(\gamma-1)}\Big]^{\frac{1}{p-1}},
                     & \text{if } 4\varphi'(r)>(p-1)(\gamma-1), \\[2ex]
            +\infty, & \text{otherwise.}
        \end{cases}
    \]
\end{definition}

\begin{lemma}[Level structure]\label{lem:level-structure}
    Under $(\Phi)$ and \eqref{eq:p-range}:
    \begin{enumerate}[label=(\roman*)]
        \item for each $r$, $h_r(u)>0$ if and only if $|u|>\alpha^*(r)$, and $h_r(u)<0$ for $0<|u|<\alpha^*(r)$.
        \item $r\mapsto\alpha^*(r)$ is nondecreasing, with $\alpha^*(0^+)=\big[\frac{2(p+1)}{(\gamma+3)-p(\gamma-1)}\big]^{\frac{1}{p-1}}=\alpha^*(0)\in(\alpha_*,\infty)$.
        \item along any interval on which $|u|$ is strictly decreasing, $r\mapsto h_r(u(r))$ changes sign at most once, and only from $+$ to $-$.
    \end{enumerate}
\end{lemma}

\begin{proof}
    (i) By \eqref{eq:pohozaev-identity}, $h_r(u)=u^2\big[\frac{4\varphi'-(p-1)(\gamma-1)}{p+1}|u|^{p-1}-2\varphi'\big]$. If $4\varphi'\le(p-1)(\gamma-1)$ both terms are nonpositive and the second is negative (as $\varphi'>0$), so $h_r<0$ for $u\neq 0$. Otherwise, the bracket is increasing in $|u|$ and vanishes exactly at $\alpha^*(r)$.

    (ii) Write $x=\varphi'(r)\in(0,1]$ and $c=(p-1)(\gamma-1)>0$. Then $\alpha^*=\big[2(p+1)x/(4x-c)\big]^{1/(p-1)}$ on $\{4x>c\}$, and $\frac{d}{dx}\frac{x}{4x-c}=\frac{-c}{(4x-c)^2}<0$, so $\alpha^*$ is decreasing in $x$. Since $x=\varphi'(r)$ is nonincreasing in $r$ by $(\Phi)$, $\alpha^*(r)$ is nondecreasing, with the convention $+\infty$ on $\{4\varphi'\le c\}$, which is a half-line by monotonicity of $\varphi'$. At $r=0^+$ we have $\varphi'\to 1$ and $4-(p-1)(\gamma-1)=(\gamma+3)-p(\gamma-1)>0$ by \eqref{eq:p-range}. The inequality $\alpha^*(0)>\alpha_*$ reduces to $\frac{2}{4-(p-1)(\gamma-1)}>\frac12$, that is $(p-1)(\gamma-1)>0$.

    (iii) $\{r:h_r(u(r))>0\}=\{r:|u(r)|>\alpha^*(r)\}$: a decreasing $|u(r)|$ and a nondecreasing level can cross at most once.
\end{proof}

\begin{proposition}[Positivity of $P$ and $P_1$, monotonicity of $\omega$]\label{prop:pohozaev-positive}
    Let $u$ be a ground state, a bound state, or a nodal solution with zeros $z_1<\cdots<z_k$ and $E>0$ on $(0,z_k]$. Then:
    \begin{enumerate}[label=(\roman*)]
        \item $\alpha>\alpha^*(0)$.
        \item $P>0$ on $(0,z_k]$, and for ground and bound states $P>0$ on $(0,\infty)$.
        \item $P_1\ge P>0$ on the same intervals, and consequently $\omega$ is strictly increasing on every subinterval of these intervals on which $u\neq 0$.
    \end{enumerate}
\end{proposition}

\begin{proof}
    Note first that $P(0^+)=0$: the bracket in \eqref{eq:pohozaev} is bounded as $r\downarrow 0$ while $e^G\to 0$ (Remark~\ref{rem:hyp-g}) and $e^G\varphi\to 0$ (Lemma~\ref{lem:phi-basic}). Also, at any zero $z$ of $u$,
    \begin{equation}\label{eq:pohozaev-zero}
        P(z)=e^{G(z)}\varphi(z)u'(z)^2>0,
    \end{equation}
    and wherever $uu'\ge 0$,
    \begin{equation}\label{eq:pohozaev-energy}
        P\;\ge\;2e^G\varphi E\;>\;0, \qquad \text{provided } E>0.
    \end{equation}
    For bound states $P(r)\to 0$ as $r\to\infty$, by Lemma~\ref{lem:decay}.

    (i) Suppose $\alpha\le\alpha^*(0)$. Proposition~\ref{prop:critical-points}(ii) applies, since $\alpha>\alpha_*\ge 1$ in all three solution classes (for nodal solutions, $E(0)>E(z_1)>0$ gives $F(\alpha)>0$). It gives $|u(r)|\le\alpha\le\alpha^*(0)\le\alpha^*(r)$ for all $r$, so $h_r(u(r))\le 0$ for all $r$. The inequality is strict wherever $0<|u(r)|<\alpha^*(r)$, in particular on a right neighborhood of $0$ if $\alpha<\alpha^*(0)$, and on any interval where $0<|u|<\alpha$ in the boundary case $\alpha=\alpha^*(0)$. Such intervals exist since $u$ is nonconstant and $\alpha^*$ is nondecreasing.
    Hence $P$ is nonincreasing, with $P(0^+)=0$, so $P\le 0$ everywhere. This contradicts \eqref{eq:pohozaev-zero} (with $z=z_1$) for nodal solutions. For ground and bound states, $P\le 0$, $P\to 0$ and $P$ nonincreasing force $P\equiv 0$, hence $P'\equiv 0$, contradicting the strict inequality $h_r(u(r))<0$ on some interval.

    (ii) We first prove $P>0$ on $(0,z_1]$. On $(0,z_1)$, $u>0$ is strictly decreasing (Proposition~\ref{prop:critical-points}) from $\alpha>\alpha^*(0)$ to $0$. By Lemma~\ref{lem:level-structure}(iii) there is a single crossing point $s_1\in(0,z_1)$ with $h_r(u(r))>0$ on $(0,s_1)$ and $<0$ on $(s_1,z_1)$, the first interval being nonempty by (i). Hence $P$ increases from $0$ on $(0,s_1]$, so $P>0$ on $(0,s_1]$, and decreases on $[s_1,z_1]$ toward $P(z_1)>0$ by \eqref{eq:pohozaev-zero}. Since $P$ is decreasing with a positive terminal value, $P>0$ on $[s_1,z_1]$ as well.

    Now we iterate over the phases. On the ascent $(z_i,c_i)$ we have $uu'\ge 0$ and $E>0$, so \eqref{eq:pohozaev-energy} gives $P>0$. On the descent $(c_i,z_{i+1})$, $|u|$ decreases strictly, so by Lemma~\ref{lem:level-structure}(iii) $h_r(u(r))$ is first $\ge 0$ and then $<0$, with at most one crossing: $P$ first increases from $P(c_i)>0$ and then decreases toward $P(z_{i+1})>0$, so in either regime $P$ stays above $\min\{P(c_i),P(z_{i+1})\}>0$. This proves $P>0$ on $(0,z_k]$. For a bound state the same argument applies on the final descent $(c_k,\infty)$, with $P(z_{i+1})>0$ replaced by $P(\infty)=0$ approached from above, and for a ground state it applies on all of $(0,\infty)$ by the single-crossing argument with $P\to 0$.

    (iii) $P_1-P=\varphi e^Gh_1(u)\ge 0$ by \eqref{eq:omega-prime}, and $\omega'=P_1/(e^Gu^2)>0$.
\end{proof}

\begin{remark}[The function $P_2$]\label{rem:P-two}
    With $A:=\varphi e^G$ one also has
    \[
        P_2:=A\Big[u'^2+\frac{\gamma-1}{\nu}uf(u)\Big]+(\gamma-1)e^Guu', \qquad \Big(\frac{P}{A}\Big)'=-\frac{\nu}{\varphi}\cdot\frac{P_2}{A},
    \]
    (identity \eqref{eq:id-P-over-A}), which reduces to $-[P/r^n]'=nP_2/r^{n+1}$ in the Euclidean case. Note that the natural generalization acquires the variable coefficient $\frac{\gamma-1}{\nu(r)}$. Its positivity on $(0,c_K]$ ($K$ as in Section~\ref{sec:later-phases}), the warped form of the monotonicity of $P/r^n$, is proved in Lemma~\ref{lem:P-two-positive}. This is only used in the pairing of Proposition~\ref{prop:pairing}.
\end{remark}

\subsection{The Pohozaev lower bound on the energy, and the classification of solutions}\label{sec:classification}
The function $P$ also bounds the energy from below at zeros and critical points (Lemma~\ref{lem:pohozaev-lower}). This is what makes the passage-time estimate \eqref{eq:zero-spacing} useful for the counting of zeros.

\begin{lemma}[Pohozaev lower bound]\label{lem:pohozaev-lower}
    Assume \textup{(G1)}, \textup{(G2)}, \textup{(G4)}, \textup{(G5)}, $(\Phi)$ and \eqref{eq:p-range}. Let $\fc_2$ be as in Lemma~\ref{lem:band-dissipation} and set
    \[
        \fc_3:=\frac{(p-1)(\gamma-1)}{p+1}\Big(\frac{\alpha_*}{2}\Big)^{p+1}>0.
    \]
    Let $[a,b]$ be an interval on which $|u|<\alpha^*(0)$, with $P(b)\ge 0$, and let $J\subset[a,b]$ be a band crossing. Then
    \[
        P(a)\;\ge\;\fc_3\fc_2\,e^{G(\inf J)}.
    \]
    In particular, if $a$ is a zero or a critical point of $u$, then $P(a)=2e^{G(a)}\varphi(a)E(a)$ and
    \begin{equation}\label{eq:energy-lower}
        E(a)\;\ge\;\frac{\fc_3\fc_2}{2\varphi(a)}\;\ge\;\frac{\fc_3\fc_2}{2a}.
    \end{equation}
\end{lemma}

\begin{proof}
    For $|u|\le\alpha_*$ the function $h_r$ of \eqref{eq:pohozaev-identity} satisfies
    \[
        -h_r(u)=2\varphi'u^2\Big[1-\Big(\frac{|u|}{\alpha_*}\Big)^{p-1}\Big]+\frac{(p-1)(\gamma-1)}{p+1}|u|^{p+1}
        \;\ge\;\frac{(p-1)(\gamma-1)}{p+1}|u|^{p+1},
    \]
    because $\frac{2}{p+1}|u|^{p-1}\le 1$ exactly when $|u|\le\alpha_*$. Note that this bound is uniform in $r$ and remains valid as $\varphi'\to 0$. On $J$ we have $|u|\ge\alpha_*/2$, so $P'=e^Gh_r(u)\le-\fc_3e^G\le-\fc_3e^{G(\inf J)}$ and $P(\sup J)\le P(\inf J)-\fc_3\fc_2e^{G(\inf J)}$. On $[a,b]$, $|u|<\alpha^*(0)\le\alpha^*(r)$, so $h_r(u(r))\le 0$ by Lemma~\ref{lem:level-structure}(i) and $P$ is nonincreasing, hence
    \[
        P(a)\ge P(\inf J)\ge
        P(\sup J)+\fc_3\fc_2e^{G(\inf J)}\ge
        P(b)+\fc_3\fc_2e^{G(\inf J)}\ge
        \fc_3\fc_2e^{G(\inf J)}.
    \]
    If $u(a)=0$ or $u'(a)=0$ then $uu'(a)=0$ and
    \[
        P(a)=
        e^{G(a)}\varphi(a)\big(u'^2+2F(u)\big)(a)=
        2e^{G(a)}\varphi(a)E(a),
    \]
    since $\inf J\ge a$, $e^{G(\inf J)}\ge e^{G(a)}$, and $\varphi(a)\le a$ by Lemma~\ref{lem:phi-basic}.
\end{proof}

\begin{proposition}[Classification and pausing]\label{prop:classification}
    Assume \textup{(G1)}, \textup{(G2)}, \textup{(G4)}, \textup{(G5)}, $(\Phi)$ and \eqref{eq:p-range}. Every solution with $\alpha>0$, $\alpha\neq 1$, has finitely many zeros, that is $N(\alpha)<\infty$, and exactly one of the following holds:
    \begin{enumerate}[label=(\alph*)]
        \item $u$ is a ground state or a bound state with $N(\alpha)$ nodes, and then $E>0$ everywhere.
        \item $u$ is trapped: beyond its last zero (or from $r=0$ if there is none), there is $r_0$ with $E(r_0)<0$. Then $u$ has no further zeros, $|u|\not\to 0$, and $N$ is constant in a neighborhood of $\alpha$ (pausing).
    \end{enumerate}
\end{proposition}

\begin{proof}
    Let $g\le M_0$ on $[R_1,\infty)$ (with $R_1\ge 1$), and let $T_0,\fc_1,\fc_2,\fc_3$ be the constants of Lemmas~\ref{lem:passage-times}, \ref{lem:band-dissipation}, \ref{lem:pohozaev-lower}.

    \emph{Finiteness.} Suppose $u$ has zeros $z_1<z_2<\cdots\to\infty$ (they cannot accumulate at a finite point, otherwise $(u,u')$ would vanish at that point). Then $E>0$ at every $z_j$, hence everywhere, and $E\downarrow E_\infty\ge 0$. Let $c_j\in(z_j,z_{j+1})$ be the critical points (Proposition~\ref{prop:critical-points}(iii)), with $|u(c_j)|>\alpha_*$, and fix $j_0$ with $z_{j_0}\ge R_1$. On the ascent $(z_j,c_j)$, $|u|$ increases from $0$ to $|u(c_j)|>\alpha_*$, so it performs a band crossing $J_j\subset(z_j,c_j)$. By Lemma~\ref{lem:band-dissipation},
    \begin{equation}\label{eq:zero-dissipation}
        E(z_j)-E(z_{j+1})\;\ge\;2\fc_1\fc_2\frac{\gamma-1}{z_{j+1}} \qquad (j\ge j_0).
    \end{equation}
    Summing, $E(z_{j_0})-E_\infty\ge 2\fc_1\fc_2(\gamma-1)\sum_{j\ge j_0}1/z_{j+1}$, so it suffices to show $\sum_j 1/z_j=\infty$, i.e.\ that the zeros do not spread out too fast.

    \emph{Case $E_\infty>0$.} By \eqref{eq:zero-spacing} with $E(z_{j+1})\ge E_\infty$, $z_{j+1}-z_j\le T_0(1+\log^+(1/E_\infty))$, so $z_j\le Cj$ and $\sum 1/z_j=\infty$: contradiction.

    \emph{Case $E_\infty=0$.} Then $F(u(c_j))=E(c_j)\to 0$ with $|u(c_j)|>\alpha_*$, so $|u(c_j)|\to\alpha_*<\alpha^*(0)$, and there is $j_1\ge j_0$ with $|u|<\alpha^*(0)$ on $[z_{j_1},\infty)$. On $[z_{j_1},\infty)$ the function $P$ is nonincreasing and $P(z_j)>0$ for all $j$ (by \eqref{eq:pohozaev-zero}), so $P>0$ on $[z_{j_1},\infty)$. Lemma~\ref{lem:pohozaev-lower}, applied with $a=z_{j+1}$, $J=J_{j+1}$ and $b=\sup J_{j+1}$, gives $E(z_{j+1})\ge \fc_3\fc_2/(2z_{j+1})$ for $j\ge j_1$, and \eqref{eq:zero-spacing} then yields
    \[
        z_{j+1}-z_j\;\le\;T_0\Big(1+\log^+\frac{2z_{j+1}}{\fc_3\fc_2}\Big) \qquad (j\ge j_1).
    \]
    For large $j$ the right-hand side is at most $z_{j+1}/2$, so $z_{j+1}\le 2z_j$, and hence $z_{j+1}-z_j\le T_0'(1+\log z_j)$ with $T_0'$ depending on $T_0,\fc_2,\fc_3$. With $\mathcal Z(x)\defeq\int_1^x\frac{dt}{1+\log t}$, this gives $\mathcal Z(z_{j+1})-\mathcal Z(z_j)\le\frac{z_{j+1}-z_j}{1+\log z_j}\le T_0'$, so $\mathcal Z(z_j)\le T_0'j+C$. Since $\mathcal Z(x)\ge\frac{x-1}{1+\log x}$, first $z_j\le j^2$ and then $z_j\le C\,j\log j$ for large $j$, so $\sum 1/z_j=\infty$ and \eqref{eq:zero-dissipation} again gives a contradiction. We proved $N(\alpha)<\infty$.

    \emph{Dichotomy.} Suppose $u$ has finitely many zeros, the last being $z_k$ (or none), and $E>0$ everywhere.
    By Proposition~\ref{prop:critical-points}(i), on $(z_k,\infty)$ every critical point is a strict local maximum of $|u|$, so (as in the proof of Proposition~\ref{prop:critical-points}(iii)) there is at most one critical point beyond $z_k$, and $|u|$ is eventually monotone, with a limit $u_\infty\ge 0$. Moreover, $u'\to 0$, since $u'^2=2(E-F(u))$ has a limit, which must vanish, for otherwise $|u|\to\infty$ and $F(u)\to\infty>E(0)$. If $u_\infty>0$ and $f(u_\infty)\neq 0$ then, from $(e^Gu')'=-e^Gf(u)$, for large $r$ and $r_1\ge R_1$,
    \[
        |e^Gu'(r)|\;\ge\;\tfrac12|f(u_\infty)|\int_{r_1}^{r}e^G-|e^{G(r_1)}u'(r_1)|\;\ge\;\tfrac14|f(u_\infty)|\,\frac{e^{G(r)}}{M_0},
    \]
    by Lemma~\ref{lem:phi-consequences}(ii) and since $e^{G(r)}\to\infty$ (Lemma~\ref{lem:phi-consequences}(i)). Hence $|u'|\ge c>0$ eventually and $|u|\to\infty$, contradicting boundedness. If $u_\infty=1$ then $E\to F(1)<0$, contradicting $E>0$ (and $u_\infty>0$ with $f(u_\infty)=0$ forces $u_\infty=1$). Hence $u_\infty=0$ and $u$ is a ground or bound state, which is alternative (a). Otherwise, $E(r_0)<0$ at some $r_0$ beyond the last zero, and Proposition~\ref{prop:trapping} with its stability clause gives (b): for nearby $\tilde\alpha$, $u(\cdot,\tilde\alpha)$ still has $E<0$ at $r_0$ and, by simplicity and continuity, exactly the same number of zeros in $[0,r_0]$.
\end{proof}

\begin{remark}\label{rem:pausing}
    The pausing statement in (b) is one half of the staircase argument of Section~\ref{sec:staircase}: on the set of $\alpha$ whose solutions are not bound states, $N$ is locally constant. The other half, the \emph{jump} at bound states, requires working with the variation. The two energy cases are handled by the same argument: the dissipation in each oscillation is only $\gtrsim 1/z_j$ (the lower bound $g>(\gamma-1)/r$ of $(\Phi)$ has the sharp order $1/r$ when $g=(n-1)/r$), and the divergence of $\sum 1/z_j$ is obtained from the period bound \eqref{eq:zero-spacing}, which in the delicate case $E_\infty=0$ uses the Pohozaev lower bound \eqref{eq:energy-lower}. For $g=(n-1)/r$ this yields the same finiteness statement of \cite{McLeodTroyWeissler1990,Tang2026} without using the explicit form of $g$.
\end{remark}

%% file: sections/14-variation-function.tex
\section{The variation equation: identities and the first phase}\label{sec:variation}

So far we know that the number of zeros $N(\alpha)$ can change only at bound states (Sections~\ref{sec:basic-properties} and~\ref{sec:pohozaev}). To see what happens at a bound state we follow the variation $v=\partial_\alpha u$, and from now on we mostly study $v$. This section prepares the tools, a family of auxiliary functions built from $u$ and $v$ whose derivatives have explicit formulas, and then uses them on the first phase, from the origin down to the first zero of $u$ and a little beyond. For a ground state the first phase is the whole half-line, so this already does most of the work for ground states. Nodal solutions have more phases, and these are the subject of Sections~\ref{sec:later-phases}--\ref{sec:renewal-below}.

\subsection{The identities}\label{sec:identities}
Let $u$ solve \eqref{eq:ode} and let $v=\partial_\alpha u$ solve the variation equation (\ref{eq:variation}). Define:
\begin{align*}
    M      & \defeq  e^G\big(u'v-uv'\big),                                                  &
    Q      & \defeq  e^G\Big[\varphi\,u'v'+\varphi f(u)v+(\gamma-1)u'v\Big],                  \\
    Q_a    & \defeq  Q+a\,e^Gu'v \quad (a\in\R),                                            &
    Q_\nu  & \defeq  Q+\nu(r)\,e^Gu'v,                                                        \\
    W_a    & \defeq  Q-a\,\varphi'M,                                                        &
    \Sigma & \defeq  e^G\big(u''v'-u'v''\big),                                                \\
    T_1    & \defeq  Q-\varphi'g_1(u)M,                                                     &
    g_1(u) & \defeq  \frac{2f(u)}{uf'(u)-f(u)}=\frac{2}{p-1}\big(1-|u|^{1-p}\big),            \\
    T_2    & \defeq  Q-\varphi'g_2(u)M,                                                     &
    g_2(u) & \defeq  \frac{2F(u)}{h_1(u)}=\frac{2}{p-1}\Big(1-\frac{p+1}{2}|u|^{1-p}\Big),    \\
    \theta & \defeq  e^G\frac{v\,h_1(u)}{u'} \quad (u'\neq0),                               &
    B_a    & \defeq  Q-a\varphi'M-2F_a(u)\frac{e^Gv}{u'},\qquad F_a\defeq F-\frac{a}{2}h_1.
\end{align*}
Here $Q=e^G\big[wv'-w'v\big]$, with $w=\varphi u'$, is the weighted Wronskian of $v$ against the multiplier field applied to $u$, while $M=e^G(u'v-uv')$ is, with the same convention, the weighted Wronskian of $u$ against $v$. Moreover, $g_1$ vanishes exactly at $|u|=1$ and $g_2$ exactly at $|u|=\alpha_*$, with
\begin{equation}\label{eq:g-relations}
    \frac{d}{du}g_1=2|u|^{-p}\sgn u,\qquad
    \frac{d}{du}g_2=(p+1)|u|^{-p}\sgn u,\qquad
    g_2-g_1=-|u|^{1-p},\qquad
    g_2(1)=-1.
\end{equation}

\begin{proposition}[Identity table]\label{prop:identity-table}
    With $\kappa_a(u)\defeq \big(1-\tfrac{p-1}{2}a\big)|u|^{p-1}-1$ and ${}'=\tfrac{d}{dr}$:
    \begin{align}
        M'             & = e^G\big(uf'(u)-f(u)\big)v=(p-1)e^G|u|^{p-1}uv, \label{eq:M-prime}                                                    \\
        Q'             & = 2\varphi'e^Gf(u)v, \label{eq:Q-prime}                                                                                \\
        \Sigma'        & = e^Gf''(u)u'^2v, \label{eq:Sigma-prime}                                                                                   \\
        Q-P\frac{v}{u} & = \omega\big(M-\theta\big)\qquad(u,u'\neq0), \label{eq:connection}                                                           \\
        T_1'           & = -M\big(\varphi'g_1(u)\big)'=-M\Big[\varphi''g_1(u)+2\varphi'|u|^{-p}(\sgn u)u'\Big], \label{eq:T-one-prime} \\
        T_2'           & = -\varphi'(p+1)\frac{uu'}{|u|^{p+1}}\big(M-\theta\big)-\varphi''g_2(u)M\qquad(u'\neq0), \label{eq:T-two-prime}                           \\
        W_a'           & = 2\varphi'e^Guv\,\kappa_a(u)-a\varphi''M, \label{eq:W-prime}                                                          \\
        Q_a'           & = e^G\Big[(2\varphi'-a)f(u)v+a\,u'v'\Big], \label{eq:Qa-prime}                                                          \\
        Q_\nu'         & = e^G\Big[(1-\gamma)f(u)v+\nu\,u'v'+2\varphi''u'v\Big], \label{eq:Qnu-prime}                                            \\
        B_a'           & = 2(\varphi'-1)e^Guv\,\kappa_a(u)-a\varphi''M-2F_a(u)\frac{Q_\nu}{\varphi\,u'^2}, \label{eq:bridge-prime}                  \\
        T_2            & = T_1+\varphi'|u|^{1-p}M. \label{eq:T-two-T-one}
    \end{align}
    Moreover, for $w_\lambda\defeq \varphi u'+\lambda(r)u$ and $L\defeq \partial_r^2+g\partial_r+f'(u)$,
    \begin{equation}\label{eq:L-w-lambda}
        Lw_\lambda=-(2\varphi'+\lambda)f(u)+\lambda f'(u)u+2\lambda'u'+(\lambda''+g\lambda')u,
    \end{equation}
    so that $\big[e^G(w_\lambda v'-w_\lambda'v)\big]'=-e^Gv\,Lw_\lambda$.
\end{proposition}

Not all of these identities are used in the proofs. The derivatives \eqref{eq:W-prime}, \eqref{eq:Qnu-prime} and \eqref{eq:bridge-prime} (and with them $\kappa_a$ and the unweighted bridge $B_a$) are recorded because they are the templates of the weighted bridge of Section~\ref{sec:transition}, and \eqref{eq:L-w-lambda} is used only with $\lambda\equiv0$, where it gives \eqref{eq:Q-prime}.

\begin{proof}
    Each identity is a finite computation from \eqref{eq:ode}, the variation equation, and $\varphi'=g\varphi-(\gamma-1)$ (whence $\varphi''=(g\varphi)'$). No other property of $\varphi$ or $g$ is used. All of them are collected in Appendix~\ref{sec:identity-table}. We indicate the two main identities from which the table follows.

    \emph{Wronskian identity:} for any $w$,
    \[
        \big[e^G(wv'-w'v)\big]'=e^G\big[w(v''+gv')-v(w''+gw')\big]=-e^Gv\,Lw
    \]
    since $Lv=0$. Identity~\eqref{eq:Q-prime} is the case $w=\varphi u'$, for which \eqref{eq:L-w-lambda} with $\lambda=0$ gives $Lw=-2\varphi'f(u)$: the coefficient of $u'$ in $Lw$, namely $\varphi''-g\varphi'-g'\varphi=\varphi''-(g\varphi)'$, vanishes identically.

    \emph{Functional identities:} \eqref{eq:T-one-prime} and \eqref{eq:T-two-prime} follow from \eqref{eq:M-prime}, \eqref{eq:Q-prime} and the algebraic identities $2f-g_1\cdot(uf'-f)=0$ and $2\varphi'f-\varphi'g_2(uf'-f)=(p-1)\varphi'u$, together with $(p-1)e^Guv=(p+1)\frac{uu'}{|u|^{p+1}}\theta$.
\end{proof}

\begin{remark}[Reading the table]\label{rem:reading-table}
    Every correction term carries $\varphi''\le0$ or $\varphi'-1\le0$ by ($\Phi$), with a definite sign once the phase of the solution fixes the signs of $u,u',v$ and of $g_1(u),g_2(u),F(u),f(u)$. The point of Section~\ref{sec:first-phase} and of Sections~\ref{sec:transition}--\ref{sec:renewal-below} is to show that these correction signs are favorable at every step of the first phase and of the first ascent, and that they are the only obstruction in the remaining steps.
\end{remark}

\subsection{The first phase}\label{sec:first-phase}
\emph{Conventions.} Let $u$ be a nodal solution or a bound state, so that $\alpha>\alpha^*(0)$ and Propositions~\ref{prop:critical-points}, \ref{prop:classification} and~\ref{prop:pohozaev-positive} apply, with zeros $z_1<z_2<\dots$. On $(0,z_1)$ we have $u>0$ and $u'<0$. Let $b_1<r_1$ be the points in $(0,z_1)$ where $u=\alpha_*$ and $u=1$, and let $\bar r_1$, $\bar b_1$ be the first points after $z_1$ where $u=-1$ and $u=-\alpha_*$, and $c_1$ the first critical point after $z_1$, when they exist. If $E>0$ on $(z_1,c_1]$, then $|u(c_1)|>\alpha_*$ by Proposition~\ref{prop:critical-points}(i), so $z_1<\bar r_1<\bar b_1<c_1$ (and $c_1<z_2$ if $u$ has a second zero). For a ground state no $z_1$ exists and the argument below on $(0,z_1]$ is to be read on $(0,\infty)$, $r_1$ being the unique point with $u(r_1)=1$. In part (iv) we additionally assume that $\bar r_1$ exists, that $E>0$ on $(z_1,\bar r_1]$, and that $\nu\ge1$ on $[z_1,\bar r_1]$. In part (v) we assume that $E>0$ on $(z_1,c_1]$.
The conditions on $E$ hold when $u$ has a second zero or is a bound state. Regarding $\nu\ge1$, see Remark~\ref{rem:nu-condition}.

\begin{theorem}[First phase]\label{thm:first-phase}
    \begin{enumerate}[label=(\roman*),font=\upshape]
        \item \textup{(Tendency.)} $v$ changes sign in $(0,r_1)$. Let $\tau_1$ be its first zero. Then $v>0$ and $v'<0$ on $(0,\tau_1)$, $v'(\tau_1)<0$, and $u'v'(\tau_1)>0$.
        \item \textup{(Resistance.)} $M>0$ and $Q>0$ on $(0,z_1]$; $T_1>0$ on $(0,r_1]$ with $T_1(r_1)=Q(r_1)$; $T_2>0$ on $(0,z_1)$; and $v<0$ on $(\tau_1,z_1]$. In particular $\tau_1$ is the only zero of $v$ in $(0,z_1]$, and $v(z_1)u'(z_1)>0$.
        \item \textup{(Ground state.)} If $u$ is a ground state, then $v$ has exactly one zero $\tau_1\in(0,r_1)$ in $(0,\infty)$, and $v<0$, $M>0$, $Q>0$ on $(\tau_1,\infty)$.
        \item \textup{(First ascent, partial.)} $M>0$ and $Q_1>0$ on $[z_1,\bar r_1]$, and $v<0$ on $(\tau_1,\bar r_1]$.
        \item \textup{(Euclidean ascent.)} If $\varphi'\equiv1$, then moreover $Q_1>0$ on $[z_1,\bar b_1]$, $Q>0$ on $[\bar b_1,c_1]$, $v<0$ on $(\tau_1,c_1]$, and $M,Q,Q_1,Q_2>0$ at $c_1$. (This part is not used later: under (H), which is trivial when $\varphi'\equiv1$, Theorem~\ref{thm:phase} gives the same conclusions for every phase.)
    \end{enumerate}
\end{theorem}

\begin{proof}
    \emph{(i) Tendency.} On $(0,r_1)$ we have $u>1$, so $f(u)>0$, $f'(u)>0$ and $f''(u)>0$. Suppose $v\ge0$ on $[0,r_1]$. As $r\downarrow0$ the bracket in $\Sigma$ is bounded (local regularity, Section~\ref{sec:well-posedness}) and $e^G\to0$, so $\Sigma(0^+)=0$. By \eqref{eq:Sigma-prime}, $\Sigma'=e^Gf''(u)u'^2v\ge0$ on $(0,r_1)$, and $>0$ on a set of positive measure since $v(0)=1$, so $\Sigma(r_1)>0$. On the other hand, at $r_1$ we have $f(u)=0$, hence from the two equations
    \[
        u''v'-u'v''=\big(-gu'-f(u)\big)v'-u'\big(-gv'-f'(u)v\big)=f'(1)u'v\quad\text{at }r=r_1,
    \]
    so $\Sigma(r_1)=(p-1)e^Gu'(r_1)v(r_1)\le0$: a contradiction. Thus $v$ vanishes in $(0,r_1)$. Let $\tau_1$ be its first zero. On $(0,\tau_1)$, $(e^Gv')'=-e^Gf'(u)v<0$ and $e^Gv'\to0$ as $r\downarrow0$, so $v'<0$ on $(0,\tau_1]$. In particular $v'(\tau_1)<0$, as it cannot vanish together with $v$, and $u'v'(\tau_1)>0$.

    \emph{(ii) Resistance.} As $r\downarrow0$ we have $M(0^+)=0$ and $Q(0^+)=0$ (bounded brackets times $e^G\to0$, respectively $e^G\varphi\to0$), and $T_1(0^+)=0$ since $g_1(u)$ is bounded near $u=\alpha$. Consider
    \[
        \varrho\defeq \sup\{r\in(0,z_1]:\ M>0\text{ on }(0,r)\}.
    \]
    On $(0,\tau_1]$ we have $M'=(p-1)e^G|u|^{p-1}uv\ge0$, $>0$ on $(0,\tau_1)$, so $M>0$ on $(0,\tau_1]$. Moreover, $M(\tau_1)=-e^Gu(\tau_1)v'(\tau_1)>0$. Hence $\varrho>\tau_1$.

    \emph{Claim 1: $v<0$ on $(\tau_1,\varrho)$.} A second zero $\tau'\in(\tau_1,\varrho)$ of $v$ would have $v'(\tau')\ge0$, hence $M(\tau')=-e^Gu(\tau')v'(\tau')\le0$ (and $v'(\tau')=0$ is impossible), contradicting $M>0$.

    \emph{Claim 2: $T_1>0$ on $(0,\min(\varrho,r_1)]$, and consequently $Q\ge T_1>0$.} By \eqref{eq:T-one-prime}, $T_1'=-M\big[\varphi''g_1(u)+2\varphi'|u|^{-p}(\sgn u)u'\big]$. On $(0,\min(\varrho,r_1))$ we have $u>1$, so $g_1(u)>0$ and $\varphi''g_1\le0$. Also, $(\sgn u)u'=u'<0$, so the second term in the bracket is negative. Hence $T_1'>0$ where $M>0$, i.e.\ on $(0,\min(\varrho,r_1))$, and with $T_1(0^+)=0$ this gives $T_1>0$. Since $g_1(u)\ge0$ and $M\ge0$ on this range ($M=0$ only at $\varrho$, if $\varrho\le r_1$), $Q=T_1+\varphi'g_1(u)M\ge T_1>0$.

    \emph{Claim 3: $\varrho>r_1$, and $Q(r_1)=T_1(r_1)>0$.} Suppose $\varrho\le r_1$. Then $M(\varrho)=0$. Eliminating $v'$ through $u'v=uv'$ at $\varrho$ gives, directly from the definitions,
    \begin{equation}\label{eq:Q-at-M-zero}
        Q(\varrho)=\frac{v(\varrho)}{u(\varrho)}P_1(\varrho).
    \end{equation}
    Now $\varrho>\tau_1$, so $v(\varrho)<0$ by Claim 1 (if $v(\varrho)=0$ then $M(\varrho)=-e^Guv'(\varrho)=0$ forces $v'(\varrho)=0$, which is impossible).  Also $u(\varrho)>0$, and $P_1(\varrho)>0$ by Proposition~\ref{prop:pohozaev-positive}. So $Q(\varrho)<0$, contradicting Claim 2, which is valid up to $\varrho$ by continuity. Hence $\varrho>r_1$. At $r_1$, $g_1(u(r_1))=g_1(1)=0$ gives $Q(r_1)=T_1(r_1)>0$.

    \emph{Claim 4: $M>0$, $Q>0$, $v<0$ on $(r_1,z_1]$, so $\varrho=z_1$.} On $(r_1,z_1)$ we have $u\in(0,1)$, so $f(u)<0$ and, as long as $v<0$, $Q'=2\varphi'e^Gf(u)v>0$: $Q$ increases from $Q(r_1)>0$. The three properties go together. Let
    \[
        \varrho_1\defeq \sup\{r\le z_1:\ M>0,\ v<0,\ Q>0\text{ on }(r_1,r)\}.
    \]
    At $\varrho_1$, a zero of $v$ is excluded as in Claim 1 (using $M>0$). A zero of $M$ is excluded by \eqref{eq:Q-at-M-zero}, which would give $Q(\varrho_1)<0$ while $Q$ has been increasing from $Q(r_1)>0$. Finally, a zero of $Q$ is excluded since $Q$ is increasing while $v<0$. Hence $\varrho_1=z_1$, and the properties hold up to $z_1$ inclusive: a zero $v(z_1)=0$ would give $Q(z_1)=e^G\varphi u'v'(z_1)<0$ against $Q$ increasing through positive values, so $v(z_1)<0$, hence $M(z_1)=e^Gu'(z_1)v(z_1)>0$, that is $v(z_1)u'(z_1)>0$, and $Q(z_1)>Q(r_1)>0$.

    On $(0,r_1]$, $T_2=T_1+\varphi'|u|^{1-p}M>0$ by \eqref{eq:T-two-T-one}. On $(r_1,z_1)$, $u\in(0,1)$ gives $g_2(u)<0$, so $T_2=Q-\varphi'g_2(u)M=Q+\varphi'|g_2(u)|M>0$. ($T_2$ is not defined at $z_1$: $g_2(u)\to-\infty$ as $u\to0$, so $T_2\to+\infty$ as $r\uparrow z_1$.)

    \emph{(iii) Ground state.} There is no $z_1$, and $r_1$ exists since $u\downarrow0$. Steps (i)--(ii) go through with $z_1$ replaced by $+\infty$: in Claim 4, $Q$ is increasing on $(r_1,\infty)$ while $v<0$, and the propagation of the three properties yields $M,Q>0$ and $v<0$ on the whole ray (here $P_1>0$ on $(0,\infty)$ by Proposition~\ref{prop:pohozaev-positive}).

    \emph{(iv) First ascent.} On $(z_1,\bar r_1]$ we have $u\in[-1,0)$ and $u'<0$, since a critical point in $(z_1,\bar r_1]$ would have $E>0$ and $|u|\le1$, against Proposition~\ref{prop:critical-points}(i). Moreover, as long as $v<0$,
    \[
        uv>0,\qquad u'v>0,\qquad f(u)>0,\qquad F(u)<0,\qquad uf'(u)-f(u)=(p-1)|u|^{p-1}u<0,
    \]
    where $f(u)>0$ because $f<0$ on $(0,1)$ and $f$ is odd. First, $M(z_1)=e^Gu'v(z_1)>0$ and $M'=(p-1)e^G|u|^{p-1}uv>0$, so $M$ increases while $v<0$. Second, $Q_1(z_1)=Q(z_1)+e^Gu'v(z_1)=Q(z_1)+M(z_1)>0$. We propagate $\{Q_1>0,\ v<0\}$ on $(z_1,\bar r_1]$:
    \begin{itemize}
        \item A first zero $\varsigma$ of $v$ would satisfy $v'(\varsigma)>0$ and $Q_1(\varsigma)=e^G\varphi u'v'(\varsigma)<0$, since all the other terms carry a factor $v$: this cannot happen while $Q_1>0$.
        \item For a first zero $\tilde q\in(z_1,\bar r_1]$ of $Q_1$, consider the \emph{anchored bridge}
              \[
                  \widetilde B\defeq Q-2\varphi_1'F(u)\frac{e^Gv}{u'},\qquad\varphi_1'\defeq \varphi'(z_1)\quad\text{(a constant)}.
              \]
              Since $\varphi_1'$ is constant, the computation behind \eqref{eq:bridge-prime} yields (identity~\eqref{eq:id-anchored-bridge})
              \[
                  \widetilde B'=2\big(\varphi'-\varphi_1'\big)e^Gf(u)v-2\varphi_1'F(u)\frac{Q_\nu}{\varphi u'^2}.
              \]
              On $(z_1,\tilde q)$ we have $\varphi'\le\varphi_1'$, as $\varphi'$ is nonincreasing, and $f(u)v<0$, so the first term is $\ge0$. Moreover, $F(u)<0$ and $Q_\nu=Q_1+(\nu-1)e^Gu'v\ge Q_1>0$, since $u'v>0$ and $\nu\ge1$, so the second term is $>0$.
              Hence $\widetilde B$ is strictly increasing on $(z_1,\tilde q)$ from $\widetilde B(z_1)=Q(z_1)>0$, the $F$ term vanishing at $z_1$ since $F(0)=0$. But at $\tilde q$ we have $Q(\tilde q)=-e^Gu'v(\tilde q)$, so
              \[
                  \widetilde B(\tilde q)=-e^G\Big[u'v+2\varphi_1'F(u)\frac{v}{u'}\Big]
                  =-\frac{e^Gv}{u'}\Big[u'^2+2\varphi_1'F(u)\Big]
                  \le-\frac{e^Gv}{u'}\Big[u'^2+2F(u)\Big]
                  =-\frac{2e^Gv}{u'}E<0,
              \]
              using $v(\tilde q)<0$ (if $v(\tilde q)=0$, then $Q_1(\tilde q)=e^G\varphi u'v'(\tilde q)=0$ would force $v'(\tilde q)=0$), so that $v/u'>0$, and $E>0$, and $\varphi_1'\le1$ with $F<0$ (so $2\varphi_1'F\ge2F$). This contradicts $\widetilde B(\tilde q)>Q(z_1)>0$.
    \end{itemize}
    Hence $Q_1>0$ and $v<0$ on $(z_1,\bar r_1]$, and $M>0$ by its monotonicity. The argument uses $|u|\le1$, that is $r\le\bar r_1$, only through $\sgn f(u)=-\sgn u$. It is exactly at $\bar r_1$ that $f(u)v$ changes sign and the first bridge term becomes useless.

    \emph{(v) Euclidean ascent.} If $\varphi'\equiv1$ then $\varphi_1'=1$, $\varphi''\equiv0$, $\nu\equiv\gamma+1>1$, and the first bridge term vanishes identically. Since only $F<0$, that is $|u|<\alpha_*$, was used in the second term, the argument of (iv) runs on all of $(z_1,\bar b_1]$, with the same contradiction at a hypothetical zero of $Q_1$. At $\bar b_1$, $F(u(\bar b_1))=F(-\alpha_*)=0$ gives $Q(\bar b_1)=\widetilde B(\bar b_1)>Q(z_1)>0$. On $[\bar b_1,c_1]$ we have $|u|>\alpha_*>1$, so $f(u)<0$ and $v<0$, hence $Q'=2e^Gf(u)v>0$ and $Q$ increases. A first zero $\varsigma\in(\bar b_1,c_1]$ of $v$ would have $v'(\varsigma)\ge0$ and $Q(\varsigma)=e^G\varphi u'v'(\varsigma)\le0<Q(\bar b_1)\le Q(\varsigma)$, which is excluded. Finally $Q_a(c_1)=Q(c_1)>0$ for every $a$, as $u'(c_1)=0$, and $M(c_1)>0$ by monotonicity.
\end{proof}

\begin{remark}\label{rem:nu-condition}
    The only place in Theorem~\ref{thm:first-phase} where a hypothesis beyond ($\Phi$) is needed is the condition $\nu\ge1$ on $[z_1,\bar r_1]$ in step (iv). It is trivial when $\gamma\ge2$, that is $n\ge3$ for warped products that are $C^2$ at the pole, since $\nu>\gamma-1$ by \eqref{eq:phi-bounds}. For $\gamma\in(1,2)$, as $\nu$ is nonincreasing with $\nu(0^+)=\gamma+1$ (Lemma~\ref{lem:phi-basic}), it holds on some initial interval $(0,r_\nu]$ and asks that the first ascent take place in $(0,r_\nu]$. Part (iv) is used only for the second half of Theorem~\ref{thm:unconditional}(i), and hence in Theorem~\ref{thm:euclidean}, where $\nu\equiv n$. It is not used in the proofs of Theorem~\ref{thm:staircase} and of Theorem~\ref{thm:unconditional}(ii)--(iv): under (H) the ascent is handled by Theorem~\ref{thm:phase} and Corollary~\ref{cor:ascent}, which need only $\varphi'+\gamma-1>0$.
\end{remark}

\begin{remark}\label{rem:Q-at-M-zero}
    Identity~\eqref{eq:Q-at-M-zero} (at any zero of $M$ with $u\neq0$ one has $Q=(v/u)P_1$) is the reason why the Pohozaev estimates of Section~\ref{sec:pohozaev} control the Wronskian identity. It follows from the definitions: if $u'v=uv'$ then $Q=(v/u)e^G\big[\varphi u'^2+\varphi uf(u)+(\gamma-1)uu'\big]=(v/u)P_1$.
\end{remark}

%% file: sections/15-later-phases.tex
\section{Later phases: conventions and the scaling field}\label{sec:later-phases}

After the first phase, the proof becomes an induction over the critical points of $u$. The induction hypothesis, the slope condition of Section~\ref{sec:transition}, says roughly that at a critical point $v$ has the sign of $u$ and is already moving toward zero, but not too fast. The induction step shows that the same holds at the next critical point. This section only sets up the induction. It fixes the notation for a general phase, and it explains, using the scaling field, why one of the three inequalities of the slope condition, the slope bound $T_2>0$, is much harder to carry over than the other two.

The conventions of this preamble are used throughout Sections~\ref{sec:later-phases}--\ref{sec:renewal-below}. The solution $u$ is a nodal solution with zeros $z_1<\dots<z_k$ ($k\ge1$), or a bound state with $k\ge1$ zeros. For a nodal solution $E>0$ on $(0,z_k]$, since $E$ is nonincreasing and $E(z_k)=\frac12u'(z_k)^2>0$. We write $K\defeq k-1$ if $u$ is a nodal solution which is not a bound state and $K\defeq k$ if $u$ is a bound state. Then $E>0$ on $(0,c_K]$. We put $c_0\defeq0$ and
\begin{equation}\label{eq:beta-levels}
    \beta_i\defeq|u(c_i)|\qquad(0\le i\le K),\qquad\beta_0=\alpha,
\end{equation}
so that $\beta_i>\alpha_*$ for $1\le i\le K$ and $\beta_i<\beta_{i-1}$, because $E$ is decreasing, $E(c_i)=F(\beta_i)$ and $F$ is increasing on $[1,\infty)$.

                \emph{Phase} $i$ is the interval $(c_{i-1},c_i]$: the \emph{descent} $(c_{i-1},z_i)$ and the \emph{ascent} $(z_i,c_i)$ are its two halves. In phase $i$ we use the points
\[
    c_{i-1}<b_i<r_i<z_i<\bar r_i<\bar b_i<c_i,\qquad|u(b_i)|=|u(\bar b_i)|=\alpha_*,\qquad
    |u(r_i)|=|u(\bar r_i)|=1,
\]
whose existence and ordering follow from Propositions~\ref{prop:critical-points} and~\ref{prop:classification}, since $|u(c_{i-1})|,|u(c_i)|>\alpha_*$ and $|u|$ is strictly monotone between consecutive critical points. The \emph{river} is the stretch $(b_i,\bar b_i)$ on which $|u|<\alpha_*$ (see Figure~\ref{fig:phase-scheme}). The ascent points $\bar r_i<\bar b_i<c_i$ are used only for $i\le K$: for a nodal solution which is not a bound state the last available phase is the descent $(c_{k-1},z_k]$, since $E$ may become negative after $z_k$ and $\bar r_k$ need not exist. For a bound state the final phase is $(c_k,\infty)$, on which only $b_{k+1}<r_{k+1}$ exist, as $|u|\downarrow0$.

The functions $M,Q,T_1,T_2,\dots$ are invariant under $(u,v)\mapsto(-u,-v)$. We therefore normalize $u>0$ on the descent of the phase under consideration, so that $u'<0$ on the whole phase, $u>0$ on the descent and $u<0$ on the ascent. Statements whose signs depend on this choice are to be read with this normalization.

The subsection below shows which identities are satisfied by the scaling field of the pure-power equation: they explain why the quantity $T_2$ in the slope bound is of lower order at large amplitude.

\subsection{The scaling field and the quantity \texorpdfstring{$T_2$}{T2}}\label{sec:scaling-field}
The function $T_2=Q-\varphi'g_2(u)M$ has a clear meaning in the scaled limit of Section~\ref{sec:scaled-limit}, which explains why the renewal of the slope bound $T_2(c_i)>0$ from one critical point to the next (Section~\ref{sec:renewal-above} and Section~\ref{sec:renewal-below}) is the delicate step of the argument.

\begin{proposition}[The scaling field]\label{prop:scaling-field}
    Let $w$ solve the pure-power Euclidean equation
    \[
        w''+\frac{\gamma}{r}w'+|w|^{p-1}w=0
    \]
    on $(0,\infty)$, $\gamma>1$, and let
    \[
        V\defeq \frac{2}{p-1}\,w+r\,w'=\frac{\partial}{\partial\lambda}\Big|_{\lambda=1}
        \lambda^{\frac{2}{p-1}}\,w(\lambda r)
    \]
    be its scaling field. Then $V''+\frac{\gamma}{r}V'+p|w|^{p-1}V=0$. Write $M[V]$, $Q[V]$, $T_2[V]$ for the functions of Section~\ref{sec:identities} with $(u,v)$ replaced by $(w,V)$. With $e^G=r^\gamma$, $\varphi=r$, $f(w)=|w|^{p-1}w$, $P_1=r^\gamma\big[r(w'^2+wf(w))+(\gamma-1)ww'\big]$ and the value $g_2\equiv\frac{2}{p-1}$,
    \[
        M[V]=P_1,\qquad Q[V]=\frac{2}{p-1}P_1,\qquad T_2[V]=Q[V]-\frac{2}{p-1}M[V]\equiv0.
    \]
\end{proposition}

\begin{proof}
    The family $w_\lambda(r)=\lambda^{2/(p-1)}w(\lambda r)$ solves the same equation, so $V=\partial_\lambda w_\lambda|_{\lambda=1}$ solves the linearized one. The three identities are direct computations, reducing $w''$ by the equation (Appendix~\ref{sec:identity-table}, \eqref{eq:id-scaling}). For instance,
    \begin{align*}
        M[V] & =r^\gamma\big(w'V-wV'\big)=r^\gamma\big(rw'^2-ww'-rww''\big) \\
             & =r^\gamma\big(rw'^2+(\gamma-1)ww'+r|w|^{p+1}\big)=P_1.
    \end{align*}
\end{proof}

\begin{remark}\label{rem:scaled-variation}
    Heuristically, the scaled limit identifies the variation with the scaling field. Put $\lambda\defeq\alpha^{\frac{p-1}{2}}$ and $s=\lambda r$. By Lemma~\ref{lem:scaled-convergence}, $u(r,\alpha)=\alpha\big(w(s)+o(1)\big)$ as $\alpha\to\infty$, uniformly for $s$ in compact sets.
    The rescaled variation $\tilde v(s)\defeq v(s/\lambda,\alpha)$ solves (with $'=d/ds$)
    \[
        \tilde v''+\lambda^{-1}g(s/\lambda)\tilde v'+\big(p\,|u/\alpha|^{p-1}-\lambda^{-2}\big)\tilde v=0,\qquad
        \tilde v(0)=1,\qquad \tilde v'(0)=0,
    \]
    and $\tilde v\to y$ uniformly on compact sets. Indeed, the coefficient $p|u/\alpha|^{p-1}-\lambda^{-2}$ is bounded by $pM^{p-1}+1$ with $M$ the bound of Lemma~\ref{lem:scaled-convergence}, so the integral form $e^{G(s/\lambda)}\tilde v'(s)=-\int_0^se^{G(t/\lambda)}\big(p|u/\alpha|^{p-1}-\lambda^{-2}\big)\tilde v\,dt$ and Gronwall's inequality bound $\tilde v$ and $\tilde v'$ on $[0,T]$ uniformly in $\lambda\ge1$ (the energy argument of Lemma~\ref{lem:scaled-convergence} has no analogue for this linear equation), and the passage to the limit is then as there, with $p|u/\alpha|^{p-1}\to p|w|^{p-1}$ uniformly on compact sets. Here
    \[
        y''+\frac{\gamma}{s}y'+p|w|^{p-1}y=0,\qquad y(0)=1,\qquad y'(0)=0.
    \]
    Since $\frac{p-1}{2}V$ solves the same problem, $v(r,\alpha)=\frac{p-1}{2}V(s)+o(1)$ on compact sets of $s$: in the scaled limit the variation is a multiple of the scaling field. (Both statements hold up to an additive error and say nothing about large $s$.)

    Proposition~\ref{prop:scaling-field} then says that, along the first phases of a bound state with many nodes, $Q[v]$ and $M[v]$ are positive with the margin of $P_1>0$, whereas $T_2[v]$ is a lower-order quantity. The sign of $T_2(c_i)$ comes from lower-order effects, namely the term $-u$ in $f$ and, unless $g=\gamma/r$, the deviation of $rg$ from $\gamma$ near the origin. Its verification requires analyzing these terms (it is the reason for the ``residual'' arguments of \cite{Tang2026}).

    The quantity $Q$ itself is not degenerate. In the scaled limit $Q[v]$, as a function of $s$, is asymptotic to a positive multiple (depending on $\alpha$) of $\frac{p-1}{2}Q[V]=P_1[w]$, and $P_1[w]>0$ on $(0,\infty)$. Indeed, the function
    \[
        P[w]\defeq r^{\gamma+1}(w'^2+2F(w))+(\gamma-1)r^\gamma ww'
    \]
    vanishes at $0^+$, and by \eqref{eq:id-pohozaev-limit} of Appendix~\ref{sec:identity-table} with $F(w)=|w|^{p+1}/(p+1)$ we have
    \[
        P[w]'=r^\gamma|w|^{p+1}\Big[\frac{2(\gamma+1)}{p+1}-(\gamma-1)\Big]>0
    \]
    under \eqref{eq:p-range}. Hence $P[w]>0$ and $P_1[w]=P[w]+r^{\gamma+1}h_1(w)\ge P[w]>0$.
\end{remark}

%% file: sections/16-critical-point-transition.tex
\section{The critical-point form of the transition}\label{sec:transition}

This section carries out the induction step. We assume the slope condition at one critical point of $u$ and follow $v$ through the next phase. On the descent $v$ changes sign once, at the zero of $u$ it has the right sign, and it gets through the river with the help of a bridge function. At the next critical point $v$ has the right signs again, but the third inequality, the slope bound, comes back here only in the easier of two cases, when the zero of $v$ lies high on the descent. The harder case is left to Section~\ref{sec:renewal-below}. In the last phase of a bound state the same analysis shows that $v$ grows at infinity with the sign opposite to $u$, which is all that Section~\ref{sec:staircase} needs. The argument follows the Euclidean proof of~\cite{Tang2026}, and on warped products it needs one extra hypothesis on $g$, which holds on Euclidean and hyperbolic spaces.

Precisely, the slope condition is \eqref{eq:slope-condition}, and the extra hypothesis is (H) (Assumption~\ref{ass:H}). Under (H), Theorem~\ref{thm:phase} controls the whole phase from the slope condition at the preceding critical point. The river $(b_i,\bar b_i)$, on which $|u|<\alpha_*$, is crossed by a bridge function that does not see the sign of $Q$ inside it, and Corollary~\ref{cor:ascent} completes the ascent with no smallness condition. When the zero $\tau_i$ of $v$ on the descent lies above the level $\beta_i=|u(c_i)|$, Theorem~\ref{thm:renewal-above} renews the slope bound $T_2(c_i)>0$. The remaining case, in which $\tau_i$ lies below that level, is the hardest step of~\cite{Tang2026}. It is treated in Section~\ref{sec:renewal-below}, where the main theorem is proved in critical-point form (Theorem~\ref{thm:critical-form}). The behavior at infinity is Theorem~\ref{thm:phase}(v), and Theorem~\ref{thm:tail-criterion} turns it into the theorems of Section~\ref{sec:uniqueness}. Throughout this section, we use the conventions of Section~\ref{sec:later-phases}, including the levels $\beta_i$ of \eqref{eq:beta-levels}.

\subsection{The hypothesis (H) and the variable-weight Wronskians}\label{sec:hyp-H}

\begin{assumption}[H]\label{ass:H}
    $e^G\varphi'$ is nondecreasing on $(0,\infty)$, or equivalently,
    \begin{equation}\label{eq:H-forms}
        g\varphi'+\varphi''\;\ge\;0,\qquad\text{or}\qquad 2g\varphi'+g'\varphi\;\ge\;0,\qquad\text{or}\qquad
        \big(e^G(D-1)\big)'\;\ge\;0,
    \end{equation}
    where $D=g\varphi/(\gamma-1)$ is the function of Definition~\ref{def:multiplier}.
\end{assumption}

The three forms are the same statement, since $(e^G\varphi')'=e^G(g\varphi'+\varphi'')$, $\varphi''=(g\varphi)'=g'\varphi+g\varphi'$, and $e^G\varphi'=(\gamma-1)e^G(D-1)$. Under $(\Phi)$ one has $\varphi''\le0$, and (H) is the quantitative bound $|\varphi''|\le g\varphi'$ on the rate at which $\varphi'$ decreases: in the last form, $D-1$ decays no faster than $e^{-G}$. On $\R^n$, $e^G\varphi'=r^{n-1}$ and (H) is trivial.

\begin{lemma}[(H) on the model spaces and a criterion]\label{lem:H-criterion}
    \begin{enumerate}[label=(\roman*)]
        \item Assume \textup{(G1)}, \textup{(G4)}, \textup{(G5)}. Put $\ell:=1/g$ and suppose $\ell'<2$ on $(0,\infty)$ (so that $g'+2g^2=(2-\ell')/\ell^2>0$). If
              \begin{equation}\label{eq:H-criterion}
                  2\ell\ell''\;\ge\;\ell'\,(\ell'-2)\qquad\text{on }(0,\infty),
              \end{equation}
              then (\textup{H}) holds.
        \item (\textup{H}) holds on $\R^n$ and on the hyperbolic spaces $\Hyp^n_\lambda$ of every curvature, for every $n\ge3$. On $\Hyp^3$ one has $g\varphi'+\varphi''=4e^{-4r}/(1-e^{-2r})>0$ explicitly.
    \end{enumerate}
\end{lemma}

\begin{proof}
    (i) Since $\varphi=(\gamma-1)e^G\cJ$ (Definition~\ref{def:multiplier}), the second form of \eqref{eq:H-forms} becomes $g'\varphi+2g(g\varphi-(\gamma-1))\ge0$, i.e.
    \begin{equation}\label{eq:upsilon}
        e^G\cJ\;\ge\;\Upsilon:=\frac{2g}{g'+2g^2}=\frac{2\ell}{2-\ell'}.
    \end{equation}
    Now $(\Upsilon e^{-G})'=(\Upsilon'-g\Upsilon)e^{-G}$, and $g\Upsilon-1=\frac{2g^2}{g'+2g^2}-1=-\frac{g'}{g'+2g^2}$, so that $(\Upsilon e^{-G})'\ge-e^{-G}$ if and only if $\Upsilon'\ge-\frac{g'}{g'+2g^2}$. A direct computation gives $\Upsilon'=\frac{2(g'^2-2g^2g'-gg'')}{(g'+2g^2)^2}$, and the condition becomes $3g'^2-2g^2g'-2gg''\ge0$, which in terms of $\ell$ is $\ell'(2-\ell')+2\ell\ell''\ge0$, that is \eqref{eq:H-criterion}. Next, $\ell'\le\frac32$ along a sequence $R_m\to\infty$: otherwise $\ell(r)\ge\ell(R)+\frac32(r-R)$ for $r\ge R$, so $G(r)\le G(R)+\frac23\log\frac{\ell(R)+\frac32(r-R)}{\ell(R)}$ and $e^{-G}$ would not be integrable, contradicting (G5). Along this sequence, $\Upsilon e^{-G}\le4\ell e^{-G}\to0$, because $\ell(r)\le\ell(1)+2(r-1)$ for $r\ge1$ by $\ell'<2$, while $re^{-G(r)}\le2\int_{r/2}^re^{-G}\to0$ by (G5), $e^{-G}$ being decreasing. Integrating $(\Upsilon e^{-G})'\ge-e^{-G}$ from $r$ to $R_m$ and letting $m\to\infty$ gives $\Upsilon(r)e^{-G(r)}\le\cJ(r)$, which is \eqref{eq:upsilon}.

    (ii) On $\R^n$, $\ell=r/\gamma$, $\ell'=1/\gamma\in(0,1)$, $\ell''=0$, and \eqref{eq:H-criterion} becomes $0\ge\frac1\gamma(\frac1\gamma-2)$.
    On $\Hyp^n$ (curvature $-1$, to which the general case reduces by scaling), $\ell=\tanh r/\gamma$, $\ell'=\sech^2r/\gamma$, $\ell''=-2\sech^2r\tanh r/\gamma$, and \eqref{eq:H-criterion} becomes $-4\tanh^2r\ge\sech^2r-2\gamma$, i.e.\ $2\gamma\ge1+3\tanh^2r$, which holds for $\gamma\ge2$. For $\Hyp^3$, $\varphi=\frac12(1-e^{-2r})$, $\varphi'=e^{-2r}$, $\varphi''=-2e^{-2r}$, $g=2\coth r$, and $g\varphi'+\varphi''=2e^{-2r}(\coth r-1)=4e^{-4r}/(1-e^{-2r})$.
\end{proof}

We use the Wronskians of Section~\ref{sec:variation} with variable weights, and a weighted bridge function:
\begin{gather}
    Q_\chi:=Q+\chi(r)\,e^Gu'v\quad(\chi\in C^1),\qquad\text{in particular}\qquad
    Q_{\varphi'},\ Q_{2\varphi'},\ Q_\nu\ \ (\nu=2\varphi'+\gamma-1),\label{eq:Q-chi}\\
    \sB_a:=Q-a\varphi'M-2\varphi'F_a(u)\frac{e^Gv}{u'}=W_a+\varphi'H_a(u)\,\frac{e^Gv}{u'},\quad
    H_a(u):=-2F_a(u)=a\,h_1(u)-2F(u),\label{eq:weighted-bridge}
\end{gather}
for a constant $a$, where $W_a=Q-a\varphi'M$ and $F_a=F-\frac a2h_1$ are as in Section~\ref{sec:variation}. For $a=g_2(\beta)$, $\beta>0$, one has $H_a(u)=u^2\big(1-(|u|/\beta)^{p-1}\big)$, which is nonnegative exactly for $|u|\le\beta$ and vanishes at $|u|=\beta$, and $\kappa_a(u)=(\alpha_*|u|/\beta)^{p-1}-1$ in \eqref{eq:W-prime}, since $1-\frac{p-1}{2}g_2(\beta)=\frac{p+1}{2}\beta^{1-p}=(\alpha_*/\beta)^{p-1}$. On $\R^n$, $\sB_a$ is the bridge $B_a$ of \cite{Tang2026}. In general it differs from the $B_a$ of Proposition~\ref{prop:identity-table} by the weight $\varphi'$ in front of $F_a$, which removes the term $2(\varphi'-1)e^Guv\kappa_a$ of \eqref{eq:bridge-prime} that has no definite sign on the river.

\begin{lemma}[Identities]\label{lem:weighted-identities}
    On every interval on which $u'\ne0$ (and $u\ne0$ where $\omega$ occurs; the identities containing $\theta=e^Gv\,h_1(u)/u'$ hold across the zeros of $u$, since $h_1$ is $C^1$):
    \begin{align}
        Q_\chi'                 & =e^G\big[\chi\,u'v'+(2\varphi'-\chi)f(u)v+\chi'\,u'v\big],\label{eq:Q-chi-prime}                                  \\
        \Big(\frac{v}{u'}\Big)' & =\frac{Q_{\varphi'}}{e^G\varphi u'^2},\qquad
        \Big(\frac{e^Gv}{u'}\Big)'=\frac{Q_\nu}{\varphi u'^2},\qquad
        \Big(\frac{Q}{e^Gu'}\Big)'=\frac{f(u)\,Q_{2\varphi'}}{e^Gu'^2},\label{eq:ratio-primes}                                                      \\
        (M-\theta)'             & =-\frac{h_1(u)\,Q_\nu}{\varphi u'^2},\label{eq:M-theta-prime}                                                     \\
        \sB_a'                  & =-2\varphi'F_a(u)\frac{Q_\nu}{\varphi u'^2}
        -\varphi''\Big(aM+2F_a(u)\frac{e^Gv}{u'}\Big)\label{eq:weighted-bridge-prime}                                                               \\
                                & =\varphi'H_a(u)\frac{Q}{\varphi u'^2}+a|\varphi''|M+c_H\,H_a(u)\frac{e^Gv}{u'},\label{eq:weighted-bridge-prime-H}
    \end{align}
    where $c_H:=\varphi'\nu/\varphi+\varphi''=g\varphi'+\varphi''+\varphi'^2/\varphi$, so that $c_H\ge\varphi'^2/\varphi>0$ under (\textup{H}); the second form \eqref{eq:weighted-bridge-prime-H} is recorded for comparison with~\cite{Tang2026} and is not used below. Moreover, at a zero $\tau$ of $v$ and at a point $s$ with $g_2(u(s))=a$,
    \begin{equation}\label{eq:bridge-values}
        \sB_a(\tau)=W_a(\tau)=M(\tau)\big(\omega(\tau)-a\varphi'(\tau)\big),\quad
        T_2(\tau)=M(\tau)\big(\omega(\tau)-\varphi'(\tau)g_2(u(\tau))\big),\quad
        \sB_a(s)=T_2(s),
    \end{equation}
    and at a critical point $c$ of $u$ with $a=g_2(|u(c)|)$, $\sB_a(r)\to T_2(c)$ as $r\to c$. Finally $Q=\frac vuP_1+\omega M$ wherever $u\ne0$.
\end{lemma}

\begin{proof}
    All are finite computations from \eqref{eq:ode}, the variation equation and $\varphi'=g\varphi-(\gamma-1)$ (Appendix~\ref{sec:identity-table}, \eqref{eq:id-Q-chi}--\eqref{eq:id-Q-split}). \eqref{eq:Q-chi-prime} is \eqref{eq:Qa-prime} with the extra term from $\chi'$.
    For the first two identities in \eqref{eq:ratio-primes}: $(v/u')'=(u'v'-u''v)/u'^2=[u'v'+gu'v+f(u)v]/u'^2$ by \eqref{eq:ode}, and $e^G\varphi\big(u'v'+f(u)v+gu'v\big)=Q+\varphi'e^Gu'v=Q_{\varphi'}$ by the definition of $Q$ in Section~\ref{sec:variation} and $g\varphi=\varphi'+\gamma-1$. The second follows from the first and $(e^G)'=ge^G$, as $\varphi'+g\varphi=\nu$. The third and \eqref{eq:M-theta-prime} are \eqref{eq:id-Q-ratio}, \eqref{eq:id-M-theta}. The first form of \eqref{eq:weighted-bridge-prime} is \eqref{eq:id-weighted-bridge}, and the second follows by writing $Q_\nu=Q+\nu e^Gu'v$. At a zero of $v$ all terms with a factor $v$ vanish, $Q=e^G\varphi u'v'$, $M=-e^Guv'$, and $\omega=-\varphi u'/u$. At $s$, $F_a(u(s))=\frac12h_1(u(s))(g_2(u(s))-a)=0$. At $c$, $F_a(u(r))=O((r-c)^2)$ and $u'(r)^{-1}=O(|r-c|^{-1})$, so the last term of $\sB_a$ tends to~$0$. The last identity is the connection identity \eqref{eq:connection} with $\omega\theta=-\varphi e^Gvh_1(u)/u$.
\end{proof}

\subsection{The phase from its critical point}\label{sec:phase}
We say that $v$ satisfies the \emph{slope condition} at a critical point $c$ of $u$, with $|u(c)|=\beta>\alpha_*$, if
\begin{equation}\label{eq:slope-condition}
    Q(c)>0,\qquad M(c)>0,\qquad T_2(c)=Q(c)-\varphi'(c)g_2(\beta)M(c)>0.
\end{equation}
The rest of Part~II is an induction over the critical points $c_0,c_1,\dots$ of $u$, and the slope condition is its inductive hypothesis. Since $u'(c)=0$, $Q(c)=e^G\varphi f(u)v(c)$ and $M(c)=-e^Guv'(c)$, so the first two say that $v(c)$ has the sign of $u(c)$ and $v'(c)$ the opposite sign. With the normalization of Section~\ref{sec:later-phases}, $u(c)=\beta>0$ and $v(c)>0$, and
\[
    T_2(c)=e^{G(c)}\big[\varphi(c)f(\beta)\,v(c)+\varphi'(c)g_2(\beta)\beta\,v'(c)\big],
\]
so the third says that the ratio $\kappa\defeq -v'(c)/v(c)>0$ satisfies
\[
    \kappa<\kappa_{\max}\defeq \frac{\varphi(c)f(\beta)}{\varphi'(c)g_2(\beta)\beta},
\]
or equivalently that $Q(c)/M(c)>\varphi'(c)g_2(\beta)$. Thus the slope condition says that $v(c)$ has the sign of $u(c)$ and that the ratio $\kappa$ lies in $(0,\kappa_{\max})$. We call the third inequality, $T_2(c)>0$, the \emph{slope bound}. In \cite[Lemma~5.1]{Tang2026} these are the inductive data, with $g_2$ replaced by $\varphi'g_2$.

\begin{theorem}[Descent, river and ascent from the slope condition at a critical point]\label{thm:phase}
    Assume (\textup{H}). Let $u$ be as above, let $2\le i\le K+1$, and suppose that \eqref{eq:slope-condition} holds at $c=c_{i-1}$. Then, with $\tau_i$ the first zero of $v$ after $c_{i-1}$:
    \begin{enumerate}[label=(\roman*)]
        \item $\tau_i\in(c_{i-1},r_i)$, $v>0$, $v'<0$ on $[c_{i-1},\tau_i)$, $v'(\tau_i)<0$, and $v<0$ on $(\tau_i,c_i]$ if $i\le K$ (on $(\tau_i,z_i]$ if $i=K+1=k$ for a nodal solution, on $(\tau_{k+1},\infty)$ if $i=k+1$ for a bound state): $v$ has exactly one zero in the phase, and $v(z_i)u'(z_i)>0$ whenever $z_i$ exists.
        \item $T_2$ is strictly increasing on $[c_{i-1},b_i]$, and in particular $T_2>0$ on $[c_{i-1},b_i]$.
        \item $Q$, $Q_{\varphi'}$, $Q_{2\varphi'}$ are positive on $[c_{i-1},\max\{b_i,\tau_i\}]$, and $Q_{\varphi'}>0$ on $[\tau_i,\bar b_i]$ if $i\le K$ (on $[\tau_i,z_i]$ if $i=K+1=k$ for a nodal solution, on $[\tau_{k+1},\infty)$ if $i=k+1$ for a bound state). Moreover, if $i\le K$, then $Q(\bar b_i)>Q(b_i)>0$ and $Q$ is increasing and positive on $[\bar b_i,c_i]$.
        \item if $i\le K$: $M>0$ on $[c_{i-1},c_i]$, $Q_{\varphi'}>0$ and $Q_\nu>0$ on $(\tau_i,c_i]$, and $Q(c_i)>0$, $M(c_i)>0$.
        \item if $i=k+1$ and $u$ is a bound state, $v$ is eventually strictly monotone and $|v(r)|\to\infty$: $v$ is dominant at infinity with the sign opposite to $u$, without passing through the final-phase condition \eqref{eq:final-phase}, and Corollary~\ref{cor:divergence} applies with the zero $\tau_{k+1}$ of \textup{(i)}.
    \end{enumerate}
    For $i=1$ the same conclusions hold with $c_0=0$ and no slope condition, the closed left ends $[c_0,\cdot\,]$ in (i)--(iv) being replaced by $(0,\cdot\,]$: Theorem~\ref{thm:first-phase}\textup{(i)--(ii)} replaces Step~0 below, and Steps~1--2 go through unchanged from $r=0$, where $T_2$, $Q$, $Q_{\varphi'}$ and $Q_{2\varphi'}$ tend to $0$. In particular, if $K\ge1$: $\tau_1\in(0,r_1)$, $v<0$ on $(\tau_1,c_1]$, $Q_{\varphi'}>0$ on $[\tau_1,\bar b_1]$, $Q(\bar b_1)>Q(b_1)>0$, $Q$ is increasing on $[\bar b_1,c_1]$, and $M,Q>0$ at $c_1$. Since $Q_1=Q_{\varphi'}+(1-\varphi')e^Gu'v\ge Q_{\varphi'}$ where $u'v>0$, and $Q_a(c_1)=Q(c_1)$ for every $a$, Theorem~\ref{thm:first-phase}\textup{(v)} holds under (\textup{H}) in place of $\varphi'\equiv1$.
\end{theorem}

\begin{proof}
    We write $c:=c_{i-1}$ and $\beta:=\beta_{i-1}$, and follow \cite[\S5.7, Steps~1--5]{Tang2026}, emphasizing the modifications.

    \emph{Step 0: tendency and signs on $(c,\tau_i)$.} By \eqref{eq:slope-condition}, $v(c)>0$ and $v'(c)<0$. Recall
    \[
        \Sigma=e^G(u''v'-u'v'')=e^G\big[f'(u)u'v-f(u)v'\big],
    \]
    so that $\Sigma'=e^Gf''(u)u'^2v$ by \eqref{eq:Sigma-prime}. Suppose that $v>0$ on $(c,r_i)$. Then $\Sigma'>0$ on $(c,r_i)$, as $u>1$ and $f''(u)>0$ there. On the other hand, since $f(1)=0$, $f'(1)=p-1$ and $v(r_i)\ge0$,
    \[
        \Sigma(c)=-e^Gf(\beta)v'(c)>0,\qquad \Sigma(r_i)=(p-1)e^Gu'v(r_i)\le0,
    \]
    a contradiction. Hence $\tau_i\in(c,r_i)$. On $(c,\tau_i)$, $u>1$ gives $f'(u)>0$, so $(e^Gv')'=-e^Gf'(u)v<0$ and $e^Gv'$ decreases from the negative value $e^{G(c)}v'(c)$. Thus $v'<0$ on $(c,\tau_i]$, whence
    \[
        u'v'>0,\qquad uv'<0,\qquad u'v<0,\qquad vv'<0\qquad\text{on }(c,\tau_i),
    \]
    and $u'v'(\tau_i)>0$. Also $M'=(p-1)e^Gu^pv>0$ on $(c,\tau_i)$, so $M>0$ on $[c,\tau_i]$, with $M(\tau_i)=-e^Guv'(\tau_i)>0$.

    \emph{Step 1: $T_2$ increases on $[c,b_i]$, and $v<0$ on $(\tau_i,b_i]$ if $\tau_i<b_i$.} On the descent, $uu'/|u|^{p+1}=u'/u^p<0$, so \eqref{eq:T-two-prime} becomes
    \begin{equation}\label{eq:T-two-descent}
        T_2'=\varphi'(p+1)\frac{|u'|}{u^p}\,(M-\theta)+|\varphi''|\,g_2(u)\,M,\qquad
        \theta=e^G\frac{v\,h_1(u)}{u'}.
    \end{equation}
    On $(c,\min\{\tau_i,b_i\})$: $M>0$, $\theta<0$ (as $v>0>u'$), $g_2(u)\ge0$ (as $u\ge\alpha_*$), so $T_2'>0$. If $\tau_i\ge b_i$ this proves (ii).

    Let $\tau_i<b_i$, and let $\tau^*$ be the supremum of the $s\in(\tau_i,b_i]$ such that $v<0$ on $(\tau_i,s]$ and $T_2'>0$ on $[\tau_i,s]$. Since $\theta(\tau_i)=0$, we have
    \[
        T_2'(\tau_i)=\varphi'(p+1)|u'|u^{-p}M(\tau_i)+|\varphi''|g_2M(\tau_i)>0.
    \]
    As moreover $v<0$ just after $\tau_i$, it follows that $\tau^*>\tau_i$. By continuity $v\le0$ on $(\tau_i,\tau^*]$ and $T_2$ is nondecreasing on $[\tau_i,\tau^*]$, so $T_2(\tau^*)\ge T_2(\tau_i)>0$. If $\tau^*<b_i$, or $\tau^*=b_i$ with $v(b_i)=0$ or $T_2'(b_i)=0$, then $v(\tau^*)=0$ or $T_2'(\tau^*)=0$ (with $v<0$ on $(\tau_i,\tau^*)$ in the second case).

    \emph{Case $T_2'(\tau^*)=0$, $v(\tau^*)<0$.} By \eqref{eq:T-two-descent}, $M-\theta=-\eta M$ at $\tau^*$ with $\eta:=|\varphi''|g_2u^p/(\varphi'(p+1)|u'|)\ge0$. Here $\theta(\tau^*)>0$ (as $v,u'<0$), so $M=\theta/(1+\eta)>0$ and $M-\theta\le0$. The connection identity \eqref{eq:connection} gives $Q(\tau^*)=P\frac vu+\omega(M-\theta)<0$, since $v/u<0$, $P>0$ (Proposition~\ref{prop:pohozaev-positive}) and $\omega>0$ on the descent, but $T_2(\tau^*)>0$ and $T_2=Q-\varphi'g_2(u)M\le Q$ ($g_2(u)\ge0$, $M>0$), a contradiction.

    \emph{Case $v(\tau^*)=0$.}
    Then $v'(\tau^*)>0$, since $v'$ cannot vanish together with $v$. By \eqref{eq:bridge-values},
    \[
        T_2(\tau^*)=M(\tau^*)\big(\omega-\varphi'g_2(u)\big)(\tau^*),\qquad M(\tau^*)=-e^Guv'(\tau^*)<0.
    \]
    Now $T_2(\tau_i)>0$ means $\omega(\tau_i)>\varphi'(\tau_i)g_2(u(\tau_i))$. Along the descent $\omega$ increases (Proposition~\ref{prop:pohozaev-positive}), while $\varphi'g_2(u)\ge0$ decreases, because $\varphi'$ is nonincreasing, $g_2$ is increasing and $u$ is decreasing. Therefore $\omega>\varphi'g_2(u)$ at $\tau^*$ too, and $T_2(\tau^*)<0$. This contradicts $T_2(\tau^*)\ge T_2(\tau_i)>0$.

    Both cases are thus impossible, so $\tau^*=b_i$ with $v(b_i)<0$ and $T_2'(b_i)>0$, and (ii) and the second assertion follow. We observe that the argument of the first case shows more generally that $M>0$ on $(\tau_i,b_i]$. Indeed, at a first zero $m$ of $M$ in $(\tau_i,b_i]$, \eqref{eq:Q-at-M-zero} would give $Q(m)=\frac{v(m)}{u(m)}P_1(m)<0$, while $Q(m)\ge T_2(m)>0$ by continuity, since $Q\ge T_2$ as long as $M\ge0$.

    \emph{Step 2: $Q,Q_{\varphi'},Q_{2\varphi'}>0$ on $[c,\max\{b_i,\tau_i\}]$.} At $c$ the three coincide with $Q(c)>0$. On $(c,\tau_i)$, by \eqref{eq:Q-prime} and \eqref{eq:Q-chi-prime}, $Q'=2\varphi'e^Gf(u)v>0$, $Q_{\varphi'}'=e^G[\varphi'(u'v'+f(u)v)+\varphi''u'v]>0$ and $Q_{2\varphi'}'=2e^G[\varphi'u'v'+\varphi''u'v]>0$, since $u'v'>0$, $f(u)v>0$ and $\varphi''u'v\ge0$ ($\varphi''\le0$, $u'v<0$). If $b_i\le\tau_i$ we are done. If $\tau_i<b_i$: on $(\tau_i,b_i]$, $M>0$ (Step~1), so $Q=T_2+\varphi'g_2(u)M\ge T_2>0$, and $u'v>0$ gives $Q_{2\varphi'}>Q_{\varphi'}>Q>0$.

    \emph{Step 3: $Q_{\varphi'}>0$ on $[\tau_i,\bar b_i]$ and $v<0$ on $(\tau_i,\bar b_i]$ (on $[\tau_i,z_i]$, resp.\ $[\tau_{k+1},\infty)$).} Let $\tilde q$ be the supremum of the $s\in[\tau_i,\bar b_i]$ with $Q_{\varphi'}>0$ on $[\tau_i,s]$. Since $Q_{\varphi'}(\tau_i)=Q(\tau_i)=e^G\varphi u'v'(\tau_i)>0$, we have $\tilde q>\tau_i$, and by Step~2, $\tilde q>\max\{b_i,\tau_i\}$. On $(\tau_i,\tilde q)$, $v\ne0$: at a first zero $\tilde\tau$ of $v$ after $\tau_i$ one has $v'(\tilde\tau)>0$ and $Q_{\varphi'}(\tilde\tau)=e^G\varphi u'v'(\tilde\tau)<0$. Hence $v<0$, $u'v>0$ and $Q_\nu=Q_{\varphi'}+(\varphi'+\gamma-1)e^Gu'v>0$ on $(\tau_i,\tilde q)$. Suppose $Q_{\varphi'}(\tilde q)=0$ and set $s_0:=\max\{b_i,\tau_i\}$. On $(s_0,\tilde q)$ we have $F(u)<0$ except at $z_i$, where $F(u)=0$, and, by \eqref{eq:weighted-bridge-prime} with $a=0$ and $Q_\nu=Q_{\varphi'}+g\varphi e^Gu'v$,
    \[
        \sB_0'=\frac{2|F(u)|}{\varphi u'^2}\Big[\varphi'Q_{\varphi'}+\big(g\varphi'+\varphi''\big)\varphi\,
        e^Gu'v\Big]>0
    \]
    by (H), except at $z_i$, where $\sB_0'=0$. Thus $\sB_0$ is strictly increasing on $[s_0,\tilde q]$.
    At $s_0$, $\sB_0=Q$ (at $\tau_i$ because $v=0$, at $b_i$ because $F=0$), and $Q(s_0)>0$ by Step~2, so $\sB_0(\tilde q)>0$. But $Q_{\varphi'}(\tilde q)=0$ means $Q(\tilde q)=-\varphi'e^Gu'v(\tilde q)$, and $v(\tilde q)<0$ (if $v(\tilde q)=0$, then $Q_{\varphi'}(\tilde q)=e^G\varphi u'v'(\tilde q)=0$ would force $v'(\tilde q)=0$). Hence
    \[
        \sB_0(\tilde q)=-\varphi'e^Gu'v-2\varphi'F(u)\frac{e^Gv}{u'}
        =-\varphi'\frac{e^Gv}{u'}\big(u'^2+2F(u)\big)=-2\varphi'\frac{e^Gv}{u'}\,E(\tilde q)<0,
    \]
    because $v/u'>0$ and $E>0$. This contradiction shows $Q_{\varphi'}>0$ on $[\tau_i,\bar b_i]$, and then $v<0$ on $(\tau_i,\bar b_i]$ by the first part. In the final phase of a bound state the same argument runs on $(\tau_{k+1},\infty)$, where $F(u)<0$ beyond $b_{k+1}$ and $E>0$. For the last descent of a nodal solution it runs on $(\tau_k,z_k]$.

    \emph{Step 4: $Q(\bar b_i)>Q(b_i)>0$.} By Step~3, $\sB_0'>0$ on $(s_0,\bar b_i)\setminus\{z_i\}$, so $\sB_0$ is strictly increasing on $[s_0,\bar b_i]$, and $\sB_0=Q$ at $s_0$ and at $\bar b_i$ (where $F=0$). If $s_0=b_i$ this yields $Q(\bar b_i)>Q(b_i)$. If $s_0=\tau_i>b_i$, then $Q(\bar b_i)>Q(\tau_i)>Q(b_i)$, the last because $Q'=2\varphi'e^Gf(u)v>0$ on $(b_i,\tau_i)$.

    \emph{Step 5: the ascent above $\alpha_*$.} On $[\bar b_i,c_i]$, $u\le-\alpha_*<-1$, so $f(u)<0$, and while $v<0$ one has $Q'=2\varphi'e^Gf(u)v>0$. A first zero $\tilde\tau\in(\bar b_i,c_i]$ of $v$ would have $v'(\tilde\tau)\ge0$ and $Q(\tilde\tau)=e^G\varphi u'v'(\tilde\tau)\le0<Q(\bar b_i)\le Q(\tilde\tau)$. Hence $v<0$ on $(\bar b_i,c_i]$, $Q$ is increasing on $[\bar b_i,c_i]$, $Q\ge Q(\bar b_i)>0$, and $Q(c_i)>0$. This completes (i) and (iii).

    (iv) $M$ increases on $(c,\tau_i)$ from $M(c)>0$, decreases on $(\tau_i,z_i)$ to $M(z_i)=e^Gu'v(z_i)>0$, and increases on $(z_i,c_i)$, where $uv>0$. Hence $M>0$ on $[c,c_i]$ and $M(c_i)>0$. On $(\tau_i,c_i]$, $u'v\ge0$ (with $u'v>0$ on $(\tau_i,c_i)$ and $u'(c_i)=0$) gives $Q_{\varphi'}\ge Q$ on $[\bar b_i,c_i]$ and, with Step~3, $Q_{\varphi'}>0$ on $(\tau_i,c_i]$, hence $Q_\nu=Q_{\varphi'}+(\varphi'+\gamma-1)e^Gu'v>0$.

    (v) By (i), $v<0$ on $(\tau_{k+1},\infty)$. Since $f'(u)<0$ once $|u|<p^{-1/(p-1)}$, we have $(e^Gv')'=-e^Gf'(u)v<0$ eventually. Hence $e^Gv'$ is eventually decreasing and $v'$ has eventually one sign, that is, $v$ is eventually strictly monotone.

    Suppose that $v$ stays bounded. Then it decays (Lemma~\ref{lem:dichotomy}), and $Q\to0$ as in the proof of Proposition~\ref{prop:nondegeneracy}. Moreover $F(u)e^Gv/u'\to0$. Indeed, eventually
    \[
        |F(u)|\le u^2,\qquad |u'|\ge cu,\qquad u\,e^G|v|\le Ce^{-(1-\varepsilon)\int m+G/2-\int\sqrt q}\to0,
    \]
    with $m$ as in Section~\ref{sec:tail}. The second bound follows from $u'/u=-\frac g2-\sqrt q+o(1)$ (as in the proof of Lemma~\ref{lem:tail-sign}), and the third holds because the exponent has rate $\varepsilon\frac g2-(2-\varepsilon)\sqrt q+o(1)<0$ for $\varepsilon$ small. Therefore $\sB_0\to0$. But $\sB_0$ is increasing on $(\max\{b_{k+1},\tau_{k+1}\},\infty)$ by Step~3 and $\sB_0\ge Q(s_0)>0$ at its left end. Hence $v$ is dominant and $|v|\to\infty$, with the sign opposite to $u$. Corollary~\ref{cor:divergence} uses only this and the single zero $\tau_{k+1}$.

    For $i=1$: Theorem~\ref{thm:first-phase}(i)--(ii) give $\tau_1\in(0,r_1)$, $v>0$ and $v'<0$ on $(0,\tau_1)$, $M,Q>0$ on $(0,z_1]$, $T_2>0$ on $(0,z_1)$ and $v<0$ on $(\tau_1,z_1]$. Steps~1--2 then run on $(0,\max\{b_1,\tau_1\}]$ with the left-end values
    \[
        T_2(0^+)=Q(0^+)=Q_{\varphi'}(0^+)=Q_{2\varphi'}(0^+)=0
    \]
    in place of the slope condition \eqref{eq:slope-condition}. The derivatives computed there are positive on $(0,\min\{b_1,\tau_1\})$, so the four functions are positive on $(0,\min\{b_1,\tau_1\}]$, and the rest of Steps~1--2 is unchanged. This is all that Steps~3--5 use at the left end.
\end{proof}

\begin{corollary}[Completing the ascent under (H)]\label{cor:ascent}
    Assume (\textup{H}), and let $1\le i\le K$ with $Q(z_i)>0$ and $M(z_i)>0$ (the latter means $v(z_i)u'(z_i)>0$, while the former is an independent condition). Then $Q_{\varphi'}>0$ and $v\ne0$ on $[z_i,\bar b_i]$, $Q(\bar b_i)>Q(z_i)$, $Q$ is increasing on $[\bar b_i,c_i]$, and $v\ne0$ on $(z_i,c_i]$ with $Q(c_i)>0$, $M(c_i)>0$. In other words, under (\textup{H}) the ascent from a zero is completed with no smallness condition, as on $\R^n$. The corollary is not needed for the main theorems, whose ascents are obtained from Theorem~\ref{thm:phase}(iii)--(iv); it is recorded for comparison with Theorem~\ref{thm:first-phase}(v).
\end{corollary}

\begin{proof}
    Run Steps~3--5 of the preceding proof from $z_i$ instead of $\tau_i$: $Q_{\varphi'}(z_i)=Q(z_i)+\varphi'M(z_i)>0$, $\sB_0(z_i)=Q(z_i)>0$ since $F(0)=0$, and $F(u)<0$ on $(z_i,\bar b_i)$.
\end{proof}

\subsection{The renewal above the level \texorpdfstring{$\beta_i$}{beta\_i}}\label{sec:renewal-above}
The slope condition \eqref{eq:slope-condition} at $c_i$ consists of $Q(c_i)>0$, $M(c_i)>0$, which Theorem~\ref{thm:phase}(iv) provides, and of the slope bound $T_2(c_i)>0$, which it does not. On $\R^n$ this is the content of \cite[Lemmas~5.3--5.6]{Tang2026}. We write $u_\tau:=|u(\tau_i)|$. In this subsection we treat the case $u_\tau\ge\beta_i$, in which the zero of $v$ on the descent lies above the level $\beta_i=|u(c_i)|$. The case $u_\tau<\beta_i$ is treated in Section~\ref{sec:renewal-below}. Both use the following form of \eqref{eq:weighted-bridge-prime}: inserting $Q_\nu=Q_{\varphi'}+g\varphi\,e^Gu'v$, with $e^Gu'v=u'^2\,e^Gv/u'$, in its first form yields, for every constant $a$ and wherever $u'\ne0$,
\begin{equation}\label{eq:weighted-bridge-renewal}
    \sB_a'=\varphi'H_a(u)\frac{Q_{\varphi'}}{\varphi u'^2}+a|\varphi''|M
    +\big(g\varphi'+\varphi''\big)H_a(u)\frac{e^Gv}{u'}.
\end{equation}

\begin{theorem}[Renewal above the level $\beta_i$]\label{thm:renewal-above}
    Assume (\textup{H}), let $1\le i\le K$, and suppose that \eqref{eq:slope-condition} holds at $c_{i-1}$ (nothing if $i=1$), so that Theorem~\ref{thm:phase} applies to phase~$i$. Let $t_i\in(c_{i-1},b_i)$ be the point of the descent where $|u|=\beta_i$. If $u_\tau\ge\beta_i$, then $T_2(c_i)\;>\;T_2(t_i)\;>\;T_2(c_{i-1})$ for $i\ge2$, and $T_2(c_1)>T_2(t_1)>0$. Consequently \eqref{eq:slope-condition} holds at $c_i$.
\end{theorem}

\begin{proof}
    By Theorem~\ref{thm:phase}(i), $v>0$ on $[c_{i-1},\tau_i)$ and $v<0$ on $(\tau_i,c_i]$. The point $t_i$ exists since $\alpha_*<\beta_i<\beta_{i-1}$, and $t_i<b_i$. Since $u$ decreases on the descent and $u(\tau_i)=u_\tau\ge\beta_i=u(t_i)$, we have $\tau_i\le t_i$. Put $a_i:=g_2(\beta_i)>0$, so that $H_{a_i}(u)=u^2\big(1-(|u|/\beta_i)^{p-1}\big)$.

    On $(t_i,c_i)$ one has $|u|<\beta_i$, so $H_{a_i}(u)\ge0$, with equality only at $z_i$. Moreover, $v/u'>0$ (as $v,u'<0$), and $M>0$, $Q_{\varphi'}>0$ (Theorem~\ref{thm:phase}(iv)). By \eqref{eq:weighted-bridge-renewal} and (H), $\sB_{a_i}'\ge0$ on $(t_i,c_i)$, and $\sB_{a_i}'>0$ wherever $|u|<\beta_i$ and $u\ne0$. Near $c_i$ the integrand is bounded: $H_{a_i}(u)=O((r-c_i)^2)$, while $u'(r)^2\ge c\,(r-c_i)^2$ because $u''(c_i)=-f(u(c_i))\ne0$, so that $H_{a_i}(u)/u'^2=O(1)$ and $H_{a_i}(u)\,e^Gv/u'=O(|r-c_i|)$. Hence, by Lemma~\ref{lem:weighted-identities},
    \[
        T_2(c_i)=\lim_{r\to c_i}\sB_{a_i}(r)=\sB_{a_i}(t_i)+\int_{t_i}^{c_i}\sB_{a_i}'\,dr.
    \]
    Now $\sB_{a_i}(t_i)=T_2(t_i)$ by \eqref{eq:bridge-values}, since $g_2(u(t_i))=g_2(\beta_i)=a_i$, and $T_2(t_i)>T_2(c_{i-1})$ by Theorem~\ref{thm:phase}(ii) (for $i=1$, $T_2(t_1)>0$ by Theorem~\ref{thm:first-phase}(ii)). The integral is positive, since its integrand is nonnegative and positive on $(t_i,c_i)\setminus\{z_i\}$. This gives the strict inequalities. The other two inequalities of the slope condition at $c_i$ are given by Theorem~\ref{thm:phase}(iv).
\end{proof}

\begin{remark}\label{rem:renewal-above}
    (a) The proof uses, besides the identities and (H), only the signs provided by Theorem~\ref{thm:phase}, among them $Q_{\varphi'}>0$ on $(\tau_i,c_i]$, which is the river crossing, and $H_{a_i}\ge0$ below the level $\beta_i$. On $\R^n$ ($\varphi''\equiv0$) it gives \cite[Lemma~5.6]{Tang2026} in the case $u_\tau\ge\beta_i$ without \cite[Lemmas~5.3--5.5]{Tang2026}: neither the comparison of $Q/(r^{n-1}|u'|)$ across the zero, nor the monotonicity of $P/r^n$, nor the critical-point information of $T_2$ on the ascent is needed. In \cite{Tang2026}, the only possible critical point of $T_2$ on $(\bar b_i,c_i)$ is the zero $\zeta_i$ of $M-\theta$ on the ascent: by \eqref{eq:M-theta-prime} and Theorem~\ref{thm:phase}(iv), $M-\theta$ decreases strictly on $[\tau_i,c_i)$ from $M(\tau_i)>0$ to $-\infty$, so it has exactly one zero beyond $z_i$. This zero is not used here: it appears only in the case $u_\tau<\beta_i$ (Section~\ref{sec:renewal-below}).

    (b) In the scaled limit all quantities in \eqref{eq:weighted-bridge-renewal} vanish: along the scaling field $T_2\equiv0$ (Proposition~\ref{prop:scaling-field}), $g_2\equiv\frac{2}{p-1}$, $H_{2/(p-1)}\equiv0$.
\end{remark}

%% file: sections/17-renewal.tex
\section{The renewal below the level \texorpdfstring{$\beta_i$}{beta\_i}}\label{sec:renewal-below}

This section finishes the induction step, and with it the analysis phase by phase. It is the longest and most technical part of the proof, and the only place where the restriction (R) is used. The idea, taken from~\cite{Tang2026}, is to compare the descent and the ascent of a phase at equal heights of $|u|$. Once this is done, the slope condition holds at every critical point, and at every bound state $v$ grows at infinity with the sign opposite to $u$ (Theorem~\ref{thm:critical-form}). The last two sections study the tail of a bound state and turn this sign into the staircase.

Section~\ref{sec:transition} carries the slope condition \eqref{eq:slope-condition} from one critical point of $u$ to the next, except for the slope bound $T_2(c_i)>0$, which Theorem~\ref{thm:renewal-above} provides when the zero $\tau_i$ of $v$ on the descent lies above the level $\beta_i=|u(c_i)|$. In this section we treat the remaining case $u_\tau<\beta_i$, which on $\R^n$ is the delicate step of~\cite{Tang2026} (Lemmas~5.3--5.6). We pair the descent with the ascent at equal levels, as in~\cite[\S5.6]{Tang2026}, but we organize the pairing by integrating $T_2'$ level by level (Theorem~\ref{thm:renewal-below}, Steps 1--4). The Euclidean pairing uses three facts that do not carry over directly to a warped product:
\begin{itemize}
    \item[(i)] the comparison of $\varpi=Q/(e^G|u'|)$ at paired levels across the zero of $u$ \cite[Lemma~5.5]{Tang2026},
    \item[(ii)] the monotonicity of $P/r^n$ \cite[Proposition~2.5(iii)]{Tang2026}, whose warped form is the positivity of $P_2$ (Remark~\ref{rem:P-two}),
    \item[(iii)] $n\ge3$, through the monotonicity of $r^{n-3}=e^G/\varphi^2$. On a warped product, $e^G/\varphi^2$ is nondecreasing if and only if $\varphi'\le\gamma-1$.
\end{itemize}
We prove (i) at every level above $\alpha_*$ (Lemma~\ref{lem:level-comparison}), by combining a weighted energy inequality (Lemma~\ref{lem:paired-energy}) with the monotonicity of $e^G\varphi'v/u'$, which is where (H) is used. We prove (ii) (Lemma~\ref{lem:P-two-positive}), and we show that (iii) is not needed: the weighted energy inequality replaces the monotonicity of $e^G/\varphi^2$ by that of $e^{2G}/\varphi^2$, which always holds.

The case $\tau_i>b_i$, in which the zero of $v$ lies below the level $\alpha_*$, is the residual case of~\cite{Tang2026}, settled for $n=3$, $1<p<2$ by~\cite[Lemma~5.4]{Tang2026}. We prove the warped analogue of that lemma for every $\gamma\ge\frac85$ and every $p>1$ (Lemma~\ref{lem:river-integral}), by a shorter argument than the one in~\cite{Tang2026} (Remark~\ref{rem:river-integral}). The river is then crossed with a bridge whose weight is frozen at a constant (Lemma~\ref{lem:frozen-bridge}). For $1<\gamma<\frac85$ we use instead the warped analogue of~\cite[Lemma~5.3]{Tang2026} (Lemma~\ref{lem:Q-nu-positive}), which requires $p\ge\frac2{\gamma-1}$: this is where the restriction (R) of Assumption~\ref{ass:R} comes from.

The result (Theorem~\ref{thm:renewal-below}) is that $T_2(c_i)>T_2(c_{i-1})$ at every phase, under the standing hypotheses of Section~\ref{sec:transition} together with (R). The main theorem in critical-point form (Theorem~\ref{thm:critical-form}) then follows, and with it the theorems of Section~\ref{sec:uniqueness}. By Lemma~\ref{lem:phi-basic}, $\nu\le\gamma+1$. The second alternative of (R), $p\ge\frac2{\gamma-1}$, is equivalent to $\gamma+1\le(p+1)(\gamma-1)$, so it implies
\begin{equation}\label{eq:nu-bound}
    \nu\le(p+1)(\gamma-1)\qquad\text{on }(0,\infty).
\end{equation}
Only this consequence is used, and only at the points $b_i$ with $b_i<\tau_i$.

\subsection{Conventions and facts used}\label{sec:renewal-setting}

Throughout Lemmas~\ref{lem:paired-energy}--\ref{lem:frozen-bridge} and Theorem~\ref{thm:renewal-below} we assume the following.
\begin{itemize}
    \item Hypothesis (H) holds.
    \item $u$, $K$, the points of phase $i$ and the levels $\beta_i$ are as in Section~\ref{sec:later-phases}.
    \item $1\le i\le K$ is fixed and, if $i\ge2$, the slope condition \eqref{eq:slope-condition} holds at $c_{i-1}$.
\end{itemize}
We write $u_\tau\defeq |u(\tau_i)|$, and recall that $\alpha_*^{p-1}=\frac{p+1}2$.
To keep long products and quotients readable, a vertical bar with a subscript means that all the terms before it are evaluated at that point. For instance,
\[
    \frac{e^Gv}{\varphi u'^2}\Big|_{r}\defeq\frac{e^{G(r)}v(r)}{\varphi(r)\,u'(r)^2}.
\]

\emph{Level points.} For $0<\mu<\beta_{i-1}$, let $r_\mu\in(c_{i-1},z_i)$ be the point with $u(r_\mu)=\mu$. For $0<\mu<\beta_i$, let $\bar r_\mu\in(z_i,c_i)$ be the point with $u(\bar r_\mu)=-\mu$. Both maps are $C^1$, with $dr_\mu/d\mu=1/u'(r_\mu)$ and $d\bar r_\mu/d\mu=-1/u'(\bar r_\mu)$, and $b_i=r_{\alpha_*}$, $\bar b_i=\bar r_{\alpha_*}$.

From Theorem~\ref{thm:phase} (for $i=1$: Theorem~\ref{thm:first-phase}(i)--(ii) and the last paragraph of Theorem~\ref{thm:phase}) we use:
\begin{enumerate}
    \item[(D1)] $\tau_i\in(c_{i-1},r_i)$ is the only zero of $v$ in $(c_{i-1},c_i]$. Moreover $v>0$ on $[c_{i-1},\tau_i)$, $v'<0$ on $(c_{i-1},\tau_i]$, and $v<0$ on $(\tau_i,c_i]$ (for $i=1$, use $(0,\tau_1)$ in place of $[c_0,\tau_1)$), in particular $u>1$ on $(c_{i-1},\tau_i]$, so $u_\tau>1$.
    \item[(D2)] $M>0$ on $(c_{i-1},c_i]$.
    \item[(D3)] $Q>0$ and $Q_{2\varphi'}>0$ on $(c_{i-1},\max\{b_i,\tau_i\}]$. Also $Q_{\varphi'}>0$ and $Q_\nu>0$ on $(\tau_i,c_i)$, $Q$ is positive and increasing on $[\bar b_i,c_i]$, and $Q(\bar b_i)>Q(b_i)$.\footnote{For $i=1$ these are contained in the last paragraph of Theorem~\ref{thm:phase}. The two not written out there follow in one line each.
          \emph{$Q_{2\varphi'}$.} On $(0,\tau_1]$ one has $Q_{2\varphi'}'=2e^G(\varphi'u'v'+\varphi''u'v)>0$ with $Q_{2\varphi'}(0^+)=0$. On $(\tau_1,b_1]$, if nonempty, $u'v>0$ and $Q>0$ (Theorem~\ref{thm:first-phase}(ii)) give $Q_{2\varphi'}>Q>0$.
          \emph{$Q_{\varphi'}$ and $Q_\nu$.} On $(\tau_1,c_1)$ we have $u'v>0$. On $(\tau_1,\bar b_1]$, $Q_{\varphi'}>0$ (stated there) gives $Q_\nu=Q_{\varphi'}+g\varphi e^Gu'v>0$. On $[\bar b_1,c_1)$, $Q>0$ gives $Q_\nu>Q_{\varphi'}>Q>0$.}
    \item[(D4)] $T_2$ is strictly increasing on $[c_{i-1},b_i]$ if $i\ge2$, while $T_2>0$ on $(0,z_1)$ and $T_2(0^+)=0$ if $i=1$ (Theorem~\ref{thm:first-phase}(ii) and its proof).
\end{enumerate}
Of (D3), only the positivity statements and $Q(\bar b_i)>Q(b_i)$ are used below, and of (D4) only $T_2(0^+)=0$.
Moreover, by Proposition~\ref{prop:pohozaev-positive}:
\begin{enumerate}
    \item[(D5)] $P>0$ on $(0,c_K]$, and $\omega$ is strictly increasing on every subinterval of $(0,c_K]$ on which $u\neq0$.
\end{enumerate}
Two further facts on the descent follow from these.
\begin{enumerate}
    \item[(D6)] The Wronskian $\Sigma=e^G(u''v'-u'v'')=e^G\big(f'(u)u'v-f(u)v'\big)$ of Section~\ref{sec:identities} is positive on $(c_{i-1},\tau_i]$. Equivalently, since $f(u)>0$ on $(c_{i-1},\tau_i]$,
          \[
              \frac{|v'|}{v}>\frac{f'(u)\,|u'|}{f(u)}\qquad\text{on }(c_{i-1},\tau_i).
          \]
    \item[(D7)] $\omega(b_i)>\gamma-1$, hence $\omega>\gamma-1$ on $[b_i,z_i)$.
\end{enumerate}
\emph{Proof of (D6)--(D7).} (D6): by \eqref{eq:Sigma-prime}, $\Sigma'=e^Gf''(u)u'^2v>0$ on $(c_{i-1},\tau_i)$, because $u>1$ and $v>0$. Also $\Sigma(0^+)=0$ if $i=1$ (as in the proof of Theorem~\ref{thm:first-phase}(i)), and $\Sigma(c_{i-1})=-e^{G(c_{i-1})}f(\beta_{i-1})v'(c_{i-1})>0$ if $i\ge2$, because $M(c_{i-1})=-e^{G(c_{i-1})}\beta_{i-1}v'(c_{i-1})>0$.

(D7): since $F(\alpha_*)=0$, $P=e^G|u'|\big(\varphi|u'|-(\gamma-1)u\big)=e^Gu|u'|\,(\omega-\gamma+1)$ at $b_i$, and $P(b_i)>0$ by (D5). Then use that $\omega$ is increasing. \qed

\begin{lemma}[Identities]\label{lem:renewal-identities}
    With $A=\varphi e^G$, the following hold.
    \begin{enumerate}
        \item[(i)] $P_2=P-\dfrac{A}{\nu}h_r(u)$ and $\Big(\dfrac PA\Big)'=-\dfrac{\nu}{\varphi}\,\dfrac{P_2}{A}$.
        \item[(ii)] With $k_r(u)\defeq \partial_uh_r(u)=u\big[(4\varphi'-(p-1)(\gamma-1))|u|^{p-1}-4\varphi'\big]$,
              \begin{equation}\label{eq:P-two-prime}
                  P_2'=-\frac{A}{\nu}\,u'\,k_r(u)+\frac{2(\gamma-1)A\,|\varphi''|}{\nu^2}\,uf(u).
              \end{equation}
        \item[(iii)] Wherever $u'\neq0$,
              \begin{equation}\label{eq:H-ratio}
                  \Big(e^G\varphi'\frac{v}{u'}\Big)'=\big(e^G\varphi'\big)'\frac{v}{u'}+\frac{\varphi'Q_{\varphi'}}{\varphi u'^2}.
              \end{equation}
        \item[(iv)] For every $\mu>0$,
              \begin{equation}\label{eq:level-energy}
                  \frac{d}{dr}\Big[e^{2G}\big(E-F(\mu)\big)\Big]=2g\,e^{2G}\big(F(u)-F(\mu)\big).
              \end{equation}
        \item[(v)] $2g\varphi=\gamma-1+\nu$. Wherever $u'\neq0$ and $v\neq0$, the function $X\defeq -e^Gv/u'$ satisfies
              \begin{equation}\label{eq:X-log}
                  \frac{X'}{X}=2g+\frac{v'}{v}+\frac{f(u)}{u'} .
              \end{equation}
        \item[(vi)] \textup{(Frozen bridge.)} For a constant $\Lambda$ put $\mathfrak V_\Lambda\defeq Q-\Lambda(M-\theta)-2\varphi'F(u)\,e^Gv/u'$. Wherever $u'\neq0$,
              \begin{align}
                  \mathfrak V_\Lambda'
                   & =\big(\Lambda h_1(u)-2\varphi'F(u)\big)\frac{Q_\nu}{\varphi u'^2}-2\varphi''F(u)\,\frac{e^Gv}{u'}
                  \label{eq:frozen-bridge-prime}                                                                       \\
                   & =\big(\Lambda h_1(u)-2\varphi'F(u)\big)\frac{Q_{\varphi'}}{\varphi u'^2}
                  +\Big(g\Lambda h_1(u)-2\big(g\varphi'+\varphi''\big)F(u)\Big)\frac{e^Gv}{u'}.
                  \label{eq:frozen-bridge-prime-alt}
              \end{align}
    \end{enumerate}
\end{lemma}

\begin{proof}
    (i) The first identity is immediate from the definitions, since $h_r=2\nu F-(\gamma-1)uf$, and the second is \eqref{eq:id-P-over-A}.

    (ii) Differentiate $P_2=P-2AF+\frac{(\gamma-1)A}{\nu}uf$ using $P'=e^Gh_r(u)$, $A'=\nu e^G$ and $\nu'=2\varphi''\le0$.

    (iii) This is the product rule together with the first identity of~\eqref{eq:ratio-primes}.

    (iv) Use $E'=-gu'^2$.

    (v) The first identity is $\nu=2\varphi'+\gamma-1$ with $\varphi'=g\varphi-(\gamma-1)$. For the second, $(\log X)'=g+v'/v-u''/u'$ and $u''=-gu'-f(u)$.

    (vi) Write $Y\defeq e^Gv/u'$, so that $\theta=h_1(u)Y$, $u'Y=e^Gv$ and, by \eqref{eq:ratio-primes}, $Y'=Q_\nu/(\varphi u'^2)$. By \eqref{eq:M-prime}, \eqref{eq:Q-prime} and \eqref{eq:M-theta-prime},
    \[
        \mathfrak V_\Lambda'=2\varphi'f(u)u'Y+\Lambda h_1(u)Y'-2\varphi''F(u)Y
        -2\varphi'f(u)u'Y-2\varphi'F(u)Y',
    \]
    which is \eqref{eq:frozen-bridge-prime}. For \eqref{eq:frozen-bridge-prime-alt}, insert $Q_\nu=Q_{\varphi'}+g\varphi e^Gu'v=Q_{\varphi'}+g\varphi u'^2Y$.
\end{proof}

\subsection{Positivity of \texorpdfstring{$P_2$}{P\_2}}

\begin{lemma}[Positivity of $P_2$]\label{lem:P-two-positive}
    Let $u$ be a nodal solution with $k\ge1$ zeros or a bound state with $k\ge1$ zeros, and let $K$ be as in Section~\ref{sec:later-phases}. Then $P_2>0$ on $(0,c_K]$. Consequently $P/A$ is strictly decreasing on $(0,c_K]$.
\end{lemma}

\begin{proof}
    First, $P_2(0^+)=0$: the brackets in $P_2$ are bounded as $r\downarrow0$, while $A$ and $e^G$ tend to $0$.

    \emph{(a) Ascents.} Let $r\in[z_j,c_j]$ with $1\le j\le K$, so that $uu'\ge0$ and $E(r)>0$.
    \begin{itemize}
        \item If $u(r)=0$, then $P_2=Au'^2>0$.
        \item If $u'(r)=0$, then $r=c_j$ and $|u|>\alpha_*>1$ by Proposition~\ref{prop:critical-points}, so $P_2=\frac{(\gamma-1)A}{\nu}uf>0$.
        \item If $|u|\ge1$ and $u'\neq0$, all three terms of $P_2$ are nonnegative and $Au'^2>0$.
        \item If $0<|u|<1$, then $E>0$ gives $u'^2>-2F(u)$. Hence
              \[
                  u'^2+\frac{\gamma-1}{\nu}uf>-2F+\frac{\gamma-1}{\nu}uf
                  =u^2\Big[\Big(1-\frac{\gamma-1}{\nu}\Big)
                  +|u|^{p-1}\Big(\frac{\gamma-1}{\nu}-\frac{2}{p+1}\Big)\Big].
              \]
              The right-hand side is affine in $|u|^{p-1}\in[0,1]$. Its values at the endpoints are $1-\frac{\gamma-1}{\nu}>0$ (Lemma~\ref{lem:phi-basic}) and $1-\frac2{p+1}>0$, so it is positive.
              Since also $(\gamma-1)e^Guu'\ge0$, $P_2>0$.
    \end{itemize}

    \emph{(b) Descents.} Let $1\le j\le K$. On $(c_{j-1},z_j]$ (or $(0,z_1]$ if $j=1$) $|u|$ is strictly decreasing. By Lemma~\ref{lem:level-structure}(i),(iii), the set where $h_r(u(r))>0$, i.e.\ where $|u|>\alpha^*(r)$, is a possibly empty initial interval $(c_{j-1},\rho)$.

    On $(c_{j-1},\rho)$, $\alpha^*(r)$ is finite, so $4\varphi'>(p-1)(\gamma-1)$ and
    \[
        \big(4\varphi'-(p-1)(\gamma-1)\big)|u|^{p-1}
        >\big(4\varphi'-(p-1)(\gamma-1)\big)\alpha^*(r)^{p-1}
        =2(p+1)\varphi'>4\varphi'.
    \]
    Thus $k_r(u)/u>0$. Since $uu'<0$, this gives $-u'k_r(u)>0$, and $uf(u)>0$ because $|u|>1$. By \eqref{eq:P-two-prime}, $P_2'>0$ on $(c_{j-1},\rho)$. Moreover $P_2(c_{j-1})>0$ by (a) if $j\ge2$, and $P_2(0^+)=0$ if $j=1$. Hence $P_2>0$ on $(c_{j-1},\rho]$.

    On $[\rho,z_j]$ we have $h_r(u)\le0$, so $P_2=P-\frac A\nu h_r\ge P>0$ by Lemma~\ref{lem:renewal-identities}(i) and Proposition~\ref{prop:pohozaev-positive}.

    Together, (a) and (b) give $P_2>0$ on $(0,c_K]$, and the last assertion follows from Lemma~\ref{lem:renewal-identities}(i).
\end{proof}

\subsection{A weighted energy inequality and the comparison lemma}

\begin{lemma}[Weighted energy at paired levels]\label{lem:paired-energy}
    For every $\mu\in[\alpha_*,\beta_i)$,
    \[
        |u'(\bar r_\mu)|<|u'(r_\mu)|
        \qquad\text{and}\qquad
        e^{G(\bar r_\mu)}|u'(\bar r_\mu)|<e^{G(r_\mu)}|u'(r_\mu)|.
    \]
\end{lemma}

\begin{proof}
    The first inequality holds because $E$ is strictly decreasing and $F(u)$ takes the same value at $r_\mu$ and $\bar r_\mu$.

    For the second, on $(r_\mu,\bar r_\mu)$ we have $|u|<\mu$, which gives $F(u)<F(\mu)$:
    \begin{itemize}
        \item For $|u|\le1$, $F(u)\le0\le F(\mu)$, with equality only if $u=0$ and $\mu=\alpha_*$, which happens at most once, at $z_i$.
        \item For $1<|u|<\mu$, $F$ is strictly increasing on $[1,\infty)$.
    \end{itemize}
    As $g>0$ by (G1), \eqref{eq:level-energy} shows that $e^{2G}(E-F(\mu))$ is strictly decreasing on $[r_\mu,\bar r_\mu]$. At both endpoints it equals $\frac12e^{2G}u'^2$.
\end{proof}

\begin{lemma}[Comparison of the levels]\label{lem:level-comparison}
    Let $\varpi\defeq Q/(e^G|u'|)$ on $(c_{i-1},c_i)$. Then
    \[
        \varpi(\bar r_\mu)>\varpi(r_\mu)>0\qquad\text{for every }\mu\in[\alpha_*,\beta_i).
    \]
\end{lemma}

\begin{proof}
    For $\mu\ge\alpha_*$ we have $r_\mu\in(c_{i-1},b_i]$, so $\varpi(r_\mu)>0$ by (D3). Put $\varpi_2\defeq Q_{2\varphi'}/(e^G|u'|^3)$.

    By the third identity of~\eqref{eq:ratio-primes}, $\big(Q/(e^Gu')\big)'=f(u)Q_{2\varphi'}/(e^Gu'^2)$. Using $\varpi=-Q/(e^Gu')$, $d\mu=u'\,dr$ on the descent, $d\mu=-u'\,dr$ on the ascent, and the oddness of $f$, we obtain
    \begin{equation}\label{eq:varpi-levels}
        \frac{d}{d\mu}\varpi(r_\mu)=f(\mu)\,\varpi_2(r_\mu),\qquad
        \frac{d}{d\mu}\varpi(\bar r_\mu)=f(\mu)\,\varpi_2(\bar r_\mu).
    \end{equation}
    Let $\Delta(\mu)\defeq \varpi(\bar r_\mu)-\varpi(r_\mu)$, which is a $C^1$ function on $[\alpha_*,\beta_i)$. At $\mu=\alpha_*$, (D3) gives $Q(\bar b_i)>Q(b_i)>0$, and Lemma~\ref{lem:paired-energy} gives $e^{G(\bar b_i)}|u'(\bar b_i)|<e^{G(b_i)}|u'(b_i)|$. Hence $\Delta(\alpha_*)>0$.

    \emph{Claim: if $\Delta(\mu)\ge0$, then $\varpi_2(\bar r_\mu)>\varpi_2(r_\mu)$.} Write $r=r_\mu$, $\bar r=\bar r_\mu$. At $\bar r$ we have $u'v>0$, so $Q_{2\varphi'}=Q+2\varphi'e^Gu'v$ gives
    \[
        \varpi_2(\bar r)=\frac{\varpi}{u'^2}\Big|_{\bar r}+\frac{2\varphi'|v|}{u'^2}\Big|_{\bar r}
        >\frac{\varpi}{u'^2}\Big|_{\bar r}.
    \]
    Since $\Delta(\mu)\ge0$, $\varpi(r)>0$ and $|u'(\bar r)|<|u'(r)|$ (Lemma~\ref{lem:paired-energy}),
    \[
        \frac{\varpi}{u'^2}\Big|_{\bar r}\ge\frac{\varpi(r)}{u'(\bar r)^2}>\frac{\varpi}{u'^2}\Big|_{r}.
    \]
    \begin{itemize}
        \item \emph{Case $r\le\tau_i$.} Then $u'v\le0$ at $r$, so $Q_{2\varphi'}(r)\le Q(r)$ and $\varpi_2(r)\le\varpi(r)/u'(r)^2$. The claim follows from the two displays above.
        \item \emph{Case $r>\tau_i$.} Then $u'v>0$ at $r$ as well, and $\varpi_2(r)=\varpi(r)/u'(r)^2+2\varphi'(r)|v(r)|/u'(r)^2$. It remains to compare the last terms. The function $\Xi\defeq e^G\varphi'v/u'$ is positive on $(\tau_i,c_i)$, and by \eqref{eq:H-ratio} it is strictly increasing: $(e^G\varphi')'\ge0$ by (H), $v/u'>0$, and $Q_{\varphi'}>0$ by (D3).
              Since $\varphi'|v|/u'^2=\Xi/(e^G|u'|)$, Lemma~\ref{lem:paired-energy} gives that $\varphi'|v|/u'^2$ is larger at $\bar r$ than at $r$.
    \end{itemize}
    This proves the claim.

    Now suppose $\Delta$ vanishes somewhere in $(\alpha_*,\beta_i)$, and let $\mu_0$ be its smallest zero. Then $\Delta>0$ on $[\alpha_*,\mu_0)$, so $\Delta'(\mu_0)\le0$. On the other hand, by \eqref{eq:varpi-levels} and the claim, $\Delta'(\mu_0)=f(\mu_0)\,[\varpi_2(\bar r_{\mu_0})-\varpi_2(r_{\mu_0})]>0$, because $f>0$ on $(1,\infty)$. This contradiction shows $\Delta>0$ on $[\alpha_*,\beta_i)$.
\end{proof}

\subsection{The negative part of \texorpdfstring{$Q_\nu$}{Q\_nu} (used only for \texorpdfstring{$\gamma<\frac85$}{gamma<8/5})}

\begin{lemma}[cf.~{\cite[Lemma~5.3]{Tang2026}}]\label{lem:Q-nu-positive}
    If $\nu(b_i)\le(p+1)(\gamma-1)$, and in particular if $p\ge\frac2{\gamma-1}$, then $Q_\nu>0$ on $[b_i,c_i]$.
\end{lemma}

\begin{proof}
    Two facts cover most of the interval. First, $Q_\nu>0$ on $(\tau_i,c_i)$ by (D3). Second, $Q_\nu=Q$ wherever $u'v=0$, and $Q=e^G\varphi u'v'>0$ at $\tau_i$ and $Q=e^G\varphi f(u)v>0$ at $c_i$ by (D1). If $b_i\ge\tau_i$ there is nothing left to prove.

    Let $b_i<\tau_i$ and suppose the claim fails. Then $\{r\in[b_i,\tau_i]:Q_\nu(r)\le0\}$ is closed, nonempty and does not contain $\tau_i$. Its maximum $s$ satisfies $s\in[b_i,\tau_i)$, $Q_\nu(s)=0$, and $Q_\nu>0$ on $(s,\tau_i]$. At $s$ we have $1<u\le\alpha_*$ (since $b_i\le s<\tau_i<r_i$), $u'<0$, $v>0$ and $v'<0$ by (D1).

    \emph{Step 1: an upper bound for $\omega(s)$.} The equation $Q_\nu(s)=0$ reads
    \[
        \varphi(u'v'+fv)=(\gamma-1+\nu)|u'|v .
    \]
    Since $u'v'>0$ and $fv>0$, this gives $(\gamma-1+\nu)v>\varphi|v'|$. With (D6) and $f'>0$,
    \[
        f|v'|>\frac{f'\varphi|u'||v'|}{\gamma-1+\nu},
        \qquad\text{that is,}\qquad
        \frac{(\gamma-1+\nu)f}{uf'}>\omega\quad\text{at }s.
    \]

    \emph{Step 2: contradiction.} By (D7), $\omega(s)>\gamma-1$, so $(\gamma-1+\nu)f>(\gamma-1)uf'$ at $s$. With $x\defeq u(s)^{p-1}\in\big(1,\frac{p+1}2\big]$, this becomes $(\gamma-1+\nu)(x-1)>(\gamma-1)(px-1)$, i.e.
    \[
        \big[\nu-(p-1)(\gamma-1)\big]\,x>\nu .
    \]
    This forces $\nu>(p-1)(\gamma-1)$ and $\frac{p+1}{2}\ge x>\frac{\nu}{\nu-(p-1)(\gamma-1)}$. The last pair of inequalities implies $\nu>(p+1)(\gamma-1)$, since
    \[
        \frac{p+1}{2}\big(\nu-(p-1)(\gamma-1)\big)-\nu=\frac{p-1}{2}\big(\nu-(p+1)(\gamma-1)\big).
    \]
    But $\nu(s)\le\nu(b_i)\le(p+1)(\gamma-1)$, since $\nu$ is nonincreasing (Lemma~\ref{lem:phi-basic}), a contradiction. The last assertion of the lemma is \eqref{eq:nu-bound}.
\end{proof}

\subsection{The river descent: a warped analogue of {\cite[Lemma~5.4]{Tang2026}}}

In this subsection $\tau_i>b_i$: the zero of $v$ on the descent lies below the level $\alpha_*$. On $[b_i,\tau_i)$ we have $1<u\le\alpha_*$, $u'<0$, $v>0$ and $v'<0$ by (D1), and we put
\[
    X\defeq \frac{e^Gv}{|u'|}=-\frac{e^Gv}{u'}>0\quad\text{on }[b_i,\tau_i),
    \qquad X(\tau_i)=0 .
\]
Put $Y\defeq e^Gv/u'=-X$. By \eqref{eq:ratio-primes}, $Y'=Q_\nu/(\varphi u'^2)$. Also $Y(\tau_i)=0$ and $u'Y=e^Gv$. Hence, for $w\in C^1([1,\alpha_*])$, integration by parts gives
\begin{equation}\label{eq:river-parts}
    \int_{b_i}^{\tau_i}w(u)\,\frac{Q_\nu}{\varphi u'^2}\,dr
    =w(\alpha_*)\,X(b_i)-\int_{b_i}^{\tau_i}w'(u)\,e^Gv\,dr .
\end{equation}
This identity uses only $\tau_i>b_i$, and it is used in Lemmas~\ref{lem:river-integral} and~\ref{lem:frozen-bridge}.

\begin{lemma}[Growth of $X$ on the river descent]\label{lem:X-growth}
    Suppose $\tau_i>b_i$, and put $m_\gamma\defeq \frac{2\gamma}{\gamma-1}$. Then $X\,u^{m_\gamma}/f(u)$ is strictly decreasing on $[b_i,\tau_i)$. Consequently, for $b_i<r<\tau_i$,
    \[
        X(r)\ <\ \Big(\frac{\alpha_*}{u}\Big)^{m_\gamma}\frac{f(u)}{f(\alpha_*)}\Big|_r\,X(b_i)
        \ \le\
        \begin{cases}
            C\,X(b_i),\quad C\defeq \frac{6}{13}\,e^{7/6},        & \text{if }\gamma\ge\frac85, \\[1ex]
            \frac{2\sqrt e}{3}\,X(b_i)\ <\ \frac{11}{10}\,X(b_i), & \text{if }\gamma\ge2.
        \end{cases}
    \]
\end{lemma}

\begin{proof}
    On $(b_i,\tau_i)$ we estimate the three terms of \eqref{eq:X-log}:
    \begin{itemize}
        \item $v'/v<-f'(u)|u'|/f(u)$, by (D6).
        \item $f(u)/u'<0$, since $u>1$ and $u'<0$.
        \item $2g=(\gamma-1+\nu)/\varphi=(\gamma-1+\nu)|u'|/(\omega u)$, by Lemma~\ref{lem:renewal-identities}(v) and $\omega=\varphi|u'|/u$. Here $\nu\le\gamma+1$ (Lemma~\ref{lem:phi-basic}) and $\omega>\gamma-1$ (D7), so $2g<m_\gamma|u'|/u$.
    \end{itemize}
    Hence
    \[
        \frac{X'}{X}<\Big(\frac{m_\gamma}{u}-\frac{f'(u)}{f(u)}\Big)|u'|
        =-\frac{d}{dr}\log\frac{u^{m_\gamma}}{f(u)},
    \]
    which is the first assertion, and the first inequality of the display follows because $u(b_i)=\alpha_*$.

    It remains to bound the scalar factor for $1<u\le\alpha_*$. Put $t\defeq p-1$, $s\defeq u^t-1\in(0,t/2]$ and $\sigma\defeq s/t\in(0,\frac12]$. Since $\alpha_*^t=1+\frac t2$ and $f(u)=u(u^t-1)$,
    \[
        \Big(\frac{\alpha_*}{u}\Big)^{m_\gamma}\frac{f(u)}{f(\alpha_*)}
        =\Big(\frac{\alpha_*}{u}\Big)^{m_\gamma-1}\frac{s}{t/2}
        =\Big(\frac{1+t/2}{1+s}\Big)^{(m_\gamma-1)/t}\,2\sigma.
    \]
    From $\log(1+x)\le x$ with $x=\frac{t/2-s}{1+s}\ge0$,
    \[
        \log\frac{1+t/2}{1+s}\le\frac t2-s,
        \qquad\text{hence}\qquad
        \Big(\frac{\alpha_*}{u}\Big)^{m_\gamma}\frac{f(u)}{f(\alpha_*)}\le2\sigma\,e^{(m_\gamma-1)(\frac12-\sigma)} .
    \]
    As $\frac12-\sigma\ge0$, the right-hand side is nondecreasing in $m_\gamma$. For $\xi>0$ the maximum of $\sigma\mapsto2\sigma e^{\xi(\frac12-\sigma)}$ over $\sigma>0$ is $\frac2\xi e^{\frac\xi2-1}$, attained at $\sigma=\frac1\xi$. If $\gamma\ge\frac85$, then $m_\gamma-1=\frac{\gamma+1}{\gamma-1}\le\frac{13}3$, and the choice $\xi=\frac{13}3$ gives the bound $\frac6{13}e^{7/6}=C$. If $\gamma\ge2$, then $m_\gamma-1\le3$, and $\xi=3$ gives $\frac23e^{1/2}$. Finally $\frac23e^{1/2}<\frac{11}{10}$ because $e<\big(\frac{33}{20}\big)^2=2.7225$.
\end{proof}

\begin{lemma}[A warped analogue of {\cite[Lemma~5.4]{Tang2026}}]\label{lem:river-integral}
    Suppose $\gamma\ge\frac85$ and $\tau_i>b_i$. For $\tilde u\in[\alpha_*,\infty]$ put
    \[
        I(\tilde u)\defeq \int_{b_i}^{\tau_i}u^2\Big(1-\Big(\frac{u}{\tilde u}\Big)^{p-1}\Big)\frac{Q_\nu}{\varphi u'^2}\,dr,
        \qquad\text{with }(u/\infty)^{p-1}\defeq 0 .
    \]
    Then $I(\tilde u)>0$ for every $\tilde u\in[\alpha_*,\infty]$. Moreover, if $\gamma\ge2$,
    \[
        I(\infty)=\int_{b_i}^{\tau_i}\frac{\varphi}{\omega^2}\,Q_\nu\,dr
        \ >\ \frac45\,\frac{e^Gv}{|u'|}\Big|_{b_i} .
    \]
    Only $I(\infty)\ge0$ is used below (in Lemma~\ref{lem:frozen-bridge}); the other conclusions are stated for comparison with \cite[Lemma~5.4]{Tang2026}.
\end{lemma}

\begin{proof}
    The integrand is continuous on $[b_i,\tau_i]$, and \eqref{eq:river-parts} applies.

    \emph{The endpoint $\tilde u=\alpha_*$.} Since $\alpha_*^{1-p}=\frac2{p+1}$, the weight is $w=u^2-\frac2{p+1}u^{p+1}=-2F(u)$. It satisfies $w(\alpha_*)=0$ and $w'=-2f$, so $I(\alpha_*)=2\int_{b_i}^{\tau_i}f(u)e^Gv\,dr>0$.

    \emph{The endpoint $\tilde u=\infty$.} Here $w=u^2$, and the first expression for $I(\infty)$ follows from $u^2/(\varphi u'^2)=\varphi/\omega^2$. Substituting $\mu=u(r)$ into the last integral of \eqref{eq:river-parts}, and using $ue^Gv=uX|u'|$,
    \[
        I(\infty)=\alpha_*^2X(b_i)-2\int_{u_\tau}^{\alpha_*}\mu\,X(r_\mu)\,d\mu .
    \]
    Let $C_\gamma$ be the constant of Lemma~\ref{lem:X-growth}: $C_\gamma=\frac{11}{10}$ if $\gamma\ge2$, and $C_\gamma=C$ if $\frac85\le\gamma<2$. Since $X(b_i)>0$, $u_\tau<\alpha_*$ and $u_\tau>1$ (D1),
    \[
        2\int_{u_\tau}^{\alpha_*}\mu\,X(r_\mu)\,d\mu
        <C_\gamma\big(\alpha_*^2-u_\tau^2\big)X(b_i)
        <C_\gamma\big(\alpha_*^2-1\big)X(b_i).
    \]
    Hence $I(\infty)>\big(C_\gamma-(C_\gamma-1)\alpha_*^2\big)X(b_i)$. Here $\alpha_*^2=(1+x)^{1/x}<e$, with $x=\frac{p-1}2$.
    \begin{itemize}
        \item If $\gamma\ge2$: $I(\infty)>\frac{11-\alpha_*^2}{10}\,X(b_i)>\frac45\,X(b_i)$, since $e<3$.
        \item If $\frac85\le\gamma<2$: since $C>1$ and $\alpha_*^2<e$, $I(\infty)>\big(C-(C-1)e\big)X(b_i)$, and $C-(C-1)e\ge0$ is equivalent to $C\le\frac e{e-1}$, that is, to $6(e-1)e^{1/6}\le13$. This holds because $e-1<1.72$ and $e^{1/6}<1.182$ (as $1.182^6>2.72>e$), so the left side is less than $6\cdot1.72\cdot1.182<12.2$. Also $C>1$, as $e^{7/6}>e>\frac{13}6$. Hence $I(\infty)>0$.
    \end{itemize}

    \emph{Generic $\tilde u$.} Write $y\defeq \tilde u^{1-p}\in[0,\frac2{p+1}]$. Then
    \[
        I(\tilde u)=I(\infty)-y\int_{b_i}^{\tau_i}u^{p+1}\frac{Q_\nu}{\varphi u'^2}\,dr
    \]
    is affine in $y$. It is positive at $y=0$ (that is, $\tilde u=\infty$) and at $y=\frac2{p+1}$ (that is, $\tilde u=\alpha_*$), hence on the whole interval.
\end{proof}

\begin{remark}[On Lemmas~\ref{lem:X-growth}--\ref{lem:river-integral}]\label{rem:river-integral}
    \hfill
    \begin{enumerate}
        \item[(a)] \emph{The Euclidean case.} For $g=2/r$ one has $e^G=r^2$, $\varphi=r$, $\gamma=2$, $\nu=3$ and $Q_\nu=Q_3$, and $I(\tilde u)$ is exactly the integral $I_i$ of~\cite[Lemma~5.4]{Tang2026}. The proof above therefore gives a short proof of that lemma, for every $1<p<5$ (not only $1<p<2$). Besides the scalar bound $\frac23\sqrt e<\frac{11}{10}$, it uses only the following facts, all available in~\cite{Tang2026}:
              \begin{itemize}
                  \item \cite[Lemma~5.2]{Tang2026}: $\tau_i<r_i$, so that $u_\tau>1$, and the signs (5.6) there on $(c_{i-1},\tau_i)$.
                  \item the Wronskian inequality~\cite[(5.11)]{Tang2026}, on all of $(b_i,\tau_i)$. It holds because the Wronskian ($\Sigma$ here, $\rho$ in~\cite{Tang2026}) is nonnegative at $c_{i-1}$ and increasing on $(c_{i-1},\tau_i)$.
                  \item $\omega(b_i)>1$~\cite[(5.12)]{Tang2026} and the monotonicity of $\omega$.
                  \item the identity $(r^2v/u')'=Q_3/(ru'^2)$, which is the case $n=3$ of~\cite[(A.20)]{Tang2026} and the case $g=2/r$ of \eqref{eq:ratio-primes}.
              \end{itemize}
              It replaces the steps S1--S8 there. This is possible because of the affine dependence on $\tilde u^{1-p}$: the endpoint $\tilde u=\alpha_*$ corresponds to the bridge $\sB_0$ (on $\R^n$, $I(\alpha_*)=\sB_0(\tau_i)-\sB_0(b_i)$) and is trivially positive, while at the endpoint $\tilde u=\infty$ the boundary term $\alpha_*^2X(b_i)$ gives a large margin.
        \item[(b)] \emph{Where $\gamma\ge\frac85$ is used.} Only through $m_\gamma-1=\frac{\gamma+1}{\gamma-1}\le\frac{13}3$ in Lemma~\ref{lem:X-growth} and the resulting inequality $C\le\frac e{e-1}$. The value $\frac85$ is not optimal for this argument, but it is chosen so that the numerical verification is elementary. Since $\xi\mapsto\frac2\xi e^{\frac\xi2-1}$ is increasing for $\xi>2$, the same proof works for every $\gamma\ge\gamma_*$, where $\gamma_*\approx1.560$ is defined by $\frac2{\xi_*}e^{\frac{\xi_*}2-1}=\frac e{e-1}$ with $\xi_*=\frac{\gamma_*+1}{\gamma_*-1}\approx4.57$. Below $\gamma_*$ the argument fails for $p$ close to $1$, where the estimates $\log(1+x)\le x$ and $\alpha_*^2<e$ become sharp.
    \end{enumerate}
\end{remark}

\subsection{The zero of \texorpdfstring{$M-\theta$}{M-theta} and the connection bounds}

\begin{lemma}\label{lem:M-theta-zero}
    On $(c_{i-1},c_i)$ let $\theta=e^Gvh_1(u)/u'$ as in Section~\ref{sec:variation}. Then:
    \begin{enumerate}
        \item[(i)] $M-\theta$ is positive on $(c_{i-1},z_i]$, strictly decreasing on $(\tau_i,c_i)$, and tends to $-\infty$ as $r\uparrow c_i$. Its unique zero $\zeta_i$ lies in $(z_i,c_i)$, and $M-\theta>0$ on $(c_{i-1},\zeta_i)$.
        \item[(ii)] $v<uQ/P$ on $(c_{i-1},\tau_i)$, and $|v|>|u|\,Q/P$ on $(z_i,\zeta_i)$.
    \end{enumerate}
\end{lemma}

\begin{proof}
    (i) On $(c_{i-1},\tau_i]$ we have $v\ge0>u'$, so $\theta\le0<M$ by (D2). By \eqref{eq:M-theta-prime}, $(M-\theta)'=-h_1(u)Q_\nu/(\varphi u'^2)$. This is negative on $(\tau_i,c_i)\setminus\{z_i\}$, because $Q_\nu>0$ by (D3) and $h_1(u)>0$ for $u\neq0$. Also $(M-\theta)(z_i)=M(z_i)>0$. As $r\uparrow c_i$ we have $M\to M(c_i)$, $v\to v(c_i)<0$, $u'\to0^-$ and $h_1(u)\to h_1(\beta_i)>0$, so $\theta\to+\infty$. Hence $M-\theta$, which is strictly decreasing on $[\tau_i,c_i)$ and positive at $\tau_i$ and at $z_i$, is positive on $(c_{i-1},z_i]$ and has a unique zero $\zeta_i\in(z_i,c_i)$, with $M-\theta>0$ on $(c_{i-1},\zeta_i)$.

    (ii) Use the connection identity \eqref{eq:connection}, $Q-Pv/u=\omega(M-\theta)$, with $P>0$ from (D5).
    On $(c_{i-1},\tau_i)$, $u>0$ and $\omega>0$, so $Q>Pv/u$. On $(z_i,\zeta_i)$, $u<0$, $u'<0$ and $v<0$, so $\omega<0$ and $Q<Pv/u=P|v|/|u|$.
\end{proof}

\subsection{The pairing at equal levels}

\begin{proposition}[Pairing]\label{prop:pairing}
    Let $\mu\in(\alpha_*,\beta_i)$ with $r_\mu<\tau_i$ and $\bar r_\mu<\zeta_i$. Then
    \begin{equation}\label{eq:pairing}
        \nu(\bar r_\mu)\,\frac{e^{G}|v|}{\varphi u'^2}\Big|_{\bar r_\mu}
        \ >\ (\gamma-1)\,\frac{e^{G}v}{\varphi u'^2}\Big|_{r_\mu}.
    \end{equation}
\end{proposition}

\begin{proof}
    Write $r=r_\mu$ and $\bar r=\bar r_\mu$. Since $\mu>\alpha_*$, we have $r\in(c_{i-1},b_i)$ and $\bar r\in(\bar b_i,\zeta_i)$, hence $Q(r),Q(\bar r)>0$ by (D3). Write $Q=e^G|u'|\varpi$. Then Lemma~\ref{lem:M-theta-zero}(ii) gives
    \[
        \frac{e^G|v|}{\varphi u'^2}\Big|_{\bar r}>\mu\,\frac{e^{2G}\varpi}{\varphi P|u'|}\Big|_{\bar r},
        \qquad
        \frac{e^Gv}{\varphi u'^2}\Big|_{r}<\mu\,\frac{e^{2G}\varpi}{\varphi P|u'|}\Big|_{r}.
    \]
    Factor $\frac{e^{2G}\varpi}{\varphi P|u'|} =\frac{e^{G}}{\varphi^2}\cdot\varpi\cdot\frac{\varphi e^G}{P}\cdot\frac1{|u'|}$. We use three facts, all involving positive quantities:
    \begin{itemize}
        \item $\varpi(\bar r)>\varpi(r)>0$ (Lemma~\ref{lem:level-comparison}),
        \item $\varphi e^G/P$ is larger at $\bar r$ than at $r$ (Lemma~\ref{lem:P-two-positive}, since $r<\bar r\le c_K$),
        \item $1/|u'(\bar r)|>e^{G(\bar r)}/(e^{G(r)}|u'(r)|)$ (Lemma~\ref{lem:paired-energy}).
    \end{itemize}
    With $X_0\defeq \frac{\varphi(r)\varpi(r)}{P(r)|u'(r)|}>0$ they give
    \[
        \frac{e^{2G}\varpi}{\varphi P|u'|}\Big|_{\bar r}>\frac{e^{2G}}{\varphi^2}\Big|_{\bar r}\,X_0,
        \qquad
        \frac{e^{2G}\varpi}{\varphi P|u'|}\Big|_{r}=\frac{e^{2G}}{\varphi^2}\Big|_{r}\,X_0 .
    \]
    Finally, $e^G/\varphi=\big((\gamma-1)\cJ\big)^{-1}$ with $\cJ=\int_r^\infty e^{-G}$ strictly decreasing, so $e^{2G}/\varphi^2$ is strictly increasing. Together with $\nu>\gamma-1$ (Lemma~\ref{lem:phi-basic}) this gives
    \[
        \nu(\bar r)\,\frac{e^{2G}}{\varphi^2}\Big|_{\bar r}
        >(\gamma-1)\,\frac{e^{2G}}{\varphi^2}\Big|_{\bar r}
        >(\gamma-1)\,\frac{e^{2G}}{\varphi^2}\Big|_{r},
    \]
    and \eqref{eq:pairing} follows.
\end{proof}

\subsection{Crossing the river with a frozen weight}

\begin{lemma}[Frozen bridge]\label{lem:frozen-bridge}
    Suppose $\tau_i>b_i$ and
    \[
        I(\infty)=\int_{b_i}^{\tau_i}u^2\,\frac{Q_\nu}{\varphi u'^2}\,dr\ \ge\ 0 .
    \]
    By Lemma~\ref{lem:river-integral} this holds whenever $\gamma\ge\frac85$. Then for every $\Lambda\in\big[0,\frac2{p-1}\varphi'(\bar b_i)\big)$,
    \[
        Q(\bar b_i)-Q(b_i)+\Lambda\Big[(M-\theta)(b_i)-(M-\theta)(\bar b_i)\Big]>0 .
    \]
\end{lemma}

\begin{proof}
    Take $\mathfrak V_\Lambda$ as in Lemma~\ref{lem:renewal-identities}(vi): it is $C^1$ on $[b_i,\bar b_i]$, where $u'\neq0$. Since $F(\pm\alpha_*)=0$, $\mathfrak V_\Lambda=Q-\Lambda(M-\theta)$ at $b_i$ and at $\bar b_i$. So the left side of the claim is $\mathfrak V_\Lambda(\bar b_i)-\mathfrak V_\Lambda(b_i)$, which we split at $\tau_i$.

    \emph{From $\tau_i$ to $\bar b_i$.} On $(\tau_i,\bar b_i)$ every term of \eqref{eq:frozen-bridge-prime-alt} is nonnegative:
    \begin{itemize}
        \item $|u|<\alpha_*$, so $F(u)\le0$ and $\Lambda h_1(u)-2\varphi'F(u)\ge0$.
        \item $Q_{\varphi'}>0$ by (D3).
        \item $g>0$, $h_1\ge0$, $g\varphi'+\varphi''\ge0$ by (H), and $e^Gv/u'>0$ since $v,u'<0$.
    \end{itemize}
    Hence $\mathfrak V_\Lambda(\bar b_i)\ge\mathfrak V_\Lambda(\tau_i)$.

    \emph{From $b_i$ to $\tau_i$.} At $\tau_i$ we have $v=0$, so $\theta(\tau_i)=0$ and $\mathfrak V_\Lambda(\tau_i)=Q(\tau_i)-\Lambda M(\tau_i)$. At $b_i$, $\theta(b_i)=-h_1(\alpha_*)X(b_i)$ with $X$ as above. By \eqref{eq:M-prime}--\eqref{eq:Q-prime},
    \[
        \mathfrak V_\Lambda(\tau_i)-\mathfrak V_\Lambda(b_i)
        =\int_{b_i}^{\tau_i}\big[2\varphi'f(u)-\Lambda(p-1)u^p\big]e^Gv\,dr+\Lambda h_1(\alpha_*)X(b_i).
    \]
    Write $\Lambda=\frac{2\eta}{p-1}$ with $0\le\eta<\varphi'(\bar b_i)$. Then $\Lambda h_1(\alpha_*)=\frac{2\eta}{p+1}\alpha_*^{p+1}=\eta\alpha_*^2$, and $\varphi'f(u)-\eta u^p=(\varphi'-\eta)f(u)-\eta u$. Hence, with \eqref{eq:river-parts} for $w=u^2$,
    \begin{align*}
        \mathfrak V_\Lambda(\tau_i)-\mathfrak V_\Lambda(b_i)
         & =2\int_{b_i}^{\tau_i}(\varphi'-\eta)f(u)e^Gv\,dr
        +\eta\Big[\alpha_*^2X(b_i)-2\int_{b_i}^{\tau_i}u\,e^Gv\,dr\Big]      \\
         & =2\int_{b_i}^{\tau_i}(\varphi'-\eta)f(u)e^Gv\,dr+\eta\,I(\infty).
    \end{align*}
    On $(b_i,\tau_i)$ we have $f(u)>0$ and $v>0$. Also $\varphi'-\eta\ge\varphi'(\bar b_i)-\eta>0$, since $\varphi'$ is nonincreasing. So the first term is positive and the second is nonnegative by hypothesis. Hence $\mathfrak V_\Lambda(\tau_i)>\mathfrak V_\Lambda(b_i)$, and the claim follows.
\end{proof}

\subsection{The renewal}

\begin{theorem}[Renewal of the slope bound]\label{thm:renewal-below}
    In the setting of \S\ref{sec:renewal-setting}, suppose that at least one of the following holds:
    \begin{itemize}
        \item[(a)] $\gamma\ge\frac85$,
        \item[(b)] $\tau_i\le b_i$,
        \item[(c)] $\nu(b_i)\le(p+1)(\gamma-1)$ \textup{(}for instance, $p\ge\frac2{\gamma-1}$\textup{)}.
    \end{itemize}
    Then
    \[
        T_2(c_i)>T_2(c_{i-1}),\qquad\text{where }T_2(c_0)\defeq T_2(0^+)=0 .
    \]
    In particular \eqref{eq:slope-condition} holds at $c_i$.
\end{theorem}

\begin{proof}
    If $u_\tau\ge\beta_i$, this is Theorem~\ref{thm:renewal-above}. Let $u_\tau<\beta_i$. Put
    \[
        r_{\lmu}\defeq \min\{\tau_i,b_i\},\qquad
        \lmu\defeq u(r_{\lmu})=\max\{u_\tau,\alpha_*\}\in[\alpha_*,\beta_i),\qquad
        \bar r_{\lmu}\in[\bar b_i,c_i)\ \text{with } u(\bar r_{\lmu})=-\lmu,
    \]
    and write $u_\zeta\defeq |u(\zeta_i)|\in(0,\beta_i)$. Recall $g_2'(\mu)=(p+1)\mu^{-p}$ for $\mu>0$ (from \eqref{eq:g-relations}), and $g_2\ge0$ on $[\alpha_*,\infty)$.

    \emph{Step 1: descent.} On $(c_{i-1},r_{\lmu}]$, \eqref{eq:T-two-prime} becomes $T_2'=\varphi'(p+1)\frac{|u'|}{u^p}(M-\theta)+|\varphi''|g_2(u)M$ (this is \eqref{eq:T-two-descent}). Here every term is nonnegative: $M-\theta>0$ (Lemma~\ref{lem:M-theta-zero}(i)), $g_2(u)\ge0$ as $u\ge\lmu\ge\alpha_*$, and $M>0$. The substitution $\mu=u(r)$ is valid on $[c_{i-1}+\varepsilon,r_{\lmu}]$ for every small $\varepsilon>0$. Let $\varepsilon\to0$ (monotone convergence), and for $i=1$ use $T_2(0^+)=0$ from (D4). This gives
    \[
        T_2(r_{\lmu})-T_2(c_{i-1})=\int_{\lmu}^{\beta_{i-1}}\mathcal A(\mu)\,d\mu,
        \qquad
        \mathcal A(\mu)\defeq g_2'(\mu)\,(M-\theta)\varphi'\big|_{r_\mu}+g_2(\mu)\,\frac{|\varphi''|M}{|u'|}\Big|_{r_\mu}.
    \]
    Here $\mathcal A\ge0$, and $\mathcal A>0$ on $(\lmu,\beta_{i-1})$.

    \emph{Step 2: across the river.} Put $\mathcal R\defeq T_2(\bar r_{\lmu})-T_2(r_{\lmu})$.
    \begin{itemize}
        \item If $\tau_i>b_i$, then $r_{\lmu}=b_i$, $\bar r_{\lmu}=\bar b_i$ and $g_2(\alpha_*)=0$. So $T_2(b_i)=Q(b_i)$ and $T_2(\bar b_i)=Q(\bar b_i)$, hence $\mathcal R=Q(\bar b_i)-Q(b_i)>0$ by (D3).
        \item If $\tau_i\le b_i$, then $r_{\lmu}=\tau_i$, $\lmu=u_\tau$. Take the bridge $\sB_a$ of \eqref{eq:weighted-bridge} with $a\defeq g_2(u_\tau)\ge0$. By \eqref{eq:bridge-values}, $\sB_a=T_2$ at $\tau_i$ and at $\bar r_{\lmu}$, where $g_2(|u|)=a$. On $(\tau_i,\bar r_{\lmu})$ we have $|u|<u_\tau$, so $H_a(u)=u^2\big(1-(|u|/u_\tau)^{p-1}\big)\ge0$. By the identity \eqref{eq:weighted-bridge-renewal},
              \[
                  \sB_a'=\varphi'H_a\frac{Q_{\varphi'}}{\varphi u'^2}+a|\varphi''|M+(g\varphi'+\varphi'')H_a\frac{e^Gv}{u'},
              \]
              and on $(\tau_i,\bar r_{\lmu})$ every term is nonnegative, by $H_a\ge0$, $a\ge0$, (D3), (D2), (H) and $v/u'>0$. Hence $\sB_a'\ge0$, and $\mathcal R\ge0$.
    \end{itemize}

    \emph{Step 3: ascent.} On $[\bar r_{\lmu},c_i)$, \eqref{eq:T-two-prime} reads $T_2'=-\varphi'(p+1)\frac{|u'|}{|u|^p}(M-\theta)+|\varphi''|g_2(u)M$. For $\mu\in[\lmu,\beta_i)$ put
    \[
        \mathcal C(\mu)\defeq -g_2'(\mu)\,(M-\theta)\varphi'\big|_{\bar r_\mu}+g_2(\mu)\,\frac{|\varphi''|M}{|u'|}\Big|_{\bar r_\mu}.
    \]
    This is continuous on $[\lmu,\beta_i)$ and positive on $(\max\{\lmu,u_\zeta\},\beta_i)$, where $M-\theta<0$. Substituting $\mu=-u(r)$ on $[\bar r_{\lmu},c_i-\varepsilon]$ and letting $\varepsilon\to0$ (monotone convergence near $\beta_i$) gives
    \[
        T_2(c_i)-T_2(\bar r_{\lmu})=\int_{\lmu}^{\beta_i}\mathcal C(\mu)\,d\mu ,
    \]
    which converges.

    \emph{Step 4: level-by-level pairing.} Adding Steps 1--3,
    \begin{equation}\label{eq:renewal-split}
        T_2(c_i)-T_2(c_{i-1})=\int_{\beta_i}^{\beta_{i-1}}\mathcal A\,d\mu
        +\mathcal R+\int_{\lmu}^{\beta_i}\big(\mathcal A+\mathcal C\big)\,d\mu .
    \end{equation}
    The first integral is positive. Dropping the nonnegative $|\varphi''|$-terms,
    \[
        \mathcal A(\mu)+\mathcal C(\mu)\ \ge\ g_2'(\mu)\Big[(M-\theta)\varphi'\big|_{r_\mu}-(M-\theta)\varphi'\big|_{\bar r_\mu}\Big].
    \]
    \begin{itemize}
        \item If $\mu\ge u_\zeta$, the bracket is positive, since $(M-\theta)(r_\mu)>0\ge(M-\theta)(\bar r_\mu)$.
        \item If $\mu<u_\zeta$, then $(M-\theta)(r_\mu)>0$ and $\varphi'(r_\mu)\ge\varphi'(\bar r_\mu)$, since $\varphi'$ is nonincreasing and $r_\mu<\bar r_\mu$. Hence the bracket is at least (the equality being \eqref{eq:M-theta-prime} integrated across the zero $z_i$ of $u$, where $M-\theta$ is $C^1$ because $h_1$ is)
              \[
                  \varphi'(\bar r_\mu)\big[(M-\theta)(r_\mu)-(M-\theta)(\bar r_\mu)\big]
                  =\varphi'(\bar r_\mu)\,\mathcal I(\mu),
                  \qquad
                  \mathcal I(\mu)\defeq \int_{r_\mu}^{\bar r_\mu}\frac{h_1(u)\,Q_\nu}{\varphi u'^2}\,dr .
              \]
    \end{itemize}
    For $\mu\in(\lmu,u_\zeta)$ (an interval that is empty if $u_\zeta\le\lmu$) split $[r_\mu,\bar r_\mu]=[r_\mu,r_{\lmu}]\cup[r_{\lmu},\bar r_{\lmu}]\cup[\bar r_{\lmu},\bar r_\mu]$, so that $\mathcal I(\mu)=\mathcal I_{\rm out}(\mu)+\mathcal I_{\rm mid}$ with
    \[
        \mathcal I_{\rm out}(\mu)\defeq \Big(\int_{r_\mu}^{r_{\lmu}}+\int_{\bar r_{\lmu}}^{\bar r_\mu}\Big)\frac{h_1(u)\,Q_\nu}{\varphi u'^2}\,dr,
        \qquad
        \mathcal I_{\rm mid}\defeq \int_{r_{\lmu}}^{\bar r_{\lmu}}\frac{h_1(u)\,Q_\nu}{\varphi u'^2}\,dr
        =(M-\theta)(r_{\lmu})-(M-\theta)(\bar r_{\lmu}).
    \]
    We claim $\mathcal I_{\rm out}(\mu)>0$. Substitute $s=u$ on $[r_\mu,r_{\lmu}]$ and $s=-u$ on $[\bar r_{\lmu},\bar r_\mu]$:
    \[
        \mathcal I_{\rm out}(\mu)
        =\int_{\lmu}^\mu h_1(s)\Big[\frac{Q_\nu}{\varphi|u'|^3}\Big|_{r_s}+\frac{Q_\nu}{\varphi|u'|^3}\Big|_{\bar r_s}\Big]\,ds .
    \]
    Let $s\in(\lmu,\mu)$. Then $r_s\in(c_{i-1},r_{\lmu})\subset(c_{i-1},\tau_i)$, so $u'v<0$ and $Q_{2\varphi'}>0$ at $r_s$ (D3). Also $\bar r_s\in(\bar r_{\lmu},\bar r_\mu)\subset(\bar b_i,\zeta_i)$, so $Q>0$ and $u'v>0$ at $\bar r_s$. Hence
    \begin{align*}
        Q_\nu & =Q_{2\varphi'}+(\gamma-1)e^Gu'v>-(\gamma-1)e^G|u'|v &  & \text{at }r_s,      \\
        Q_\nu & =Q+\nu e^Gu'v>\nu e^G|u'||v|                        &  & \text{at }\bar r_s.
    \end{align*}
    The bracket is therefore larger than
    \[
        \nu(\bar r_s)\frac{e^G|v|}{\varphi u'^2}\Big|_{\bar r_s}-(\gamma-1)\frac{e^Gv}{\varphi u'^2}\Big|_{r_s},
    \]
    which is positive by Proposition~\ref{prop:pairing}, applied with $s$ in place of $\mu$. Its hypotheses $s\in(\alpha_*,\beta_i)$, $r_s<\tau_i$ and $\bar r_s<\zeta_i$ hold.

    Put $\Lambda\defeq \int_{\lmu}^{\max\{\lmu,u_\zeta\}}g_2'(\mu)\varphi'(\bar r_\mu)\,d\mu\ge0$. Combining the above,
    \[
        \int_{\lmu}^{\beta_i}\big(\mathcal A+\mathcal C\big)\,d\mu
        \ \ge\ \int_{\lmu}^{\max\{\lmu,u_\zeta\}}g_2'(\mu)\varphi'(\bar r_\mu)\big[\mathcal I_{\rm out}(\mu)+\mathcal I_{\rm mid}\big]\,d\mu
        \ \ge\ \Lambda\,\mathcal I_{\rm mid},
    \]
    and \eqref{eq:renewal-split} gives
    \begin{equation}\label{eq:renewal-middle}
        T_2(c_i)-T_2(c_{i-1})\ >\ \mathcal R+\Lambda\,\mathcal I_{\rm mid}.
    \end{equation}

    \emph{Step 5: the middle stretch.} It remains to show $\mathcal R+\Lambda\mathcal I_{\rm mid}\ge0$.
    Recall $\mathcal R\ge0$ and $\Lambda\ge0$.
    \begin{itemize}
        \item \emph{Case $\tau_i\le b_i$.} On $[r_{\lmu},\bar r_{\lmu}]=[\tau_i,\bar r_{\lmu}]$ we have $Q_\nu\ge0$, by (D3) and $Q_\nu(\tau_i)=Q(\tau_i)>0$. Hence $\mathcal I_{\rm mid}\ge0$.
        \item \emph{Case $\tau_i>b_i$ and $\nu(b_i)\le(p+1)(\gamma-1)$.} Here $[r_{\lmu},\bar r_{\lmu}]=[b_i,\bar b_i]$, and $Q_\nu>0$ on $[b_i,\bar b_i]$ by Lemma~\ref{lem:Q-nu-positive}. Hence $\mathcal I_{\rm mid}\ge0$.
        \item \emph{Case $\tau_i>b_i$ and $\gamma\ge\frac85$.} Here $\lmu=\alpha_*$, $\mathcal R=Q(\bar b_i)-Q(b_i)$ and $\mathcal I_{\rm mid}=(M-\theta)(b_i)-(M-\theta)(\bar b_i)$.
              Moreover $\bar r_\mu\ge\bar b_i$ for $\mu\ge\alpha_*$, and $\varphi'$ is nonincreasing, so
              \[
                  \Lambda\le\varphi'(\bar b_i)\big(g_2(\max\{\alpha_*,u_\zeta\})-g_2(\alpha_*)\big)
                  <\frac2{p-1}\varphi'(\bar b_i),
              \]
              since $g_2<\frac2{p-1}$ and $g_2(\alpha_*)=0$. By Lemmas~\ref{lem:river-integral} and~\ref{lem:frozen-bridge}, $\mathcal R+\Lambda\mathcal I_{\rm mid}>0$.
    \end{itemize}
    If (b) holds we are in the first case. If (a) or (c) holds and $\tau_i>b_i$, we are in the second or third case. Hence in all cases \eqref{eq:renewal-middle} gives $T_2(c_i)>T_2(c_{i-1})\ge0$.

    Finally, $Q(c_i)>0$ and $M(c_i)>0$ by Theorem~\ref{thm:phase}(iv), so \eqref{eq:slope-condition} holds at $c_i$.
\end{proof}

\begin{remark}[On the proof]\label{rem:renewal-below}
    \hfill
    \begin{enumerate}
        \item[(a)] \emph{The case $\tau_i\le b_i$.} Here no hypothesis beyond those of Section~\ref{sec:transition} is used. In the Euclidean case the argument is that of~\cite[Lemma~5.6]{Tang2026}, organized by levels, and Proposition~\ref{prop:pairing} is the counterpart of the pairing \cite[(5.41)--(5.45)]{Tang2026}. Of the three facts listed at the beginning of this section, (i) is Lemma~\ref{lem:level-comparison}, where (H) is used only when $\tau_i<b_i$; (ii) is Lemma~\ref{lem:P-two-positive}; and (iii) is not used.
        \item[(b)] \emph{The case $\tau_i>b_i$ (the residual case of~\cite{Tang2026}).} Here the middle stretch $[b_i,\bar b_i]$ may carry a negative part of $Q_\nu$ on the descent. Step~5 does not require $\mathcal I_{\rm mid}\ge0$. It uses the margin $Q(\bar b_i)-Q(b_i)$ of the river crossing. On $(b_i,\tau_i)$ the combined weight in \eqref{eq:frozen-bridge-prime} is
              \[
                  \Lambda h_1-2\varphi'F=\eta\,u^2+(\varphi'-\eta)(-2F),
                  \qquad \Lambda=\tfrac{2\eta}{p-1},
              \]
              a combination with nonnegative coefficients of the weights of the two endpoints $\tilde u=\infty$ and $\tilde u=\alpha_*$ of Lemma~\ref{lem:river-integral}. The second endpoint is trivially positive, which is why only the endpoint $\tilde u=\infty$ has to be proved. In the Euclidean argument~\cite[Lemma~5.6]{Tang2026} the river is crossed by the bridge at the level $u_\zeta$, which exceeds $\alpha_*$ in the case where that bridge is needed. Its weight on $(b_i,\tau_i)$ is $u^2(1-(u/u_\zeta)^{p-1})$, which corresponds to the intermediate value $\tilde u=u_\zeta$ of Lemma~\ref{lem:river-integral}.
        \item[(c)] \emph{Why the common factor $\varphi'(\bar r_\mu)$ matters.} One can also bridge from the descent point $r_{u_\zeta}$ to $\zeta_i$ with $\sB_{g_2(u_\zeta)}$, as in~\cite[Lemma~5.6]{Tang2026}. This leads to the same pairing, but with the local weight $\varphi'$ inside the integrand, i.e.\ with coefficients $(\gamma-1)\varphi'+\varphi\varphi''$ on the descent and $\varphi'\nu+\varphi\varphi''$ on the ascent. Since $\varphi'(r_\mu)\ge\varphi'(\bar r_\mu)$ goes in the opposite direction for the pairing, that route needs one more monotonicity condition. The level integration of $T_2'$ in Step~4 factors out $\varphi'(\bar r_\mu)$ once, which is why no further condition on the geometry is needed here.
        \item[(d)] \emph{Monotonicity along the staircase.} Together with Theorem~\ref{thm:renewal-above}, the theorem gives $0<T_2(c_1)<T_2(c_2)<\dots<T_2(c_K)$ whenever its hypothesis holds at every phase.
        \item[(e)] \emph{An invariant cone.} For $i\ge2$ the proofs of Theorems~\ref{thm:phase}, \ref{thm:renewal-above} and~\ref{thm:renewal-below} use $v$ only through the variation equation and the slope condition at $c_{i-1}$. Hence, if (a) or (c) holds, every solution of the variation equation that satisfies the slope condition at $c_{i-1}$ satisfies it at $c_i$ as well: the cone of initial values with $v(c_{i-1})>0$ and $0<\kappa<\kappa_{\max}$, in the notation following \eqref{eq:slope-condition}, is invariant under the transition from $c_{i-1}$ to $c_i$. An earlier version of this work used this invariance for a computer-assisted proof. The present proof does not need it.
    \end{enumerate}
\end{remark}

\subsection{The main theorem in critical-point form}\label{sec:critical-form}

\begin{theorem}[Uniqueness in critical-point form]\label{thm:critical-form}
    Let $f$ be as in \eqref{eq:model-f} with $p$ as in \eqref{eq:p-range}, let $g$ satisfy (G1)--(G6), $(\Phi)$ and (H), and assume (R) \textup{(}Assumption~\ref{ass:R}\textup{)}. Then:
    \begin{enumerate}
        \item[(i)] At every $k$-node bound state, $k\ge1$: the slope condition \eqref{eq:slope-condition} holds at $c_1,\dots,c_k$, with $0<T_2(c_1)<\dots<T_2(c_k)$; $v$ has exactly one zero $\tau_i$ in each phase $(c_{i-1},c_i]$, located in $(c_{i-1},r_i)$, and exactly one zero $\tau_{k+1}\in(c_k,r_{k+1})$ on the final phase. Moreover, $u'v'>0$ at every $\tau_i$, $v(z_i)u'(z_i)>0$ for $i\le k$, $Q_{\varphi'}>0$ on $[\tau_i,c_i]$ and $Q>0$ on $[\bar b_i,c_i]$. Finally, $v$ is dominant at infinity with the sign opposite to $u$, i.e.\ $M(\infty)>0$.
        \item[(ii)] At every zero $z_j$ of every nodal solution, $v(z_j)u'(z_j)>0$. Consequently $N$ is nondecreasing on $(\alpha^*(0),\infty)$.
    \end{enumerate}
\end{theorem}

\begin{proof}
    Let $u$ be a bound state ($K=k$) or a nodal solution which is not a bound state ($K=k-1$), with $k\ge1$ zeros. By (R), at every phase $i\le K$ hypothesis (a) or hypothesis (c) of Theorem~\ref{thm:renewal-below} holds: (a) if $\gamma\ge\frac85$, and (c) by \eqref{eq:nu-bound} if $p\ge\frac2{\gamma-1}$. We argue by induction on $i=1,\dots,K$, with the slope condition at $c_{i-1}$ as induction hypothesis:
    \begin{itemize}
        \item The slope condition \eqref{eq:slope-condition} holds at $c_{i-1}$ (nothing is required for $i=1$), so Theorem~\ref{thm:phase} (for $i=1$: Theorem~\ref{thm:first-phase}(i)--(ii) and the last paragraph of Theorem~\ref{thm:phase}) gives the structure of phase~$i$.
        \item Theorem~\ref{thm:renewal-below} gives $T_2(c_i)>T_2(c_{i-1})$, so the slope condition holds at $c_i$, which is the induction step.
    \end{itemize}
    Hence the slope condition holds at $c_1,\dots,c_K$ with $0<T_2(c_1)<\dots<T_2(c_K)$, and Theorem~\ref{thm:phase}(i),(iii),(iv) gives, in each phase $i\le K$, exactly one zero $\tau_i$ of $v$, located in $(c_{i-1},r_i)$, with $u'v'(\tau_i)>0$, $v(z_i)u'(z_i)>0$, $Q_{\varphi'}>0$ on $[\tau_i,c_i]$ and $Q>0$ on $[\bar b_i,c_i]$.
    \begin{itemize}
        \item If $u$ is a bound state, Theorem~\ref{thm:phase} with $i=k+1$ gives the final phase: exactly one zero $\tau_{k+1}\in(c_k,r_{k+1})$, with $u'v'(\tau_{k+1})>0$, and, by Theorem~\ref{thm:phase}(v), $v$ is dominant at infinity with the sign opposite to $u$. By Lemma~\ref{lem:tail-sign}, $M(\infty)>0$. This proves (i).
        \item If $u$ is a nodal solution which is not a bound state, Theorem~\ref{thm:phase} with $i=k$, in its descent-only form, gives $v(z_k)u'(z_k)>0$. With the phases $i\le k-1$ this gives $v(z_j)u'(z_j)>0$ at every zero. At bound states this is contained in (i). Lemma~\ref{lem:zero-motion} then gives the monotonicity of $N$. This proves (ii).
    \end{itemize}
    The conclusions of Theorem~\ref{thm:staircase} follow in Section~\ref{sec:main-proofs}.
\end{proof}

\begin{remark}[The Euclidean case, and the complement of (R)]\label{rem:R-complement}
    \hfill
    \begin{enumerate}
        \item[(a)] On $\R^n$, $n\ge3$ ($e^G=r^{n-1}$, $\varphi=r$, $\nu\equiv n$, $\varphi''\equiv0$, $\gamma=n-1\ge2$), (H) and (R) hold (Theorem~\ref{thm:euclidean}). In the renewal step, Theorem~\ref{thm:renewal-above} and the case $\tau_i\le b_i$ of Theorem~\ref{thm:renewal-below} take the place of~\cite[Lemma~5.6]{Tang2026} (Remarks~\ref{rem:renewal-above}(a) and~\ref{rem:renewal-below}(a)), and Lemmas~\ref{lem:river-integral} and~\ref{lem:frozen-bridge} replace~\cite[Lemmas~5.3 and~5.4]{Tang2026} and the last part of the proof of~\cite[Lemma~5.6]{Tang2026} (Remark~\ref{rem:river-integral}).
        \item[(b)] (R) excludes only $1<\gamma<\frac85$ with $p<\frac2{\gamma-1}$. For these $\gamma$ and $p$, Theorem~\ref{thm:renewal-below} still gives $T_2(c_i)>T_2(c_{i-1})$ at every phase with $u_\tau\ge\beta_i$, with $\tau_i\le b_i$, or with $\nu(b_i)\le(p+1)(\gamma-1)$. At a phase with $\tau_i>b_i$ it suffices, by Step~5 of its proof, that $I(\infty)\ge0$ (Lemma~\ref{lem:frozen-bridge}) or that $(M-\theta)(b_i)\ge(M-\theta)(\bar b_i)$.
        \item[(c)] The conclusions of \cite[Lemma~5.1]{Tang2026} and of its proof in \cite[\S5.7]{Tang2026} are those of Theorems~\ref{thm:phase}, \ref{thm:renewal-above} and~\ref{thm:renewal-below}: one zero of $v$ per phase, $Q_{\varphi'}>0$ on the river, $Q>0$ on $[\bar b_i,c_i]$, and the slope condition at the critical points. They do not include the positivity of $Q$ inside the river $(b_i,\bar b_i)$. Since $Q'=2\varphi'e^Gf(u)v$ and $v<0$ after $\tau_i$, $Q$ increases on $(b_i,\tau_i)$ if $\tau_i>b_i$, decreases on $(\tau_i,r_i)$, increases on $(r_i,z_i)$, decreases on $(z_i,\bar r_i)$ and increases on $(\bar r_i,\bar b_i)$. So $Q$ has two local minima in the river, at $r_i$ and at $\bar r_i$, and therefore $Q>0$ on the river if and only if $Q(r_i)>0$ and $Q(\bar r_i)>0$. In \cite[\S4.3 and \S5.7]{Tang2026} it is observed that inside the river $u$, $f(u)$ and $f'(u)$ all change sign and $Q$ may decrease, so that the positivity of $Q$ is uncertain. In fact $Q$ can change sign inside the river, while $Q_{\varphi'}>0$ on $[\tau_i,\bar b_i]$ still holds, as Theorem~\ref{thm:phase} asserts. The bridge $B_0$ is introduced there to take $Q$ from $Q(b_i)>0$ to $Q(\bar b_i)>Q(b_i)$ without using it.\footnote{In arXiv:2409.06915v2 these remarks appear as footnotes.} Nothing in Sections~\ref{sec:renewal-below}--\ref{sec:staircase} uses the sign of $Q$ inside the river.
    \end{enumerate}
\end{remark}

%% file: sections/18-tail.tex
\section{Tail behavior: decay, dichotomy and nondegeneracy}\label{sec:tail}

We now leave the phases and look at the tail of a bound state, after its last critical point. There $u$ decays exponentially, so $f'(u)\to-1$ and $v$ solves a linear equation close to $v''+gv'-v=0$. Hence $v$ either decays or grows exponentially, and Section~\ref{sec:staircase} needs to know which of the two happens at every bound state, and with which sign. This section proves the dichotomy, and shows that a simple condition at the start of the tail rules out decay and fixes the sign of the growth. It uses only Sections~\ref{sec:multiplier}--\ref{sec:variation}, so it can also be read right after Section~\ref{sec:variation}.

Throughout this section $u$ is a bound state with $k\ge 0$ zeros, normalized so that $u>0$ on the tail, and
\[
    m(r) \defeq  \frac{g(r)+\sqrt{g(r)^2+4}}{2}
\]
denotes the modulus of the decaying root of the linearization at $u=0$, with $g$ treated as a constant: the roots of $\lambda^2+g\lambda-1=0$ are $\lambda_-=-m<0<\lambda_+=1/m$.

\begin{lemma}[Decay]\label{lem:decay}
    Assume (G1), (G2), (G4)--(G6). Then $u'<0$ on the tail and
    \[
        \frac{-u'(r)/u(r)}{m(r)} \longrightarrow 1 \qquad (r\to\infty).
    \]
    Consequently, for every $\varepsilon>0$ there are $C_\varepsilon, R$ with
    \[
        C_\varepsilon^{-1}e^{-(1+\varepsilon)\int_R^r m}
        \;\le\; u(r) \;\le\;
        C_\varepsilon e^{-(1-\varepsilon)\int_R^r m} \qquad (r\ge R),
    \]
    and, if moreover $(\Phi)$ holds, $P(r)\to 0$, with $P$ as in \eqref{eq:pohozaev}.
\end{lemma}

\begin{proof}
    Let $y\defeq u'/u$ on $(c_k,\infty)$, where $u>0$ and $u'<0$ (Proposition~\ref{prop:critical-points}(iv)), so $y<0$ and
    \begin{equation}\label{eq:riccati}
        y' = -y^2-gy+1-u^{p-1} = (\lambda_+-y)(y-\lambda_-)-u^{p-1},
    \end{equation}
    with $\lambda_-=-m\le -1$, $\lambda_+=1/m\in(0,1]$ and $\lambda_+-\lambda_-=\sqrt{g^2+4}$. Note that $|m'|\le|g'|$ and, by (G6), for every $\delta>0$ there is $R_\delta$ with $|m'|\le\delta m^2$ and $u^{p-1}\le\delta$ on $[R_\delta,\infty)$, using $m\ge 1$, $g'\to 0$, together with $u\to 0$.

    Fix $\varepsilon\in(0,\frac12)$ and choose $R=R_\delta$ with $\delta\defeq \frac{\varepsilon^2}{100}$. We claim that $-(1+\varepsilon)m(r)<y(r)<-(1-\varepsilon)m(r)$ for all $r\ge R'$, for some $R'\ge R$.

    \emph{The region $\{y\le-(1+\varepsilon)m\}$ is forward invariant after $R$ and leads to a contradiction.} On its boundary $y=-(1+\varepsilon)m$ we have $\lambda_+-y\ge(1+\varepsilon)m$ and $y-\lambda_-=-\varepsilon m$, so by~\eqref{eq:riccati}
    \[
        y' \le -\varepsilon(1+\varepsilon)m^2 < -(1+\varepsilon)\delta m^2
        \le -(1+\varepsilon)|m'| \le \frac{d}{dr}\big[-(1+\varepsilon)m\big],
    \]
    the second inequality because $\delta=\varepsilon^2/100\le\varepsilon/4$. So the flow crosses the boundary only downward. In the interior, while $y\le-(1+\varepsilon)m$, we have $\lambda_+-y\ge|y|$ and $y-\lambda_-\le-\varepsilon m\le-\varepsilon$, hence $y'\le-\varepsilon|y|$. Thus $|y|$ grows at least exponentially, and once $|y|\ge 2m$ we get $y-\lambda_-\le y/2$ and $y'\le-y^2/2$, so $y\to-\infty$ in finite time. Then $u$, with $u'/u\to-\infty$ at a finite point and $u'$ bounded, reaches $0$ at that point: a zero of $u$ beyond $z_k$, a contradiction. Hence $y>-(1+\varepsilon)m$ for all $r\ge R$.

    \emph{The region $\{y\ge-(1-\varepsilon)m\}$ cannot be visited at late times.}
    On its boundary $y=-(1-\varepsilon)m$ we have $\lambda_+-y\ge(1-\varepsilon)m\ge m/2$ and $y-\lambda_-=\varepsilon m$, so
    \[
        y'\ge\frac{\varepsilon}{2}m^2-\delta\ge\frac{\varepsilon}{4}m^2
        >(1-\varepsilon)\delta m^2\ge(1-\varepsilon)|m'|
        \ge\frac{d}{dr}\big[-(1-\varepsilon)m\big],
    \]
    and the flow crosses only upward, that is into that region. In the interior, while $y\in[-(1-\varepsilon)m,0)$, we have $\lambda_+-y\ge\lambda_+=1/m$ and $y-\lambda_-\ge\varepsilon m$, so $y'\ge\varepsilon-\delta\ge\frac{\varepsilon}{2}>0$. Hence, once in the region, $y$ increases at a uniform rate and reaches $0$ in finite time. But $y(r_*)=0$ means $u'(r_*)=0$, contradicting $u'<0$ on the tail. So $y$ never enters the region after some $R'$.

    As $\varepsilon\in(0,\frac12)$ was arbitrary, $y/(-m)\to 1$. The integrated bounds follow from $u(r)=u(R)\exp\int_R^r y$. Finally, for $P$: since $|F(u)|\le Cu^2$ and $|u'|\le 2mu$ on the tail,
    \[
        |P| \le e^G\varphi\big(u'^2+2|F(u)|\big)+(\gamma-1)e^G|uu'|
        \le C\varphi(1+m^2)e^Ge^{-2(1-\varepsilon)\int_R^r m}.
    \]
    Using $2m-g=\sqrt{g^2+4}$ and $m\le\sqrt{g^2+4}$, we get, for $\varepsilon<\frac14$,
    \[
        2(1-\varepsilon)m-g=\sqrt{g^2+4}-2\varepsilon m
        \ge(1-2\varepsilon)\sqrt{g^2+4}\ge c(1+g).
    \]
    Therefore $e^Ge^{-2(1-\varepsilon)\int m}\le Ce^{-cr}$, while $\varphi\le r$ (Lemma~\ref{lem:phi-basic}, which is where $(\Phi)$ is used) and $m$ is bounded on the tail by (G5). Hence $P\to 0$.
\end{proof}

\begin{remark}\label{rem:decay}
    The proof uses the boundedness of $m$ (that is, of $g$) only in the step $|y|\ge2m$ of the first region, which can be avoided (for $y\le-(1+\varepsilon)m$ we have $m\le|y|/(1+\varepsilon)$, so $y-\lambda_-\le\varepsilon y/(1+\varepsilon)$ and $y'\le-\varepsilon y^2/(1+\varepsilon)$), and in the last step ($P\to0$). The choice of $R_\delta$ uses only $g'\to0$ and $m\ge1$, and no lower bound on $g$ is needed. In particular it applies directly to $g=(n-1)/r$, where $\varphi=r$ and $m\to 1$.
\end{remark}

\begin{lemma}[Dichotomy for the variation]\label{lem:dichotomy}
    Assume (G1), (G2), (G4)--(G6). The equation $v''+gv'+f'(u)v=0$ has solutions $V_\pm$ with
    \[
        V_\pm(r) = q^{-1/4}e^{-G/2}\exp\Big(\pm\int_R^r\sqrt{q}\Big)\big(1+o(1)\big),
        \qquad
        \frac{V_\pm'}{V_\pm} = -\frac{g}{2}\pm\sqrt{q}+o(1),
    \]
    and every solution is a linear combination $c_+V_++c_-V_-$. Moreover $V_+(r)\to\infty$ and $V_-\to 0$. Consequently, for the variation $v$ of a bound state exactly one of the following holds:
    \begin{enumerate}[label=(\alph*)]
        \item \emph{($v$ decays)} $c_+=0$ and $v(r)\to 0$.
        \item \emph{($v$ dominant)} $|v(r)|\to\infty$, and $v,v'$ are eventually strictly of one sign, with $\sgn v=\sgn c_+$ and $v'/v=-\frac{g}{2}+\sqrt{q}+o(1)$, which is bounded below by a positive constant for large $r$.
    \end{enumerate}
\end{lemma}

\begin{proof}
    Substituting $v=e^{-G/2}z$ turns the equation into
    \[
        z''=\big(q-p|u|^{p-1}\big)z,
    \]
    since $1+f'(u)=p|u|^{p-1}$. By (G5)--(G6), $q=1+g^2/4+g'/2$ is of class $C^2$ on $[R,\infty)$, bounded, and bounded below by a positive constant (indeed $q-1-g^2/4=g'/2\to 0$ and $g$ is bounded), although it need not converge. We apply \cite[Ch.~6, Thm.~2.2]{Olver1974} to this equation on $[R,\infty)$, with main term $q$ and perturbation $-p|u|^{p-1}$. Its error-control function
    \[
        \int\Big(q^{-1/4}\big(q^{-1/4}\big)''+p|u|^{p-1}q^{-1/2}\Big)\,dr
    \]
    has finite variation on $[R,\infty)$: the first term is integrable because $q$ is bounded below and $\int_R^\infty(|q''|+q'^2)\,dr<\infty$ by (G6), and the second because $u$ decays exponentially (Lemma~\ref{lem:decay}). The theorem gives solutions $z_\pm=q^{-1/4}e^{\pm\int_R^r\sqrt q}(1+\varepsilon_\pm)$, with explicit bounds on $\varepsilon_\pm$ and $q^{-1/2}\varepsilon_\pm'$ in terms of this variation. For $z_-$ the variation is taken over $[r,\infty)$, so $\varepsilon_-$ and $\varepsilon_-'$ tend to $0$. For $z_+$ it is taken over $[R,r]$, so the bound does not tend to $0$, but it can be as small as we like if $R$ is large enough. Since any two growing solutions differ by a constant factor and a multiple of $z_-$, it follows that $z_+\big/\big(q^{-1/4}e^{\int_R^r\sqrt q}\big)$ has a nonzero limit and that $z_+'/z_+=\sqrt q+o(1)$, and we normalize $z_+$ so that this limit is $1$. As $q'\to0$ by (G6), this yields the basis $V_\pm=e^{-G/2}z_\pm$ and the logarithmic derivatives.

    For the growth of $V_+$, its exponent is
    \[
        \int\Big(\sqrt{q}-\frac{g}{2}\Big)
        = \int\frac{q-g^2/4}{\sqrt{q}+g/2}
        = \int\frac{1+g'/2}{\sqrt{q}+g/2},
    \]
    and since $g'\to 0$ and $\sqrt{q}+g/2$ is bounded by (G5), the integrand is $\ge c>0$ for large $r$, so the exponent diverges linearly. The factor $q^{-1/4}$ is bounded below, hence $V_+\to\infty$. Similarly $V_-\to 0$. The dichotomy for $v$ follows. In case (b) the eventual monotonicity and the signs are read off from $v=c_+V_+(1+o(1))$ and $v'=c_+V_+'(1+o(1))$, with $V_+'/V_+=-\frac{g}{2}+\sqrt{q}+o(1)$ and $-\frac{g}{2}+\sqrt{q}=\frac{1+g'/2}{\sqrt{q}+g/2}\ge c>0$ eventually ($g$ itself need not converge).
\end{proof}

\subsection{Nondegeneracy}\label{sec:nondegeneracy}
Let $u$ be a bound state with $k\ge 0$ zeros, normalized so that $u>0$ on the tail, and let $r_{k+1}>c_k$ be the point with $u(r_{k+1})=1$. We say that the variation $v$ satisfies the \emph{final-phase condition} if
\begin{equation}\label{eq:final-phase}
    v<0 \quad\text{on } [r_{k+1},\infty) \qquad\text{and}\qquad
    Q(r_{k+1})>0.
\end{equation}
Then $Q$ is increasing on $[r_{k+1},\infty)$, by \eqref{eq:Q-prime} and $f(u)<0$, so that $Q\ge Q(r_{k+1})>0$ on the tail. This condition holds unconditionally when $k=0$, by Theorem~\ref{thm:first-phase}(iii). The next proposition shows that the growth of $|v|$ is a consequence of it, so that no separate hypothesis is needed. For $k\ge 1$, Theorem~\ref{thm:phase}(v) proves $|v|\to\infty$ directly, under (H) and the slope condition at the last critical point, without passing through this condition.

\begin{proposition}[Nondegeneracy]\label{prop:nondegeneracy}
    Let $u$ be a bound state at which the variation $v$ satisfies the final-phase condition \eqref{eq:final-phase}. Then $v$ does not decay: alternative (b) of Lemma~\ref{lem:dichotomy} holds, with $v\to-\infty$. In particular, every ground state is nondegenerate.
\end{proposition}

\begin{proof}
    On $[r_{k+1},\infty)$ we have $u\in(0,1)$, so $f(u)<0$, and $v<0$. Hence $Q'=2\varphi'e^Gf(u)v>0$ and
    \[
        Q(r) \;\ge\; Q(r_{k+1}) \;>\; 0 \qquad (r\ge r_{k+1}).
    \]
    Suppose $v$ decayed. Then $v,v'$ decay like $e^{-G/2-\int\sqrt{q}}$ up to bounded factors (Lemma~\ref{lem:dichotomy}, using $V_-'/V_-=-\frac{g}{2}-\sqrt{q}+o(1)$ and (G5)), while $u,u'$ decay like $e^{-(1-\varepsilon)\int m}$ with $|u'|\le 2mu$ (Lemma~\ref{lem:decay}). Since $m=\frac{g}{2}+\sqrt{1+g^2/4}$ and $\sqrt{q}=\sqrt{1+g^2/4+g'/2}$ with $g'\to 0$, one has $\frac{g}{2}+\sqrt{q}\ge(1-\varepsilon)m$ for large $r$, so
    \[
        |Q| \;\le\; e^G\Big[\varphi|u'v'|+\varphi|f(u)v|+(\gamma-1)|u'v|\Big]
        \;\le\; C\,\varphi\,(1+m^2)\,e^Ge^{-2(1-\varepsilon)\int_R^r m}
        \;\longrightarrow\; 0,
    \]
    exactly as in the last step of the proof of Lemma~\ref{lem:decay}: $2(1-\varepsilon)m-g\ge c(1+g)$ for $\varepsilon<\frac14$, $\varphi\le r$ and $m$ is bounded. This contradicts $Q\ge Q(r_{k+1})>0$. Thus $v=c_+V_+(1+o(1))$ with $c_+\ne 0$, and $c_+<0$ because $v<0$ on the tail.
\end{proof}

\begin{remark}\label{rem:nondegeneracy}
    No condition on $\varphi''$ is used for the monotonicity: $Q$ itself is increasing on $[r_{k+1},\infty)$ as soon as $v$ has the sign opposite to $u$. (The concavity of $\varphi$ is used only through the bound $\varphi\le r$ in the step $|Q|\to0$, and even there it can be avoided: $\varphi e^{-G}=(\gamma-1)\cJ$ is decreasing, so $\varphi\le Ce^G$ on the tail, and the exponent $2g-2(1-\varepsilon)m=-2/m+2\varepsilon m$ is still negative once $\varepsilon<1/\sup m^2$.) A pointwise condition of the type $(p-1)\varphi'|u'|\ge|\varphi''||u|$ on the tail, which would make $T_1$ nondecreasing, is not needed. On $\Hyp^3$ such a condition fails for $p<2\sqrt{2}-1\approx 1.83$, whereas Proposition~\ref{prop:nondegeneracy} covers the whole range $1<p<5$. If $v$ decays with the sign opposite to $u$, $Q$ is still eventually monotone and tends to $0$, but from below, against the sign given by this condition, so the condition rules it out.
\end{remark}

\begin{corollary}[Divergence]\label{cor:divergence}
    Let $u$ be a bound state at which the variation does not decay and has eventually the sign $-\sigma$, $\sigma\defeq \sgn u$ on the tail (in the applications, $v$ has a last zero $\tau_{k+1}\in(c_k,r_{k+1})$ after which it does not vanish; the location of this zero is not used here). Then
    \[
        |v(r)|\to\infty, \qquad
        \sgn v=\sgn v'=-\sigma \quad\text{eventually.}
    \]
\end{corollary}

\begin{proof}
    By hypothesis $v$ has eventually the sign $-\sigma$ and does not decay, so Lemma~\ref{lem:dichotomy} gives alternative (b) with $\sgn c_+=-\sigma$.
\end{proof}

The hypotheses of Corollary~\ref{cor:divergence} are verified, for $k=0$, by Theorem~\ref{thm:first-phase}(iii) and Proposition~\ref{prop:nondegeneracy} (with $\tau_1$) and, for $k\ge1$ under (H) and (R), by Theorem~\ref{thm:phase}(i),(v). The corollary itself uses only Lemma~\ref{lem:dichotomy}.

%% file: sections/19-assembly.tex
\section{Monotonicity, jumps and the staircase}\label{sec:staircase}

This is the last step. For every bound state we now know that $v$ grows at infinity with the sign opposite to $u$: for the ground state by Sections~\ref{sec:variation} and~\ref{sec:tail}, and for nodal bound states, under (H) and (R), by the induction of Sections~\ref{sec:transition}--\ref{sec:renewal-below}. This section turns that sign into the staircase. We first show how the zeros of $u$ move as $\alpha$ grows, and that $N$ jumps by exactly one at a bound state. We then show that the sign of $v$ at infinity, checked only at the bound states, is enough by itself (Theorem~\ref{thm:tail-criterion}), and we prove the theorems stated in Section~\ref{sec:uniqueness}.

\begin{lemma}[Persistence and monotone motion of zeros]\label{lem:zero-motion}
    Assume, for the statements involving $i\ge 2$, that $v(z_i)u'(z_i)>0$ at the $i$-th zero of every solution considered.
    \begin{enumerate}[label=(\roman*)]
        \item Zeros of $r\mapsto u(r,\alpha)$ are simple, and each zero $z_i(\alpha)$ is $C^1$ in $\alpha$ on its interval of existence, with
              \[
                  z_i'(\alpha) = -\frac{v(z_i(\alpha),\alpha)}{u'(z_i(\alpha),\alpha)} < 0,
              \]
              the sign being $v(z_i)u'(z_i)>0$: for $i=1$ by Theorem~\ref{thm:first-phase}(ii), for $i\ge 2$ by hypothesis.
        \item If $u(\cdot,\alpha_0)$ has $N(\alpha_0)\ge k$ zeros, then $N(\alpha)\ge k$ for all $\alpha\ge\alpha_0$. In particular $N$ is nondecreasing on $(\alpha^*(0),\infty)$.
    \end{enumerate}
\end{lemma}

\begin{proof}
    (i) Simplicity means $(u,u')\ne(0,0)$, so the implicit function theorem applies at each $z_i$, and differentiating $u(z_i(\alpha),\alpha)=0$ gives the formula.

    (ii) Continue the first $k$ zeros in $\alpha$ from $\alpha_0$ upward. Along the continuation $z_i(\alpha)$ remains the $i$-th zero of $u(\cdot,\alpha)$: zeros are simple, so by the implicit function theorem no new zero can appear in the interior of a bounded interval, and none can enter through $r=0$, since $u>0$ near $0$ uniformly for $\alpha$ in compact sets. Hence (i) applies to each branch. On the maximal interval of continuation each $z_i(\alpha)$ is strictly decreasing by (i). The zeros stay distinct, as two of them colliding would produce a double zero (by Rolle's theorem $u'$ would vanish between them, and in the limit $(u,u')=(0,0)$), and are bounded: $z_i(\alpha)\le z_i(\alpha_0)$, and $z_1(\alpha)\ge\underline{z}>0$ because $u(0,\alpha)=\alpha\ge\alpha_0$ and, on compact sets of $\alpha$, $u>0$ on a fixed neighborhood of $0$ by continuous dependence. If the continuation stopped at some $\alpha_\dagger$, the limits $z_i(\alpha_\dagger^-)$ would exist, being monotone and bounded, and would be zeros of $u(\cdot,\alpha_\dagger)$ by continuity, distinct and simple, so the continuation would extend past $\alpha_\dagger$, a contradiction. Hence all $k$ zeros persist for every $\alpha\ge\alpha_0$.
\end{proof}

Under (H) and (R), the hypothesis of Lemma~\ref{lem:zero-motion} for $i\ge2$ is provided by Theorem~\ref{thm:critical-form}(ii). The lemma itself uses only Theorem~\ref{thm:first-phase}(ii).

\begin{lemma}[Jumping at bound states]\label{lem:jump}
    Let $\hat\alpha$ be a bound state with $k\ge 0$ nodes at which the variation is dominant at infinity with the sign opposite to $u$ (Corollary~\ref{cor:divergence}), and assume $v(z_j)u'(z_j)>0$ at the zeros of the nodal solutions with $\alpha<\hat\alpha$ near $\hat\alpha$ (for $k=0$ both hypotheses are provided by Theorem~\ref{thm:first-phase}(ii),(iii), Proposition~\ref{prop:nondegeneracy} and Corollary~\ref{cor:divergence}, under (G1)--(G6) and $(\Phi)$ only).
    Then there is $\varepsilon>0$ such that
    \[
        N(\alpha)=k \quad (\hat\alpha-\varepsilon<\alpha\le\hat\alpha), \qquad
        N(\alpha)=k+1 \quad (\hat\alpha<\alpha<\hat\alpha+\varepsilon),
    \]
    and for $\alpha\in(\hat\alpha,\hat\alpha+\varepsilon)$ the solution is trapped in the sense of Proposition~\ref{prop:classification}(b).
\end{lemma}

The statement is superseded by Lemma~\ref{lem:tail-jump} below, which drops the hypothesis on the nodal solutions. The proof is given in full because Lemma~\ref{lem:tail-jump} and Theorem~\ref{thm:unconditional}(ii) reuse it from the paragraph \emph{The extra zero above $\hat\alpha$} on.

\begin{proof}
    Write $\hat u:=u(\cdot,\hat\alpha)$, normalized with $\hat u>0$ on the tail, and $v=\partial_\alpha u(\cdot,\hat\alpha)$. Let $g\le M_0$ on $[R_1,\infty)$, and let $T_0,\fc_1,\fc_2,\fc_3$ be the constants of Lemmas~\ref{lem:passage-times}, \ref{lem:band-dissipation}, \ref{lem:pohozaev-lower}, computed for a compact range of $\alpha$ around $\hat\alpha$ (they depend on $\alpha$ only through $S$).

                    \emph{Below and at $\hat\alpha$.} The $k$ zeros of $\hat u$ are simple, so for $\alpha$ near $\hat\alpha$, on either side, $u(\cdot,\alpha)$ has at least $k$ zeros near them: $N(\alpha)\ge k$. For $\alpha\le\hat\alpha$, Lemma~\ref{lem:zero-motion}(ii) gives $N(\alpha)\le N(\hat\alpha)=k$ (this is the only use of the hypothesis on $v(z_j)u'(z_j)$). Hence $N\equiv k$ on some $(\hat\alpha-\varepsilon,\hat\alpha]$.

    \emph{The extra zero above $\hat\alpha$.} By hypothesis (Corollary~\ref{cor:divergence}), $v\to-\infty$ with $v'<0$ eventually. Let $u_\dagger:=p^{-1/(p-1)}$, so that $f'<0$ on $(0,u_\dagger)$, and choose $\bar r>\max(c_k,R_1)$, beyond the last zero of $v$, with
    \[
        0<\hat u(\bar r)<\frac{u_\dagger}{2}, \qquad v(\bar r)<0, \qquad
        v'(\bar r)<0.
    \]
    For $\alpha>\hat\alpha$ close, set $w:=u(\cdot,\alpha)-\hat u$. By $C^1$ dependence on compacts, $w(\bar r)<0$ and $w'(\bar r)<0$, and $u(\cdot,\alpha)$ has exactly $k$ zeros in $[0,\bar r]$, all near those of $\hat u$. Suppose, for a contradiction, that $u(\cdot,\alpha)>0$ on $[\bar r,\infty)$, and let $\varrho_w:=\sup\{r\ge\bar r:\ w<0\text{ on }[\bar r,r)\}$. The difference $w$ solves $w''+gw'+f'(\bar u)w=0$, where $\bar u$ is an intermediate value in $(u(\cdot,\alpha),\hat u)$. On $[\bar r,\varrho_w)$ we have $0<u(\cdot,\alpha)<\hat u<u_\dagger$, so $f'(\bar u)\in(-1,0)$.

    Suppose that $w'$ vanishes at some first $\varsigma\in(\bar r,\varrho_w)$. Then $w''(\varsigma)=-f'(\bar u)w(\varsigma)<0$, so $w'$ would be decreasing through $0$. This is impossible at a first zero of $w'$ reached with $w'<0$ before it. Hence $w'<0$ on $[\bar r,\varrho_w)$, so $w$ is decreasing and $\varrho_w=\infty$. This gives $u(r,\alpha)\le\hat u(r)+w(\bar r)\to w(\bar r)<0$ as $r\to\infty$, which contradicts positivity.

    Hence the supposition fails: there is a first $\zeta=\zeta(\alpha)>\bar r$ with $u(\zeta,\alpha)=0$, distinct from the $k$ persisted zeros, so $N(\alpha)\ge k+1$. Moreover, $\zeta(\alpha)\to\infty$ as $\alpha\downarrow\hat\alpha$, and the same argument shows that $w$ is negative and decreasing on $[\bar r,\zeta]$, with $w(\zeta)=-\hat u(\zeta)$.

    \emph{The energy at the extra zero is exponentially small.} Since $w$ is decreasing and negative on $[\bar r,\zeta]$, $|w|\le|w(\zeta)|=\hat u(\zeta)$.
    Integrating $(e^Gw')'=-e^Gf'(\bar u)w$ from $\bar r$ to $\zeta$, with $|f'(\bar u)|\le 1$ and Lemma~\ref{lem:phi-consequences}(iii),
    \[
        |w'(\zeta)| \;\le\; e^{G(\bar r)-G(\zeta)}|w'(\bar r)|
        +\hat u(\zeta)\,e^{-G(\zeta)}\int_{\bar r}^{\zeta}e^G
        \;\le\; |w'(\bar r)|+(\zeta-\bar r)\,\hat u(\zeta).
    \]
    At the fixed point $\bar r$ we have, as $\alpha\to\hat\alpha$,
    \[
        w(\bar r)=(\alpha-\hat\alpha)\big(v(\bar r)+o(1)\big),
        \qquad
        w'(\bar r)=(\alpha-\hat\alpha)\big(v'(\bar r)+o(1)\big).
    \]
    Since $v(\bar r)\ne 0$, it follows that $|w'(\bar r)|\le C_{\bar r}|w(\bar r)|\le C_{\bar r}\hat u(\zeta)$ for $\alpha$ close to $\hat\alpha$, with $C_{\bar r}:=2|v'(\bar r)/v(\bar r)|$. By Lemma~\ref{lem:decay}, $|\hat u'(\zeta)|\le 2m(\zeta)\hat u(\zeta)\le 2M_1\hat u(\zeta)$, where $M_1:=\sup_{[R_1,\infty)}m<\infty$ by (G5). Together with the bound on $|w'(\zeta)|$ above, this gives
    \[
        |u'(\zeta,\alpha)|\le(2M_1+C_{\bar r}+\zeta)\,\hat u(\zeta).
    \]
    Finally, the upper bound of Lemma~\ref{lem:decay} with $\varepsilon=\frac14$ and $m\ge 1$ gives $\hat u(\zeta)^2\le Ce^{-\frac32(\zeta-R)}$, so that
    \begin{equation}\label{eq:extra-zero-energy}
        E_\alpha(\zeta)=\tfrac12u'(\zeta,\alpha)^2 \;\le\; C\,\zeta^2\,\hat u(\zeta)^2
        \;\le\; C\,\zeta^2e^{-(\zeta-R)} \eqdef \eta(\zeta),
    \end{equation}
    where $E_\alpha$ denotes the energy of $u(\cdot,\alpha)$ and $C,R$ depend only on $\hat u$ and $\bar r$.

    \emph{No second extra zero, and trapping.} Suppose that $u(\cdot,\alpha)$ has a further zero $\zeta'>\zeta$, or that it is a bound state with $k+1$ nodes. In both cases $E_\alpha\ge 0$ on $[\zeta,\infty)$. (In the first case $E_\alpha\ge E_\alpha(\zeta')\ge 0$ up to $\zeta'$, and we only use $[\zeta,\zeta']$.) By Proposition~\ref{prop:critical-points} there is a first critical point $c'>\zeta$, with $|u(c')|>\alpha_*$. After $c'$ the function $|u|$ decreases toward the next zero (respectively toward $0$), performing a band crossing $J'\subset(c',\zeta')$ (respectively $J'\subset(c',\infty)$). Since $F(|u(c')|)=E_\alpha(c')\le E_\alpha(\zeta)\le\eta(\zeta)$, for $\zeta$ large we have $|u(c')|<\alpha^*(0)$. Hence $|u|<\alpha^*(0)$ on $[c',\zeta']$ (respectively on $[c',\infty)$), and therefore $P$ is nonincreasing there. Moreover, $P(\zeta')>0$ by \eqref{eq:pohozaev-zero} (respectively $P>0$ on $(0,\infty)$ by Proposition~\ref{prop:pohozaev-positive}). Lemma~\ref{lem:pohozaev-lower} with $a=c'$ and $J=J'$ gives $E_\alpha(c')\ge \fc_3\fc_2/(2c')$. Then \eqref{eq:ascent-time} on $[\zeta,c']$ yields
    \[
        c'-\zeta \;\le\; T_0\Big(1+\log\frac{2c'}{\fc_3\fc_2}\Big), \qquad\text{hence}\qquad
        c'\le 2\zeta \quad\text{for } \zeta\ge\zeta_1,
    \]
    with $\zeta_1$ depending only on $T_0,\fc_2,\fc_3$. Finally, the ascent from $\zeta$ to $c'$ performs a band crossing $J\subset(\zeta,c')$, and Lemma~\ref{lem:band-dissipation} gives
    \[
        E_\alpha(c') \;\le\; E_\alpha(\zeta)-2\fc_1\fc_2\frac{\gamma-1}{c'}
        \;\le\; \eta(\zeta)-\fc_1\fc_2\frac{\gamma-1}{\zeta} \;<\; 0
        \qquad\text{for } \zeta\ge\zeta_2,
    \]
    since $\eta(\zeta)\zeta\to 0$. This contradicts $E_\alpha(c')\ge 0$. Choose $\varepsilon$ so small that $\zeta(\alpha)\ge\max(\zeta_1,\zeta_2)$ for $\alpha\in(\hat\alpha,\hat\alpha+\varepsilon)$. Then $u(\cdot,\alpha)$ has no zero after $\zeta$ and is not a bound state, so $N(\alpha)=k+1$ and, by Proposition~\ref{prop:classification}, $u(\cdot,\alpha)$ is trapped.
\end{proof}

\begin{remark}\label{rem:jump}
    The argument is that of Proposition~\ref{prop:classification}, with the estimate \eqref{eq:extra-zero-energy} for the energy at the extra zero. Alternatively, the last step follows from Part~I: Lemma~\ref{lem:no-gaps}, which needs only (G3), a consequence of $(\Phi)$ (Lemma~\ref{lem:phi-consequences}(i)), and Remark~\ref{rem:no-gaps-trapped} give $N(\alpha)\le k+1$ and trapping for $\alpha$ near $\hat\alpha$ directly.
\end{remark}

\subsection{The tail criterion: uniqueness from the sign of the variation at
    infinity}\label{sec:tail-criterion}
The phase analysis of Sections~\ref{sec:transition}--\ref{sec:renewal-below} is used in the proof of the main theorems through two of its consequences only. The first is the sign $v(z_j)u'(z_j)>0$ at the zeros of the nodal solutions near a bound state (Lemma~\ref{lem:zero-motion}, hence the monotonicity of $N$). The second is the dominance of $v$ at infinity, with the sign opposite to $u$, at the bound state itself (Corollary~\ref{cor:divergence}). We show here that the second consequence alone, at the bound states alone, suffices for the whole staircase, and that it is the sign of a single number.

For a $C^2$ function $y$ write $M[y]:=e^G(u'y-uy')$, so that $M=M[v]$. With $Lz:=z''+gz'+f'(u)z$ one has
\[
    M[y]'=e^G\big(u''y-uy''+g(u'y-uy')\big)=e^G\big(y(u''+gu')-u(y''+gy')\big),
\]
that is
\begin{equation}\label{eq:M-y-prime}
    M[y]'=e^G\big(y\,Lu-u\,Ly\big)=e^G\big((p-1)|u|^{p-1}u\,y-u\,Ly\big),
\end{equation}
since $Lu=f'(u)u-f(u)=(p-1)|u|^{p-1}u$ for the model nonlinearity.

\begin{lemma}[The tail sign]\label{lem:tail-sign}
    Let $u$ be a bound state with $k\ge 0$ nodes and $v=\partial_\alpha u$. Then the limit
    \begin{equation}\label{eq:M-infinity}
        M(\infty):=\lim_{r\to\infty}M(r)=(p-1)\int_0^\infty e^G|u|^{p-1}u\,v\,dr
    \end{equation}
    exists, the integral converging absolutely, and the following are equivalent:
    \begin{enumerate}[label=(\roman*)]
        \item \emph{(the tail sign)} $v$ is dominant at infinity, with $\sgn v=-\sgn u$ eventually.
        \item $M(\infty)>0$.
        \item $v$ is dominant, and its number $N_v$ of zeros in $(0,\infty)$ satisfies $N_v\equiv k+1\pmod{2}$.
    \end{enumerate}
    Moreover $M(\infty)=0$ if and only if $v$ decays, and $M(\infty)<0$ if and only if $v$ is dominant with $\sgn v=\sgn u$ eventually.
\end{lemma}

\begin{proof}
    By \eqref{eq:M-prime}, $M'=(p-1)e^G|u|^{p-1}uv$, and $M(0^+)=0$ (proof of Theorem~\ref{thm:first-phase}(ii)), so $M(r)=(p-1)\int_0^re^G|u|^{p-1}uv$. On the tail, $u$ is itself a decaying solution of the linear equation $z''+gz'+(|u|^{p-1}-1)z=0$, whose coefficient differs from that of $Lz=0$, namely $f'(u)=p|u|^{p-1}-1$, by the exponentially decaying $(p-1)|u|^{p-1}$ (Lemma~\ref{lem:decay}). The substitution $z=e^{-G/2}y$ turns this equation into $y''=\big(q-|u|^{p-1}\big)y$, to which \cite[Ch.~6, Thm.~2.2]{Olver1974} applies as in the proof of Lemma~\ref{lem:dichotomy}, with perturbation $-|u|^{p-1}$. Since $u$ decays, this gives $u=D_\infty\tilde V_-$ with $\tilde V_-=q^{-1/4}e^{-G/2-\int\sqrt{q}}(1+o(1))$, $u'/u=-\frac{g}{2}-\sqrt{q}+o(1)$, where $D_\infty\ne 0$ has the sign of $u$ on the tail. Write $v=c_+V_++c_-V_-$. The decaying part $c_-V_-$ contributes $e^G(u'V_--uV_-')=e^GuV_-(u'/u-V_-'/V_-)\to 0$, while
    \[
        e^G\big(u'V_+-uV_+'\big)=e^G\,u\,V_+\Big(\frac{u'}{u}-\frac{V_+'}{V_+}\Big),
        \qquad e^G\,\tilde V_-V_+=q^{-1/2}(1+o(1)),
        \qquad \frac{u'}{u}-\frac{V_+'}{V_+}=-2\sqrt{q}+o(1),
    \]
    so that $M(\infty)$ exists and $M(\infty)=-2c_+D_\infty$. Thus $M(\infty)$ has the sign of $-c_+D_\infty$, which is $-\sgn v(\infty)\sgn u(\infty)$ when $c_+\ne 0$, and vanishes if and only if $c_+=0$: this is (i)$\Leftrightarrow$(ii) and the last sentence. The integral converges absolutely since $e^G|u|^{p-1}|uv|\le Ce^{-(p-1)(1-\varepsilon)\int m}$ on the tail. For (iii): $u(0)=\alpha>0$ and $u$ has $k$ simple zeros, so $\sgn u=(-1)^k$ eventually; $v(0)=1$ and $v$ has $N_v<\infty$ zeros (all simple, in a bounded set by Lemma~\ref{lem:dichotomy}), so $\sgn v=(-1)^{N_v}$ eventually. Hence (i) is $c_+\ne 0$ together with $(-1)^{N_v}=-(-1)^k$.
\end{proof}

\begin{lemma}[At least one zero of $v$ per monotonicity interval]\label{lem:v-zeros}
    Assume (G1), (G2), (G4), (G5), (G6) and $g'<0$ on $(0,\infty)$ (which holds under $(\Phi)$: $g'\varphi=\varphi''-g\varphi'<0$). Let $u$ be a bound state with $k$ nodes and $0=c_0<c_1<\cdots<c_k$ its critical points. Then every nontrivial solution $y$ of $Ly=0$ has a zero in each of the $k+1$ sets $[c_0,c_1),\ldots,[c_{k-1},c_k),[c_k,\infty)$. In particular $N_v\ge k+1$ for the variation, and $N_v=k+1$ exactly when $v$ has one zero in each of them. These zeros interlace with the critical points of $u$ (weakly: a zero may lie at some $c_i$, $i\ge1$).
\end{lemma}

\begin{proof}
    \emph{Step 1: the function $\Omega$.} Differentiating the equation gives $L(u')=-g'u'$. Set $\Omega:=e^G(yu''-y'u')$. The identity behind \eqref{eq:M-y-prime}, $\big(e^G(yz'-y'z)\big)'=e^G(y\,Lz-z\,Ly)$, with $z=u'$ then gives
    \[
        \Omega'=-g'e^Gu'y=|g'|e^Gu'y.
    \]
    At a critical point,
    \[
        \Omega(c_j)=e^Gy(c_j)u''(c_j)=-e^{G(c_j)}y(c_j)f(u(c_j)),
    \]
    and $f(u(c_j))$ has the sign of $u(c_j)$ because $|u(c_j)|>\alpha_*>1$ ($E>0$ at a bound state).

    \emph{Step 2: $\Omega(0^+)=0$.} We have $u'(0)=0$, and $u''=-gu'-f(u)$ is bounded near $0$, as $|u'(r)|\le Cr$ and $rg$ is bounded by (G2). Since $e^G\to0$, it follows that $\Omega\to0$ whenever $y$ and $y'$ are bounded near $0$, which covers $y=v$. In general, by reduction of order, $y=a_1v+a_2\tilde y$ with $\tilde y\defeq v\int_r^{s_0}e^{-G}v^{-2}$ on an interval $(0,s_0]$, $s_0<\tau_1$, on which $v>0$. Then $e^G|\tilde y|\le Ce^G\cJ=C\varphi/(\gamma-1)\to0$ by Lemma~\ref{lem:phi-basic}. Moreover, $e^G\tilde y'=e^Gv'\int_r^{s_0}e^{-G}v^{-2}-1/v$ is bounded, so again $\Omega(0^+)=0$.

    \emph{Step 3: the intervals $[c_{i-1},c_i)$.} Let $I=[c_{i-1},c_i)$ with $u'<0$ on its interior, so that $u(c_{i-1})>0>u(c_i)$, and suppose $y>0$ on $I$. Then $\Omega$ is strictly decreasing on $I$ and $\Omega(c_{i-1})\le 0$. On the other hand, $\Omega(c_i)=-e^Gy(c_i)f(u(c_i))\ge 0$ (if $y(c_i)=0$, then $\Omega(c_i)=-e^Gy'u'(c_i)=0$). This is a contradiction. The case $y<0$ on $I$, and the case $u'>0$, are symmetric.

    \emph{Step 4: the interval $[c_k,\infty)$.} Say $u>0$ decreases to $0$ on $[c_k,\infty)$, and suppose $y>0$ there. Then $\Omega$ decreases strictly from $\Omega(c_k)\le 0$, so $\Omega(\infty)<0$. This is impossible. Indeed, if $y$ decays, then $y$ and $u$ both decay like $e^{-G/2-\int\sqrt{q}}$ up to bounded factors, and $\Omega\to 0$. If $y$ is dominant with $y\to+\infty$, set $\hat\lambda_\pm\defeq-\frac{g}{2}\pm\sqrt{q}$ (not the $\lambda_\pm$ of Section~\ref{sec:tail}) and write $u=D_\infty\tilde V_-$ as in the proof of Lemma~\ref{lem:tail-sign}. Then
    \begin{align*}
        \Omega & =e^GuV_+\big(u''/u-(y'/y)(u'/u)\big)(c_++o(1))                           \\
               & =c_+D_\infty\,q^{-1/2}\hat\lambda_-(\hat\lambda_--\hat\lambda_+)(1+o(1)) \\
               & =2c_+D_\infty\Big(\frac{g}{2}+\sqrt{q}\Big)(1+o(1))>0.
    \end{align*}
    The case $y<0$ on $[c_k,\infty)$ is symmetric. Hence $y$ vanishes in each of the $k+1$ disjoint sets.
\end{proof}

The finer interlacing with the zeros of $u$ in Theorem~\ref{thm:staircase}(iv) comes from Theorem~\ref{thm:critical-form}(i), which is not used in Lemma~\ref{lem:v-zeros}.

\begin{lemma}[Jumping from the tail sign]\label{lem:tail-jump}
    Let $\hat\alpha$ be a bound state with $k\ge 0$ nodes at which the tail sign holds, $M(\infty)>0$. Then there is $\varepsilon>0$ such that
    \[
        N(\alpha)=k \quad (\hat\alpha-\varepsilon<\alpha\le\hat\alpha), \qquad
        N(\alpha)=k+1 \quad (\hat\alpha<\alpha<\hat\alpha+\varepsilon),
    \]
    and every $u(\cdot,\alpha)$ with $\alpha\in(\hat\alpha-\varepsilon,\hat\alpha+\varepsilon)\setminus\{\hat\alpha\}$ is trapped. No information on the nodal solutions near $\hat\alpha$, and none on $v$ away from the tail, is used.
\end{lemma}

\begin{proof}
    Normalize $\hat u:=u(\cdot,\hat\alpha)>0$ on the tail. By Lemma~\ref{lem:tail-sign} and Lemma~\ref{lem:dichotomy}(b), $v\to-\infty$ with $v'<0$ eventually. Choose $\bar r>\max(c_k,R_1)$, so that $\hat u$ is positive and strictly decreasing on $[\bar r,\infty)$, with $0<\hat u(\bar r)<u_\dagger/2$, $v(\bar r)<0$, $v'(\bar r)<0$, where $u_\dagger=p^{-1/(p-1)}$, and $\hat u(\bar r)$ so small that the constants below apply. For $\alpha$ near $\hat\alpha$ set $w:=u(\cdot,\alpha)-\hat u$. By $C^1$ dependence, $u(\cdot,\alpha)$ has exactly $k$ zeros in $[0,\bar r]$, all near those of $\hat u$, and $w(\bar r),w'(\bar r)$ have the sign of $(\alpha-\hat\alpha)v(\bar r)$, i.e.\ of $\hat\alpha-\alpha$.

    \emph{Above $\hat\alpha$.} This is the proof of Lemma~\ref{lem:jump}, which from this point on uses only $w(\bar r)<0$, $w'(\bar r)<0$ and the estimates of Lemmas~\ref{lem:phi-consequences}(iii), \ref{lem:passage-times}, \ref{lem:band-dissipation}, \ref{lem:pohozaev-lower}, \ref{lem:decay} and Proposition~\ref{prop:pohozaev-positive}: an extra zero $\zeta(\alpha)\to\infty$ appears, the energy at $\zeta(\alpha)$ is exponentially small \eqref{eq:extra-zero-energy}, and there is not enough energy left for a further zero. Hence $N(\alpha)=k+1$ and the solution is trapped.

    \emph{Below $\hat\alpha$.} Now $w(\bar r)>0$, $w'(\bar r)>0$. Let $\varrho:=\sup\{r\ge\bar r:\ u(\cdot,\alpha)<u_\dagger\text{ on }[\bar r,r)\}$.
    On $[\bar r,\varrho)$ we have $0<\hat u<u(\cdot,\alpha)<u_\dagger$, so $w''+gw'+f'(\bar u)w=0$ with $f'(\bar u)\in(-1,0)$. At a first zero $\varsigma$ of $w'$ one would have $w''(\varsigma)=-f'(\bar u)w(\varsigma)>0$, impossible for $w'$ decreasing through $0$. Hence $w'>0$ and $w$ is increasing on $[\bar r,\varrho)$. In particular, $u(\cdot,\alpha)\ge\hat u+w(\bar r)>0$ on $[\bar r,\varrho)$: \emph{no zero can occur while $u(\cdot,\alpha)<u_\dagger$}.

    Let $\varrho_1:=\inf\{r\ge\bar r:\ w(r)\ge\hat u(r)\}$. Since $w\ge w(\bar r)>0$ and $\hat u\downarrow 0$, we have $\varrho_1<\infty$, and $\varrho_1<\varrho$ because $u(\varrho_1,\alpha)=2\hat u(\varrho_1)<u_\dagger$. Moreover, $\varrho_1\to\infty$ as $\alpha\uparrow\hat\alpha$, since $w\to0$ uniformly on compact subsets of $[\bar r,\infty)$ while $\hat u$ is positive on $[\bar r,\infty)$. On $[\bar r,\varrho_1]$ we have $0<w\le w(\varrho_1)=\hat u(\varrho_1)$. The computation leading to \eqref{eq:extra-zero-energy}, with $\zeta$ replaced by $\varrho_1$ and $|w|\le\hat u(\varrho_1)$, then gives $|w'(\varrho_1)|\le(C_{\bar r}+\varrho_1)\hat u(\varrho_1)$. Since $F\le 0$ on $(0,\alpha_*)$, it follows that
    \[
        E_\alpha(\varrho_1) \;\le\; \tfrac12u'(\varrho_1,\alpha)^2
        \;\le\; C\varrho_1^2\hat u(\varrho_1)^2
        \;\le\; \eta(\varrho_1):=C\varrho_1^2e^{-(\varrho_1-R)}.
    \]
    Suppose now that $u(\cdot,\alpha)$ has a zero $\zeta'>\bar r$, or is a bound state. In both cases $E_\alpha\ge 0$ on $[\varrho_1,\zeta']$, respectively on $[\varrho_1,\infty)$. A zero requires $u(\cdot,\alpha)$ to exceed $u_\dagger$ first, and a bound state requires $u(\cdot,\alpha)\to 0$, which is impossible while $u(\cdot,\alpha)\ge\hat u+w(\bar r)$. Thus in both cases $u(\cdot,\alpha)$ reaches $u_\dagger$. Since $E_\alpha\ge 0$ excludes a local minimum below $\alpha_*$, it increases from $\varrho_1$ to a first critical point $c'>\varrho_1$ with $u(c',\alpha)>\alpha_*$. After $c'$ it descends to $\zeta'$, respectively to $0$, performing a band crossing $J'$.

    From here the proof of Lemma~\ref{lem:jump} applies verbatim with $\zeta$ replaced by $\varrho_1$. In detail, $F(u(c'))=E_\alpha(c')\le\eta(\varrho_1)$ forces $u(c')<\alpha^*(0)$, and Lemma~\ref{lem:pohozaev-lower} gives $E_\alpha(c')\ge \fc_3\fc_2/(2c')$. On $[\varrho_1,c']$ the solution increases from $2\hat u(\varrho_1)\le\delta$ through the band to $u(c')$, so Lemma~\ref{lem:passage-times}(i)--(iii) gives
    \[
        c'-\varrho_1 \;\le\; T_0\Big(1+\log\frac{2c'}{\fc_3\fc_2}\Big), \qquad\text{hence}\qquad
        c'\le 2\varrho_1.
    \]
    Finally, by Lemma~\ref{lem:band-dissipation}, the band crossing on the ascent costs
    \[
        E_\alpha(c') \;\le\; E_\alpha(\varrho_1)-2\fc_1\fc_2\frac{\gamma-1}{c'}
        \;\le\; \eta(\varrho_1)-\fc_1\fc_2\frac{\gamma-1}{\varrho_1} \;<\; 0
    \]
    for $\varrho_1$ large, contradicting $E_\alpha(c')\ge 0$. Hence for $\alpha<\hat\alpha$ close to $\hat\alpha$, $u(\cdot,\alpha)$ has no zero after $\bar r$ and is not a bound state: $N(\alpha)=k$, and the solution is trapped by Proposition~\ref{prop:classification}.
\end{proof}

\begin{theorem}[Tail criterion]\label{thm:tail-criterion}
    Suppose that the tail sign holds at every bound state:
    \begin{equation}\label{eq:tail-sign}
        M(\infty)=(p-1)\int_0^\infty e^G|u|^{p-1}u\,v\,dr \;>\; 0
        \qquad\text{at every bound state } u,\ v=\partial_\alpha u.
    \end{equation}
    Then conclusions (i)--(iii) of Theorem~\ref{thm:staircase} hold, and at each $\alpha_k$ the variation is dominant, $|v(r)|\to\infty$, with $N_v\ge k+1$ zeros, $N_v\equiv k+1\pmod{2}$. Lemma~\ref{lem:zero-motion} is not used.
\end{theorem}

\begin{proof}
    Let $B$ be the set of bound states. By Proposition~\ref{prop:classification}, every $\alpha$ not in $B$ is trapped, $N$ is constant near such $\alpha$, and nearby solutions are trapped as well. By Lemma~\ref{lem:tail-jump}, every point of $B$ is isolated in $B$. A limit point $\alpha_\infty$ of $B$ cannot be trapped (nearby solutions would be trapped) and cannot be a bound state (it would be isolated). Hence $B$ is closed and discrete in $(0,\infty)$, i.e.\ locally finite. Moreover, $B\subset(\alpha_*,\infty)$, since $N\equiv 0$ and no bound state exists on $(0,\alpha_*]$ (Proposition~\ref{prop:trapping} and the remark after \eqref{eq:energy-model}). As $N(\alpha)\to\infty$ (Lemmas~\ref{lem:scaled-convergence}, \ref{lem:limit-zeros} and the argument at the beginning of Section~\ref{sec:existence-proof}) and $N$ changes only at points of $B$, the set $B$ is infinite: $B=\{\hat\alpha_1<\hat\alpha_2<\ldots\}$, $\hat\alpha_j\to\infty$.

    On $(0,\hat\alpha_1)$ all solutions are trapped and $N$ is locally constant, hence $N\equiv 0$. By Lemma~\ref{lem:tail-jump} at $\hat\alpha_1$, $N(\hat\alpha_1)=0$, so $\hat\alpha_1$ is a ground state and $N=1$ just above it. Next, $N$ is constant on $(\hat\alpha_1,\hat\alpha_2)$, so $N\equiv 1$, $N(\hat\alpha_2)=1$ and $\hat\alpha_2$ is the bound state with one node. Inductively, $N\equiv j$ on $(\hat\alpha_j,\hat\alpha_{j+1}]$ and $\hat\alpha_{j+1}$ is a bound state with exactly $j$ nodes.

    Thus $\alpha_k:=\hat\alpha_{k+1}$ is the unique $k$-node bound state, $N=k$ on $(\alpha_{k-1},\alpha_k]$ with trapped interior solutions, $N$ is nondecreasing with unit jumps exactly at the $\alpha_k$, and $\alpha_k\to\infty$: this is (i)--(iii), and $\alpha_0>\alpha^*(0)$ by Proposition~\ref{prop:pohozaev-positive}(i). The statements on $v$ are Lemmas~\ref{lem:tail-sign} and~\ref{lem:v-zeros}.
\end{proof}

\begin{remark}\label{rem:tail-criterion}
    Theorem~\ref{thm:tail-criterion} reduces what has to be proved. The phase analysis of Sections~\ref{sec:transition}--\ref{sec:renewal-below} is a family of local, phase-by-phase statements, while \eqref{eq:tail-sign} is one global sign per bound state. Its left-hand side is
    \[
        (p-1)\lim_{R\to\infty}\partial_\alpha\int_0^R \frac{e^G|u|^{p+1}}{p+1},
    \]
    where the derivative is taken at fixed $R$ (for $\alpha$ near a bound state the solutions are trapped and the integral over $(0,\infty)$ diverges). By Lemma~\ref{lem:tail-sign}(iii), \eqref{eq:tail-sign} is also a parity statement: \emph{at every $k$-node bound state the variation is dominant at infinity and its number of zeros satisfies $N_v\equiv k+1\pmod{2}$}. Lemma~\ref{lem:v-zeros} gives the lower bound $N_v\ge k+1$. The missing statement is nondegeneracy together with the upper bound $N_v\le k+1$, or merely with the parity, or merely the sign of $M(\infty)$.
    (In spectral terms, $N_v$ is the radial Morse index of $u$, but we do not use this.)
    Sections~\ref{sec:transition}--\ref{sec:renewal-below} prove the sign phase by phase, under (H) and (R) (Theorem~\ref{thm:critical-form}), while for the ground state it holds unconditionally (Theorem~\ref{thm:unconditional}(ii)).
\end{remark}

\subsection{Proofs of the theorems of Section~\ref{sec:uniqueness}}\label{sec:main-proofs}

\begin{proof}[Proof of Theorem~\ref{thm:staircase}]
    Let $u$ be a nodal bound state. By Theorem~\ref{thm:critical-form}(i), $v$ is dominant at infinity with the sign opposite to $u$, i.e.\ $M(\infty)>0$ by Lemma~\ref{lem:tail-sign}. At the ground state the same holds unconditionally: Theorem~\ref{thm:first-phase}(iii) gives the final-phase condition, Proposition~\ref{prop:nondegeneracy} nondegeneracy and Corollary~\ref{cor:divergence} the sign. Thus \eqref{eq:tail-sign} holds at every bound state, and Theorem~\ref{thm:tail-criterion} gives (i)--(iii). For (iv): at $\alpha_k$ with $k\ge 1$, Theorem~\ref{thm:critical-form}(i) gives exactly $k+1$ zeros $\tau_i\in(c_{i-1},r_i)\subset(z_{i-1},z_i)$ of $v$ ($z_0:=0$, $z_{k+1}:=\infty$), $u'v'(\tau_i)>0$ at each of them, and $|v|\to\infty$ by Theorem~\ref{thm:phase}(v). For $k=0$ this is Theorem~\ref{thm:unconditional}(i)--(ii).
\end{proof}

\begin{proof}[Proof of Theorem~\ref{thm:unconditional}]
    Parts (i) and (iii) are Theorem~\ref{thm:first-phase} and Proposition~\ref{prop:pohozaev-positive}.
    Part (iv) is Proposition~\ref{prop:classification} together with Lemma~\ref{lem:tail-jump}, applied at every bound state at which the tail sign holds. At every ground state the tail sign holds by (ii), and under (H) and (R) it holds at every nodal bound state by Theorem~\ref{thm:critical-form}(i) and Lemma~\ref{lem:tail-sign}. For (ii), let $\hat\alpha<\tilde\alpha$ be two ground states. Theorem~\ref{thm:first-phase}(iii) provides the final-phase condition \eqref{eq:final-phase} for $\hat u$ unconditionally, so Proposition~\ref{prop:nondegeneracy} and Corollary~\ref{cor:divergence} apply to $\hat u$: $v$ is dominant with the sign opposite to $\hat u$. The paragraph \emph{The extra zero above $\hat\alpha$} of the proof of Lemma~\ref{lem:jump}, which uses only this dominance and no information on nodal solutions, gives a first zero $z_1(\alpha)$ for every $\alpha$ slightly above $\hat\alpha$. By Theorem~\ref{thm:first-phase}(ii), which is unconditional, $v(z_1)u'(z_1)>0$ for \emph{every} solution with a first zero, so Lemma~\ref{lem:zero-motion}(ii) applies with $k=1$ and gives $N(\alpha)\ge 1$ for all $\alpha>\hat\alpha$. In particular $N(\tilde\alpha)\ge 1$, contradicting the fact that $\tilde u$ is a ground state.
    The last assertion of (ii) is Corollary~\ref{cor:divergence}.
\end{proof}

\begin{proof}[Proof of Theorem~\ref{thm:euclidean}]
    Let $g=(n-1)/r$ with $n\ge 3$. We first check the hypotheses of Theorem~\ref{thm:unconditional}.
    \begin{itemize}
        \item (G1)--(G4) are immediate, with $\gamma=n-1>1$.
        \item (G5) holds since $\int^\infty s^{1-n}\,ds<\infty$ and $g$ is bounded on $[1,\infty)$.
        \item For (G6), note that
              \[
                  q=1+\frac{(n-1)^2}{4r^2}-\frac{n-1}{2r^2}=1+\frac{(n-1)(n-3)}{4r^2},
              \]
              so $g'\to 0$ and $q',q''=O(r^{-3}),O(r^{-4})$ are integrable.
        \item $(\Phi)$ holds with $D\equiv\frac{n-1}{n-2}$ by Lemma~\ref{lem:space-forms}, and $\nu\equiv n$.
    \end{itemize}
    Hence Theorem~\ref{thm:unconditional} applies, including the second part of (i), as $\nu\ge1$. On $\R^n$ the hypotheses (H) and (R) of Theorem~\ref{thm:staircase} hold (Lemma~\ref{lem:H-criterion}(ii), and $\gamma=n-1\ge 2$), so conclusions (i)--(iv) of Theorem~\ref{thm:staircase} hold. Sections~\ref{sec:renewal-above} and~\ref{sec:renewal-below} take the place of the renewal lemmas of~\cite{Tang2026} (Remark~\ref{rem:R-complement}(a)), and nondegeneracy is Proposition~\ref{prop:nondegeneracy} for $k=0$ and Theorem~\ref{thm:phase}(v) for $k\ge1$. With $\varphi\equiv r$, $e^G=r^{n-1}$ and $\nu\equiv n$ the functions of Sections~\ref{sec:pohozaev}--\ref{sec:variation} are those of~\cite{Tang2026}.
\end{proof}

Corollary~\ref{cor:warped} is Theorem~\ref{thm:manifold-existence} combined with Theorem~\ref{thm:staircase}: (G1) and (G2), with $\gamma=n-1$, hold as in the proof of Theorem~\ref{thm:manifold-existence}, (G3) follows from $(\Phi)$ by Lemma~\ref{lem:phi-consequences}(i), and (R) holds (Section~\ref{sec:main-results}). Corollary~\ref{cor:ground-state} is, in the same way, Theorem~\ref{thm:manifold-existence} combined with Theorem~\ref{thm:unconditional}(ii). On $\Hyp^n_\lambda$ the hypotheses hold by the hyperbolic example of Section~\ref{sec:manifolds} for (G1)--(G3), Remark~\ref{rem:extra-hyp} and the proof of Lemma~\ref{lem:space-forms}(ii) for (G4)--(G6), Lemma~\ref{lem:space-forms}(ii) for $(\Phi)$ and Lemma~\ref{lem:H-criterion}(ii) for (H).

%% file: sections/20-restriction-on-f.tex
\section{Why uniqueness requires a restriction on \texorpdfstring{$f$}{f}}\label{sec:restriction-f}

Theorem~\ref{thm:existence} holds for the whole class (f1)--(f4), whereas Theorem~\ref{thm:staircase} is stated only for the model nonlinearity $f(u) = |u|^{p-1}u - u$. We first record which properties of $f$ the proof in Part~II actually uses, and then observe that some restriction of this kind is unavoidable: uniqueness fails within the class (f1)--(f4), already for the ground state and already in Euclidean space.

\subsection{What the proof uses about \texorpdfstring{$f$}{f}}\label{sec:f-conditions}
With a uniqueness theorem for more general nonlinearities in mind, the following table lists, for each part of the proof in Part~II, the structural conditions on an odd $f$ that it uses. Here $f$ is normalized so that its unique positive zero is at $u = 1$, and $\alpha_*$ denotes the unique positive zero of $F$.

\begin{table}[ht]
  \centering
  \begin{tabular}{p{0.275\textwidth}p{0.625\textwidth}}
    \toprule
    Part of the proof                                           & Conditions on $f$ used                                                                                                                                                                                                                                                                                                                                                                                                                       \\
    \midrule
    Basic properties (§\ref{sec:basic-properties})              & $f$ odd and locally Lipschitz, the sign pattern of $f$ across $u = 1$ with $f$ nondecreasing on $[1,\infty)$ (Lemma~\ref{lem:passage-times}(iii)), the positive zero $\alpha_*$ of $F$, $F \ge F(1) > -\infty$ and $f'(0) < 0$.                                                                                                                                                                                                                                                                                                      \\
    Pohozaev estimates (§\ref{sec:pohozaev})                    & the $h_r$ structure: for each $t \in (0,1]$, playing the role of $\varphi'$, the function $h_t(u) := 2(2t+\gamma-1)F(u)-(\gamma-1)uf(u)$ is negative for $0<|u|<\alpha^*(t)$ and positive for $|u|>\alpha^*(t)$, for a threshold $\alpha^*(t)\in(\alpha_*,+\infty]$ nonincreasing in $t$; and $h_1 = uf - 2F > 0$ for $u\ne0$. The strict sign is needed in Lemma~\ref{lem:pohozaev-lower}, which uses the bound $-h_t\ge(\gamma-1)h_1$ on $|u|\le\alpha_*$, from $h_t=4tF-(\gamma-1)h_1$, uniformly in $t$. (For the model $f$: Lemma~\ref{lem:level-structure}.)                                                                                                                                            \\
    First phase (§\ref{sec:variation})                          & in addition: $f'' > 0$ on $(1,\infty)$ (tendency), $uf' - f > 0$ for $u > 0$, $g_1$ increasing with $g_1(1) = 0$ (resistance), and $g_2 \le 0$ on $(0,\alpha_*]$.                                                                                                                                                                                                                                                                            \\
    Transitions (§\ref{sec:later-phases}--\ref{sec:transition}) & in addition: $g_2$ increasing in $|u|$, i.e.\ $h_3 := uf^2 - fF - uf'F$ has the sign of $u$ (indeed $\frac{d}{du}g_2 = 2h_3/h_1^2$), and $g_2(\alpha_*) = 0$ (no power structure is needed for \eqref{eq:T-two-prime}, which for a general $f$ reads $T_2' = -\varphi'\,\frac{dg_2}{du}(u)\,u'(M-\theta) - \varphi'' g_2(u)M$); the sign structure of $H_a = ah_1 - 2F$ (bridges of Sections~\ref{sec:transition}--\ref{sec:renewal-below}). \\
    Renewal (§\ref{sec:renewal-below})                          & in addition, quantitatively: the explicit forms of $F$, $h_1$ and $h_r$ (Lemmas~\ref{lem:P-two-positive}, \ref{lem:Q-nu-positive}, \ref{lem:river-integral} and~\ref{lem:frozen-bridge}), the bound $\alpha_*^2<e$, and the bound $(\alpha_*/u)^{m_\gamma}f(u)/f(\alpha_*)\le C$ for $1<u\le\alpha_*$ (Lemma~\ref{lem:X-growth}).                                                                                                            \\
    Tail (§\ref{sec:tail})                                      & $f'(0) < 0$ and $f'(u) - f'(0) = O(|u|^{\varepsilon})$ at $0$ for some $\varepsilon > 0$ (decay and dichotomy, where for $f'(0) \ne -1$ the functions $m$ and $q$ are defined with $-f'(0)$ in place of $1$), and the sign $f < 0$ on $(0,1)$ (monotonicity of $Q$ on the tail, Proposition~\ref{prop:nondegeneracy}).                                                                                                                       \\
    Main theorems (§\ref{sec:staircase})                        & the basic properties, the Pohozaev estimates (Proposition~\ref{prop:pohozaev-positive} and Lemma~\ref{lem:pohozaev-lower}, used in Lemmas~\ref{lem:jump} and~\ref{lem:tail-jump}) and the tail lemmas of Section~\ref{sec:tail}.                                                                                                                                                                                                             \\
    \bottomrule
  \end{tabular}
\end{table}

In summary, everything up to the Pohozaev estimates survives for odd nonlinearities satisfying the structural conditions listed in the first two rows, and so does the deduction of the main theorems from the tail sign (Theorem~\ref{thm:tail-criterion}), given the tail conditions of the row for Section~\ref{sec:tail}. The Wronskian identities additionally need the monotonicity of $g_1$ and the convexity of $f$ beyond its zero, and the computations with $T_2$ need only the monotonicity of $g_2$. The renewal of Section~\ref{sec:renewal-below} uses the power structure quantitatively. Some conditions of this kind are \emph{necessary}: as the next subsection shows, (f1)--(f4) alone are compatible with nonuniqueness.

\subsection{Nonuniqueness within the class (f1)--(f4)}\label{sec:nonuniqueness}

\begin{proposition}\label{prop:nonuniqueness}
  There exist odd nonlinearities $f$ satisfying \textup{(f1)--(f4)} (for instance the odd extensions of $f(u) = -u + u^p + \lambda u^q$ with $1 < q < 3$, $\lambda$ large, and then $p < 5$ sufficiently close to $5$, depending on $\lambda$) for which \eqref{eq:ode} with $g = 2/r$ possesses at least three distinct ground states. Hence no uniqueness theorem can hold under \textup{(f1)--(f4)} alone, for any class of damping functions $g$ containing the three-dimensional Euclidean case $g = 2/r$.
\end{proposition}

\begin{proof}
  D\'avila, del Pino and Guerra \cite{DavilaDelPinoGuerra2013} prove that for $n = 3$ and $1 < q < 3$, if $\lambda$ is fixed sufficiently large and then $p < 5$ is taken sufficiently close to $5$ (depending on $\lambda$), the equation $\Delta u - u + u^p + \lambda u^q = 0$ in $\R^3$ admits at least three positive radial solutions decaying at infinity. Each of them is a ground state of \eqref{eq:ode} with $g = 2/r$, and they correspond to three distinct values of the shooting parameter $\alpha$.

  The odd extension $f(u) = -u + |u|^{p-1}u + \lambda|u|^{q-1}u$ satisfies (f1)--(f4):
  \begin{itemize}
    \item (f1) is clear, as $q > 1$,
    \item (f2) holds since $f(u)/u \to -1$ as $u \to 0$,
    \item for (f3), the quotient $f(u)/u = -1 + |u|^{p-1} + \lambda|u|^{q-1}$ is strictly increasing in $|u|$, so $f$ has a unique positive zero $u_0$ and $f > 0$ on $(u_0, \infty)$, while $F < 0$ on $(0, \beta)$ for the unique positive zero $\beta > u_0$ of $F$ (and $f > 0$ on $[\beta, \infty)$),
    \item (f4) holds with $k_\pm = 1$.
  \end{itemize}
\end{proof}

\FloatBarrier

%% file: sections/21-singular-tips-uniqueness.tex
\section{Uniqueness on manifolds with singular tips}\label{sec:tip-uniqueness}

Corollaries~\ref{cor:warped} and~\ref{cor:ground-state} are stated for manifolds that are $C^2$ at the pole, where $\gamma=n-1$. As in Section~\ref{sec:tips}, the manifold enters Part~II only through the damping $g_\psi=(n-1)\psi'/\psi$, and the theorems of Section~\ref{sec:uniqueness} allow any real $\gamma>1$. In this section $M$ is the warped product \eqref{eq:tip-metric} over a closed Riemannian manifold $\cN$ of dimension $n-1\ge1$, and we state the uniqueness theorem on $M$. At the tip Part~II needs only the tip index, as Part~I does. The structural hypothesis $(\Phi)$, however, sees more of the geometry than the index, and Section~\ref{sec:excluded-tips} describes what it excludes. Elliptic equations at conical singularities have been studied with other methods: semilinear Dirichlet problems for cone-degenerate operators in~\cite{ChenLiuWei2012}, the Yamabe equation on spaces with conical and edge singularities in~\cite{AkutagawaCarronMazzeo2014}, and the prescribed curvature equation on surfaces with conical points in~\cite{Troyanov1991}.

\subsection{What the uniqueness proof uses at the origin}\label{sec:tip-origin}
Part~II uses the origin only through the following facts:
\begin{itemize}
    \item the limits of Lemma~\ref{lem:phi-basic}, which require (G2) with $\gamma>1$,
    \item the limits $e^G\to0$ and $e^G\varphi\to0$, which give the vanishing at $0^+$ of $P$, $P_2$, $M$, $Q$, $Q_\chi$, $T_1$, $T_2$, $\Sigma$ and $\Omega$ (Proposition~\ref{prop:pohozaev-positive}, Theorems~\ref{thm:first-phase} and~\ref{thm:phase}, Lemmas~\ref{lem:P-two-positive} and~\ref{lem:v-zeros}),
    \item through Part~I, the local theory of Section~\ref{sec:well-posedness}, the value $u''(0)=-f(\alpha)/(\gamma+1)$ and the scaled limit of Lemma~\ref{lem:scaled-convergence}, which require (G1) and (G2) only (proof of Lemma~\ref{lem:origin}).
\end{itemize}
No rate is needed in the limit $rg(r)\to\gamma$, and $\gamma$ need not be an integer.

The Euclidean proof of~\cite{Tang2026} also uses $n\ge3$, through the monotonicity of $r^{n-3}=e^G/\varphi^2$. On a warped product, $(e^G/\varphi^2)'=e^G(\gamma-1-\varphi')/\varphi^3$, which is negative near a tip with $1<\gamma<2$, and Section~\ref{sec:renewal-below} does not use it.

The next lemma identifies the solutions on $M$ to which the shooting method applies.

\begin{lemma}[Bounded solutions near the tip]\label{lem:bounded-tip}
    Let $g$ satisfy \textup{(G1)} and \textup{(G2)} with $\gamma>1$, let $f$ be locally Lipschitz, and let $u\in C^2((0,r_0))$ solve \eqref{eq:ode} on $(0,r_0)$. If $u$ is bounded near $0$, then $e^Gu'\to0$ as $r\to0$, the limit $\alpha\defeq\lim_{r\to0}u(r)$ exists, and $u$ is the solution of \eqref{eq:ode} with $u(0)=\alpha$ and $u'(0)=0$.
\end{lemma}

\begin{proof}
    Since $e^G$ is bounded near $0$ and $u$ is bounded, $(e^Gu')'=-e^Gf(u)$ is bounded near $0$, so $e^Gu'$ has a limit $\ell$ as $r\to0$. By (G2), $e^{-G(s)}\ge cs^{-\gamma'}$ near $0$ for some $\gamma'>1$ (proof of Lemma~\ref{lem:phi-basic}), so $e^{-G}$ is not integrable at $0$. If $\ell\ne0$, then $|u'|\ge\frac{|\ell|}2e^{-G}$ near $0$ and $u$ would be unbounded. Hence $\ell=0$ and $u'(r)=-e^{-G(r)}\int_0^re^Gf(u)$, so that $|u'(r)|\le Cr$, as $G$ is increasing. Therefore $u$ has a limit $\alpha$ at $0$ and $u\in C^1([0,r_0))$ with $u'(0)=0$, and the converse part of \S\ref{sec:local} gives the last claim.
\end{proof}

For a warped product $e^Gu'$ is a constant multiple of the flux $\psi^{n-1}u'$, and Lemma~\ref{lem:bounded-tip} is the statement of Remark~\ref{rem:tip-capacity} that a tip with $\gamma>1$ has zero capacity: no boundary condition is needed at the tip, only boundedness. Boundedness cannot be replaced by the weak formulation \eqref{eq:weak-tip} alone. For the pure power $f(w)=|w|^{p-1}w$ and $g=\gamma/r$, that is, for the scaled limit \eqref{eq:limit-ode}, and for $p>\frac{\gamma+1}{\gamma-1}$, the function $w=Cr^{-\beta}$ with $\beta=\frac2{p-1}$ and $C^{p-1}=\beta(\gamma-1-\beta)$ is a solution whose flux $r^\gamma w'$ tends to $0$, and it satisfies \eqref{eq:weak-tip} for test functions supported near the tip.

\begin{theorem}[Uniqueness with a singular tip]\label{thm:tip-uniqueness}
    Let $n\ge2$, let $\cN$ be a closed Riemannian manifold of dimension $n-1$, and let $M$ be as in \eqref{eq:tip-metric}, where $\psi\in C^3((0,\infty))$ satisfies \textup{($\Psi$1)} and \textup{($\Psi$2)} with tip index $a$ such that $\gamma\defeq(n-1)a>1$. Suppose that $g_\psi$ satisfies \textup{(G5)}, \textup{(G6)} and $(\Phi)$, and let $1<p<\frac{\gamma+3}{\gamma-1}$.
    Consider the radial solutions $u=u(r)$ of
    \begin{equation}\label{eq:tip-model}
        \Delta_{g_M}u-u+|u|^{p-1}u=0\quad\text{on }M,\qquad u(r)\to0\quad(r\to\infty),
    \end{equation}
    that are bounded near the tip.
    \begin{enumerate}[label=(\roman*)]
        \item Exactly one of them is positive.
        \item If $g_\psi$ also satisfies \textup{(H)} and $(\gamma,p)$ satisfies \textup{(R)}, then for every $k\ge0$ exactly one of them, up to the reflection $u\mapsto-u$, has $k$ nodes.
    \end{enumerate}
    These solutions are of the form $u(\cdot,\alpha)$. They are $C^1$ up to the tip with $u'(0)=0$, they are weak solutions across the tip in the sense of \eqref{eq:weak-tip}, and their variation satisfies $|v(r)|\to\infty$ as $r\to\infty$.
\end{theorem}

\begin{proof}
    Since $\psi\in C^3$ and $\psi,\psi'>0$, $g_\psi$ is positive and of class $C^2$ on $(0,\infty)$. With ($\Psi$2) this gives (G1), (G2) and (G4) with $\gamma=(n-1)a$, and (G3) follows from $(\Phi)$ by Lemma~\ref{lem:phi-consequences}(i). Let $u$ be a nontrivial radial solution of \eqref{eq:tip-model} that is bounded near the tip. It solves \eqref{eq:tip-ode}, so by Lemma~\ref{lem:bounded-tip} it is $u(\cdot,\alpha)$ with $\alpha=u(0)\ne0$. Since $f$ is odd, $u(\cdot,-\alpha)=-u(\cdot,\alpha)$, so we may take $\alpha>0$, and $u$ is then a bound state in the sense of Section~\ref{sec:terminology}. Existence for every $k$ is Theorem~\ref{thm:tip-existence}. Uniqueness is Theorem~\ref{thm:unconditional}(ii) in case (i) and Theorem~\ref{thm:staircase}(i) in case (ii). The growth of $|v|$ is Proposition~\ref{prop:nondegeneracy} for $k=0$ and Theorem~\ref{thm:staircase}(iv) for $k\ge1$, and the weak formulation is proved as in Theorem~\ref{thm:tip-existence}.
\end{proof}

Examples of tips to which Theorem~\ref{thm:tip-uniqueness} applies are given in Appendix~\ref{sec:unique-tip-examples}.

\subsection{What is excluded}\label{sec:excluded-tips}
Theorem~\ref{thm:tip-uniqueness} covers fewer tips than Theorem~\ref{thm:tip-existence}. The requirement $\gamma>1$ excludes the necks and wormholes of Example~\ref{ex:necks}, the exterior problem of Example~\ref{ex:exterior}, the sub-power tips of Example~\ref{ex:subpoly-tips} and two-dimensional cones: for these $\gamma\le1$, and the multiplier field $\varphi=(\gamma-1)e^G\cJ$ degenerates. For two-dimensional cones the radial equation is the Euclidean one with $n=2$ (Example~\ref{ex:cones}), so the $k$-node bound state is unique there by~\cite[Theorem~1.1]{CortazarGarciaHuidobroYarur2011}, whose hypothesis $(F/f)'\ge0$ on $(\alpha_*,\infty)$ holds for the model nonlinearity by \eqref{eq:F-over-f}. The other restrictions come from $(\Phi)$, which is a global condition. Its first form, $D>1$ and $D$ nonincreasing, does not involve $\gamma$, and it already determines the behavior of $g$ at the tip.

\begin{proposition}[$(\Phi)$ at the tip]\label{prop:phi-tip}
    Let $g$ satisfy \textup{(G1)} and $\int_1^\infty e^{-G}<\infty$, and let $D$ be nonincreasing on $(0,\infty)$.
    \begin{enumerate}[label=(\roman*)]
        \item If $\limsup_{r\to0}rg(r)>1$, then $\gamma=\lim_{r\to0}rg(r)$ exists, $\gamma>1$, and $D(0^+)=\frac{\gamma}{\gamma-1}$.
        \item If moreover $rg(r)\to m>1$ as $r\to\infty$, then $\gamma\le m$, and $\gamma=m$ only if $g=\gamma/r$ on $(0,\infty)$.
    \end{enumerate}
\end{proposition}

\begin{proof}
    Put $\tilde\varphi\defeq e^G\cJ>0$, so that $\tilde\varphi'=g\tilde\varphi-1=D-1$ and $\tilde\varphi''=D'\le0$. As $\tilde\varphi$ is concave and positive, the limits $\tilde\varphi(0^+)\in[0,\infty)$ and $\kappa\defeq\tilde\varphi'(0^+)\in(-\infty,\infty]$ exist, and $rg=r(\tilde\varphi'+1)/\tilde\varphi$.

    (i) If $\tilde\varphi(0^+)>0$, then $r\tilde\varphi'(r)\to0$, since $\tilde\varphi'$ is bounded near $0$ when $\kappa<\infty$, and $\frac r2\tilde\varphi'(r)\le\tilde\varphi(r)-\tilde\varphi(r/2)\to0$ when $\kappa=\infty$. Hence $rg\to0$. If $\tilde\varphi(0^+)=0$ and $\kappa=\infty$, then $\tilde\varphi(r)=\int_0^r\tilde\varphi'\ge r\tilde\varphi'(r)$, so $rg\le1+1/\tilde\varphi'(r)\to1$. Both cases contradict $\limsup rg>1$, so $\tilde\varphi(0^+)=0$ and $\kappa<\infty$.
    Moreover $\kappa>0$, for otherwise $\tilde\varphi'\le0$ and $\tilde\varphi\le\tilde\varphi(0^+)=0$. Then $\tilde\varphi(r)/r\to\kappa$, so $rg=(\tilde\varphi'+1)/(\tilde\varphi/r)\to1+\frac1\kappa\eqdef\gamma>1$, and $D(0^+)=1+\kappa=\frac\gamma{\gamma-1}$.

    (ii) As $r\to\infty$, $\cJ$ and $re^{-G}$ tend to $0$, and l'H\^opital's rule gives $\tilde\varphi(r)/r=\cJ(r)/(re^{-G(r)})\to\frac1{m-1}$, as in the proof of Lemma~\ref{lem:phi-basic}, so $D=rg\,\tilde\varphi/r\to\frac m{m-1}$. Since $D$ is nonincreasing and $x\mapsto\frac x{x-1}$ is decreasing on $(1,\infty)$, $\gamma\le m$. If $\gamma=m$, then $D$ is constant, so $\varphi'=(\gamma-1)(D-1)\equiv1$, $\varphi=r$ and $g=(\varphi'+\gamma-1)/\varphi=\gamma/r$.
\end{proof}

\begin{remark}[Tips excluded by $(\Phi)$]\label{rem:excluded-tips}
    (a) By Proposition~\ref{prop:phi-tip}(i), no tip of Section~\ref{sec:no-tangent-cone} with $(n-1)a_+>1$ satisfies $(\Phi)$, whatever the warping function away from the tip, and those with $(n-1)a_+\le1$ violate (G4): for the log-periodic tips of Example~\ref{ex:log-periodic}, $rg_\psi$ has no limit while $\limsup rg_\psi=(n-1)(1+\varepsilon)>1$. (The restriction $\limsup rg>1$ in Proposition~\ref{prop:phi-tip}(i) cannot be dropped: there are warping functions with $D>1$ nonincreasing for which $rg_\psi$ oscillates below $1$.) So Theorem~\ref{thm:blowups} has no uniqueness counterpart within the present method.

    (b) By Proposition~\ref{prop:phi-tip}(ii), the effective dimension $\gamma+1$ of the tip is at most that of a conical end, with equality only for exact cones and horns. This excludes the global horns $r^a(1+r)^{1-a}$ of Example~\ref{ex:horns} with $a>1$, the metric of Example~\ref{ex:cones} whose cone angle changes from $1$ to $c\ne1$, and the mildly oscillating metric of Section~\ref{sec:manifolds}. Ends on which $g$ has a positive limit, such as hyperbolic ends, have $D\to1$ and impose no such condition.

    (c) Unlike Theorem~\ref{thm:tip-existence}, the uniqueness theorem sees the curvature at the tip. For the rough tip $\psi(r)=r+r^3\sin(1/r)$ of Example~\ref{ex:rough-tips}, $r(rg_\psi)'=-(n-1)\sin(1/r)+O(r)$. If $(\Phi)$ held, then $\varphi/r\to1$ and $|r(\varphi/r)'|=|\varphi'-\varphi/r|\le1-\varphi'\to0$ by \eqref{eq:phi-bounds}, so $r\varphi''=r\big((rg)\,\varphi/r\big)'=-(n-1)\sin(1/r)+o(1)$, which is positive along a sequence $r\to0$. For the tip $\psi(r)=r|\log r|^b$ of Example~\ref{ex:rough-tips} with $b>0$ one has $e^{G(r)-G(s)}=(r/s)^\gamma\big(\log r/\log s\big)^{\gamma b}$ for $r<s$ near $0$, and with $(1-y)^{-\gamma b}\ge1+\gamma by$ this gives $\varphi(r)/r\ge1+\frac{\gamma b}{(\gamma-1)|\log r|}\big(1+o(1)\big)>1$ for small $r$, against \eqref{eq:phi-bounds}. For $b<0$, tips of this kind do satisfy $(\Phi)$ (Example~\ref{ex:log-pinched}).
\end{remark}

%% file: sections/A-singular-tips.tex
\section{Singular tips: examples and non-unique blow-ups}\label{sec:tip-appendix}

This appendix collects the examples of Sections~\ref{sec:tips} and~\ref{sec:tip-uniqueness} and the proof of Theorem~\ref{thm:blowups}. Throughout, $M$ is the warped product \eqref{eq:tip-metric} over a closed Riemannian manifold $\cN$ of dimension $n-1\ge1$.

\subsection{Examples for Section~\ref{sec:tips}}\label{sec:tip-examples}
In each example $\psi$ is prescribed near the tip and may be completed arbitrarily on $[r_0,\infty)$ subject to ($\Psi$1) and ($\Psi$3). When a global formula is given, ($\Psi$3) is checked explicitly. Curvatures are computed from the warped-product formulas $\sec(\partial_r,X)=-\psi''/\psi$ and $\sec(X,Y)=(\sec_{\cN}(X,Y)-\psi'^2)/\psi^2$ for $X,Y$ tangent to $\cN$~\cite{Petersen2016}.

\begin{example}[Cones over arbitrary links]\label{ex:cones}
    $\psi(r)=cr$ near $0$, $c>0$: $a=1$, $\gamma=n-1$, Euclidean range $p<p_c(n-1)$, i.e.\ $p<\frac{n+2}{n-2}$ for $n\ge3$ and every $p>1$ for $n=2$. Equation~\eqref{eq:tip-ode} is literally the Euclidean radial equation, so for $\cN=\mathbb{S}^{n-1}/\Gamma$ (orbifold points) and for two-dimensional cones of angle $2\pi c$ the solutions are the Euclidean ones. What is new is that $\cN$ is arbitrary. Examples include cones over products $C(\mathbb{S}^{d_1}\times\mathbb{S}^{d_2})$, cones over Einstein links such as the conifold $C(T^{1,1})$ with $n=6$ and $p<2$~\cite{CandelasDeLaOssa1990}, and asymptotically conical spaces whose cone angle changes along the way, such as
    \[
        \psi(r)=r\sqrt{\frac{1+c^2r^2}{1+r^2}},\qquad
        \frac{\psi'}{\psi}=\frac1r+\frac{(c^2-1)\,r}{(1+c^2r^2)(1+r^2)}>0,
    \]
    with cone angle $1$ at the tip and $c$ at infinity. Here $g_\psi\sim(n-1)/r$ at infinity, so $r^2g_\psi\to\infty$ and ($\Psi$3) holds.

    Conical singularities also appear in physics, for instance around topological defects such as global monopoles~\cite{BarriolaVilenkin1989}, and equations of this type are studied on such backgrounds~\cite{Hiscock1996}.
\end{example}

\begin{example}[Horns: fractional effective dimension]\label{ex:horns}
    $\psi(r)=r^a$ near $0$, or globally $\psi(r)=r^a(1+r)^{1-a}$, for which $\psi'/\psi=(a+r)/(r(1+r))>0$ and $g_\psi\sim(n-1)/r$ at infinity, so that ($\Psi$3) holds. Then $\gamma=(n-1)a$ and $N_{\mathrm{eff}}=1+(n-1)a$: the radial Laplacian in non-integer dimension, realized geometrically. For $a>1$ the model is a metric horn~\cite{Cheeger1980}: its radial curvature $-a(a-1)/r^2$ has the conical rate, while its tangential curvature $(\sec_{\cN}-a^2r^{2a-2})/r^{2a}$ blows up like $r^{-2a}$, strictly faster than $r^{-2}$, wherever $\sec_{\cN}\ne0$. This plays no role in the proof.
\end{example}

\begin{example}[Sub-polynomial tips]\label{ex:subpoly-tips}
    $\psi(r)=1/|\log r|$ or $\psi(r)=\exp(-|\log r|^\theta)$ with $0<\theta<1$ near $0$: $\psi(0)=0$, the tip is at finite distance, but $\varphi_\psi(r)=1/|\log r|$, resp.\ $\theta|\log r|^{\theta-1}$, tends to $0$, so $a=0$, $N_{\mathrm{eff}}=1$, and every $p>1$ is admissible. For $\theta>1$ instead $\varphi_\psi\to\infty$, and likewise for $\psi=e^{-1/r}$: this is the case this method cannot handle. The range $p<p_c(\gamma)$ is empty as $\gamma\to\infty$, and indeed the scaled family $v(\cdot,\lambda)$ of Lemma~\ref{lem:scaled-convergence} converges to the constant $1$, and the damping ``freezes'' the solution near the tip. The threshold $\theta=1$ is $\psi=r$.
\end{example}

\begin{example}[Rough tips with the Euclidean range]\label{ex:rough-tips}
    $\psi(r)=r+r^3\sin(1/r)$: $\psi'=1+3r^2\sin(1/r)-r\cos(1/r)\to1$ and $\psi'>0$ near $0$, but $\psi''=6r\sin(1/r)-4\cos(1/r)-r^{-1}\sin(1/r)$, so the radial curvature $-\psi''/\psi=r^{-2}\sin(1/r)+4r^{-1}\cos(1/r)+O(1)$ is unbounded with oscillating sign. Still $a=1$ and the range is the Euclidean one. Similarly $\psi(r)=r|\log r|^b$ has $\varphi_\psi=1+b/\log r\to1$ although $\psi'(0)$ is $+\infty$ for $b>0$ and $0$ for $b<0$. The tip condition sees only the index, not $\psi'(0)$, $\psi''$ or the curvature.
\end{example}

\begin{example}[Necks and wormholes]\label{ex:necks}
    If $\psi(0)=\psi_0>0$ and $r\psi'(r)\to0$ as $r\to0^+$ (for instance if $\psi'$ is bounded near $0$), then $\varphi_\psi\to0$, $a=0$, and every $p>1$ is admissible in any dimension. The ``tip'' is now a copy of $\cN$, the condition $u'(0)=0$ is a Neumann condition, and by even reflection in $r$ one obtains solutions on the two-ended space $(\R\times \cN,\ dr^2+\psi(|r|)^2g_{\cN})$ depending only on the signed distance to the neck $\{r=0\}$ and having exactly $2k$ nodal hypersurfaces $\{r=\pm r_i\}$ parallel to it. When $\psi'(0)\ne0$ the doubled metric is only Lipschitz across the neck and $u$ is a weak solution across the neck in the sense of \eqref{eq:weak-tip}. Since ($\Psi$1) forbids $\psi'<0$, placing the neck at $r=0$ and reflecting is the only way Theorem~\ref{thm:existence} applies to a neck. Three instances, all with $\psi'(0)=0$:
    \begin{itemize}
        \item The spatial slice of the Ellis--Bronnikov wormhole $dr^2+(r^2+b^2)g_{\mathbb{S}^2}$~\cite{Ellis1973,Bronnikov1973}: $g_\psi=2r/(r^2+b^2)$ and $r^2g_\psi\sim2r\to\infty$. The Ellis spacetime $-dt^2+dr^2+(r^2+b^2)g_{\mathbb{S}^2}$ is ultrastatic, so standing waves $e^{i\omega t}u$ of $\Box u-m^2u+|u|^{p-1}u=0$ satisfy exactly \eqref{eq:tip-ode} with $f(u)=-(m^2-\omega^2)u+|u|^{p-1}u$, which satisfies (f1)--(f4) for $\omega^2<m^2$.
              Nonlinear waves on this spacetime are studied in~\cite{BizonKahl2015}, where each topological degree of equivariant wave maps has a unique static solution.
        \item The $t=\mathrm{const}$ slice of Schwarzschild (the Einstein--Rosen bridge~\cite{EinsteinRosen1935}), parametrized by proper distance $r$ from the throat: $\psi'=\sqrt{1-2M/\psi}$, $\psi(0)=2M$, $\psi(r)=2M+r^2/8M+O(r^4)$, $\psi\sim r$ at infinity, and ($\Psi$1), ($\Psi$3) hold. This concerns the Riemannian slice. Standing waves on the Lorentzian Schwarzschild spacetime involve the lapse and are not of the form \eqref{eq:tip-ode} with autonomous $f$.
        \item Fuchsian hyperbolic manifolds $\R\times \cN$ with $dr^2+\cosh^2r\,g_{\cN}$, $\cN$ closed hyperbolic: the curvatures are $-\cosh''/\cosh=-1$ and $(-1-\sinh^2r)/\cosh^2r=-1$, so these are complete manifolds of constant curvature $-1$ with a totally geodesic copy of $\cN$. Here $g_\psi=(n-1)\tanh r$ and $r^2g_\psi\to\infty$. For every $p>1$, including exponents supercritical for the dimension $n$, and every $k$, there is a solution depending only on the signed distance to $\cN$ with $2k$ nodal hypersurfaces.
    \end{itemize}
\end{example}

\begin{example}[Exterior Neumann problem]\label{ex:exterior}
    $\psi(r)=r+R_0$ with $\cN=\mathbb{S}^{n-1}$ is $\R^n\setminus B_{R_0}$ in the variable $r=|x|-R_0$: $g_\psi=(n-1)/(r+R_0)$, $a=0$ and $r^2g_\psi\to\infty$. Hence for every $p>1$ and $k\ge0$ there is a radial solution of $\Delta u+f(u)=0$ in $\R^n\setminus B_{R_0}$ with $\partial_\nu u=0$ on $\partial B_{R_0}$, $u\to0$ at infinity and exactly $k$ nodal spheres. As noted in Remark~\ref{rem:tip-capacity}, the Dirichlet analogue is not covered.
\end{example}

\begin{remark}[On the end]\label{rem:end}
    The examples above single out the tip, while the end is governed by ($\Psi$3) alone. Ends that approach a cylinder fast fail ($\Psi$3), e.g.\ the cigar $\psi=\tanh r$, for which $r^2g_\psi\to0$. Some condition at infinity is needed for the shooting argument. On such ends $g_\psi$ is bounded and integrable at infinity, so the energy $E$ need not tend to $0$. A solution whose energy has a positive limit has infinitely many zeros. Otherwise, after its last zero it would be eventually monotone by Lemma~\ref{lem:positive-critical}, with $u\to\ell$, $u'\to0$ and $F(\ell)>0$, and then $u''=-g_\psi u'-f(u)\to-f(\ell)\ne0$, which is impossible. So $N(\alpha)$ can jump from a finite value to $\infty$, which Lemma~\ref{lem:no-gaps} excludes when ($\Psi$3) holds. We do not know whether bound states with $k\ge2$ nodes exist on such ends. However, an end that approaches a cylinder slowly enough passes, e.g.\ $\psi=2-r^{-1/2}$ near infinity, for which $r^2g_\psi\sim\frac{n-1}{4}r^{1/2}\to\infty$. Cusps ($\psi=e^{-r}$) are excluded by $\psi'>0$, while ends opening only logarithmically, $\psi=(\log r)^b$, pass, since $r^2g_\psi=(n-1)b\,r/\log r\to\infty$.
\end{remark}

\subsection{Non-unique blow-ups}\label{sec:blowups}
Here we prove Theorem~\ref{thm:blowups}. The proof replaces Lemma~\ref{lem:limit-zeros} by an oscillation lemma for a general damping. Let $h:(0,\infty)\to[0,\infty)$ be locally Lipschitz with $h(r)\le\gamma_+/r$ for some $\gamma_+\ge0$, let $H$ be a primitive of $h$, and put
\begin{equation}\label{eq:V-h}
    V_h(r):=e^{-H(r)}\int_0^r e^{H(s)}\,ds,\qquad
    \hrho_h:=\sup_{r>0}\Big(h(r)V_h(r)-\tfrac12\Big).
\end{equation}
Since $H$ is nondecreasing, $V_h$ is well defined, $V_h'=1-hV_h$, and $V_h(r)\le r$. Moreover, $H(r)-H(s)\le\gamma_+\log(r/s)$ for $s<r$, so $V_h(r)\ge r/(\gamma_++1)$. Hence $\hrho_h\le\gamma_+-\frac12$. For $h=\gamma/r$ one has $V_h(r)=r/(\gamma+1)$ and $\hrho_h=(\gamma-1)/(2(\gamma+1))$, so that $(1-\hrho_h)/\hrho_h=p_c(\gamma)$ when $\gamma>1$, while $\hrho_h\le0$ when $\gamma\le1$.

\begin{lemma}[Oscillation for a general damping]\label{lem:oscillation}
    Let $h$ be as above, let $k_\pm>0$, and let $1<p<p_c(\gamma_+)$, with in addition $p<(1-\hrho_h)/\hrho_h$ if $\hrho_h>0$. Then the solution of
    \begin{equation}\label{eq:damped-limit}
        w''+h(r)w'+k_\pm|w|^{p-1}w=0,\qquad w(0)=1,\quad w'(0)=0,
    \end{equation}
    has infinitely many zeros, all of them simple.
\end{lemma}

\begin{proof}
    Existence and uniqueness of $w$ follow from \S\ref{sec:local} as in the proof of Lemma~\ref{lem:origin}, since $h\ge0$. Zeros are simple by uniqueness. With $f_k(w):=k_{\sgn w}|w|^{p-1}w$ and $F_k(w):=k_{\sgn w}|w|^{p+1}/(p+1)$ as in the proof of Lemma~\ref{lem:limit-zeros}, put $\mu:=2\max\{\hrho_h,0\}$ and
    \[
        P:=e^H\Big[V_h\big(w'^2+2F_k(w)\big)+\mu\,ww'\Big].
    \]
    Using \eqref{eq:damped-limit} and $V_h'=1-hV_h$, a direct computation gives
    \[
        P'=e^H\Big[\big(1+\mu-2hV_h\big)w'^2+\Big(\frac{2}{p+1}-\mu\Big)k_{\sgn w}|w|^{p+1}\Big].
    \]
    By the definition of $\hrho_h$, $1+\mu-2hV_h\ge1+2\hrho_h-2hV_h\ge0$ if $\hrho_h\ge0$, and $1-2hV_h\ge-2\hrho_h>0$ if $\hrho_h<0$. Moreover, $\mu<\frac2{p+1}$, which for $\hrho_h>0$ is the condition $p<(1-\hrho_h)/\hrho_h$. Hence $P'\ge0$, with $P'>0$ wherever $w\ne0$. As $r\downarrow0$, $e^HV_h=\int_0^re^H\to0$ and $e^Hww'\to0$, because $e^H$ is bounded near $0$ and $w'(0)=0$. Thus $P(0^+)=0$, and $P(r)>0$ for every $r>0$.

    Suppose that $w$ has finitely many zeros. Steps~2 and~3 of the proof of Lemma~\ref{lem:limit-zeros} hold verbatim with $\gamma/r$ replaced by $h\ge0$ and $r^\gamma$ by $e^H$: normalizing $w>0$ after the last zero and writing $k=k_+$, this means that there is a critical point $R'$ such that $w>0$, $w'<0$ and $(e^Hw')'=-ke^Hw^p$ on $(R',\infty)$. For $r\ge2R'$,
    \[
        e^{-H(r)}\int_{R'}^re^{H(s)}\,ds\ge\int_{R'}^r\Big(\frac sr\Big)^{\gamma_+}ds\ge\frac{r}{2(\gamma_++1)} .
    \]
    Hence, as in Step~4 there, $w$ decreases to $0$ (a positive limit $w_\infty$ would give $w'\le-kw_\infty^p\,r/(2(\gamma_++1))$ for large $r$), and since $w$ is decreasing,
    \[
        w'(r)\le-kw(r)^pe^{-H(r)}\int_{R'}^re^H\le-\frac{k}{2(\gamma_++1)}\,r\,w(r)^p,
    \]
    so that $w(r)\le Cr^{-2/(p-1)}$ for large $r$. As $e^{H(r)}\le e^{H(1)}r^{\gamma_+}$ for $r\ge1$,
    \[
        \int^\infty e^Hw^{p+1}\,dr\le C\int^\infty r^{\gamma_+-\frac{2(p+1)}{p-1}}\,dr<\infty,
    \]
    because $\gamma_+-\frac{2(p+1)}{p-1}<-1$ is equivalent to $p<p_c(\gamma_+)$. Multiplying $(e^Hw')'=-ke^Hw^p$ by $w$ and integrating over $[R',r]$, where $e^Hww'\le0$ at $r$ and $w'(R')=0$, gives
    \[
        \int_{R'}^re^Hw'^2\le k\int_{R'}^re^Hw^{p+1},
    \]
    hence $\int^\infty e^H\big(w'^2+2F_k(w)\big)<\infty$, and therefore
    \[
        \liminf_{r\to\infty}re^{H(r)}\big(w'^2+2F_k(w)\big)(r)=0.
    \]
    But on $(R',\infty)$ we have $\mu ww'\le0$ and $V_h(r)\le r$, so $P(r)\le re^{H(r)}\big(w'^2+2F_k(w)\big)(r)$, and $\liminf_{r\to\infty}P(r)=0$ contradicts $P(r)\ge P(R'+1)>0$ for $r\ge R'+1$.
\end{proof}

For $h=\gamma/r$ with $\gamma\ge1$ one has $\mu=(\gamma-1)/(\gamma+1)$, and $P$ is $1/(\gamma+1)$ times the Pohozaev function of Lemma~\ref{lem:limit-zeros}. Note that $\hrho_h$ is a supremum over the whole real half-line: the proof uses the monotonicity of $P$ from $P(0^+)=0$ on, so a condition on $h$ at infinity alone does not suffice for this argument.

Now let $\psi$ satisfy \textup{($\Psi$1)} and \textup{($\Psi$2$'$)}, with $a_+$ and $\gamma_+$ as in Section~\ref{sec:no-tangent-cone}. Since
\begin{equation}\label{eq:blowup-g}
    g_\lambda(r)=(n-1)\frac{\varphi_\psi(r/\lambda)}{r},\qquad
    g_\lambda'(r)=(n-1)\frac{s\varphi_\psi'(s)-\varphi_\psi(s)}{r^2}\bigg|_{s=r/\lambda},\qquad
    s\varphi_\psi'(s)=\varphi_\psi+s^2\frac{\psi''}{\psi}-\varphi_\psi^2,
\end{equation}
hypothesis ($\Psi$2$'$) says that the blow-ups $\{g_\lambda\}_{\lambda\ge\lambda_0}$ are bounded and equi-Lipschitz on compact subsets of $(0,\infty)$. By Arzel\`a--Ascoli the family is relatively compact in $C^0_{\mathrm{loc}}((0,\infty))$, and we define the \emph{blow-up set}
\begin{equation}\label{eq:blowup-set}
    \mathcal{L}(\psi):=\Big\{h\in C^0((0,\infty)):\ g_{\lambda_j}\to h\text{ in }
    C^0_{\mathrm{loc}}((0,\infty))\text{ for some }\lambda_j\to\infty\Big\}.
\end{equation}
Every $h\in\mathcal{L}(\psi)$ is locally Lipschitz and satisfies $0\le h(r)\le\gamma_+/r$. The function $V_{g_\psi}$ of \eqref{eq:rho-star} is $V_h$ of \eqref{eq:V-h} with $h=g_\psi$ and $H=G$. Moreover $\hrho_*\ge(\gamma_+-1)/(2(\gamma_++1))$: given $\varepsilon>0$, for $s$ small one has $tg_\psi(t)\le\gamma_++\varepsilon$ on $(0,s]$, hence $e^{G(t)-G(s)}\ge(t/s)^{\gamma_++\varepsilon}$ and $V_{g_\psi}(s)\ge s/(\gamma_++\varepsilon+1)$, and along a sequence $s_j\downarrow0$ with $s_jg_\psi(s_j)\to\gamma_+$ this gives $\limsup g_\psi V_{g_\psi}\ge\gamma_+/(\gamma_++\varepsilon+1)$. Under ($\Psi$2), $\mathcal{L}(\psi)=\{\gamma/r\}$ and, since $sg_\psi(s)\to\gamma$ and $V_{g_\psi}(s)=\frac{s}{\gamma+1}(1+o(1))$ as $s\downarrow0$ (proof of Lemma~\ref{lem:origin}), $\hrho_*=(\gamma-1)/(2(\gamma+1))$.

\begin{proof}[Proof of Theorem~\ref{thm:blowups}]
    \emph{The limit problems.} We first check that every $h\in\mathcal{L}(\psi)$ satisfies the hypotheses of Lemma~\ref{lem:oscillation}. It is locally Lipschitz, with $0\le h(r)\le\gamma_+/r$. The family of blow-ups is scale-invariant in the following sense: from $V_{g_\lambda}(t)=\int_0^te^{G(s/\lambda)-G(t/\lambda)}\,ds =\lambda V_{g_\psi}(t/\lambda)$ we get
    \[
        \big(g_\lambda V_{g_\lambda}\big)(t)=\big(g_\psi V_{g_\psi}\big)(t/\lambda).
    \]
    If $g_{\lambda_j}\to h$ in $C^0_{\mathrm{loc}}$, then $V_{g_{\lambda_j}}(t)=\int_0^t\exp\big(-\int_s^tg_{\lambda_j}\big)\,ds\to V_h(t)$ for every $t>0$ by dominated convergence, the integrand being bounded by $1$. Hence
    \[
        (hV_h)(t)=\lim_j(g_\psi V_{g_\psi})(t/\lambda_j)\le\limsup_{s\to0^+}(g_\psi V_{g_\psi})(s),
    \]
    that is $\hrho_h\le\hrho_*$. (Conversely, if $s_j\to0$ realizes the $\limsup$ in \eqref{eq:rho-star}, a limit $h$ of a subsequence of $g_{1/s_j}$ has $(hV_h)(1)=\hrho_*+\frac12$, so $\hrho_*=\max_{h\in\mathcal L(\psi)}\hrho_h$.) Since $(1-\rho)/\rho$ is decreasing in $\rho>0$, \eqref{eq:blowup-p-range} gives $p<(1-\hrho_h)/\hrho_h$ whenever $\hrho_h>0$. Thus, by Lemma~\ref{lem:oscillation}, the solution $w_h$ of \eqref{eq:damped-limit} has infinitely many zeros, all simple.

    \emph{Scaled limit.} Lemma~\ref{lem:scaled-convergence} is modified as follows. Given $\lambda_j\to\infty$, pass to a subsequence along which $g_{\lambda_j}\to h\in\mathcal{L}(\psi)$ in $C^0_{\mathrm{loc}}$. The bounds on $v,v'$ are unchanged, so along a further subsequence $v(\cdot,\lambda_j)\to w$ locally uniformly. In \eqref{eq:scaled-integral} we have
    \[
        e^{G(s/\lambda_j)-G(r/\lambda_j)}=\exp\Big(-\int_s^rg_{\lambda_j}\Big)\to e^{H(s)-H(r)},\qquad 0<s\le r,
    \]
    with the left side bounded by $1$. So dominated convergence identifies $w$ as the solution of the integral form of \eqref{eq:damped-limit} with this $h$, i.e.\ $w=w_h$, and $v(\cdot,\lambda_j)\to w_h$ in $C^1_{\mathrm{loc}}$ as before. The limit now depends on the subsequence, which is harmless for the only use made of Lemma~\ref{lem:scaled-convergence}: if $A_k$ were unbounded, pick $\alpha_j\in A_k$ with $\alpha_j\to\infty$, set $\lambda_j:=\alpha_j^{(p-1)/2}$, and extract as above. Then $w_h$ has at least $k+1$ simple zeros in some $[0,T]$, so $v(\cdot,\lambda_j)$, and hence $u(\cdot,\alpha_j)$, has at least $k+1$ zeros for $j$ large, contradicting $\alpha_j\in A_k$.

    \emph{The origin.} The remainder of the proof of Theorem~\ref{thm:existence} involves the origin only through Proposition~\ref{prop:strict-extrema} at $r_0=0$. Its proof used $u''(0)=-f(\alpha)/(\gamma+1)$, but under ($\Psi$2$'$), $u''(0)$ need not exist. If $f(\alpha)\ne0$, the integral form $u'(r)=-e^{-G(r)}\int_0^re^{G}f(u)$ shows that $u'$ has the sign of $-f(\alpha)$ on some interval $(0,\delta)$. So $0$ is a strict local maximum if $f(\alpha)>0$ and a strict local minimum if $f(\alpha)<0$, and this is all that the proof of Proposition~\ref{prop:strict-extrema} uses. The passage from the ODE to $M$ is as in Theorem~\ref{thm:tip-existence}.
\end{proof}

Explicit tips without a unique tangent cone are not easy to write down. The following simple example has a family of distinct log-self-similar blow-ups parametrized by a circle.

\begin{example}[Log-periodic tips]\label{ex:log-periodic}
    Let $0<\varepsilon<1$ and $\psi(r)=r\,e^{\varepsilon\sin\log r}$ near $0$. Then $\varphi_\psi(r)=1+\varepsilon\cos\log r\in[1-\varepsilon,1+\varepsilon]$, so $\psi'>0$, and $r\varphi_\psi'(r)=-\varepsilon\sin\log r$ is bounded, so by \eqref{eq:blowup-g} ($\Psi$2$'$) holds while ($\Psi$2) fails. The unnormalized pointed blow-ups $\lambda\psi(s/\lambda)=s\,e^{\varepsilon\sin(\log s-\log\lambda)}$ converge, along $\log\lambda\to\theta\pmod{2\pi}$, to $\Psi_\theta(s)=s\,e^{\varepsilon\sin(\log s-\theta)}$: a family, parametrized by a circle, of distinct, non-conical, log-self-similar limit spaces. The tangent cone at the tip is therefore not unique, and none of the limits is a metric cone (compare with the uniqueness results of~\cite{ColdingMinicozzi2014}, whose hypotheses do not hold here). Correspondingly
    \[
        \mathcal{L}(\psi)=\Big\{h_\theta(r)=\frac{\gamma+\varepsilon'\cos(\log r-\theta)}{r}:\
        \theta\in\R/2\pi\mathbb{Z}\Big\},\qquad \gamma=n-1,\ \varepsilon'=(n-1)\varepsilon,
    \]
    with $H_\theta(r)=\gamma\log r+\varepsilon'\sin(\log r-\theta)$. Substituting $s=rt$ in \eqref{eq:V-h} gives the exact identity $h_\theta(r)V_{h_\theta}(r)=\Pi(\log r-\theta)$ with the $2\pi$-periodic function
    \begin{equation}\label{eq:log-periodic-Phi}
        \Pi(\sigma)=\big(\gamma+\varepsilon'\cos\sigma\big)\,e^{-\varepsilon'\sin\sigma}
        \int_0^1t^{\gamma}\,e^{\varepsilon'\sin(\sigma+\log t)}\,dt,
    \end{equation}
    so $\hrho_{h_\theta}=\max\Pi-\frac12$ is independent of $\theta$ and equals $\hrho_*$ (near $0$, $g_\psi=h_0$, so $g_\psi V_{g_\psi}=\Pi(\log s)$ for small $s$). This value reduces to $(\gamma-1)/(2(\gamma+1))$ for $\varepsilon=0$, and Theorem~\ref{thm:blowups} applies with $p<\min\{p_c(\gamma_+),(1-\hrho_*)/\hrho_*\}$, $\gamma_+=(n-1)(1+\varepsilon)$.
\end{example}

\subsection{Examples for Section~\ref{sec:tip-uniqueness}}\label{sec:unique-tip-examples}
In the examples $\cN$ is any closed manifold of dimension $n-1$. When only $g$ is given, the warping function is $\psi=\exp\big(G/(n-1)\big)$ up to a constant factor.

\begin{example}[Cones and hyperbolic cones]\label{ex:hyperbolic-cones}
    For $\psi(r)=cr$ and $\psi(r)=c\lambda^{-1}\sinh(\lambda r)$, with $c,\lambda>0$ and $n\ge3$, the damping $g_\psi$ is $\frac{n-1}{r}$, respectively $(n-1)\lambda\coth(\lambda r)$, that is, the damping of $\R^n$, respectively of $\Hyp^n_\lambda$. All the hypotheses of Theorem~\ref{thm:tip-uniqueness} hold (Theorem~\ref{thm:euclidean}, Lemmas~\ref{lem:space-forms} and~\ref{lem:H-criterion}(ii), and $\gamma=n-1\ge2$), so for every $k\ge0$ and $1<p<\frac{n+2}{n-2}$ the radial $k$-node bound state is unique. This applies to the cones of Example~\ref{ex:cones}: orbifold points $\R^n/\Gamma$, cones over products of spheres and the conifold $C(T^{1,1})$, with $n=6$ and $1<p<2$. It also applies to their hyperbolic analogues, which have a conical singularity at the tip unless $c=1$ and $\cN$ is the round sphere. For the time-dependent cubic Schr\"odinger equation on metric cones over arbitrary links, see~\cite{ZhangZheng2017}.
\end{example}

\begin{example}[Horns: real dimension]\label{ex:real-horns}
    For $\psi(r)=r^a$ with $\gamma=(n-1)a>1$ one has $g_\psi=\gamma/r$, $\varphi(r)=r$, $D\equiv\frac{\gamma}{\gamma-1}$, $\nu\equiv\gamma+1$ and $e^G\varphi'=r^\gamma$, so (G4)--(G6), $(\Phi)$ and (H) hold, as in the proof of Theorem~\ref{thm:euclidean} with $n-1$ replaced by $\gamma$.
    Theorem~\ref{thm:tip-uniqueness} gives uniqueness of the ground state for $1<p<\frac{\gamma+3}{\gamma-1}$, and of the $k$-node bound state for every $k$ when moreover $\gamma\ge\frac85$ or $p\ge\frac2{\gamma-1}$. This is the radial equation in the real dimension $N_{\mathrm{eff}}=\gamma+1$ of Remark~\ref{rem:effective-dim}. Horns with $N_{\mathrm{eff}}<\frac{13}5$ are geometric examples in which (R) restricts $p$. For $1<\gamma<\frac85$, part of the range excluded by (R) is covered by~\cite{CortazarGarciaHuidobroHerreros2019}. Their weight is $e^G=r^\gamma$, for which they verify all their hypotheses on the weight when $\gamma\le3$, with $\lim\varphi'=1$ and $\sup_{r>0}\big(\frac g\varphi\int_0^r\varphi-\frac12\big)=\frac{\gamma-1}2$. By Remark~\ref{rem:literature}(b) and \eqref{eq:F-over-f}, \cite[Theorem~3(ii)]{CortazarGarciaHuidobroHerreros2019} then gives uniqueness of the bound state with one node for $1<p\le\frac{3-\gamma}{\gamma-1}$, and \cite[Theorem~4]{CortazarGarciaHuidobroHerreros2019} gives uniqueness of the $k$-node bound state for every $k$ when $2\le p\le\frac{3-\gamma}{\gamma-1}$, or when $1<p<2$ and $8-(2-p)^2\ge4(\gamma-1)(p+1)$. Neither these results nor Theorem~\ref{thm:tip-uniqueness} cover the nodal bound states for $\frac{3-\gamma}{\gamma-1}<p<\frac2{\gamma-1}$. Dimension two is included: on a surface with $\psi(r)=r^a$ and $a>1$ one has $\gamma=a$, and for instance for $a=3$ the $k$-node bound state is unique for every $k$ and every $1<p<3$.
\end{example}

\begin{example}[Lipschitz poles and exponential ends]\label{ex:lipschitz-poles}
    Let $\gamma>1$, $\mu>0$ and $g(r)=\frac{\gamma}{r}+\mu$, that is $\psi(r)=r^ae^{\mu r/(n-1)}$ with $a=\gamma/(n-1)$. Then $\ell=1/g=r/(\gamma+\mu r)$ is concave, $\ell'\le\frac1\gamma$, and
    \[
        (\gamma+\mu r)^4\big(2\ell\ell''-\ell'(\ell'-2)\big)
        =\gamma\big(2\mu^2r^2+4\mu(\gamma-1)r+\gamma(2\gamma-1)\big)>0,
    \]
    so $(\Phi)$ and (H) hold by Lemmas~\ref{lem:concave-reciprocal}(ii) and~\ref{lem:H-criterion}(i), for every $\gamma>1$, and (G4)--(G6) are immediate. For $a=1$ and $\mu=1$ this is the damping $\frac{n-1}{r}+1$ of Remark~\ref{rem:extra-hyp}, whose metric is only Lipschitz at the pole. The same holds for $\psi(r)=e^r-1$ of Section~\ref{sec:manifolds}, for which $g_\psi=\frac{n-1}{1-e^{-r}}$, $\ell=\frac{1-e^{-r}}{n-1}$ is concave, and $(n-1)^2e^r\big(2\ell\ell''-\ell'(\ell'-2)\big)=2(n-2)+e^{-r}$. Theorem~\ref{thm:tip-uniqueness} applies to both when $n\ge3$, and to $\psi(r)=r^ae^{\mu r/(n-1)}$ in every dimension $n\ge2$ with $(n-1)a>1$, with (R) when $\gamma<\frac85$. In particular, for $1<\gamma<2$ the exponential growth of the end does not prevent (H). Compare Example~\ref{ex:hyperbolic-horns}.
\end{example}

\begin{example}[Thin tips]\label{ex:thin-tips}
    Let $n\ge3$, $\frac1{n-1}<a<1$ and $\psi(r)=r^a(1+r)^{1-a}$, the global horn of Example~\ref{ex:horns} with $a<1$. Then $g_\psi=(n-1)\frac{a+r}{r(1+r)}$, $\gamma=(n-1)a\in(1,n-1)$ and $rg_\psi\to n-1$ at infinity. Here $(n-1)\ell=\frac{r(1+r)}{r+a}$, so $\ell''=-\frac{2a(1-a)}{(n-1)(r+a)^3}<0$, and $(n-1)^2(r+a)^4\big(2\ell\ell''-\ell'(\ell'-2)\big)$ is a polynomial in $r$ with the positive coefficients $\frac{2\gamma}a-1$, $8\gamma-4a$, $10a\gamma-6a+2\gamma$, $4a(a\gamma+\gamma-1)$ and $a^2(2\gamma-1)$. Hence $(\Phi)$ and (H) hold, and Theorem~\ref{thm:tip-uniqueness} applies, with (R) when $\gamma<\frac85$. The tip has effective dimension $\gamma+1<n$, and the admissible range $p<\frac{\gamma+3}{\gamma-1}$ is wider than the range $p<\frac{n+2}{n-2}$ of the natural dimension.
\end{example}

\begin{example}[Logarithmically pinched tips]\label{ex:log-pinched}
    Let $\gamma>1$, $c>0$ and
    \[
        g(r)=\frac1r\Big(\gamma+\frac{c}{L(r)}\Big),\qquad L(r)\defeq\log\big(e+\tfrac1r\big).
    \]
    Then $rg\to\gamma$ as $r\to0$ and $rg\to\gamma+c$ as $r\to\infty$, but $g-\gamma/r$ behaves like $c/(r\log\frac1r)$ near $0$ and is not integrable at $0$. The warping function satisfies $\psi(r)\sim C\,r^a|\log r|^{-c/(n-1)}$ as $r\to0$, with $a=\gamma/(n-1)$, so that $\psi(r)/r^a\to0$: the link is pinched by a logarithmic factor. (G4)--(G6) hold, and $\ell=1/g$ satisfies
    \[
        \ell''=-\frac{c\,(\gamma L+2\gamma+c)}{r(1+er)^2(\gamma L+c)^3}<0.
    \]
    With $x=er$, so that $L-1=\log(1+\frac1x)$, a direct computation shows that $(1+x)^2(\gamma L+c)^4\big(2\ell\ell''-\ell'(\ell'-2)\big)$ equals $2c\big((1+x)\log(1+\frac1x)-1\big)$ plus a polynomial in $x$, $L-1$, $\gamma-1$ and $c$ with positive coefficients. Since $\log(1+\frac1x)\ge\frac1{1+x}$, it is positive, so $(\Phi)$ and (H) hold by Lemmas~\ref{lem:concave-reciprocal}(ii) and~\ref{lem:H-criterion}(i), and Theorem~\ref{thm:tip-uniqueness} applies, with (R) when $\gamma<\frac85$. Such tips are covered because (G4) asks only for the limit $\gamma$, and not for the integrability of $g-\gamma/r$ near $0$.
\end{example}

\begin{example}[Hyperbolic horns]\label{ex:hyperbolic-horns}
    For $\psi(r)=\big(\lambda^{-1}\sinh(\lambda r)\big)^a$ with $\gamma=(n-1)a>1$ one has $g_\psi=\gamma\lambda\coth(\lambda r)$, and (G4)--(G6) hold as for $\Hyp^n_\lambda$. The proof of Lemma~\ref{lem:space-forms}(ii) uses $n-1$ only as a real number larger than $1$, so $(\Phi)$ holds for every $\gamma>1$, and Theorem~\ref{thm:tip-uniqueness}(i) gives a unique positive radial solution. The computation in the proof of Lemma~\ref{lem:H-criterion}(ii) gives (H) when $\gamma\ge2$, and then Theorem~\ref{thm:tip-uniqueness}(ii) applies. For $1<\gamma<2$, (H) fails: the identity $D-1=ge^G\int_r^\infty\ell'e^{-G}$ of the proof of Lemma~\ref{lem:concave-reciprocal} gives $e^G\varphi'=(\gamma-1)e^G(D-1)\sim C e^{(\gamma-2)\lambda r}\to0$ as $r\to\infty$, while $e^G\varphi'\to0$ also as $r\to0$. Uniqueness of the nodal bound states is then open for these tips. By Example~\ref{ex:lipschitz-poles}, the obstruction comes from this particular damping and not from the exponential growth of the end.
\end{example}

%% file: sections/B-identity-table.tex
\section{Table of identities}\label{sec:identity-table}
\allowdisplaybreaks

All the identities of Lemma~\ref{lem:pohozaev-identities}, Proposition~\ref{prop:identity-table}, Proposition~\ref{prop:scaling-field}, Lemma~\ref{lem:weighted-identities} and Lemma~\ref{lem:renewal-identities}, as well as the Pohozaev identity of Lemma~\ref{lem:limit-zeros}, follow by direct computation, reducing derivatives with the three rules
\[
    u'' = -gu' - f(u), \qquad v'' = -gv' - f'(u)v, \qquad \varphi' = g\varphi + (1-\gamma),
\]
together with $G' = g$, for \emph{unspecified} sufficiently smooth functions $g, \varphi, u, v$ ($C^2$, and $C^3$ for \eqref{eq:id-Sigma}, which also needs $f\in C^2$ along $u$) and an unspecified potential $F$ with $f = F'$ (general case), and additionally for $f(u) = u^p - u$ on the branch $u > 0$ (power case). The functions are those of Sections~\ref{sec:pohozaev}--\ref{sec:renewal-below}. Recall that $h_1 = uf - 2F$, $h_3 = uf^2 - fF - uf'F$, $\theta = e^G vh_1(u)/u'$, $\nu = 2\varphi' + \gamma - 1$ and $A = \varphi e^G$.

\noindent\textbf{General $f$, general $g$.}
\begin{align}
     & E' = -gu'^2, \tag{I1}\label{eq:id-energy}                                                            \\
     & \big(e^{2G}E\big)' = 2ge^{2G}F(u), \tag{I2}\label{eq:id-weighted-energy}                                      \\
     & M' = e^G\big(uf'(u) - f(u)\big)v, \tag{I3}\label{eq:id-M}                                       \\
     & Q' = 2\varphi' e^G f(u)v, \tag{I4}\label{eq:id-Q}                                               \\
     & P' = e^G\big[2\nu F(u) - (\gamma-1)uf(u)\big], \tag{I5}\label{eq:id-P}                          \\
     & \omega' = P_1/(e^G u^2), \tag{I6}\label{eq:id-omega}                                                \\
     & \Sigma' = e^G f''(u)u'^2 v, \tag{I7}\label{eq:id-Sigma}                                             \\
     & Q - P\tfrac{v}{u} = \omega(M - \theta), \tag{I8}\label{eq:id-connection}                                 \\
     & T_1' = -M\big(\varphi' g_1(u)\big)', \tag{I9}\label{eq:id-T-one}                                    \\
     & T_2' = 2\varphi' e^G v\tfrac{h_3}{h_1} - M\big(\varphi' g_2(u)\big)', \tag{I10}\label{eq:id-T-two} \\
     & W_a' = \varphi' e^G v\big[2f - a(uf' - f)\big] - a\varphi'' M, \tag{I11}\label{eq:id-W}        \\
     & Q_a' = e^G\big[(2\varphi' - a)fv + au'v'\big], \tag{I12}\label{eq:id-Qa}                        \\
     & Q_\nu' = e^G\big[(1-\gamma)fv + \nu u'v' + 2\varphi'' u'v\big]. \tag{I13}\label{eq:id-Qnu}
\end{align}
\begin{equation}\label{eq:id-P-over-A}\tag{I14}
    \Big(\frac{P}{A}\Big)' = -\frac{\nu}{\varphi}\cdot\frac{P_2}{A}, \qquad P_2 = A\Big[u'^2 + \frac{\gamma-1}{\nu}uf(u)\Big] + (\gamma-1)e^G uu',
\end{equation}
\begin{gather}
    B_a' = (\varphi' - 1)e^G v\big[2f(u) - a\big(uf'(u) - f(u)\big)\big] - a\varphi'' M - 2F_a(u)\frac{Q_\nu}{\varphi u'^2}, \tag{I15}\label{eq:id-bridge}\\
    L(\varphi u' + \lambda u) = -(2\varphi' + \lambda)f(u) + \lambda f'(u)u + 2\lambda' u' + (\lambda'' + g\lambda')u. \tag{I16}\label{eq:id-L}
\end{gather}

\noindent\textbf{Power case $f = |u|^{p-1}u - u$ (on $u > 0$).}
\begin{align}
     & h_1 = \tfrac{p-1}{p+1}u^{p+1}, \tag{S1}\label{eq:id-h-one}                                                    \\
     & h_3 = \tfrac{p-1}{2}uh_1, \tag{S2}\label{eq:id-h-three}                                                         \\
     & 2f - a(uf' - f) = 2u\kappa_a(u), \tag{S3}\label{eq:id-kappa}                                                  \\
     & g_2 - g_1 = -u^{1-p}, \tag{S4}\label{eq:id-g-difference}                                                             \\
     & T_2 = T_1 + \varphi' u^{1-p}M, \tag{S5}\label{eq:id-T-difference}                                                    \\
     & \tfrac{d}{du}g_2 = (p+1)u^{-p}, \quad \tfrac{d}{du}g_1 = 2u^{-p}, \tag{S6}\label{eq:id-g-prime}                                         \\
     & h_r(u) = \frac{4\varphi' - (p-1)(\gamma-1)}{p+1}u^{p+1} - 2\varphi' u^2, \tag{S7}\label{eq:id-h-r}          \\
     & W_a' = 2\varphi' e^G uv\,\kappa_a(u) - a\varphi'' M, \tag{S8}\label{eq:id-power-W}                              \\
     & T_2' = -\varphi'(p+1)\frac{uu'}{u^{p+1}}\big(M - \theta\big) - \varphi'' g_2(u)M, \tag{S9}\label{eq:id-power-T-two} \\
     & (p-1)e^G uv = (p+1)\frac{uu'}{u^{p+1}}\,\theta. \tag{S10}\label{eq:id-theta}
\end{align}
The anchored bridge used in the proof of Theorem~\ref{thm:first-phase}(iv) is the special case of the same computation with the \emph{constant} multiplier $\varphi_1'=\varphi'(z_1)$ in place of $\varphi'$:
\begin{equation}\label{eq:id-anchored-bridge}\tag{I15$'$}
    \Big[Q - 2\varphi_1' F(u)\frac{e^G v}{u'}\Big]' = 2(\varphi' - \varphi_1')e^G f(u)v - 2\varphi_1' F(u)\frac{Q_\nu}{\varphi u'^2} \qquad (\varphi_1' \in \R\text{ constant}).
\end{equation}
The Pohozaev identity of Lemma~\ref{lem:limit-zeros}, with $g = \gamma/r$, $\varphi = r$, $\nu = \gamma + 1$, is the special case of \eqref{eq:id-P}:
\begin{multline}\tag{I5$'$}\label{eq:id-pohozaev-limit}
    \Big[r^{\gamma+1}\big(w'^2 + 2F(w)\big) + (\gamma-1)r^\gamma ww'\Big]' = r^\gamma\big[2(\gamma+1)F(w) - (\gamma-1)wf(w)\big]\\
    \text{for } w'' + \tfrac{\gamma}{r}w' + f(w) = 0.
\end{multline}

\noindent\textbf{The scaling field (Section~\ref{sec:later-phases}).} For the pure-power Euclidean equation $w'' + \frac{\gamma}{r}w' + |w|^{p-1}w = 0$ with $V = \frac{2}{p-1}w + rw'$, $e^G = r^\gamma$, $\varphi = r$, $f(w) = |w|^{p-1}w$:
\begin{equation}
    V'' + \tfrac{\gamma}{r}V' + p|w|^{p-1}V = 0, \qquad M[V] = P_1, \qquad Q[V] = \tfrac{2}{p-1}P_1, \qquad Q[V] - \tfrac{2}{p-1}M[V] = 0. \tag{N}\label{eq:id-scaling}
\end{equation}

\noindent\textbf{The critical-point form (Section~\ref{sec:transition}).} With $Q_\chi = Q + \chi(r)e^G u'v$ for a $C^1$ function $\chi$, $\sB_a = Q - a\varphi' M - 2\varphi' F_a e^G v/u'$, $H_a = -2F_a$, $c_H = \varphi'\nu/\varphi + \varphi''$, in the power case (the second line of \eqref{eq:id-weighted-bridge} and \eqref{eq:id-renewal-bridge} use $|\varphi''|=-\varphi''$, i.e.\ $\varphi''\le0$):
\begin{align}
     & Q_\chi' = e^G\big[\chi\,u'v' + (2\varphi' - \chi)f(u)v + \chi'u'v\big], \tag{K1}\label{eq:id-Q-chi}                                               \\
     & (v/u')' = Q_{\varphi'}/(e^G\varphi u'^2), \tag{K2}\label{eq:id-v-ratio}                                                                             \\
     & (e^G v/u')' = Q_\nu/(\varphi u'^2), \tag{K3}\label{eq:id-weighted-v-ratio}                                                                                   \\
     & \big(Q/(e^G u')\big)' = f(u)Q_{2\varphi'}/(e^G u'^2), \tag{K4}\label{eq:id-Q-ratio}                                                                 \\
     & (M - \theta)' = -h_1(u)Q_\nu/(\varphi u'^2), \tag{K5}\label{eq:id-M-theta}                                                                          \\
     & \sB_a' = -2\varphi' F_a Q_\nu/(\varphi u'^2) - \varphi''\big(aM + 2F_a e^G v/u'\big) \tag{K6}\label{eq:id-weighted-bridge}                                  \\
     & \phantom{\sB_a'} = \varphi' H_a Q/(\varphi u'^2) + a|\varphi''|M + c_H H_a e^G v/u', \notag                                                 \\
     & T_2' = -(p+1)\varphi'\tfrac{uu'}{|u|^{p+1}}(M - \theta) - \varphi'' g_2(u)M, \tag{K7}\label{eq:id-crit-T-two}                                          \\
     & W_a' = 2\varphi' e^G uv\,\kappa_a(u) - a\varphi'' M, \qquad 1 - \tfrac{p-1}{2}g_2(\beta) = \tfrac{p+1}{2}\beta^{1-p}, \tag{K8}\label{eq:id-crit-W} \\
     & W_a|_{v=0} = M(\omega - a\varphi'), \qquad T_2|_{v=0} = M(\omega - \varphi' g_2(u)), \tag{K9}\label{eq:id-v-zero}                                  \\
     & \sB_a = T_2 \text{ where } g_2(u) = a, \notag                                                                                               \\
     & (e^G\varphi')' = e^G(g\varphi' + \varphi''), \tag{K10}\label{eq:id-H-weight}                                                                        \\
     & Q = \tfrac{v}{u}P_1 + \omega M. \tag{K11}\label{eq:id-Q-split}
\end{align}

\noindent\textbf{The renewal (Section~\ref{sec:renewal-below}).} With $A = \varphi e^G$, $k_r = \partial_u h_r$, $X = -e^G v/u'$, $\mathfrak{V}_\Lambda = Q - \Lambda(M - \theta) - 2\varphi' F(u)e^G v/u'$ for a constant $\Lambda$, and the other notation as above, in the power case:
\begin{align}
     & P_2 = P - \tfrac{A}{\nu}h_r(u), \qquad k_r(u) = u\big[(4\varphi' - (p-1)(\gamma-1))u^{p-1} - 4\varphi'\big], \tag{V1}\label{eq:id-P-two-form}                                                  \\
     & P_2' = -\tfrac{A}{\nu}u'k_r(u) - \tfrac{2(\gamma-1)A\varphi''}{\nu^2}uf(u), \tag{V2}\label{eq:id-P-two-prime}                                                                                   \\
     & \big(e^G\varphi' v/u'\big)' = (e^G\varphi')'v/u' + \varphi' Q_{\varphi'}/(\varphi u'^2), \tag{V3}\label{eq:id-H-ratio}                                                                      \\
     & \big[e^{2G}\big(E - F(\mu)\big)\big]' = 2ge^{2G}\big(F(u) - F(\mu)\big) \quad (\mu > 0 \text{ constant}), \tag{V4}\label{eq:id-level-energy}                                                     \\
     & 2g\varphi = \gamma - 1 + \nu, \qquad X'/X = 2g + v'/v + f(u)/u', \tag{V5}\label{eq:id-X-log}                                                                                              \\
     & \mathfrak{V}_\Lambda' = \big(\Lambda h_1 - 2\varphi' F\big)\frac{Q_\nu}{\varphi u'^2} - 2\varphi'' F\frac{e^G v}{u'} \tag{V6}\label{eq:id-frozen-bridge}                                          \\
     & \phantom{\mathfrak{V}_\Lambda'} = \big(\Lambda h_1 - 2\varphi' F\big)\frac{Q_{\varphi'}}{\varphi u'^2} + \big(g\Lambda h_1 - 2(g\varphi' + \varphi'')F\big)\frac{e^G v}{u'}, \notag \\
     & \big(u^2 e^G v/u'\big)' = 2ue^G v + u^2 Q_\nu/(\varphi u'^2), \tag{V7}\label{eq:id-river-parts}                                                                                                 \\
     & \big(-2F e^G v/u'\big)' = -2fe^G v - 2F Q_\nu/(\varphi u'^2), \notag                                                                                                                \\
     & u^2/(\varphi u'^2) = \varphi/\omega^2, \qquad -2F = u^2\big(1 - \alpha_*^{1-p}u^{p-1}\big), \qquad \tfrac{2}{p-1}\,h_1(\alpha_*) = \alpha_*^2, \tag{V8}\label{eq:id-river-algebra}                \\
     & \sB_a' = \varphi' H_a\frac{Q_{\varphi'}}{\varphi u'^2} + a|\varphi''|M + (g\varphi' + \varphi'')H_a\frac{e^G v}{u'}, \tag{V9}\label{eq:id-renewal-bridge}                                          \\
     & Q_\nu = Q_{2\varphi'} + (\gamma-1)e^G u'v. \notag
\end{align}

%% file: sections/C-notation.tex
\section{Notation for Part II}\label{sec:notation}

The following table collects the symbols of Part~II, with the place where each is defined.

\begin{center}
    \small
    \renewcommand{\arraystretch}{1.12}
    \begin{tabular}{>{\raggedright\arraybackslash}p{0.2\textwidth}>{\raggedright\arraybackslash}p{0.5\textwidth}>{\raggedright\arraybackslash}p{0.2\textwidth}}
        \toprule
        Symbol                                 & Meaning                                                                                                      & Defined in                                                 \\
        \midrule
        \multicolumn{3}{l}{\emph{The damping}}                                                                                                                                                                             \\
        $G$, $\gamma$                          & primitive of $g$ with $G(1)=0$, $\gamma=\lim_{s\to0}sg(s)$                                                   & \S\ref{sec:setting}                                              \\
        $\cJ$                                  & $\cJ(r)=\int_r^\infty e^{-G}$                                                                                & Def.~\ref{def:multiplier}                                        \\
        $\varphi$, $D$, $\nu$                  & multiplier $\varphi=(\gamma-1)e^G\cJ$, $D=g\varphi/(\gamma-1)$, effective dimension $\nu=2\varphi'+\gamma-1$ & Def.~\ref{def:multiplier}                                        \\
        $\ell$, $\Upsilon$                     & $\ell=1/g$, $\Upsilon=2g/(g'+2g^2)$                                                                          & Lem.~\ref{lem:concave-reciprocal}, \ref{lem:H-criterion}                        \\
        $q$, $m$                               & $q=1+\frac{g^2}4+\frac{g'}2$, $m=\frac12\big(g+\sqrt{g^2+4}\big)$                                            & (G6), \S\ref{sec:tail}                                       \\
        \multicolumn{3}{l}{\emph{The nonlinearity}}                                                                                                                                                                        \\
        $f$, $F$                               & $f(u)=|u|^{p-1}u-u$, $F'=f$, $F(0)=0$                                                                        & \eqref{eq:model-f}                                            \\
        $\alpha_*$, $\alpha^*(r)$              & positive zero of $F$, moving Pohozaev level                                                                  & \eqref{eq:levels}, Def.~\ref{def:pohozaev-level}                       \\
        $h_1$, $h_r$                           & $h_1=uf-2F$, $h_r=2\nu F-(\gamma-1)uf$                                                                       & \eqref{eq:pohozaev}, \eqref{eq:pohozaev-identity}                           \\
        $g_1$, $g_2$, $\kappa_a$, $F_a$, $H_a$ & auxiliary functions in the Wronskian identities, with $F_a=F-\frac a2h_1$, $H_a=-2F_a$                            & \S\ref{sec:identities}, \eqref{eq:weighted-bridge}                          \\
        \multicolumn{3}{l}{\emph{Solutions and their points}}                                                                                                                                                              \\
        $u=u(\cdot,\alpha)$, $v$               & solution with $u(0)=\alpha$, variation $v=\partial_\alpha u$                                                 & \eqref{eq:ivp}, \S\ref{sec:terminology}                           \\
        $E$, $N(\alpha)$                       & energy, number of zeros of $u(\cdot,\alpha)$                                                                 & Def.~\ref{def:energy}, \S\ref{sec:no-gaps}                       \\
        $z_i$, $c_i$                           & zeros and critical points of $u$ ($c_0=0$)                                                                   & Prop.~\ref{prop:critical-points}                                      \\
        $b_i,r_i$ / $\bar r_i,\bar b_i$         & points of the descent / ascent of phase $i$ with $|u|=\alpha_*,1$ / $1,\alpha_*$                             & \S\ref{sec:later-phases}                                             \\
        $K$                                    & index of the last phase with $E>0$ on $(0,c_K]$                                                              & \S\ref{sec:later-phases}                                             \\
        $\beta_i$, $t_i$                       & $\beta_i=|u(c_i)|$, point of the descent with $|u|=\beta_i$                                                  & \eqref{eq:beta-levels}, Thm.~\ref{thm:renewal-above}                       \\
        $\tau_i$, $u_\tau$, $\zeta_i$          & zero of $v$ in phase $i$, $u_\tau=|u(\tau_i)|$, zero of $M-\theta$                                           & Thm.~\ref{thm:first-phase}(i), \ref{thm:phase}, \S\ref{sec:renewal-above}, Lem.~\ref{lem:M-theta-zero} \\
        $r_\mu$, $\bar r_\mu$                  & points of the descent / ascent with $|u|=\mu$                                                                & \S\ref{sec:renewal-setting}                                           \\
        $\fc_1,\fc_2,\fc_3$, $T_0$             & constants of the dissipation estimates                                                                       & Lem.~\ref{lem:passage-times}--\ref{lem:band-dissipation}, \ref{lem:pohozaev-lower}        \\
        \multicolumn{3}{l}{\emph{Auxiliary functions}}                                                                                                                                                                     \\
        $P$, $P_1$, $P_2$, $\omega$            & Pohozaev functions, $\omega=-\varphi u'/u$                                                                   & \eqref{eq:pohozaev}, Rem.~\ref{rem:P-two}, \eqref{eq:omega}      \\
        $M$, $Q$, $\Sigma$, $\theta$           & $M=e^G(u'v-uv')$, $Q=e^G(wv'-w'v)$ with $w=\varphi u'$, $\Sigma=e^G(u''v'-u'v'')$, $\theta=e^Gvh_1(u)/u'$                                 & \S\ref{sec:identities}                                           \\
        $Q_a$, $Q_\nu$, $Q_\chi$               & $Q+a\,e^Gu'v$, with $a$ constant, $a=\nu$, or $a=\chi(r)$                                                    & \S\ref{sec:identities}, \eqref{eq:Q-chi}                          \\
        $T_1$, $T_2$, $W_a$                    & $Q-\varphi'g_1(u)M$, $Q-\varphi'g_2(u)M$, $Q-a\varphi'M$                                                     & \S\ref{sec:identities}                                           \\
        $B_a$, $\sB_a$                         & bridges of Proposition~\ref{prop:identity-table} and their $\varphi'$-weighted form                                    & \S\ref{sec:identities}, \eqref{eq:weighted-bridge}                          \\
        $\varpi$, $\varpi_2$, $X$              & $Q/(e^G|u'|)$, $Q_{2\varphi'}/(e^G|u'|^3)$, $-e^Gv/u'$                                                      & Lem.~\ref{lem:level-comparison}, \ref{lem:renewal-identities}(v)                        \\
        $\mathfrak V_\Lambda$                  & bridge with frozen weight                                                                                    & Lem.~\ref{lem:renewal-identities}(vi)                                    \\
        $V_\pm$, $D_\infty$                    & solutions of the linearized equation at infinity, $u=D_\infty\tilde V_-$                                     & Lem.~\ref{lem:dichotomy}, \ref{lem:tail-sign}                        \\
        \bottomrule
    \end{tabular}
\end{center}